\documentclass[11pt, a4paper,reqno]{amsart}
\pdfoutput=1
\usepackage{tikz,enumitem,rotating,latexsym,bm,stmaryrd,caption,}
\usetikzlibrary{positioning,intersections,decorations,shadings}
\usetikzlibrary{shapes}
\usetikzlibrary{arrows}
\usetikzlibrary{patterns}
\usepackage{tkz-euclide}
\usetikzlibrary{calc}
\usetikzlibrary{shapes.geometric}
\tikzset{snake it/.style={decorate, decoration=snake}}
\tikzset{zigzag/.style={decorate, decoration=zigzag}}

\definecolor{ao(english)}{rgb}{0.0, 0.5, 0.0}

\def\dover#1{\underline{\underline{#1}}}
\newcommand{\uvee}{{\underline{v} }}
\newcommand{\vx}{{\color{violet} \underline{x} }}

\pgfdeclareverticalshading{rainbow}{100bp}
{color(0bp)=(orange);color(25bp)=(orange);
color(50bp)=(magenta); color(75bp)=(cyan);
 color(100bp)=(cyan)}
\usetikzlibrary{decorations.pathmorphing}

	\definecolor{eng}{rgb}{0.0, 0.5, 0.0}
\definecolor{apple}{rgb}{0.55, 0.71, 0.0}
\definecolor{cadmium}{rgb}{0.0, 0.42, 0.24}
\definecolor{darkspringgreen}{rgb}{0.09, 0.45, 0.27}
\definecolor{amethyst}{rgb}{0.6, 0.4, 0.8}
\definecolor{ao}{rgb}{0.0, 0.0, 1.0}
\definecolor{atomictangerine}{rgb}{1.0, 0.6, 0.4}
\definecolor{carmine}{rgb}{0.59, 0.0, 0.09}

 \newcommand{\blor}{
 \begin{tikzpicture}
\draw[orange,->](0,0) to(0.2,0);
\draw[blue,->](0,0) to(-0.2,0);
  \end{tikzpicture}}
 \newcommand{\orbl}{
 \begin{tikzpicture}
\draw[blue,->](0,0) to(0.2,0);
\draw[orange,->](0,0) to(-0.2,0);
  \end{tikzpicture}}
\newcommand{\blbl}{
 \begin{tikzpicture}
\draw[blue,->](0,0) to(0.2,0);
\draw[blue,->](0,0) to(-0.2,0);
  \end{tikzpicture}}
\newcommand{\oror}{
 \begin{tikzpicture}
\draw[orange,->](0,0) to(0.2,0);
\draw[orange,->](0,0) to(-0.2,0);
  \end{tikzpicture}}

\definecolor{toggle}{rgb}{1.0, 0.94, 0.96}

\usepackage[normalem]{ulem}
 \newcommand{\plus}{ }
 \newcommand{\pDelta}{\Delta}
 \newcommand{\dom}{{\sf dom}}
  \newcommand{\cod}{{\sf cod}}
  \newcommand{\on}{{\sf 1}}

 \newcommand{\gsigma}{{{{\color{gray}\bm\sigma}}}}

 \newcommand{\brown}{{{{\color{brown}\bm\pi}}}}
 \newcommand{\violet}{{{{\color{violet}\bm\rho}}}}
  \newcommand{\purple}{{{{\color{violet}\bm\rho}}}}
 \newcommand{\green}{{{{\color{green!80!black}\bm\beta}}}}
 \newcommand{\orange}{{{{\color{orange}\bm\gamma}}}}
 \newcommand{\pink}{{{{\color{magenta}\bm\alpha}}}}
 \newcommand{\blue}{{{{\color{cyan}\bm\tau}}}}
 
 \newcommand{\colMAP}{\imath}
 \newcommand{\MAP}{\jmath }
 \newcommand{\TRNC}{\varphi}
 \newcommand{\defect}{{\sf def}}
 \newcommand{\grey}{{{{\color{gray}\bm\alpha}}}}

 \newcommand{\two}{{{{\color{orange}2}}}}
 \newcommand{\zero}{{{{\color{magenta}0}}}}
 \newcommand{\one}{{{{\color{cyan}1}}}}
 \newcommand{\trid}{{\sf trid}}

 \newcommand{\cpi}{{{{\color{violet}\bm\pi}}}}
 \newcommand{\csigma}{{{{\color{magenta}\bm\sigma}}}}
 \newcommand{\ctau}{{{{\color{cyan}\bm\tau}}}}
\newcommand{\csigmaw}{{{{\color{magenta}\bm\sigma}}}}
\newcommand{\ctauw}{{{{\color{cyan}\bm\tau}}}}

 \newcommand{\al}{{{{\color{magenta}\bm \sigma}}}}
 \newcommand{\crho}{{{{ \color{green!80!black}\bm\rho}}}}
   
\newcommand{\exx}{{b_\al }}
\newcommand{\deltax}{{{\color{magenta}\delta_\alpha}}}
 
\newcommand{\eps}{ \varepsilon}

\newcommand{\SSTX}{{\sf X}}
\newcommand{\SSTY}{{\sf Y}}

\newcommand{\gap}{{\sf gap}}
\newcommand{\fork}{{\sf fork}}
\newcommand{\braid}{{\sf braid}}
\newcommand{\spot}{{\sf spot}}

\newcommand{\isit}{{i}}
\newcommand{\isitone}{\al(i+1)}%+1}}
\newcommand{\Shl}{\widehat{\mathfrak{S}}_{h\ell}}
\newcommand{\w}{{\underline{w}}}
\newcommand{\vvv}{{\underline{v} }}
\newcommand{\uuu}{{\underline{u} }}
\newcommand{\y}{{\underline{y}}}
\newcommand{\x}{{\underline{x}}}

\newcommand{\dil}{{\sf dil}_\ctau }

\newcommand{\unvvv}{{{z}}}
\newcommand{\grade}{q}
\newcommand{\dgrm}{\mathcal{D}}
\newcommand{\dgrmdeg}{\mathcal{D}^{\rm deg}}
\newcommand{\dgrmf}{\mathcal{D}_{\rm f}}
\newcommand{\dgrmBS}{\mathcal{D}_{\rm BS}}
\newcommand{\dgrmBSdeg}{\mathcal{D}_{\rm BS}^{\rm deg}}
\newcommand{\dgrmBSdegsum}{\mathcal{D}_{\rm BS}^{{\rm deg},\oplus}}
\newcommand{\dgrmBSF}{\mathcal{D}_{\rm BS}^F}
\newcommand{\dgrmF}{\mathcal{D}^F}
\newcommand{\dgrmBSsumshift}{\mathcal{D}_{\rm BS}^{\oplus,(-)}}
\newcommand{\dgrmstd}{\mathcal{D}_{\rm std}}
\newcommand{\dgrmBSstd}{\mathcal{D}_{{\rm BS},{\rm std}}}
\newcommand{\dgrmBSpast}{\mathcal{D}_{{\rm BS},p|\ast}}
\newcommand{\dgrmBSpastsumshift}{\mathcal{D}_{{\rm BS},p|\ast}^{\oplus,\langle - \rangle}}
\newcommand{\dgrmpast}{\mathcal{D}_{p|\ast}}
\newcommand{\eee}{{\sf e}}

\newcommand{\Wf}{W_{\rm f}}
\newcommand{\W}{W}
\newcommand{\Wp}{\W_p}
\newcommand{\Ssf}{S_{{\rm f}}}
\newcommand{\Ss}{S}
\newcommand{\Ssp}{\Ss_p}
\newcommand{\Sspexpr}{\expr{\Ss}_p}
\newcommand{\Sspone}{\Ss_{p \cup 1}}
\newcommand{\sh}{s_{\rm h}}
\newcommand{\linkexpr}{\expr{\Ss}_{p|1}}
\newcommand{\Wpcosets}{\prescript{p}{}{\W}}
\newcommand{\Alc}{\text{\bf Alc}}
 \newcommand{\Pdiptwo}{{\sf M}_{i,i+2}}
 \newcommand{\Pdipj}{{{\sf M}_{i,j}}} 

 \newcommand{\vstwo}{{{{\color{violet}\bm \sigma_2}}}}
 \newcommand{\vsi}{{{{\color{violet}\bm \sigma_i}}}}
  \newcommand{\vsione}{{{{\color{violet}\bm \sigma_{i+1}}}}}
  \newcommand{\vsone}{{{{\color{violet}\bm \sigma_1}}}}
    \newcommand{\vsk}{{{{\color{violet!70!white}\bm \sigma_k}}}}
      \newcommand{\vsj}{{{{\color{violet}\bm \sigma_j}}}}

 \newcommand{\nodelabel}{0}

\def\down{\vee}
\def\up{\wedge}

\usepackage{standalone}

\usepackage{etex}
 
\tikzset{
  variable line width/.style={
    every variable line width/.append style={#1},
    to path={%
      \pgfextra{%
        \draw[every variable line width/.try,line width=\pgfkeysvalueof{/tikz/thickness}] (\tikztostart) -- (\tikztotarget);
      }%
      (\tikztotarget)
    },
  },
  thickness/.initial=0.6pt,
  every variable line width/.style={line cap=round, line join=round},
}

\usepackage{todonotes}
\newcommand{\warning}[1]{\todo[color=red!75]{#1}}

\usepackage{tikz}
\usetikzlibrary{matrix,  intersections, calc, decorations.pathreplacing} %, fit, hobby}

\usepackage{tikz-cd}

\newlength{\superthick}
\newlength{\cornerradius}
\tikzstyle{corner}=[rounded corners=\cornerradius]
\tikzstyle{dot}=[circle, inner sep=0pt, minimum size=4.8pt]
\tikzstyle{string}=[line width=\superthick]
\tikzstyle{std}=[string,dash pattern=on 0.9pt off 0.9pt]
\definecolor{realcyan}{rgb}{0,1,1}

 \usepackage{amsmath,amsthm,amsfonts,amssymb,mathrsfs,pb-diagram}
\usepackage[%pagebackref=true, 
bookmarks=true,colorlinks=true,linktoc=page,citecolor=green!80!black,linkcolor=green!80!black,urlcolor=green!80!black]{hyperref}
\usepackage{caption}
\usepackage{lipsum,wasysym}
\usepackage{mathtools}
\usepackage[a4paper,margin=2.2cm]{geometry}
\usepackage{cleveref}
 \usepackage{amsmath}
\mathchardef\mhyphen="2D
\usepackage{color}
\usepackage{xcolor}
\usepackage{ifthen}
\usepackage{sidecap}   
\definecolor{mediumblue}{rgb}{0.0, 0.0, 0.8}

\newcommand{\Res}{{\rm Res}}
\newcommand{\Rem}{{\rm Rem}}
\newcommand{\Add}{{\rm Add}}
\newcommand{\Ind}{{\rm Ind}}
\newcommand{\hstar}{\mathfrak{h}^*}
\newcommand{\mptn}{{\mathscr{P}_{(W,P)}}}
\newcommand{\mtau}{{\mathscr{P}^\ctau_{(W,P)}}}
\newcommand{\mptnmax}{{\mathscr{P}_{(W,P)}^\circ}}
\newcommand{\mptnl}{{\mathscr{P}^\ell_{(W,P)}}}
\newcommand{\mptntau}{{\mathscr{P}_{(W,P)^\ctau}}}

\renewcommand{\geq}{\geqslant}
\renewcommand{\leq}{\leqslant}
 \newcommand{\Q}{{\mathbb Q}}

\tikzset{wei/.style= 
{red,double=red,double
distance=0.5pt}}

\newcommand{\fA}{\mathfrak{A}}
\newcommand{\fB}{\mathfrak{B}}
\newcommand{\fC}{\mathfrak{C}}
\newcommand{\fD}{\mathfrak{D}}
\newcommand{\fE}{\mathfrak{E}}
\newcommand{\fF}{\mathfrak{F}}
\newcommand{\fG}{\mathfrak{G}}
\newcommand{\fH}{\mathfrak{H}}
\newcommand{\fI}{\mathfrak{I}}
\newcommand{\fJ}{\mathfrak{J}}
\newcommand{\fK}{\mathfrak{K}}
\newcommand{\fL}{\mathfrak{L}}
\newcommand{\fM}{\mathfrak{M}}
\newcommand{\fN}{\mathfrak{N}}
\newcommand{\fO}{\mathfrak{O}}
\newcommand{\fP}{\mathfrak{P}}
\newcommand{\fQ}{\mathfrak{Q}}
\newcommand{\fR}{\mathfrak{R}}
\newcommand{\fS}{\mathfrak{S}}
\newcommand{\fT}{\mathfrak{T}}
\newcommand{\fU}{\mathfrak{U}}
\newcommand{\fV}{\mathfrak{V}}
\newcommand{\fW}{\mathfrak{W}}
\newcommand{\fX}{\mathfrak{X}}
\newcommand{\fY}{\mathfrak{Y}}
\newcommand{\fZ}{\mathfrak{Z}}
\newcommand{\fa}{\mathfrak{a}}
\newcommand{\fb}{\mathfrak{b}}
\newcommand{\fc}{\mathfrak{c}}
\newcommand{\fd}{\mathfrak{d}}
\newcommand{\fe}{\mathfrak{e}}
\newcommand{\ff}{\mathfrak{f}}
\newcommand{\ffg}{\mathfrak{g}}
\newcommand{\fh}{\mathfrak{h}}
\newcommand{\ffi}{\mathfrak{i}}
\newcommand{\fj}{\mathfrak{j}}
\newcommand{\fk}{\mathfrak{k}}
\newcommand{\fl}{\mathfrak{l}}
\newcommand{\fm}{\mathfrak{m}}
\newcommand{\fn}{\mathfrak{n}}
\newcommand{\fo}{\mathfrak{o}}
\newcommand{\fp}{\mathfrak{p}}
\newcommand{\fq}{\mathfrak{q}}
\newcommand{\fr}{\mathfrak{r}}
\newcommand{\fs}{\mathfrak{s}}
\newcommand{\ft}{\mathfrak{t}}
\newcommand{\fu}{\mathfrak{u}}
\newcommand{\fv}{\mathfrak{v}}
\newcommand{\fw}{\mathfrak{w}}
\newcommand{\fx}{\mathfrak{x}}
\newcommand{\fy}{\mathfrak{y}}
\newcommand{\fz}{\mathfrak{z}}

\newcommand{\sA}{\mathscr{A}}
\newcommand{\sB}{\mathscr{B}}
\newcommand{\sC}{\mathscr{C}}
\newcommand{\sD}{\mathscr{D}}
\newcommand{\sE}{\mathscr{E}}
\newcommand{\sF}{\mathscr{F}}
\newcommand{\sG}{\mathscr{G}}
\newcommand{\sH}{\mathscr{H}}
\newcommand{\sI}{\mathscr{I}}
\newcommand{\sJ}{\mathscr{J}}
\newcommand{\sK}{\mathscr{K}}
\newcommand{\sL}{\mathscr{L}}
\newcommand{\sM}{\mathscr{M}}
\newcommand{\sN}{\mathscr{N}}
\newcommand{\sO}{\mathscr{O}}
\newcommand{\sP}{\mathscr{P}}
\newcommand{\sQ}{\mathscr{Q}}
\newcommand{\sR}{\mathscr{R}}
\newcommand{\sS}{\mathscr{S}}
\newcommand{\sT}{\mathscr{T}}
\newcommand{\sU}{\mathscr{U}}
\newcommand{\sV}{\mathscr{V}}
\newcommand{\sW}{\mathscr{W}}
\newcommand{\sX}{\mathscr{X}}
\newcommand{\sY}{\mathscr{Y}}
\newcommand{\sZ}{\mathscr{Z}}

\newcommand{\hd}{\operatorname{hd}}
\newcommand{\cA}{ \mathscr{A}}
\newcommand{\cB}{\mathcal{B}}
\newcommand{\cC}{\mathcal{C}}
\newcommand{\cD}{\mathcal{D}}
\newcommand{\cE}{{\rm M}}
\newcommand{\cF}{\mathcal{F}}
\newcommand{\cG}{\mathcal{G}}
\newcommand{\cH}{\mathcal{H}}
\newcommand{\cI}{\mathcal{I}}
\newcommand{\cJ}{\mathcal{J}}
\newcommand{\cK}{\mathcal{K}}
\newcommand{\cL}{\mathcal{L}}
\newcommand{\cLR}{\mathcal{LR}}
\newcommand{\cM}{\mathcal{M}}
\newcommand{\cN}{\mathcal{N}}
\newcommand{\cO}{\mathcal{O}}
\newcommand{\cP}{\mathcal{P}}
\newcommand{\cQ}{\mathcal{Q}}
\newcommand{\cR}{\mathcal{R}}
\newcommand{\cS}{\mathcal{S}}
\newcommand{\cT}{\mathcal{T}}
\newcommand{\cU}{\mathcal{U}}
\newcommand{\cV}{\mathcal{V}}
\newcommand{\cW}{\mathcal{W}}
\newcommand{\cX}{\mathcal{X}}
\newcommand{\cY}{\mathcal{Y}}
\newcommand{\cZ}{\mathcal{Z}}

\newcommand{\Z}{\mathbb{Z}}
\newcommand{\R}{\mathbb{R}}
\newcommand{\N}{\mathbb{N}}
\newcommand{\C}{\mathbb{C}}

\DeclareMathOperator{\reg}{reg}
\DeclareMathOperator{\image}{im}

 \newcommand{\alphar}{{{\color{magenta}\boldsymbol \alpha}}}
 \newcommand{\bet}{{{\color{cyan}\boldsymbol \tau}}}
  \newcommand{\gam}{{{{\color{orange!95!black}\boldsymbol\gamma}}}}

 \newcommand{\betar}{{{\color{green!70!black}\boldsymbol \beta}}}

\tikzset{wei2/.style={red,double=red,double
distance=0.5pt}}

\allowdisplaybreaks
\numberwithin{equation}{section}
\usepackage{scalefnt}

\newtheorem{thm}{Theorem}[section]
\newtheorem{cor}[thm]{Corollary}

\newtheorem{lem}[thm]{Lemma}

\newtheorem{prop}[thm]{Proposition}

\newtheorem*{prop*}{Proposition}
\newtheorem*{thmA}{Theorem A}

\newtheorem*{thmB}{Theorem B}
\newtheorem*{thmC}{Theorem C}\newtheorem*{thm*}{Theorem D}
\newtheorem*{cor*}{Corollary}

\newtheorem*{conj*}{Conjecture A}

\newtheorem*{conj1*}{Conjecture B}
\newtheorem*{Acknowledgements*}{Acknowledgements}

\theoremstyle{rmk}

\theoremstyle{defn}
\newtheorem{rmk}[thm]{Remark}
\newtheorem{defn}[thm]{Definition}
\newtheorem{eg}[thm]{Example}

\newcommand{\great}{>}
\newcommand{\less}{<}
\newcommand{\greatoreq}{\geq}
\newcommand{\lessoreq}{\leq}
\newcommand{\codeg}{\mathrm{codeg}}
\newcommand{\triv}{\mathrm{triv}}
\newcommand{\id}{\mathrm{id}}
\newcommand{\TL}{\mathrm{TL}}
\newcommand{\rad}{\mathrm{rad}}
\newcommand{\res}{\mathrm{res}}
\newcommand{\ik}{{k}}
\newcommand{\Std}{{\rm Std}}
\newcommand{\SStd}{{\rm Path}}
\renewcommand{\det}{{\rm det}}
\newcommand{\epsilonLIRONdontchange}{\epsilon}
\newcommand{\TSStd}{\operatorname{\mathcal{T}}}
\newcommand{\Shape}{\operatorname{Shape}} 
\newcommand{\Path}{{\rm Path}}
\newcommand{\CPath}{{\rm Path}_{\underline{w}}}
\newcommand{\la}{\lambda}
\newcommand{\I}{i}
\newcommand{\J}{j}
\newcommand{\K}{k}
\newcommand{\M}{m}
\renewcommand{\L}{l}
\def\Ca{\mathcal C}
\newcommand{\Lead}{\operatorname{Lead}}

\newcommand{\mcompose}{ {\; \color{magenta}\circledast \; } }
\newcommand{\ocompose}{ {\; \color{orange}\circledast \; } }
\newcommand{\gcompose}{ {\; \color{green!80!black}\circledast \; } }
\newcommand{\greycompose}{ {\; \color{gray}\circledast \; } }
\newcommand{\compose}{ {\; \color{cyan}\circledast \; } }
\newcommand{\pcompose}{ {\; \color{violet}\circledast \; } }

\newcommand{\SSTS}{\mathsf{S}}
\newcommand{\tSSTT}{\overline{\mathsf{T}}} 
\newcommand{\tla}{\overline{\x}}
\newcommand{\tmu}{\overline{\y }}
\newcommand{\SSTT}{\mathsf{T}}  %semistandard index
\newcommand{\SSTP}{\mathsf{P}}  %semistandard index
\newcommand{\SSTU}{\mathsf{U}}  %semistandard index 
\newcommand{\SSTV}{\mathsf{V}}  %semistandard index 
\newcommand{\SSTQ}{\mathsf{Q}}  %semistandard index
\newcommand{\SSTR}{\mathsf{R}}  %semistandard index
\newcommand{\sts}{\mathsf{s}}  %standard index
\newcommand{\stt}{\mathsf{t}}  %standard index
\newcommand{\stu}{\mathsf{u}}  %standard index
\newcommand{\stv}{\mathsf{v}}  %standard index
\newcommand{\ZZ}{{\mathbb Z}}
\newcommand{\NN}{{\mathbb N}}
\newcommand{\g}{\ell}
 
\newcommand{\CC}{{\mathbb{C}}}
\newcommand{\RR}{{\mathbb R}}
\newcommand{\Hyp}{{\mathbb E}_{\alpha,me}}
\newcommand{\length}{{t}}
\DeclareMathOperator\noedge{\:\rlap{\hspace*{0.25em}/}\text{---}\:}
\DeclareMathOperator{\Hom}{Hom}
\def\Mod{\textbf{-Mod}}
\let\<=\langle
\let\>=\rangle

\newcommand\Dec[1][A]{\mathbf{D}_{#1}(t)}
\newcommand\Cart[1][A]{\mathbf{C}_{#1}(t)}

\newcommand{\north}{top }
\newcommand{\northT}{{\sf T}}
\newcommand{\south}{bottom } 
\newcommand{\southT}{{\sf B}}

\newcommand\mydots{\makebox[1em][c]{\color{cyan}.\hfil\color{magenta}.\hfil\color{cyan}.\hfil\color{magenta}.}}
\tikzset{
ultra thin/.style= {line width=0.05pt},
very thin/.style=  {line width=0.2pt},
thin/.style=       {line width=0.1pt},
semithick/.style=  {line width=0.6pt},
thick/.style=      {line width=0.8pt},
very thick/.style= {line width=1.2pt},
ultra thick/.style={line width=1.6pt}
}

\crefname{ques}{Question}{Questions}
\crefname{defn}{Definition}{Definitions}
\crefname{thm}{Theorem}{Theorems}
\crefname{prop}{Proposition}{Propositions}
\crefname{lem}{Lemma}{Lemmas}
\crefname{cor}{Corollary}{Corollaries}
\crefname{conj}{Conjecture}{Conjectures}
\crefname{section}{Section}{Sections}
\crefname{subsection}{Subsection}{Subsections}
\crefname{eg}{Example}{Examples}
\crefname{figure}{Figure}{Figures}
\crefname{rem}{Remark}{Remarks}
\crefname{rmk}{Remark}{Remarks}
\crefname{equation}{equation}{equation}

\Crefname{ques}{Question}{Questions}
\Crefname{defn}{Definition}{Definitions}
\Crefname{thm}{Theorem}{Theorems}
\Crefname{prop}{Proposition}{Propositions}
\Crefname{lem}{Lemma}{Lemmas}
\Crefname{cor}{Corollary}{Corollaries}
\Crefname{conj}{Conjecture}{Conjectures}
\Crefname{section}{Section}{Sections}
\Crefname{subsection}{Subsection}{Subsections}
\Crefname{eg}{Example}{Examples}
\Crefname{figure}{Figure}{Figures}
\Crefname{rem}{Remark}{Remarks}
\Crefname{rmk}{Remark}{Remarks}

  \newcommand{\Mull}{{\rm M}}
    
\usepackage[hang,flushmargin]{footmisc}
 
\newcommand\Dim[2][t]{\text{\rm Dim}_{#1}#2}

\newcommand\REMOVETHESE[2]{{{{\mathsf{M}}_{#1}^{#2}	}}}
\newcommand\ADDTHIS[2]{{{{\mathsf{P}}_{#1}^{#2}}}}

 \newcommand{\Spotzero}{\SSTS_{0,\al}}
  \newcommand{\Spotone}{\SSTS_{1,\al}}
   \newcommand{\Spottwo}{\SSTS_{2,\al}}
    \newcommand{\Spotthree}{\SSTS_{3,\al}}
 \newcommand{\Spotq}{\SSTS_{q,\al}}
  \newcommand{\Spotb}{\SSTS_{\exx,\al}}
  \newcommand{\Spotqplus}{\SSTS_{q+1,\al}}

 \newcommand{\Forkq}{{\sf F}_{q, \al}}
  \newcommand{\Forkqplus}{{\sf F}_{q+1, \al}}
 \newcommand{\Forkone}{{\sf F}_{1, \al}}
 \newcommand{\Forkzero}{{\sf F}_{0, \al}} 
  \newcommand{\Forktwo}{{\sf F}_{2, \al}}
    \newcommand{\Forkthree}{{\sf F}_{3, \al}}
  \newcommand{\Forkb}{{\sf F}_{\exx, \al}}

 \newcommand{\TForkq}{{\sf F}_{q,\al }}
  \newcommand{\TForkqplus}{{\sf F}_{q+1,\al}}
 \newcommand{\TForkone}{{\sf F}_{1,\al }}
 \newcommand{\TForkzero}{{\sf F}_{0,\al }} 
  \newcommand{\TForktwo}{{\sf F}_{2,\al }}

\renewcommand{\labelitemi}{$\circ $}

\def\Item{\item\abovedisplayskip=0pt\abovedisplayshortskip=5pt~\vspace*{-\baselineskip}} 

\begin{document}
  
% \title  [The homogeneous  representations of anti-spherical Hecke categories]
% {The homogeneous  representations  \\  of anti-spherical Hecke categories}

 \title [The anti-spherical  Hecke categories    for Hermitian symmetric pairs   ]
 { The anti-spherical  Hecke categories
 %,
% and \\  oriented Temperley--Lieb algebras
  %\\
 \\   for Hermitian symmetric pairs   
  }
 \author{Chris Bowman}
       \address{Department of Mathematics, 
University of York, Heslington, York,  UK}
\email{Chris.Bowman-Scargill@york.ac.uk}
  
 \author{Maud De Visscher}
	\address{Department of Mathematics, City, University of London,   London, UK}
\email{Maud.DeVisscher@city.ac.uk}

		\author{Amit  Hazi}
%		\address{Department of Mathematics, City, University of London,   London, UK}
\address{School of Mathematics, University of Leeds, Leeds, UK}
\email{A.Hazi@leeds.ac.uk}

%    \author{Emily Norton}
%\address{FB Mathematik, TU Kaiserslautern, Postfach 3049,
%        67653 Kaisers\-lautern, Germany.}
%\email{Norton@mathematik.uni-kl.de}
\author{Emily Norton}
\address{
School of Mathematics, Statistics and Actuarial Science, University of Kent, Canterbury, UK  }
 \email{E.Norton@kent.ac.uk}

 \maketitle 
 
   \begin{abstract}
 We calculate the  $p$-Kazhdan--Lusztig polynomials  for  %generalised
   Hermitian symmetric pairs and prove that the corresponding  
  anti-spherical  Hecke categories 
    categories are standard  Koszul.     
    We   prove that the combinatorial invariance conjecture 
    can be lifted to the level of   graded Morita equivalences between subquotients of these Hecke categories.  
 \end{abstract}

 \section*{Introduction}
Anti-spherical Hecke categories  first rose to mathematical celebrity  
as the centrepiece of the  proof of the Kazhdan--Lusztig  positivity conjecture 
and its anti-spherical counterpart \cite{MR3245013,MR4437613}.  
Understanding the  $p$-Kazhdan--Lusztig polynomials of these   categories   subsumes the problem of determining prime divisors of Fibonacci numbers \cite{w13};  this is  a notoriously difficult problem in number theory, for which a combinatorial solution is highly unlikely.  
As $p\to \infty$ the situation simplifies and we encounter the  classical Kazhdan--Lusztig polynomials; 
these are important combinatorial objects which can be calculated via a recursive algorithm.    
%On the other hand,  classical Kazhdan--Lusztig polynomials 
%are important combinatorial objects which can be calculated via a recursive algorithm.    
We seek to understand this gulf between  the 
combinatorial and non-combinatorial realms within  $p$-Kazhdan--Lusztig theory.  
%The representation theory of  
%Hecke categories is encoded in 
%the $p$-Kazhdan--Lusztig polynomials ${^p}n_{\la,\mu}$.  
%For $p=\infty$, these polynomials are the classical Kazhdan--Lusztig polynomials, which are central objects in algebraic geometry and combinatorics. 
% For $p$ a prime, we have a step change in difficulty: 
%understanding  general 
%${^p}n_{\la,\mu}$ incorporates    the calculation of prime divisors of Fibonacci numbers as a subproblem --- this is  a notoriously difficult problem in number theory for which a combinatorial understanding  is extremely unlikely.  
% We wish to understand this boundary   separating
%the  combinatorial and number theoretic realms within $p$-Kazhdan--Lusztig theory.  

Over fields of infinite characteristic,     the  families  
of  anti-spherical  Kazhdan--Lusztig polynomials which are best understood combinatorially are those for  Hermitian symmetric pairs, $P\leq W$.  
These polynomials  admit inexplicably simple     
        combinatorial formulae in terms of  Dyck paths or  Temperley--Lieb diagrams          \cite{MR957071,MR2363142,MR2465813,MR2918294,MR3563723}.  Their importance derives from their universality: these polynomials 
  control  the structure  of parabolic Verma modules for Lie algebras \cite{MR888703}; 
 algebraic supergroups \cite{MR1937204};  Khovanov arc algebras 
  \cite{MR2881300,MR3644792,BarWang,MR4818649,
 stropandus,BDDHMS,BDDHMS2};    
   Brauer and walled Brauer algebras \cite{MR2955190,MR3544626,MR3286499,MR2813567};  
     categories $\mathcal{ O}$ for Grassmannians \cite{MR646823,MR2781018};    
   topological and algebraic Springer fibres,   Slodowy slices,        and   $W$-algebras \cite{MR3563723}.  

\smallskip
\noindent {\bf Koszulity and $p$-Kazhdan--Lusztig theory.  } The first main result  of 
this paper extends  our understanding of the Kazhdan--Lusztig theory for Hermitian symmetric pairs to all fields. 
 
\begin{thmA}
   Let  $\Bbbk$ be a field of characteristic $p\geq 0$ and $(W,P)$   a  
  Hermitian symmetric pair.   The Hecke category,  
 $\mathcal{H}_{(W,P)}$, is standard  Koszul \color{black!99}  (in the sense of \cite[Introduction]{MR1960515}) and the %associated 
 $p$-Kazhdan--Lusztig polynomials  are $p$-independent    (and hence 
 admit closed combinatorial  interpretations).  
\end{thmA}
 
% This provides the first family of $p$-Kazhdan--Lusztig polynomials to be calculated since Williamson's famous torsion explosion examples to the Lusztig conjecture \cite{w13}.
It is very unusual that our infinite families of  $p$-Kazhdan--Lusztig polynomials are independent of $p\geq0$.  
%We believe  that   the     Hermitian symmetric pairs 
%$( W,P)$  are the only infinite families of parabolic Coxeter systems for which  the polynomials   ${^p}n_{\la,\mu}$ are  entirely independent  of $p$.  
\color{black!99} Indeed, it was pointed out to us by both Pramod Achar and the anonymous referee 
  that our result is surprising from  a geometric perspective: 
 in the non-parabolic setting, a typical geometric explanation  for characteristic-free behaviour is the existence of   small %Bott--Samelson
  resolutions for the relevant Schubert varieties. 
% This is not the case in the parabolic setting,  where the connection is necessarily less direct:  in \cite{MR705051,MR2360316} it is shown that many {\em (but not   all!)}
%   Schubert varieties for Hermitian symmetric pairs 
%    admit small %(generalised) Bott--Samelson
%     resolutions. 
\color{black}  In contrast, for our parabolic setting of Hermitian symmetric pairs 
Theorem~A provides  infinite families of characteristic-free $p$-Kazhdan--Lusztig polynomials for which it is known that 
many {\em (but not   all!)}
   Schubert varieties %for Hermitian symmetric pairs 
    admit small  
     resolutions  \cite{MR705051,MR2360316}. 
\color{black}

{\color{black}Over a field of characteristic zero, Elias--Williamson and Libedinsky--Williamson have shown that these $0$-Kazhdan--Lusztig polynomials are the classical 
 anti-spherical Kazhdan--Lusztig polynomials \cite{MR3245013,MR4437613}}.
Thus our Theorem A implies that the  $p$-Kazhdan--Lusztig polynomials  are equal to the (classical) anti-spherical Kazhdan--Lusztig polynomials of \cite{DD1,MR1444322}.   
Thus  the combinatorial interpretations of these polynomials 
alluded to in  Theorem A  can be found  in terms of the tiling language of this paper     in   \cite{MR2465813}
(building on earlier work of \cite{MR646823,MR957071,MR888703}).
 Combinatorial interpretations in terms of oriented Temperley--Lieb algebras can be found in \cite{FARR} (building on earlier work \cite{MR2955190,MR2813567}).
   
\color{black}

%However, whilst lots of Schubert varieties in Hermitian symmetric pairs do  admit small resolutions  
%but also exhibits examples that do *not* admit small resolutions. In contrast your Theorem A is just a blanket statement for all the relevant p-KL-polynomials. So whatever connection there is must be somewhat indirect.

\medskip
\noindent {\bf Tetris presentations and combinatorial invariance.  }
The first step towards proving Theorem A is to reduce to simply laced types.  
The graph automorphisms of  Coxeter graphs of type  $A$ and $D$  
give  rise to fixed point subgroups of types $B$ and $C$, respectively. 
 A crucial step in the   proof of Theorem~A  is 
to lift this  to the level of the corresponding  Hecke categories of Hermitian symmetric pairs of types $(B_n, B_{n-1})$ and $(A_{2n-1}, A_{2n-2})$ and of type $(C_n, A_{n-1})$ and $(D_{n+1}, A_n)$ ---  thus  categorifying an observation of Boe, namely that the Kazhdan--Lusztig polynomials for these pairs coincide.  
 This is an example of Lusztig--Dyer--Marietti's   combinatorial invariance conjecture: which states that 
anti-spherical  Kazhdan--Lusztig polynomials depend only on local isomorphisms of the strong Bruhat graphs (and proven   by Brenti \cite{MR2465813}  for Hermitian symmetric pairs).  
\color{black!99}
%For a Hermitian symmetric pairs $(W,P)$, the strong Bruhat  graph can be encapsulated in terms of 
%  tilings   of   corresponding admissible regions  
% \cite{MR888703} (see  \cref{typeAtiling,pm1,pm2,hhhh}).
% ; 
% this tiling combinatorics  is crucial and will form  the language of all results in this paper. 
% For Hermitian symmetric pairs, this conjecture was proven by Brenti \cite{MR2465813}. 
%  We categorify this  result and generalise it to fields of positive characteristic as follows:  

\begin{thmB}
  Let  $\Pi=[\la,\mu]$ and $\Pi=[\la',\mu']$ be  subquotients of the Bruhat graphs of   Hermitian symmetric pairs $\la,\mu\in (W,P)$ and $\la',\mu'\in (W',P')$.  If  $\Pi$ and $\Pi'$   are isomorphic as partially ordered sets, then the corresponding subquotients   
 $ \mathcal{H}^{\Pi}_{(W,P)}$ 
and 
 $ \mathcal{H}^{\Pi'}_{(W',P')}$ %are highest-weight   graded Morita equivalent.  
% Let  $(W,P)$, $(W',P')$ be    Hermitian symmetric pairs. 
%  Let  $\Pi \subseteq \mathscr{A}_{(W,P)}$ and $\Pi'\subseteq \mathscr{A}_{(W,P)}$ be 
%% tilings of the same shape.
%  Then the corresponding subquotients   
% $ \mathcal{H}^{\Pi}_{(W,P)}$ 
%and 
% $ \mathcal{H}^{\Pi'}_{(W',P')}$
  are       Morita equivalent 
 (in the sense of \cite[Section 2.2]{MR1644252})
 and this equivalence preserves the 
 grading, cellular, and highest-weight structures of these algebras.

\end{thmB}

 In other words, Theorem B says that all important representation theoretic information is preserved.  
 \color{black}
 In order to prove Theorems A and B we must provide  new presentations of the
  $\mathcal{H}_{(W,P)}$ for ${(W,P)}$ a Hermitian symmetric pair, see  \cref{easydoesit}.  
 While the original presentations have many advantages, they 
 are ill-equipped for tackling the combinatorial invariance conjecture. 
 This is because these  are ``too local" and therefore cannot possibly hope to
   reflect the wider structure of the Bruhat graph.
 Defining these new presentations requires the full power of Soergel  diagrammatics  
 and the development of new ``Tetris style" closed combinatorial formulas for manipulation of diagrams in  $\mathcal{H}_{(W,P)}$.  
 This   provides  an extremely  thorough understanding of these Hecke categories, and we expect
that  it will serve as a springboard for further combinatorial analysis of more general Hecke categories.

% When $(W,P)$ and $(W',P')$ are both simply laced,
%the subquotients of Theorem B are actually {\em isomorphic} and this isomorphism can be constructed {\em without reference} to the combinatorics of   light leaves bases.  
%This gives hope that the richer structure of the Hecke category   offers new tools for tackling   the  classical combinatorial invariance conjecture. 
% Remarkable new advances in our understanding of   combinatorial invariance for   parabolic Coxeter systems 
% have come from both mathematicians \cite{MR3836662,combinv} and even Google's  artificial intelligence  \cite{deepmind,deepmind2} --- 
% we pose our  categorification of this  conjecture in full generality in  \cref{conk}.
%

\medskip
\noindent {\bf Singular Soergel diagrammatics and proof of Koszulity.  }
 \color{black!99}The standard Koszulity  property is a particularly beautiful property which is characteristic of {\em complex} Lie theory (\cite{bgs96,MR1960515});   
a much-loved consequence of this property is that 
we can explicitly calculate the radical filtrations  of  projective and cell modules    by way of   the grading structure.
Many well-loved objects in Lie theory are Koszul over the complex field, for example the
 quantum Schur algebras \cite{MR2915315},  extended  Khovanov arc algebras \cite{MR2600694}, and the (diagrammatic)
 Cherednik algebras \cite{RSVV,losev,Webster}.   
  \color{black}

  Koszulity  of Lie theoretic objects   is usually difficult to prove and it  
  is an incredibly rare attribute   over fields of  
   characteristic $p>0$.     
  Our proof of Koszulity   explicitly constructs  linear projective resolutions of standard modules using the following theorem, which recasts the results of 
 Enright--Shelton's monograph  \cite{MR888703}
 in  the setting of    Hecke categories  and   generalises  
their  results to  fields of positive characteristic. 
 We hence  make headway on the difficult problem of 
 constructing   singular Soergel diagrammatics.

\begin{thmC}\color{black!99}
Let $(W,P)$ be a simply laced Hermitian symmetric pair, and suppose $\ctau  $ is a simple reflection in $W$.
We explicitly construct the $\ctau$-singular Hecke category\footnote{This is a diagrammatic, anti-spherical analogue of the category $J\mathbb{BS}\mathrm{Bim}$ introduced in \cite{MR3502025} for $J=\{\ctau\}$. Repeated applications of our construction will give an analogue of $J\mathbb{BS}\mathrm{Bim}$ for $J$ arbitrary.}
%$\Bbbk \otimes_{R^{\ctau}} \mathcal{S}\mathbb{S}\mathrm{Bim}(\{\ctau\},\{\ctau\})$, i.e.~the cyclotomic quotient of the endomorphism category of the singleton subset $\{\ctau\}$ in the bicategory of singular Soergel bimodules for $(W,P)$.}
 $\mathcal{H}^\ctau_{(W,P)}$ as a subcategory of $\mathcal{H} _{(W,P)}$. 
 We prove that 
  $\mathcal{H}^\ctau_{(W,P)}$  is isomorphic to 
  the Hecke category of  
  a Hermitian symmetric pair $(W,P)^\ctau$ of smaller rank.  
\end{thmC}

 \begin{figure}[ht!]
 $$ \; \begin{minipage}{2.6cm}
\end{minipage}
$$
\caption{On the left
 we depict the embedding of the Bruhat graph   of  $ (  A_3, A_1 \times A_1)$ into 
 $  ( A_5, A_2 \times A_2)$.
% The tri-coloured edges reflecting how the former graph is embedded into the latter.
On the right we depict  the corresponding  dilation map  on  the  fork  
 generator.    
 Here $i$ is a primitive 4th root of 1.  The  tri-colouring  
 $\color{magenta}2\color{orange}4\color{cyan}3$  
  of single edges in the Bruhat graph  and Soergel   diagram    comes from a single tricoloured node in the truncated  Coxeter  diagram. 
}
\label{coxeterlabelA22NOW}
\end{figure}
 
 The combinatorial shadow of Theorem C is  a graded bijection between    paths in the  smaller Bruhat graph of $(W,P)^\ctau$ and paths  in a  truncation of the larger 
   Bruhat graph of $(W,P)$.    
This can be categorified  to the level of    ``dilation"  
 homomorphisms between the  
anti-spherical    Hecke categories.  
   \Cref{hom1,hom3} are dedicated to constructing these dilation
  maps and proving that they are indeed homomorphisms.  
 We depict an example  of the embedding   of Bruhat graphs and  the effect of the homomorphism on the fork generator  in  \cref{coxeterlabelA22NOW}.  

\bigskip
\noindent {\bf Structure of the paper.} The paper is organised as follows. Section 1 contains the basic definitions needed in this paper, namely, the tile combinatorics of Hermitian symmetric pairs and the original definition of the Hecke category for an arbitrary parabolic Coxeter system. % by generators and relations.
  In Section 2, we prove that in the case of Hermitian symmetric pairs, the presentation of the Hecke category can be simplified dramatically, lifting to the Hecke category a result of Stembridge which states  that these parabolic quotients   are fully commutative. 
 In Section 3, we recall the construction of the light leaves basis for these Hecke categories. 
 Section 4 constructs Tetris style presentations and proves Theorem B.  In particular, we show that the Hecke categories corresponding to non-simply laced Hermitian symmetric pairs are graded Morita equivalent to Hecke categories of simply laced types. Section 5 constructs $\ctau$-singular Hecke categories for simply-laced types by truncating the original categories and identifies them with Hecke categories for Hermitian symmetric pairs of smaller ranks. (The proof is given in Section 6.) This construction of $\ctau$-singular Hecke categories allows us to prove results by induction on the rank.  In Section 7, we use this, together with the reduction to simply-laced types, to give a description of the graded decomposition numbers and prove Koszulity of the Hecke categories for Hermitian symmetric pairs.
 
  \section{The   Hecke categories for Hermitian symmetric pairs }
 
 Let $(W, S_W)$ be a Coxeter system:  $W$ is the group generated by the finite set $S_W$ subject to the relations $(\csigma\ctau)^{m_{\csigma\ctau}} = 1$ for  
 $\csigma,\ctau\in S_W$, $ {m_{\csigma\ctau}}\in \NN\cup\{\infty\}$ 
 satisfying ${m_{\csigma\ctau}}= {m_{\ctau\csigma}}$, and ${m_{\csigma\ctau}}=1$ if and only if  $\csigma= \ctau$.  
   Let $\ell : W \to \mathbb{N}$  be the corresponding length function. 
%   Let $\mathcal{L} = \mathbb{Z}[\grade,\grade^{-1}]$ be the ring of Laurent polynomials with integer coefficients in  variable $\grade$.
  Consider $S_P \subseteq S_W$ a  subset and $(P, S_P)$
 its corresponding Coxeter system. We say that $P$ is the parabolic subgroup corresponding to  $S_P\subseteq S_W$.
Let  $^ PW \subseteq W$ denote a set of minimal coset representatives in $P\backslash W$. 
 For $\w=\sigma_1\sigma_2\cdots  \sigma_\ell$ an expression, we define a subword  to be a sequence
$\underline{t}=(t_1,t_2,\dots ,t_\ell)\in\{0,1\}^\ell$ and we set 
 $\w^{\underline{t}}  :=\sigma_1^{t_1}\sigma_2^{t_2}\cdots \sigma_\ell^{t_\ell}$. % and we emphasise that $s_i^0=1  \in W$.   
  We let $\leq $ denote the strong  Bruhat order on $^PW$: namely $y\leq w$ if for some  
reduced expression $\w$ there exists   a subword ${\underline{t}}$ and a reduced expression  $\y$ such that $\w^{\underline{t}}=\y$.   The Hasse diagram of this partial ordering is called the {Bruhat graph} of $(W,P)$.  
 \color{black}
       For the remainder of this paper we will assume that $W$ is a Weyl group and indeed that $(W,P)$ is a Hermitian symmetric pair, which are classified as follows:

    \begin{figure}[ht!]
$$
 \end{minipage} 
 $$
 
% \!\!\!\!\!\!\!\!\!\!\!\!
 \caption{ 
 Enumeration of nodes in the parabolic Dynkin diagram of types
 of type $(A_{n },   A_{k-1} \times A_{n-k } )$, 
 $(C_n, A_{n-1})  $ and $(B_n, B_{n-1}) $, 
  $(D_n , A_{n-1})  $ and $(D_n,  D_{n-1}) $ and $(E_6 , D_5 )  $ and $(E_7 , E_6) $ respectively.  The single node not belonging to the parabolic is highlighted in pink in each case.  }
\label{coxeterlabelD2}
\end{figure}

%\begin{defn}
 \color{black!99} Let  $W$ be a  finite Coxeter group and $P$ a parabolic subgroup. 
   The $(W,P)$ corresponding to Hermitian symmetric spaces were first studied by Cartan \cite{Cartan} and have been classified (see for example \cite{MR957071}). There are five infinite families  
%  Then 
%$(W,P)$ is
%  a Hermitian symmetric pair if and only if it is one of the following: 
  $(A_n  ,  A_{k-1} \times A_{n-k} )$  with $1 \leq k\leq n$, 
    $(D_n , A_{n-1}  )$,  $(D_n , D_{n-1}  )$,  $(B_n , B_{n-1}  )$, 
     $(C_n , A_{n-1}  )$
     and two exceptional ones  
      for  $n\geq 2$, $(E_6, D_5)$, or $(E_7,E_6)$.  
      Our main interest in these pairs $(W,P)$ stems from the  rich combinatorial 
     and  representation theoretic structures associated to them (see for example
      \cite{MR921987,MR888703,MR957071,MR2363142,MR2465813,MR3363009});  
     and the fact that they are  tractable and diverse enough to  serve as milestones  
     of our understanding of Lie theoretic objects.  
% The purpose of this paper is 
This paper % This  extend the  combinatorics of  
extends the work above on Hermitian symmetric pairs in order 
 to provide a new milestone in our understanding of 
 anti-spherical Hecke categories.
% \end{defn}
 
 \color{black}
 
In \cref{coxeterlabelD2},  we   recall the Dynkin diagrams of  Hermitian symmetric pairs   explicitly.  For type $D$, we use a slightly unusual labelling of nodes, which  will allow us to pass between types $C$ and $D$ more easily.  
   The remainder of this section is dedicated to the combinatorics  of Hermitian symmetric pairs.
  This has been lifted from   \cite[Appendix: diagrams of Hermitian type]{MR3363009}, but has been translated into a more diagrammatic language.

\subsection{Tile partitions   }\label{tileeeeee}
  \label{Dyck}
The Bruhat graphs of Hermitian symmetric pairs can be encapsulated 
 in terms of   tilings of  ``admissible regions" of the plane, which we now define.   
 In type $(A_n, A_{k-1} \times A_{n-k})$, the admissible region is simply a  $(k\times (n-k+1))$-rectangle, and the tilings governing the combinatorics are Young diagrams which fit in this rectangle. 
% The notion of admissible regions and their tilings, which we now define, gives the appropriate generalisation of this combinatorics to the other Coxeter types of Hermitian symmetric pairs.
 The general picture is as follows:

\begin{defn}  Let $(W,P)$ be a Hermitian symmetric pair of classical type.  
We call a point $[r,c] \in \NN^2$ a {\sf tile}.
The {\sf admissible region} $\mathscr{A}_{(W,P)}$ is a certain finite subset of tiles defined as follows:
\begin{itemize}[leftmargin=*]
\item for type $(W,P)=(A_n, A_{k-1} \times A_{n-k})$, the admissible region is the subset of tiles $$\{ [r,c] \mid   r\leq n-k+1,   c\leq k\}.$$
\item for types  $(W,P)=  (C_n, A_{n-1} )$ and $   (D_n, A_{n-1}) $, the admissible region is the subset of tiles 
$$\{ [r,c] \mid   r,c \leq  n \text{ and }r-c\geq 0 \}.$$

\item for type $(W,P)= (B_n,  B_{n-1})$,   the admissible region is the subset of tiles 
$$
\{ [r,c] \mid r=1 \text{ and }c< n\}
\sqcup
\{ [r,c] \mid c=n \text{ and }r\leq n\}.
$$

\item for type $(W,P)=  ( D_n,D_{n-1} )$,   the admissible region is the subset of tiles 
$$
\{ [r,c] \mid r=1 \text{ and }c< n\}
\sqcup
\{ [r,c] \mid c=n \text{ and }r\leq n\}
\sqcup
\{ [2,n-1] \}.
$$

\end{itemize}
We draw tiles and admissible regions in the ``Russian'' style, with rows (i.e.~fixed values of $r$) pointing northwest and columns (i.e.~fixed values of $c$) pointing northeast.   %west.
   	\end{defn}	
  
  \begin{eg} We illustrate the admissible region for type
  $(A_8 , A_4\times A_3)$ in \cref{typeAtiling}, 
  for types    $(D_6,A_5)$ and $(C_6,A_5)$ in \cref{pm1}, and
  for types $(B_6, B_5)$ and $(D_7,D_6)$ in \cref{pm2}. For the two exceptional 
  types $(E_6,D_5)$ and $(E_7,E_6)$, the admissible region consists of the subset of tiles pictured  in \cref{hhhh}.  
  \end{eg}   
  
\begin{figure}[ht!]
$$\scalefont{0.9}

$$
\caption{On the left we picture the admissible region for $(A_8 , A_4\times A_3)$. We then picture two tilings; the  first of which is a tile partition, but the latter is not (the  tile $\color{orange}[2,4]$ is not supported). }
\label{typeAtiling}
\end{figure}

  Each tile $[r,c] \in \mathscr{A}_{(W,P)}$ carries a coloured label, inherited from the Dynkin diagram of $W$.  
 This is explained in detail in  \cite[Appendix]{MR3363009}, but can be deduced easily from \cref{typeAtiling,pm1,pm2,hhhh}. 
 Given $[r,c] \in \mathscr{A}_{(W,P)}$, we let $s_{[r,c]}$ denote the corresponding simple reflection in $S$.  
 For example, in types $(A_n,A_{k-1}\times A_{n-k})$ and $(C_n,A_{n-1})$ the reflection $s_{[r,c]}$ is determined simply by the $x$-coordinate of the tile $[r,c] \in \mathscr{A}_{(W,P)}$ (i.e. it is determined by $c-r$).  
Given $\ctau \in W$ a label of the Dynkin diagram, we refer to a   tile $[r,c] \in \mathscr{A}_{(W,P)}$ 
as a {\sf $\ctau$-tile} if   
  $s_{[r,c]}=\ctau$.  
We emphasise this connection by colouring the tile,   $\color{cyan}[r,c]\color{black} \in \mathscr{A}_{(W,P)}$, when appropriate.

\begin{figure}[ht!]
%\vspace{-0.2cm}
$$\scalefont{0.9}

 $$
\caption{We picture two admissible regions $\mathscr{A}_{(W,P)}$ for types    $(D_6,A_5)$ and $(C_6 ,A_5 )$ and a tile-partition $\la=(1,2,3,4  ) \in   \mptn $ for types $(C_6 ,A_5 )$.  
The $\pm$ signs are explained in \cref{pm3}. }
\label{pm1}
\end{figure}

 We say that a pair of tiles  
  are {\sf neighbouring} if they meet  at an edge (which necessarily has  an angle of  $45^\circ$ 
or $135^\circ$ {\color{blue} to the horizontal axis}, by construction).  
% Given a pair of neighbouring tiles $X$ and $Y$, we say that  $Y$ {\sf  
%  supports} $X$ if 
% $X$ appears above $Y$.  
% We say that a tile, $X \in  \mathscr{A}_{(W,P)}$, is {\sf  supported} if every admissible tile which can support $X$ does support $X$.  
%  We say that a collection of tiles, $\la\subseteq \mathscr{A}_{(W,P)}$ is a
%  {\sf tile-partition}    if every tile in $\la$ is  
%   supported.
 \color{black!99} Given a pair of neighbouring tiles $X$ and $Y$, we write $Y<X$
  if 
 $X$ appears above $Y$ (i.e. the $y$-coordinate of $X$ is strictly larger than that of $Y$). 
 We extend $<$ to a partial order on the tiles in $  \mathscr{A}_{(W,P)}$ by transitivity and we say that $Y$ {\sf supports} $X$ if $Y<X$ in this ordering. 
%We say that a tile, $X \in  \mathscr{A}_{(W,P)}$, is {\sf  supported} if 
% for every $Y \in  \mathscr{A}_{(W,P)}$  we have that 
  We say that a collection of tiles, $\la\subseteq \mathscr{A}_{(W,P)}$ is a
  {\sf tile-partition}    if  for every tile  $X \in \la$ and 
  every $Y \in  \mathscr{A}_{(W,P)}$ such that $Y<X$, we have that $Y \in \la$.
%   supported.
 \color{black}  %%%%
 %%%%   
     We let     $\mathscr{P}_{(W,P)}  $ denote the set of all tile-partitions. 
  We  depict a tile-partition  $\la$ by  colouring the tiles of $\la$. See  \cref{typeAtiling} %and \ref{typeE6D5tiling}
    for examples and non-examples of tile-partitions.

  We define the {\sf length} of a tile-partition $\la$ to be the total number of tiles $[r,c]\in \la$.  
\color{black!99}  We let $\mathscr{P} _{(W,P)}^\ell$ denote the subset of all tile partitions of length $\ell$. 
\color{black}
  There is a natural bijection between ${^P}W$ and 
$\mathscr{P}_{(W,P)}$ (see   \cite[Appendix]{MR3363009}) under which the length functions coincide.  
For $\lambda  ,\mu  \in \mathscr{P} _{(W,P)}$, we define the {\sf Bruhat order} on tile partitions by
 $\lambda  \leq \mu $ if 
%$\\lambda  \cap \mu = \\lambda $.   
$$\{[r,c] \mid [r,c] \in \lambda \}
\subseteq \{[r,c] \mid [r,c] \in \mu \}.$$
Given $\la \leq \nu$, we define the {\sf skew tile-partition} $\nu\setminus \la$  to be the  
set difference of $\la$ and $\nu$.

\begin{figure}[ht!]
$$\scalefont{0.9}
 $$
\caption{The admissible regions for types $(B_6,B_5)$ and $ (D_7,D_6)$. The $\pm$ signs are explained in \cref{pm3}.}
\label{pm2}
\end{figure}

 \begin{figure}[ht!]
 $$
   %
 
$$ 
\caption{The admissible regions $\mathscr{A}_{(W,P)}$ for  the   pairs
 $(E_6,D_5)$ and  $(E_7,E_6)$ respectively. We think of the violet node as being free to swing from left to right; in this way we continue to associate tiles with the $x$-coordinates of nodes in the Coxeter graph.}
 \label{hhhh}
 \end{figure}

\begin{defn}\label{flashback}
Let $[r,c] \in  \mathscr{A}_{(W,P)}$ denote any tile.  
We define $\la_{[r,c]}$ to be the tile partition  
$$\la_{[r,c]}:=\{[x,y] \in  \mathscr{A}_{(W,P)} \mid x \leq r, y\leq c \}   $$
\end{defn}

\begin{rmk}
In type $(A_{n },   A_{k-1} \times A_{n-k } )  $ a tile-partition $\la$ is the  Young diagram (in Russian notation)  of a classical partition with at most $k$ columns and $(n-k+1)$ rows.   
In this case  $\la_{[r,c]} $ is the $(r \times c)$-rectangle.  In other types, it is this rectangle intersected with the region $ \mathscr{A}_{(W,P)}$.  
 \end{rmk}

%\begin{defn}
%Given $\la \subseteq \nu$, we define the {\sf skew tile partition} $\nu\setminus \la$  to be the  
%set difference of $\la$ and $\nu$.  
%\end{defn}

\subsection{Tile tableaux } \label{tiletableaux}
The combinatorics of reduced and non-reduced words 
for Hermitian symmetric pairs can be encapsulated 
in terms of {\em tile-paths} or {\em tile-tableaux}, which we now define.   
Given $\mu\in  \mathscr{P}_{(W,P)}	$,
 we define     the set of all {\sf addable} and {\sf removable}
 tiles to be
  ${\rm Add}(\mu)=\{ [r,c] \mid \mu \cup [r,c] \in \mathscr{P}_{(W,P)}		\}$
 and 
  ${\rm Rem}(\mu)=\{ [r,c] \mid \mu \setminus [r,c]  \in \mathscr{P}_{(W,P)}		\}$ 
respectively. Abusing notation, we will write $\mu+[r,c]$ for $ \mu \cup [r,c]$ and $\mu-[r,c]$ for  $\mu \setminus [r,c]$.
Any tile-partition $\mu$  has at most one addable or removable tile of any given colour $\ctau\in S_W$.  
Thus  given    $\color{cyan}[r,c]\color{black} \in \mathscr{A}_{(W,P)}$ with 
  $\ctau=\color{cyan}s_{[r,c]}\color{black}  $,  
 we often write $\ctau \in   {\rm Add}(\mu)$ or 
 $\ctau \in   {\rm Rem}(\mu)$
  and we write $\mu+\ctau$ or simply $\mu\ctau$ for 
   $\mu\ctau:= \mu+ 
 \color{cyan}[r,c]$; we write 
   or $\mu-\ctau:= \mu-
 \color{cyan}[r,c]$.

\begin{defn}\label{oieurioweurg} For $\la  \in \mathscr{P}_{(W,P)} $  we define a   
 {\sf tile-tableau}    of length $\ell$  and shape $\la$  to be a path 
 $$\SSTT: \varnothing ={\la}_0 \xrightarrow{ \ [r_1,c_1] \ } {\la}_1  \xrightarrow{ \ [r_2,c_2] \ } {\la}_2 
 \xrightarrow{ \ [r_3,c_3] \ } 
  \cdots 
 \xrightarrow{ \ [r_\ell,c_\ell] \ } {\la}_\ell  = \la$$such that for each $k=1,\ldots, \ell$, ${\la}_k   \in \mathscr{P}_{(W,P)} $   
  and ${\la}_k $ satisfies one of the following
  \begin{itemize}
  \item[$(i)$] ${\la}_k={\la}_{k-1} + [r_k,c_k]$ with $[r_k,c_k] \in {\rm Add}({\la}_{k-1})$; or 
  \item[$(ii)$] ${\la}_k={\la}_{k-1} - [r_k,c_k]$ with  $[r_k,c_k] \in {\rm Rem}({\la}_{k-1})$; or  
  \item[$(iii)$] ${\la}_k={\la}_{k-1}$		with $[r_k,c_k] \in {\rm Rem}({\la}_{k-1})$ or $ {\rm Add}({\la}_{k-1})$; 
  \end{itemize}We let $\Path_\ell(\la)$ denote the set of all tile-tableaux of  shape $\la$ and length $\ell$.  
%    We let $\Std^+_\ell(\la)\subseteq \Std_\ell(\la)$ denote the subset of   tile-tableaux, $\SSTT$,
%    such that $\SSTT(k)\in \mathscr{P}_{(W,P)}$  for all $0\leq k \leq \ell$.  
%   \in \{ \SSTT(k-1) + [r,c]  , \SSTT(k-1)  , \SSTT(k-1) -[r,c]\}$ for some addable or removable tile $[r,c]$.    
% 
 We say that a tile tableau, $\SSTT_\la\in \Path _\ell(\la)$, is {\sf reduced} if $\la  \in \mathscr{P}^\ell_{(W,P)} $ and we denote the set of all such tableaux by
$ \Std (\la)$.  We use uppercase (respectively lowercase)  san-serif letters for general tile-tableaux  
(respectively reduced tile-tableaux).  
 \end{defn}
 
  \begin{rmk}
 In type $(A_{n },   A_{k-1} \times A_{n-k } )  $, the notions of addable and removable tiles correspond to the familiar  notions of addable and removable boxes for Young diagrams.  
The set of reduced tile tableaux coincides with the usual notion of standard Young tableaux.  
 \end{rmk}

 For $\la \in\mathscr{P}^\ell_{(W,P)}$ we   identify a reduced tableau
 $\stt \in  \Std  (\la)$ with a bijective map 
 $\stt: \la \to \{1,\dots,\ell\}$ and we record this by placing the entry $\stt^{-1}(k)$ in the $[r,c]$th tile, in the usual manner.  
 In this fashion, we can identify $\Std (\la)$ with the set of all possible
  fillings of $\la$ with the numbers $ \{1,\dots,\ell\}$ in such a way that these numbers increase alongs the rows and columns of $\la$.  
For $\nu \setminus \la$ a skew tile partition,   we can 
define $\Std (\nu\setminus \la)$  
 in the obvious fashion.  
Given $1\leq k \leq \ell$,  we let $\stt{\downarrow}_{\{1,\dots,k\}}$ denote the restriction of the map to the pre-image of ${\{1,\dots,k\}}$.   
 %Given $\nu\setminus \la $ a skew tile-partition, 
 %we define the {\sf superstandard} reduced tableau, $\stt_{\nu \setminus \la}$, to be 
 %the tile tableau obtained by placing the entry $\ell$ in the 
%rightmost tile $[r,c] \in {\rm Rem}(\la)$,    
%that is the removable tile $[r,c]$ such that $r-c\in \NN$ is maximal.  {\color{magenta} This is not unique. For instance, in Figure 9 you could switch the positions of 9 and 10 in the tableau. I've tried to correct the definition following the pictures:}
% Given $\nu\setminus \la $ a skew tile-partition, 
%% we define the {\sf superstandard} reduced tableau, $\stt_{\nu \setminus \la}$, to be 
%% the unique reduced tile tableau obtained by filling the first row of tiles $[1,c]$, then the second row of tiles $[2,c]$, and so on, with the numbers $1,\ldots, \ell$ in increasing order.
%often use the notation $\stt_{\nu \setminus \la}$  to depict a preferred choice of  reduced tile tableau.   
% \begin{eg}
% For  $\mu $ a tile partition of type  $(A_{n },   A_{k-1} \times A_{n-k } )  $, 
% the tableau $\stt_\mu$ is the usual superstandard tableau.
   Examples are  depicted in \cref{super}.
% \end{eg}

\begin{defn}\label{asjkdjsfkhgd} %Let $\mu \in \mptnl.$ 
Let $\la,\mu \in \mptn$ and fix    $\stt\in \Std(\mu)$
such that 
 $\stt ( [x_k,y_k])=k$ for  $1\leq k \leq \ell$.  We say that a tile-tableau
 $$\SSTT: \varnothing ={\la}_0 \xrightarrow{ \ [r_1,c_1] \ } {\la}_1  \xrightarrow{ \ [r_2,c_2] \ } {\la}_2 
 \xrightarrow{ \ [r_3,c_3] \ } 
  \cdots 
 \xrightarrow{ \ [r_\ell,c_\ell] \ } {\la}_\ell  = \la$$is obtained by {\sf folding-up}   $\stt\in \Std(\mu)$ if 
   $s_{[r_k,c_k] }=s_{[x_k,y_k] }$    for $1\leq k \leq \ell$.  
 We let $\Path(\la,\stt)$ denote the set of all paths obtained  in this manner.  
 \end{defn}

     \begin{figure}[ht!] \color{black!99}
 $$\begin{minipage}{5cm}
\end{minipage}
$$
 \caption{\color{black!99}Tableaux $\sts,\stt, \stu$ of shape $ (4,3,2^2)$, $ (4,3,2,1)$  and 
 $(5^2,4,3)\setminus (4,3,2^2)$ respectively, in type $(A_8,A_4\times A_3) $. 
 We have that $\sts=\stt \otimes \ctau$ for $\ctau=s_7$ and 
$ \stt =\sts{\downarrow}_{\{1,\dots, 10\}}$.
 }
 \label{super}
 \end{figure}

\begin{defn}
Given  $\la \in \mptnl$, $\stt \in \Std(\la)$ and 
 $\color{cyan} [r,c] \color{black} \in {\rm Add}(\la)$, 
  we let $\stt \otimes \ctau  \in \Std(\la\ctau) $ denote the reduced tableau 
 uniquely determined by $(\stt\otimes \ctau)[x,y] \color{black} = \stt   [x,y]$
  for $[x,y]\neq \color{cyan} [r,c]$.  

\end{defn}

\begin{eg}\color{black!99}
We now provide an example of \cref{asjkdjsfkhgd}. 
An element of $\Path((3,2), \sts)$ for $\sts \in \Std(4,3,2^2)$ as in \cref{super} is given as follows: \color{black}
$$
\varnothing 
\;{\color{magenta}\to}\; (1) 
\; {\color{green}\to} 			\;(2) 
 \;{\color{orange}\to} \;		 (3) 
\; {\color{yellow}\to} \;				(3) 
 \;{\color{brown}\to} \;				(3,1) 
\; {\color{magenta}\to}			\; (3,2) 
\; {\color{green}\to} 			\;		 (3,2) 
\; {\color{cyan}\to} 			\;			(3,2,1) 
\; {\color{brown}\to}\;				 (3,2,1) 
\; {\color{pink}\to} 			\; 				(3,2,1) 
\; {\color{cyan}\to}	\;		(3,2) 
$$\color{black!99}where the tile-labels of the arrows can be deduced from their colours (which match with the colours of the tiles  for $\sts$ pictured in  \cref{super}). 
The 1st, 2nd, 3rd, 5th, 6th, 8th steps are all of the form $(i)$ \cref{oieurioweurg};
 the 11th step is of the form  $(ii)$ \cref{oieurioweurg};
 the 4th, 7th,   9th, and 10th steps are of the form $(iii)$ \cref{oieurioweurg}.
\end{eg}
%\begin{rmk}
%We freely identify  a  tile  $\color{cyan}[r,c]\color{black} \in \mathscr{A}_{(W,P)}$ with its reflection $\ctau = \color{cyan}s_{[r,c]}$
%
%\end{rmk}

\begin{rmk}
\color{black!99} For $\la \in {^PW}$, the set of reduced tableaux $\Std(\la)$ is in bijection with the set of reduced expressions for $\la$ in a natural fashion.  
For example, the tableau $\stt \in \Std(4,3,2,1)$ in \cref{super} corresponds to the reduced word 
$ 
  {\color{magenta}s_5}
  {\color{green}s_4}
  {\color{orange}s_3}
  {\color{yellow}s_2}
  {\color{brown}s_6}
  {\color{magenta}s_5}
  {\color{green}s_4}
  {\color{cyan}s_7}
   {\color{brown}s_6}
  {\color{pink}s_8}
$.
\end{rmk}

 \subsection{Parity conditions for non-simply laced tiles} 
\label{pm3}
\color{black!99}
Finally, we are now ready to explain the existence of $\pm$ signs in \cref{pm1,pm2}.  
For  non-simply-laced Weyl groups (that is, $ (W,P)=(C_n,A_{n-1})  $ or $(B_n,B_{n-1}$) we  allow a simple reflection 
  $s_i \in S_W$ to  carry a parity  label,  $s_{\pm i} $).  
  Given $ \la =\sigma_1\sigma_2\cdots  \sigma_\ell \in {^PW}$ we 
 use this parity label to  record the  odd/even  number of prior appearances of this reflection (read from left-to-right).  
 For example, in \cref{pm1} we picture the element
 $ \la = {\color{magenta}s_{+1}}
 {\color{brown}s_{2}}
 {\color{magenta}s_{-1}}
 {\color{orange}s_{3}}
  {\color{brown}s_{2}}
   {\color{green!80!black}s_{4}}
 {\color{magenta}s_{+1}}
 {\color{orange}s_{3}}
  {\color{violet}s_{5}}  {\color{cyan}s_{6}}\in {^PW}$ for $(W,P)=(C_6 ,A_5 )$.  
  We note that this $\pm$ label is  well-defined for elements of ${^PW}$ 
  (as these cosets are all ``fully commuting'' in the sense of \cite{MR1406459}).  We record this parity label   in terms of the $y$-coordinate of tiles in $\mathscr{A}_{(W,P)}$ as in \cref{pm1,pm2}.  
%For $W$ a non-simply-laced Weyl group, we refine the colouring of the graph by dividing certain tiles into $\pm$-tiles, depending on the parity condition demonstrated in \cref{pm1,pm2}. 
Given $\la,\mu \in \mathscr{P}_{(W,P)}$ and $\sts \in \Std(\mu)$, we let 
 $\Path^{\pm}(\la,\sts)\subseteq \Path (\la,\sts)$ denote the subset of paths which preserve this parity condition.  More precisely, using the notation from \cref{asjkdjsfkhgd}, we insist that if $s_{[x_k,y_k]} = s_{\pm i}$ then we also have $s_{[r_k, c_k]} = s_{\pm i}$ (with the same parity).
 \begin{eg}  Let $(W,P)=(C_6 ,A_5 )$ as in \cref{pm1}, we have that 
%$$
%\varnothing
%{	\to}
%(1)
%{	\to}
%(1^2)
%{	\to}
%(1^3)
%{	\to}
%(1^4)
%{	\to}
%(1^5)
%{	\to}
%(1^6)
%{	\to}
%(1,2,1^4)
%{	\to}
%(1,2^2,1^3)
%{	\to}
%(1,2^2,1^3)
%{	\to}
%(1,2^2,1^3)
%{	\to}
%(1,2,1^4)
%{	\to}
%(   1^6)
% $$
%$$
%\varnothing
%\hspace{0.05cm}{\color{magenta}	\to}\hspace{0.05cm}
%(1)
%\hspace{0.05cm}{\color{brown}	\to}\hspace{0.05cm}
%(1^2)
%\hspace{0.05cm}{\color{orange}	\to}\hspace{0.05cm}
%(1^3)
%\hspace{0.05cm}{\color{green!80!black}	\to}\hspace{0.05cm}
%(1^4)
%\hspace{0.05cm}{\color{violet}	\to}\hspace{0.05cm}
%(1^4)
% \hspace{0.05cm}{\color{magenta}	\to}\hspace{0.05cm}
%(1,2,1^2)
%{\color{brown}	\to}\hspace{0.05cm}
%(1,2^2,1 )
%{\color{magenta}	\to}\hspace{0.05cm}
%(1,2^2,1 )
%{\color{orange}	\to}\hspace{0.05cm}
%(1,2^2,1 )
%{\color{brown}		\to}\hspace{0.05cm}
%(1,2,1^2)
%{\color{magenta}	\to}\hspace{0.05cm}
%(   1^4)
% $$
$$
\varnothing
\hspace{0.05cm}{\color{magenta}	\to}\hspace{0.05cm}
(1)
\hspace{0.05cm}{\color{brown}	\to}\hspace{0.05cm}
(1^2)
\hspace{0.05cm}{\color{orange}	\to}\hspace{0.05cm}
(1^3)
\hspace{0.05cm}{\color{green!80!black}	\to}\hspace{0.05cm}
(1^4)
  \hspace{0.05cm}{\color{magenta}	\to}\hspace{0.05cm}
(1,2,1^2)
{\color{brown}	\to}\hspace{0.05cm}
(1,2^2,1 )
{\color{magenta}	\to}\hspace{0.05cm}
(1,2^2,1 )
{\color{orange}	\to}\hspace{0.05cm}
(1,2^2,1 )
{\color{brown}		\to}\hspace{0.05cm}
(1,2,1^2)
{\color{magenta}	\to}\hspace{0.05cm}
(   1^4)
 $$and 
$$
\varnothing
\hspace{0.05cm}{\color{magenta}	\to}\hspace{0.05cm}
(1)
\hspace{0.05cm}{\color{brown}	\to}\hspace{0.05cm}
(1^2)
\hspace{0.05cm}{\color{orange}	\to}\hspace{0.05cm}
(1^3)
\hspace{0.05cm}{\color{green!80!black}	\to}\hspace{0.05cm}
(1^4)
%\hspace{0.05cm}{\color{violet}	\to}\hspace{0.05cm}(1^4)
  \hspace{0.05cm}{\color{magenta}	\to}\hspace{0.05cm}
(1,2,1^2)
{\color{brown}	\to}\hspace{0.05cm}
(1,2^2,1 )
{\color{magenta}	\to}\hspace{0.05cm}
(1,2,3,1 )
{\color{orange}	\to}\hspace{0.05cm}
(1,2,3,2 )
{\color{brown}		\to}\hspace{0.05cm}
(1,2,3,2 )
{\color{magenta}	\to}\hspace{0.05cm}
( 1,2^3 )
 $$are both elements of $ \Path (\la,\sts)$ for some $\la=(1^4), (1,2^3) \in {^PW}$ and  for the same $\sts \in \Path(1,2,3,4)$. The former of these paths belongs to  $\Path^{\pm}(1,2,3,4)$ whereas the latter does not as $s_{[x_{10}, y_{10}]} = s_{[4,4]} = s_{-1}$ but $s_{[r_{10}, c_{10}]} = s_{[3,3]} = s_{+1}$.
 
 \end{eg}

 \color{black}

  \subsection{The diagrammatic Hecke categories}    
\label{coulda}
 Almost everything from this section is  lifted from   Elias--Williamson's original paper \cite{MR3555156}  or is an extension of their results to the parabolic setting \cite{MR4437613}.  
%  In this paper, we  suppress mention of category theoretic terminology because it jars slightly with our representation theoretic aims.  
% Instead of  the definition of the Hecke category  via objects and morphisms, we favour   the language of  Brundan--Stroppel's ``locally unital associative algebras", see \cite[Remark 2.3]{BSBS} for more details.   
Let $(W,S)$ denote a   Coxeter system for $W$ a Weyl group.   
Given $\csigma\in S_W$ we define the {\sf monochrome Soergel generators} to be the framed graphs 
\begin{align}\label{genrdiag}\tag{G1}
{\sf 1}_{\emptyset } =
\begin{minipage}{1.5cm}
\end{minipage}
\end{align}
 and given any $\csigma,\ctau\in S_W $ with $m_{\csigma\ctau}=m=2,3$ or $4$ we have the {\sf bi-chrome generator} ${\sf braid}_{\csigma\ctau }^{ \ctau\csigma }(m)$ which is pictured as follows 
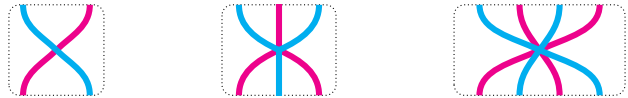
\begin{align}\label{genrdiag2}\tag{G2}
\begin{minipage}{2cm}
%
 \end{minipage}
\end{align}for $m$ equal to $2, 3$ or $4$ respectively. 
(We will also sometimes write ${\sf braid}_{\csigma\ctau}^{\ctau \csigma}$, ${\sf braid}_{\csigma \ctau \csigma}^{\ctau \csigma \ctau}$, and ${\sf braid}_{\csigma \ctau \csigma \ctau}^{\ctau \csigma \ctau \csigma}$ for ${\sf braid}_{\csigma \ctau}^{\ctau \csigma}(m)$ with $m=2$, $m=3$, and $m=4$ respectively.) 
 Pictorially, we define the duals of these generators to be the graphs obtained by reflection through their horizontal axes.  Non-pictorially, we simply swap the sub- and superscripts.  We sometimes denote duality by $\ast$.  For example, the dual of the fork generator is pictured as follows
$$
  {\sf fork}^{\csigma\csigma}_{\csigma}=  
  \begin{minipage}{2cm}  
 \begin{tikzpicture}[scale=-1.2]
 \draw[densely dotted,rounded corners](-0.75cm,-0.5cm)
  rectangle (0.75cm,0.5cm);
  \clip (-0.75cm,-0.5cm)
  rectangle (0.75cm,0.5cm);
%   \clip (-0.5cm,-0.5cm)
%  rectangle (0.5cm,0.5cm);
 \draw[line width=0.08cm, magenta](0,0)to [out=-30, in =90] (10pt,-15pt);
 \draw[line width=0.08cm, magenta](0,0)to [out=-150, in =90] (-10pt,-15pt);
 \draw[line width=0.08cm, magenta](0,0)--++(90:1);
 \end{tikzpicture}
\end{minipage}.$$We define the northern/southern reading word of a Soergel generator (or its dual) to be word in the alphabet $S$ obtained by reading the colours of the northern/southern edge of the frame respectively.  
Given   two (dual) Soergel generators $D$ and $D'$ we define $D\otimes D'$ to be the diagram obtained by horizontal concatenation (and we extend this linearly).  The northern/southern colour sequence of $D\otimes D'$ is  the concatenation of those of $D$ and $D'$ ordered from left to right.    
Given any two (dual) Soergel generators, we define their product $D\circ D'$ (or simply $DD'$) to be the vertical concatenation of $D$ on top of $D'$ if the southern reading word of $D$ is equal to 
 the northern reading word of $D'$ and  zero otherwise.  
We define a {\sf Soergel diagram} to be any graph obtained by repeated  horizontal and vertical concatenation of  (dual) Soergel generators.

%      \subsection{Some specific graphs}      

For $\w=\sigma_1\dots\sigma_\ell$ an expression, we define 
$ {\sf 1}_\w={\sf1}_{\sigma_1}\otimes {\sf1}_{\sigma_2}\otimes \dots\otimes {\sf1}_{\sigma_\ell} 
$ and given $k>1$ and $\csigma,\ctau\in S_W$ we set $
  {\sf 1}_{ \csigma\ctau} ^{k }   = 
{\sf 1}_{ \csigma}\otimes {\sf 1}_\ctau 
\otimes{\sf 1}_{ \csigma}\otimes {\sf 1}_\ctau \dots 
$ 
to be the alternately coloured idempotent on $k$ strands (so that the final strand is  $\csigma$- or $\ctau$-coloured if $k$ is odd or even respectively).  
   Given $\csigma,\ctau \in S_W$ with $m_{\csigma\ctau}=2$,
    let  
 $\w=\rho_1\cdots \rho_k (\csigma\ctau ) \rho_{k+3}\cdots \rho_\ell$ and $
\underline{\w}=\rho_1\cdots \rho_k (\ctau\csigma ) \rho_{k+3}\cdots \rho_\ell$ 
be two reduced expressions for $w\in W$.   We say that $\w$ and $\underline{\w}$ are {\sf adjacent} and we set 
$${\sf braid}^\w_{\underline{\w}}={\sf 1}_{\rho_1}\otimes \cdots \otimes {\sf 1}_{ \rho_k}
\otimes 
{\sf braid}^{\csigma\ctau}_{\ctau\csigma}(2) 
\otimes {\sf 1}_{\rho_{k+3}}\otimes  \cdots \otimes {\sf 1}_{ \rho_\ell}.
 $$Given a {\em  fixed}  sequence of adjacent reduced expressions, 
 $\w=\w^{(1)}, \w^{(2)},\dots, \w^{(q)}=\underline{\w}$ 
and the value $q$ is minimal such that 
this sequence exists, then we set 
$${\sf braid}^\w_{\underline{\w}}
=
\prod _{1\leq p < q}{\sf braid}^{\w^{(p)}}_{{\w^{(p+1)}} }
 .$$
%
%
%While this element is not uniquely defined,   only the minimality will  matter for our purposes.  
%Given   $\w=\sigma_1\dots\sigma_\ell$ and 
%${\underline{\w}}= \rho_1\dots\rho_\ell$  two reduced expressions which differ {\em only} by application of the 
%{\em commuting} Coxeter relations, we let 
%$$${\sf braid}^\w_{\underline{\w}}$$ 
%denote the unique  

 Given $\csigma$, we define  the corresponding ``barbell" and ``gap" diagrams  to be the elements  
 $${\sf bar}(\csigma)=  
  {\sf spot}_ \csigma^\emptyset
   {\sf spot}^ \csigma_\emptyset
   \qquad
 {\sf gap}(\csigma)= {\sf spot}^ \csigma_\emptyset
  {\sf spot}_ \csigma^\emptyset ,$$respectively.  
 {\color{black!99} Let  $\la, \mu \in \mathscr{P}_{(W,P)} $ with $\ell=\ell(\mu)-\ell(\la)$ and   $\stt \in \Std (\mu\setminus \la)$   
such that 
 $\stt ( [x_k,y_k])=k$ for  $1\leq k \leq \ell$.  
We  let 
$${\sf 1}_\stt = {\sf 1}_{s_{[x_1,y_1]}} \otimes  {\sf 1}_{s_{[x_2,y_2]}} \otimes \dots \otimes  {\sf 1}_{s_{[x_\ell,y_\ell]}} $$and 
for $\csigma= \color{magenta}s_{[x_k,y_k]}$  we set }
   $$
  {\sf gap}(\stt-	[x_k,y_k]	) = {\sf 1}_{\stt{\downarrow}_{\{1,\dots,k-1\}}} \otimes {\sf gap}(\csigma)\otimes {\sf 1}_{\stt{\downarrow}_{\{k+1,\dots,\ell\}}}.  
  $$We also define the corresponding ``double fork'' diagram to be the element
  $$
  {\sf dork}^{\csigma\csigma} _{\csigma\csigma} = {\sf fork}^{\csigma\csigma}_{\csigma} {\sf fork}^{\csigma}_{\csigma\csigma} \text{.}
  $$
   It is standard (in Soergel diagrammatics) to draw the  element 
${\sf cap}^\emptyset _{\ctau\ctau}:={\sf spot}_\ctau^\emptyset {\sf fork }_{\ctau\ctau}^\ctau$ simply as a strand which starts and ends on the southern edge of the frame.  (We define ${\sf cup}^{\ctau\ctau}_\emptyset:=({\sf cap}^\emptyset _{\ctau\ctau})^\ast$.)  
%For $\x  = \y \csigma $ a word, we inductively define 
% $$
% {\sf cap }_{\x \x^{-1} }^\emptyset := {\sf cap }_{{\y }{\y^{-1} } }^{\emptyset }
%({\sf 1}_{\y}  \otimes 
%{\sf cap}_{{\csigma }{\csigma } }^{\emptyset}
%  \otimes {\sf 1}_{\y^{-1}} ).
%$$
For $\x=\sigma_1\sigma_2\cdots \sigma_\ell$ a word, we define $\x_{\mathsf{rev}}=\sigma_\ell\cdots\sigma_2\sigma_1$. Then we inductively define
$$
 {\sf cap }_{\x \;\x_{\mathsf{rev}} }^\emptyset :=  {\sf cap }_{{\y }\;{\y_{\mathsf{rev}} } }^{\emptyset }
 ({\sf 1}_{\y}  \otimes 
 {\sf cap}_{{\sigma_\ell }{\sigma_\ell } }^{\emptyset}
 \otimes {\sf 1}_{\y_{\mathsf{rev}}} )
 $$
 where $\y=\sigma_1\sigma_2\cdots \sigma_{\ell-1}.$ This diagram can be visualised as a rainbow of concentric arcs (with ${\sf cap}_{\sigma_\ell \sigma_\ell}^\emptyset$ the innermost arc).   
Since we have $\x\x_{\mathsf{rev}}=1_W$ when evaluated in the group $W$, we will simply write $ {\sf cap }_{\x \x^{-1} }^\emptyset $ for $ {\sf cap }_{\x \;\x_{\mathsf{rev}} }^\emptyset$.

In order to make our notation  less dense, 
we will often  suppress mention of  idempotents by including them in the sub- and super-scripts  of other generators.  
This is made possible by recording where the 
edits to the underlying words are  with emptysets.  
For example 
$${\sf spot}_{ \alphar\betar\blue}^{  \emptyset\betar \emptyset }
:=  {\sf spot}^\emptyset_\alphar
\otimes {\sf 1}_\betar
\otimes {\sf spot}^\emptyset_\blue
\quad
{\sf fork}_{\gam\gam\betar\alphar\alphar}^{\gam  \betar \alphar }
:={\sf fork}_{\gam\gam}^\gam \otimes {\sf 1}_\betar \otimes
{\sf fork}_{\alphar\alphar}^\alphar
\quad 
{\sf bar}(\pink\orange)= {\sf bar}(\pink)\otimes {\sf bar}(\orange).
$$
We make use of all of the above notational shorthands (even within the same equation).  
%One could make an even denser shorthand combining both these ideas, but 
%we do not do this as we think this is a step too far and makes things harder to read and understand, rather than easier.  
Finally, for distinct $\csigma,\ctau \in S_W$ we recall the entries of the Cartan matrix corresponding to the Dynkin diagram of $(W,P)$:
$$
\langle \alpha_\csigma^\vee, \alpha_\ctau \rangle = 
\begin{cases}
0 & \text{if $\csigma 
                  \mathrel{\mkern2mu}% a small advancement
                  \arrownot % a short slash
                  \mathrel{\mkern-2mu}% compensate
                  \mathrel{-}% minus as a relation
                  \joinrel\joinrel % some backup
                  \mathrel{-}% minus as a relation
  \ctau$,} \\
-1 & \text{if $\csigma 
%                  \mathrel{\mkern2mu}% a small advancement
%                  \arrownot % a short slash
%                  \mathrel{\mkern-2mu}% compensate
                  \mathrel{-}% minus as a relation
                  \joinrel\joinrel % some backup
                  \mathrel{-}% minus as a relation
  \ctau$,} \\
-1 & \text{if $\csigma \Longrightarrow \ctau$,} \\
-2 & \text{if $\csigma \Longleftarrow \ctau$.}
\end{cases}
$$

%\begin{rmk}
%\color{black!99}
%
%
%\end{rmk}

%    \subsection{The diagrammatic Hecke categories }
%We now define the Hecke category.

   \begin{defn}\label{the algebra} 
     \renewcommand{\vvv}{{\underline{w} }} 
\renewcommand{\w}{{\underline{x}}}
\renewcommand{\x}{{\underline{y}}}
\renewcommand{\y}{{\underline{z}}}

Let $ W $ be a Weyl group, $P$ be a parabolic subgroup,  
\color{purple}
and let $\Bbbk$ be a field.
%For convenience let us suppose for now that $\Bbbk$ is a field of characteristic not equal to $2$ (see \cref{demazure} for how to remove this assumption).
\color{black}
 We define  $\mathcal{H}_{(W,P)}$ to be the locally-unital  associative $\Bbbk$-algebra spanned by all Soergel-graphs with multiplication given by $\circ$-concatenation modulo the following local relations and their vertical and horizontal flips.
\begin{align}\label{R1}\tag{R1}
{\sf 1}_{\csigma} {\sf 1}_{\ctau}& =\delta_{\csigma,\ctau}{\sf 1}_{\csigma},
 & {\sf 1}_{\emptyset} {\sf 1}_{\csigma} & =0,
  & {\sf 1}_{\emptyset}^2& ={\sf 1}_{\emptyset},
  \\
 \label{R2}\tag{R2}
{\sf 1}_{\emptyset} {\sf spot}_{\csigma}^\emptyset {\sf 1}_{\csigma}& ={\sf spot}_{\csigma}^{\emptyset},
 &
  {\sf 1}_{\csigma} {\sf fork}_{\csigma\csigma}^{\csigma} {\sf 1}_{\csigma\csigma}& ={\sf fork}_{\csigma\csigma}^{\csigma} ,
  &
{\sf 1}_{\ctau\csigma}^m {\sf braid}_{\csigma\ctau}^{\ctau\csigma}(m) {\sf 1}_{\csigma\ctau}^m & ={\sf braid}_{\csigma\ctau}^{\ctau\csigma}(m),
\end{align}For each  $\csigma \in S $  we have the fork-spot, double-fork, circle-annihilation relations 
\begin{align}\label{R3}\tag{R3}
({\sf spot}_\csigma^\emptyset \otimes {\sf 1}_\csigma){\sf fork}^{\csigma\csigma}_{\csigma}
=
{\sf 1}_{\csigma},
 \quad 
  ({\sf 1}_\csigma\otimes {\sf fork}_{\csigma\csigma}^{ \csigma} )
({\sf fork}^{\csigma\csigma}_{\csigma}\otimes {\sf 1}_{\csigma})
=
{\sf fork}^{\csigma\csigma}_{\csigma}
{\sf fork}^{\csigma}_{ \csigma\csigma},
\quad
{\sf fork}_{\csigma\csigma}^{\csigma}
{\sf fork}^{\csigma\csigma}_{\csigma}=0,
\end{align}
pictured in \cref{onecolour1} {together with the cinching relation} 
\begin{align}\label{R4}  \tag{R4}
 {\sf 1}_{\csigma}\otimes {\sf 1}_{\csigma}=	{\sf spot}_\csigma^\emptyset \otimes {\sf fork}^{\csigma\csigma}_\csigma+
{\sf spot}^\csigma_\emptyset \otimes {\sf fork}_{\csigma\csigma}^\csigma		
- {\sf bar}(\csigma) \otimes {\sf dork}^{\csigma\csigma} _{\csigma\csigma}	
\end{align}pictured in \cref{shorter}.
 For every ordered  pair  $(\csigma,\ctau) \in S_W^2$ with $\csigma \neq \ctau$,   the   bi-chrome relations:  The $\csigma\ctau$-barbell, 
\begin{align}\label{R5}
\tag{R5}
{\sf bar}(\ctau)\otimes {\sf 1}_{\csigma}
-
 {\sf 1}_{\csigma} \otimes {\sf bar}(\ctau)
 = \langle \alpha_\csigma^\vee, \alpha_\ctau \rangle ({\sf gap}(\csigma)-
 {\sf 1}_{\csigma} \otimes {\sf bar}(\csigma))
.\end{align}

\begin{figure}[h!]
\begin{equation*} 
   \begin{minipage}{1.5 cm}
  \end{minipage}
 \;=\;
  0
\end{equation*}

\caption{\color{black!99}The    fork-spot, double-fork, and circle-annihilation relations of \eqref{R3}. }
\label{onecolour1}
\end{figure}

\begin{figure}[h!]

  $$  \begin{minipage}{1.5cm}%
  \end{minipage}$$
  \caption{\color{black!99} The cinching relation of \eqref{R4}.		}
  
\label{shorter}
  \end{figure}

\begin{figure}[h!]
  \begin{align*} 
  \begin{minipage}{1.281cm}%
\end{minipage}\; \right)
 \end{align*}
 \caption{\color{black!99} The two-colour barbell relation of \eqref{R5}.}
 
 \end{figure}

\noindent For $m=m_{\csigma\ctau}\in \{2,3,4\}$ we also have  the fork-braid relations 
 \begin{align}\tag{R6}\begin{split}  \label{fork-braid-text1}  
 {\sf braid}_{\csigma \ctau\cdots \ctau \csigma}^{\ctau\csigma\cdots \csigma\ctau} 
({\sf fork}^{\csigma}_{\csigma\csigma}\otimes{\sf 1}_{ \ctau\csigma} ^{m-1} 	)
 ( {\sf 1}_\csigma \otimes{\sf braid}^{  \csigma \ctau\cdots \ctau \csigma}_
 {  \ctau\csigma\cdots \csigma\ctau} )
&=
   ({\sf 1}_{ \ctau\csigma}^{m-1}   \otimes {\sf fork}_{\ctau\ctau}^{ \ctau})
( 
 {\sf braid}_{ \csigma \ctau\cdots \ctau \csigma}^{  \ctau\csigma\cdots \csigma\ctau} 
\otimes {\sf 1}_\ctau) 
%$$
% $$
%\end{split}
\\
%\begin{split} 
% \label{fork-braid-text2}
 {\sf braid}_{\csigma \ctau\cdots  \csigma \ctau}^{\ctau\csigma\cdots \ctau \csigma} 
({\sf fork}^{\csigma}_{\csigma\csigma}\otimes{\sf 1}_{\ctau\csigma} ^{m-1} 	)
 ( {\sf 1}_\csigma \otimes{\sf braid}^{  \csigma \ctau\cdots  \csigma\ctau }_
 {  \ctau\csigma\cdots  \ctau\csigma} )
&=
%  {\sf 1}_{  \ctau\csigma\cdots \ctau \csigma} 
%  ^{\ctau\csigma\cdots \ctau \csigma} 
  ({\sf 1}_{ \ctau\csigma}^{m-1}   \otimes {\sf fork}_{\csigma\csigma}^{\csigma})
( 
  {\sf braid}_{  \csigma \ctau\cdots  \csigma \ctau}^{  \ctau\csigma\cdots \ctau \csigma} 
\otimes {\sf 1}_\csigma) 
 \end{split}
\end{align}
for $m$ odd and even, respectively --- these are pictured in \cref{forkbraid}.  
We   require the cyclicity  relation,  
\begin{align}\tag{R7} \begin{split}\label{R7}
({\sf 1}_{ \ctau\csigma}^{m }\otimes (   {\sf cap}^\emptyset_{\csigma\csigma}))({\sf 1}_\ctau \otimes
 {\sf braid}^{\csigma\ctau}_{ \ctau\csigma}(m)  
\otimes {\sf 1}_\csigma) 
(   {\sf cup}_{\emptyset }^{\ctau\ctau}  \otimes 
{\sf 1}_{ \csigma \ctau}^{m })
&=
  {\sf braid}^{\ctau \csigma\dots \csigma \ctau}_{\csigma \ctau\cdots \ctau\csigma} 
\\
({\sf 1}_{ \ctau\csigma}^{m }\otimes(  {\sf cap}^\emptyset_{\ctau\ctau}))({\sf 1}_\ctau \otimes
 {\sf braid}^{\csigma\ctau}_{ \ctau\csigma}(m)  
\otimes {\sf 1}_\ctau) 
(    {\sf cup} ^{\ctau\ctau}_\emptyset \otimes 
{\sf 1}_{ \csigma \ctau}^{m })
&=
  {\sf braid}^{\ctau \csigma\dots \ctau \csigma}_{\csigma \ctau\cdots \csigma \ctau}    
 \end{split}
\end{align}
for $m$ odd or even, respectively.

\begin{figure}[h]
$$
 \begin{minipage}{1.4cm}
 \end{minipage}$$ 
   
   \caption{\color{black!99}The fork braid relations of \eqref{fork-braid-text1}   for  $m( \csigma , \ctau)=2$ and $3$  and $4$ respectively.  }
   \label{forkbraid}
   \end{figure}

\noindent For $m=2,3,$ or $4 $ we have the double-braid relations\footnote{The double-braid relations can   replace the usual Jones--Wenzl relations, by 
 \cite[Exercise 9.39(1)]{MR4220642}.}
\begin{align} \tag{R8}
%\begin{split}
 \label{braidrelation1}
 {\sf 1}_{  \ctau\csigma } = 
 {\sf braid}_{\csigma\ctau}^{\ctau\csigma}
 {\sf braid}^{\csigma\ctau}_{\ctau\csigma}
\qquad 
% \label{braidrelation2}
 {\sf 1}  _{\csigma\ctau\csigma}=
 {\sf braid}^{\csigma\ctau\csigma}_{\ctau\csigma
\ctau}
{\sf braid}_{\csigma\ctau\csigma}^{\ctau\csigma
\ctau}
-
{\sf spot}^{\csigma\ctau \csigma}_{\csigma\emptyset\csigma}
{\sf fork}^{\csigma\csigma}_\csigma
{\sf fork}_{\csigma\csigma}^\csigma
{\sf spot}_{\csigma\ctau \csigma}^{\csigma\emptyset\csigma}
\end{align}
\begin{align} 
 \nonumber
   {\sf 1}_{\ctau\csigma\ctau\csigma}
   =&\; {\sf braid}
^{\ctau\csigma\ctau\csigma}
_{\csigma\ctau\csigma\ctau}
%{\sf spot}^{\csigma\ctau\csigma\ctau}_{\emptyset \ctau\csigma\ctau}
{\sf braid}
_{\ctau\csigma\ctau\csigma}
^{\csigma\ctau\csigma\ctau}
+ 
 \langle \alpha_\csigma^\vee , \alpha_\ctau%not^\vee
  \rangle 
{\sf spot}^{\ctau\csigma\ctau\csigma}
_{\ctau \emptyset\ctau \csigma}
{\sf fork}^{\ctau\ctau \csigma}_{\ctau\csigma}
{\sf fork}_{\ctau\ctau \csigma}^{\ctau\csigma}
{\sf spot}^{\ctau \emptyset \csigma \ctau }
_{\ctau\csigma\ctau\csigma}
\\ \label{braidrelation3}\tag{R9}
&
 + 
 \langle \alpha_\ctau^\vee , \alpha_\csigma%not^\vee
  \rangle 
{\sf spot}^{\ctau\csigma\ctau\csigma}
_{\ctau\csigma\emptyset\csigma}
{\sf fork}^{\ctau\csigma\csigma}_{\ctau\csigma}
{\sf fork}_{\ctau\csigma\csigma}^{\ctau\csigma}
{\sf spot}^{\ctau\csigma\emptyset}
_{\ctau\csigma\ctau\csigma}
  \\  & \nonumber   
  -
{\sf spot}_{\ctau \emptyset\ctau \csigma}^{\ctau\csigma\ctau\csigma}
{\sf fork}^{\ctau \ctau\csigma}_{\ctau\csigma}
{\sf fork}_{\ctau \csigma\csigma}^{\ctau\csigma}
{\sf spot}_{\ctau\csigma\ctau\csigma}^{\ctau\csigma \emptyset\csigma}
-
{\sf spot}^{\ctau\csigma\ctau\csigma}_{\ctau\csigma \emptyset\csigma}
{\sf fork}^{\ctau \csigma\csigma}_{\ctau\csigma}
{\sf fork}_{\ctau \ctau\csigma}^{\ctau\csigma}
{\sf spot}^{\ctau \emptyset\ctau \csigma}_{\ctau\csigma\ctau\csigma}
 \end{align}
 respectively, pictured in \cref{2braid1,2braid2}.  For   $(\csigma, \ctau ,\crho)\in S_W^3$ 
with   $m_{\csigma \crho}=
m_{  \crho\ctau}=2$ and 
$m_{\csigma \ctau}=m$,  we have the commutating-braids relation
\begin{align}\tag{R10}\label{R10-com-braid}
\begin{split}
({\sf braid}_{\csigma\ctau\dots \ctau\csigma}^{\ctau\csigma\dots \csigma \ctau}  \otimes {\sf 1}_
\crho)
{\sf braid}^{\csigma \ctau \cdots \ctau\csigma  \crho}_{\crho\csigma \ctau \cdots \ctau \csigma}%(2)
&=
{\sf braid}^{\ctau \csigma\cdots \csigma\ctau \crho}_{\crho\ctau \csigma\cdots \csigma\ctau}%(2)
({\sf 1}_
\crho\otimes {\sf braid}_{\csigma\ctau\dots  \ctau\csigma}^{\ctau\csigma\dots \csigma\ctau} )
\\
({\sf braid}_{\csigma\ctau\dots  \csigma \ctau}^{\ctau\csigma\dots   \ctau\csigma}  \otimes {\sf 1}_
\crho)
{\sf braid}^{\csigma \ctau \cdots \csigma\ctau  \crho}_{\crho\csigma \ctau \cdots \csigma\ctau}%(2)
&=
{\sf braid}^{\ctau \csigma\cdots \csigma\ctau \crho}_{\crho\ctau \csigma\cdots \csigma\ctau}%(2)
({\sf 1}_
\crho\otimes {\sf braid}_{\csigma\ctau\dots  \ctau\csigma}^{\ctau\csigma\dots \csigma\ctau} )
.\end{split}
\end{align}
for $m$ odd or even respectively, this is pictured in \cref{combraid}.

\begin{figure}[h!]

 \begin{align*}
 \begin{minipage}{1.1cm}
\end{minipage} 
     \end{align*}

  \caption{\color{black!99}Double braid relations for $m( \csigma , \ctau)=2$ and $3$ of \eqref{braidrelation1} respectively.   }
  \label{2braid1}
  \end{figure}

\begin{figure}[ht!]

 \begin{align*}   
  \begin{minipage}{2cm}    %
\end{minipage}
  \end{align*}

  \caption{\color{black!99}Double braid relations for $ \csigma = {\color{magenta} s_2}$ and 
 $ \ctau = {\color{cyan} s_1}$  in type $W=C_n$ as in 
\eqref{braidrelation3}.}
  \label{2braid2}  
  \end{figure}

\begin{figure}[h]
$$
 \begin{minipage}{1.4cm}
 %
 \end{minipage}  
   $$ 
   
   \caption{\color{black!99}The commuting braids relation of \eqref{R10-com-braid} for
      $m( \csigma , \crho)=2=m( \ctau , \crho)$  and 
      $m( \csigma , \ctau)=2$ and $3$  and $4$ respectively.  }
   \label{combraid}
   \end{figure}

\noindent  Finally, we have the  Zamolodchikov relations:    for a   triple $\al,\bet, \rho\in S_W $ with $m_{\al\bet}=3=m_{\al\gam}$ and $m_{\al\gam}=2$ we have  that 
  \begin{align}\tag{R11}\begin{split} 
  &{\sf braid}^{ \gam\al\gam \bet\al\gam}_{ \al\gam\al  \bet\al\gam    }    
  {\sf braid}^{ \al\gam  \al\bet\al \gam    }_{  \al\gam \bet\al\bet    \gam   }    
   {\sf braid}^{ \al\gam\bet \al \bet\gam  }_{ \al\bet\gam \al  \gam\bet   }  
   {\sf braid}^{   \al\bet	\gam\al\gam \bet   }_{ \al\bet \al\gam\al   \bet }   
 {\sf braid}^{  \al\bet	 \al  \gam\al \bet  } _{  \bet \al\bet  \gam\al \bet    }   
{\sf braid}^{\bet\al \bet\gam\al\bet     }_{ \bet\al \gam\bet \al\bet     } 
\\[5pt] = \;\;\;&  	{\sf braid}_{\gam\al\bet\gam \al\gam  }^{ \gam\al \gam\bet\al\gam  }  %%%
%%%
  {\sf braid}^{\gam\al\bet \gam\al \gam }_{\gam\al\bet\al\gam\al} 
   {\sf braid}^{\gam\al\bet\al\gam\al}_{\gam\bet\al\bet\gam\al}
 {\sf braid}^{\gam\bet\al\bet\gam\al}_{\bet\gam\al\gam\bet\al}   
 {\sf braid} ^{\bet\gam\al\gam\bet\al} 
 _{\bet \al\gam\al\bet\al} 
  {\sf braid}^{\bet\al\gam\al\bet\al}_{\bet\al\gam\bet\al\bet} .
  \end{split}
\intertext{ For a     triple $\al,\bet,\gam\in S$ such that 
$m_{\al\gam}=4$, $m_{\bet\gam}=2$, $m_{\al\bet}=3$, we have that 
}
\tag{R12}
\begin{split} 
  &  {\sf braid}_{\gam\al\gam {\al\bet\al}{\gam\al\bet}}^{\gam\al\gam{\bet\al\bet} {\gam\al\bet}} 
    {\sf braid}^{\gam\al\gam{\al\bet\al}{\gam\al\bet}}
    _ { \al\gam{\al\gam\bet\al}{\gam\al\bet}}
     {\sf braid} 
    ^ { \al\gam{\al\gam\bet\al}{\gam\al\bet}}
    _ { \al\gam \al\bet \gam \al \gam\al\bet} 
    {\sf braid}^{ \al\gam \al\bet \gam \al \gam\al\bet} _
    { \al\gam \al\bet  \al \gam\al \gam\bet} 
%\\
    {\sf braid}^ 
    { \al\gam \al\bet  \al \gam\al \gam\bet} 
_    { \al\gam \bet  \al \bet  \gam\al  \gam \bet} 
% \\
 \times
\\[4pt]  
%\begin{split} 
&   {\sf braid}^ 
     { \al\gam \bet  \al \bet  \gam\al   \gam\bet} _
  { \al  \bet \gam \al    \gam \bet  \al  \bet\gam}  {\sf braid}^ 
  { \al  \bet \gam \al    \gam \bet  \al  \bet\gam} _
    { \al  \bet \gam \al    \gam   \al  \bet  \al \gam} 
 {\sf braid}^ 
     { \al  \bet \gam \al    \gam   \al  \bet  \al \gam} 
     _
         { \al  \bet   \al    \gam   \al \gam \bet  \al \gam} 
 {\sf braid}^ 
         { \al  \bet   \al    \gam   \al \gam \bet  \al \gam} 
         _
  {   \bet   \al  \bet  \gam   \al\gam\bet     \al \gam}   
 {\sf braid}^ 
      {   \bet   \al  \bet  \gam   \al\gam\bet     \al \gam}   
      _
  {   \bet   \al     \gam\bet   \al \bet  \gam   \al \gam}         
\\[7pt]  %\end{align*}
%is equal to 
%   \begin{align*}
   =\;\;\;& {\sf braid}  
    ^
      {      \gam   \al \gam  \bet	 \al   	 \bet 	 \gam    \al  \bet   }   
_
  {      \gam   \al \gam  \bet	 \al 	 \gam   	 \bet   \al  \bet   }   
    {\sf braid}  
    ^
    {      \gam   \al \gam  \bet	 \al 	 \gam   	 \bet   \al  \bet   }   
    _
        {      \gam   \al    \bet	\gam \al 	 \gam   	  \al  \bet    \al }   
  {\sf braid}  
    ^
         {      \gam   \al    \bet	\gam \al 	 \gam   	  \al  \bet    \al }   
         _
                 {      \gam   \al    \bet		  \al 	 \gam   	  \al 	\gam			 \bet    \al }   
  {\sf braid}  
    ^
                          {      \gam   \al    \bet		  \al 	 \gam   	  \al 	\gam			 \bet    \al }   
                          _
                 {      \gam      \bet		  \al    \bet	 \gam   	  \al 	\gam			 \bet    \al }                             
  {\sf braid}  
    ^
                {      \gam      \bet		  \al    \bet	 \gam   	  \al 	\gam			 \bet    \al }                             
_
                 {          \bet	 \gam  	  \al    \gam   \bet	  	  \al 	 		 \bet \gam	   \al }                                                 
     \times
\\[4pt]     % 
 & {\sf braid}  
    ^
    {          \bet	 \gam  	  \al    \gam   \bet	  	  \al 	 		 \bet \gam	   \al } 
_
    {          \bet	 \gam  	  \al    \gam     \al   \bet     \al 	 \gam	   \al }   {\sf braid}  
    ^
      {          \bet	 \gam  	  \al    \gam     		\al   \bet     \al 	 \gam	   \al }     
      _
          {    \bet	   \al    \gam     \al  \gam  	 	 \bet     \al 	 \gam	   \al }     
 {\sf braid}  
    ^
          {    \bet	   \al    \gam     \al  \gam  	 	 \bet     \al 	 \gam	   \al }     
          _
    {    \bet	   \al    \gam     \al   	 \bet   \gam  	   \al 	 \gam	   \al }               
 {\sf braid}  
    ^
         {    \bet	   \al    \gam     \al   	 \bet   \gam  	   \al 	 \gam	   \al }               
_
    {    \bet	   \al    \gam     \al   	 \bet     	   \al 	 \gam	   \al \gam 		}                        
 {\sf braid}  
    ^
        {    \bet	   \al    \gam     \al   	 \bet     	   \al 	 \gam	   \al \gam 		}      
_      {    \bet	   \al    \gam      	 \bet     	   \al 	 	 \bet      \gam	   \al \gam 		}      
                         .   \end{split}   \end{align}
 Further,  we require the %bifunctoriality relation
  interchange law
\begin{align}\label{R8}\tag{R13}
 \big( {\sf D  }_1 \otimes   {\sf D}_2   \big)\circ  
\big(  {\sf D}_3  \otimes {\sf D }_4 \big)
=
 ({\sf D}_1 \circ   {\sf D_3}) \otimes ({\sf D}_2 \circ  {\sf D}_4)
\end{align} 
and the monoidal unit relation
\begin{align}\label{R9}\tag{R14}
{\sf 1}_{\emptyset} \otimes {\sf D}_1={\sf D}_1={\sf D}_1 \otimes {\sf 1}_{\emptyset}
\end{align}
for all diagrams ${\sf D}_1,{\sf D}_2,{\sf D}_3,{\sf D}_4$. 
%\end{defn}
%
%
%
%\begin{defn}\label{cyclotomic22}Let $ W $ be a Weyl group  and $P$ be a parabolic subgroup.  
%We define  $\mathcal{H}_{(W,P)}$ to be quotient of  $\mathcal{H}_{W}$ by the 
Finally, we require the following {\em non-local } cyclotomic relations 
\begin{align}\label{R10}\tag{R15}
{\sf bar}(\csigma) \otimes D =0 \qquad &\text{for all $\csigma \in S_W$ and $D$ any diagram,  }
 \\
\label{R11}\tag{R16}
{\sf 1}_\ctau \otimes D=0
 \qquad &\text{for  all
  $\ctau \in S_P\subset  S_W$ and $D$ any diagram.   }
\end{align}
  
 \end{defn}

\renewcommand{\w}{{\underline{w}}}
\renewcommand{\vvv}{{\underline{z} }}
\renewcommand{\y}{{\underline{y}}}
\renewcommand{\x}{{\underline{x}}}

The algebra    $\mathcal{H}_{  (W,P)}$ can be equipped with 
 a $\mathbb Z$-grading which preserves the duality~$\ast$.  The degrees  of  the generators under this grading are defined  as follows:
 $$
 {\sf deg}({\sf 1}_\emptyset)=0
 \quad
  {\sf deg}({\sf 1}_\al)=0
  \quad
  {\sf deg} ({\sf spot}^\emptyset_\al)=1
    \quad
  {\sf deg} ({\sf fork}^\al_{\al\al})=-1
    \quad
  {\sf deg} ({\sf braid}^{\al\bet}_{\bet\al}(m))=0
 $$for $\al,\bet \in S_W$   arbitrary and $m\geq 2$.  
%\end{rmk}
\color{black!99}We will define the $p$-Kazhdan--Lusztig polynomials of these categories in \eqref{hjhjhjhas}.

\begin{rmk}
\color{black!99}Of the defining  relations of \cref{the algebra}, we have only diagrammatically depicted those  
which   will explicitly appear  in the arguments of this paper; the   remaining relations can be found for example in  \cite{MR3555156,mybook}).
\end{rmk}

\begin{rmk}
We can pre- and post-multiply the relation depicted in  \cref{shorter} with spot generators. We hence obtain the usual ``one colour Demazure relation''  as follows:
 \color{black} 
 \begin{equation} \label{onecolourdemazure}
  \begin{minipage}{1.5cm}  \end{minipage}  \end{equation}
%That the one-colour Demazure relation implies 
%  the relation in \cref{shorter}  
%    \cite[Equation 5.15]{MR3555156}
  \end{rmk}

 \color{purple}

 \begin{rmk} \label{demazure} %[AMIT PLEASE ADD TO THIS... maybe use Bowman--Hazi--Norton]
The Hecke category is usually defined in the literature using the one-colour Demazure relation \eqref{onecolourdemazure} instead of the cinching relation \eqref{R4}, plus an additional technical assumption called {\em Demazure surjectivity} \cite[Assumption 3.7]{MR3555156}. 
In this setting Demazure surjectivity is necessary to prove the cinching relation, which is essential for the Hecke category to be well behaved (more precisely, for the light leaves construction in \cref{lightleaves} to yield a basis). 
However a careful analysis of the proof in \cite[\S 7]{MR3555156} shows that Demazure surjectivity is \emph{not} necessary for the Hecke category to be well behaved if one assumes the cinching relation to begin with!
In other words, our definition of $\mathcal{H}_{ (W,P) }$ is \emph{always} well behaved, and is equivalent to the usual definition of the Hecke category in the literature when the latter is well behaved.
We believe this trivial observation has been overlooked until now due to historical motivation of the diagrammatic Hecke category from Soergel bimodules, which rely on Demazure surjectivity much more heavily.
 \end{rmk}

 \color{black}

 \section{Lifting full-commutativity to the \\ Hecke categories  of Hermitian symmetric pairs  }

% We now provide a dramatic simplification of the presentation of $\mathcal{H}_{(W,P)}$ for Hermitian symmetric pairs.  This is essential for the proof of our main theorem.  In a nutshell it says that any (non-commuting) braid generator is zero in  $\mathcal{H}_{(W,P)}$ for $(W,P)$ any Hermitian symmetric pair.  

Stembridge proved that the parabolic quotients   for Hermitian symmetric pairs are fully commutative \cite[Theorem 6.1]{MR1406459}. In other words, 
%when passing between two words for two cosets
  in $^PW$ the non-commuting  braid {\em relations} are redundant.  
  (This is discussed in more detail in terms of Temperley--Lieb  diagrammatics in the companion paper \cite{FARR}.) 
We now lift this idea to the 2-categorical level;  
  in $\mathcal{H}_{(W,P)}$ the non-commuting braid {\em generators} are redundant.

\begin{rmk}
 \label{locality}\color{black!99}
% By construction, any diagram $D \in \mathcal{H}_{(W,P)}$
%We will now state and prove
By a    ``local relation''    of the algebra $\mathcal{H}_{(W,P)}$ we mean any relation that can be applied in {\em an arbitrary} local neighbourhood of a diagram 
(as opposed to the relations in \eqref{R10} and \eqref{R11} which can only be applied to the leftmost edge of a diagram).  
In this section we state and  prove new {\bf local relations} of $\mathcal{H}_{(W,P)}$ --- our proofs will make  use of  the {\bf non-local} relations of  \eqref{R10} and \eqref{R11}.  
To see how this works, we first observe that  
any element  of  $  \mathcal{H}_{(W,P)}$ can be obtained 
by vertical concatenation  of diagrams of the form ${\sf 1}_{\x} \otimes D \otimes {\sf 1}_{\y}$ for $\x$ and $\y$ two expressions for $x,y \in W$ and $D$ a diagram 
as in \eqref{genrdiag} or \eqref{genrdiag2}. 
 Thus proving a local relation $ D_1 \circ D_2 = D_3$ is equivalent to proving the non-local relation 
 $${\sf 1}_{\x} \otimes D_1 \circ D_2 \otimes {\sf 1}_{\y}=   {\sf 1}_{\x} \otimes D_3 \otimes {\sf 1}_{\y}$$for all possible $\x$ and $\y$ for $x,y \in W$. 
 In fact,  by \cite[Theorem 5.3]{MR4437613}, it is enough to consider only   {\em reduced} expressions $\x,\y$ for $x,y \in W$.
\end{rmk}

\begin{thm}\label{simpligy}\label{makeahole}
%Let $w\in {^P}W$ and let $\w$ be any expression for $w$.  
In $\mathcal{H}_{(W,P)}$ we have the  local relation 
 $
  {\sf braid}^{\csigma\ctau}_{\ctau\csigma}(m)=0
 $ for any $m=m(\csigma,\ctau)>2$.   
\end{thm}

%\color{black}  Given a tile-partition $\mu$ and $\sts\in\Std(\mu)$, we denote by ${\sf 1}_\sts$ the element 
%${\sf 1}_{\sigma_1}\otimes \cdots \otimes {\sf 1}_{\sigma_\ell}\in\mathcal{H}_{(W,P)}$ where $\sigma_i$ is the colour of the tile $\sts^{-1}(i)$ for each $i=1,\ldots, \ell=\ell(\mu)$. 
%\color{black}
%
% \begin{thm}\label{simpligy}\label{makeahole}
% Let $(W,P)$ be a  Hermitian symmetric pair.  
% Let $\sts\in \Std(\mu)$   for any $\mu\in \mptn$.  
%In $\mathcal{H}_{(W,P)}$  we have that 
%\begin{equation}\label{needed2}
% {\sf 1}_\sts \otimes {\sf braid}^{\csigma\ctau}_{\ctau\csigma}(m)=0
%\end{equation}
% for any $m=m(\csigma,\ctau)>2$.  
% \end{thm}

%
% Before proceeding, we first record a useful relation which we will use in what follows. For $\csigma \in S $, we have that 
%
% which can be found (with a proof) in   \cite[Equation 5.15]{MR3555156}.   We also record the following useful observation.
% 
 
\begin{prop}\label{iftheshoecommutes}
 Let $w\in {^P}W$ and let $\w, \w'$ be  a pair of reduced expressions for $w$.  
We have that 
$${\sf1}_\w= {\sf braid}^{\w}_{\w'}
{\sf braid}_{\w}^{\w'}
\qquad
\quad
{\sf1}_{\w'}=  
{\sf braid}_{\w}^{\w'}  {\sf braid}^{\w}_{\w'} $$
 \end{prop}

\begin{proof}
For $(W,P)$ a Hermitian symmetric pair,   \cite[Theorem 6.1]{MR1406459} implies that any two words $\w$ and $\w'$ in ${^PW}$ differ only by application of the commuting braid relations of the Coxeter group.   These lift to commuting braid generators in $\mathcal{H}_{(W,P)}$ and the result follows.
\end{proof}

     We are now ready to prove the main result of this section.  

 \begin{proof}[Proof of \cref{makeahole}]
%Let $w\in {^P}W$ and let $\w$ be any reduced  expression for $w$.  
As in \cref{locality}, it will suffice to  prove that 
$$
{\sf 1}_{\sts } \otimes {\sf braid}^{\csigma\ctau}_{\ctau\csigma}(m)=0,
$$for all   $\sts \in \Std(\la)$ and all    $\la \in \mptn$. 	
We will proceed   by induction on the   $\ell=\ell(\la)$.
If $\la =\varnothing$, then ${\sf braid}^{\csigma\ctau}_{\ctau\csigma}(m)=0$ for all $m\geq 2$ by relation \ref{R11}.  
  In what follows, we only explicitly consider the cases for which neither of  
$ {\sf 1}_{\sts  } \otimes  {\sf 1}_{\csigma } $
 or $ {\sf 1}_{\sts  } \otimes  {\sf 1}_{ \ctau }$ is equal to zero 
 by application of the commutativity and cyclotomic relations (as these cases are trivial).  
Now assume that $\ell(\la)\geq 1$.   
%By Stembridge's theorem, neither $\w \csigma\ctau\csigma$
%nor $\w  \ctau\csigma\ctau$ are reduced words, this will be key in what follows. 
% 
%%
%
We have two cases to consider. 

\smallskip
\noindent
{\bf Case 1. }
% Either $\w=\x\csigma$ or $\w=\x\ctau$  for 
 Either $ \csigma $ or $\ctau \in {\sf Rem}(\la)$. 
% some $\x$ a reduced expression of some $x\in {^PW}$.  
We consider the former case, as the latter is identical and we set $\mu = \la-\csigma$ and 
we can assume that $\stt \in \Std(\mu)$ so that $\stt \otimes \csigma = \sts$ (by \cref{iftheshoecommutes}).  
For $m>2$, we have that
\begin{align*}
&{\sf 1}_{\sts} \otimes {\sf braid}^{\csigma\ctau}_{\ctau\csigma}(m)
\\
=&{\sf 1}_{\stt} \otimes {\sf 1}_\csigma\otimes 
{\sf braid}^{\csigma\ctau}_{\ctau\csigma}(m)
\\
=&{\sf 1}_{\stt} \otimes 
({\sf 1}_\csigma\otimes {\sf 1}_\csigma\otimes {\sf 1}_{\ctau\csigma }^{m-1})(
{\sf 1}_\csigma\otimes 
{\sf braid}^{\csigma\ctau}_{\ctau\csigma}(m))
\\
=&{\sf 1}_{\stt} \otimes 
((	{\sf spot}_\csigma^\emptyset \otimes {\sf fork}^{\csigma\csigma}_\csigma+
{\sf spot}^\csigma_\emptyset \otimes {\sf fork}_{\csigma\csigma}^\csigma		
- {\sf bar}(\csigma) \otimes {\sf dork}^{\csigma\csigma} _{\csigma\csigma}	)
\otimes {\sf 1}_{\ctau\csigma}^{m-1})
(
{\sf 1}_\csigma\otimes 
{\sf braid}^{\csigma\ctau}_{\ctau\csigma}(m)).
\end{align*}{\color{black!99}where the first two equalities are trivial and the 
final equality is an application of  the cinching relation  \eqref{R4} visualised in \cref{shorter}.  
This is depicted in diagrammatically in \cref{apicforme}}. 
The first term is zero by induction, since it factors through ${\sf 1}_\stt\otimes {\sf braid}^{\csigma\ctau}_{\ctau\csigma}(m)
$ with $\ell(\stt)=\ell(\sts)-1$.  The other two terms factor through a diagram of the form
\begin{align*}
{\sf 1}_{\stt} \otimes
 ({\sf fork}_{\csigma\csigma}^\csigma	\otimes {\sf 1}_{\ctau\csigma}^{m-1})
(  {\sf 1}_{\csigma}	
\otimes  {\sf braid}^{\csigma\ctau}_{\ctau\csigma}(m) ),
\end{align*}
and  we can  apply the  {\color{black!99} fork-braid relation
 \eqref{fork-braid-text1}   visualised in 
   \cref{forkbraid}}
and hence 
  obtain 
 \begin{align*}
{\sf 1}_{\stt} \otimes 
{\sf braid}^{\csigma\ctau\csigma}
_{ \ctau\csigma\ctau}
({\sf 1}_{\ctau\csigma}\otimes {\sf fork}^{\ctau}_{\ctau\ctau}) 
({\sf braid}_{\csigma\ctau\csigma}
^{ \ctau\csigma\ctau}\otimes {\sf 1}_\ctau)
%\end{align*}
%for $m=3$ and 
% \begin{align*}
\\
{\sf 1}_{\stt} \otimes 
{\sf braid}^{\csigma\ctau\csigma\ctau}
_{ \ctau\csigma\ctau\csigma}
({\sf 1}_{\ctau\csigma\ctau}\otimes {\sf fork}^{\csigma}_{\csigma\csigma}) 
({\sf braid}^{\ctau\csigma\ctau\csigma}
_{\csigma \ctau\csigma\ctau}\otimes {\sf 1}_\csigma)
\end{align*}
for $m=3 $ or $4$ respectively. In both cases, this is zero by induction.

\begin{figure}[ht!]

$$  
{\sf 1}_{\sts}\otimes \begin{minipage}{1.25cm}
\end{minipage} 
 $$

  \caption{\color{black!99}Case 1  of the proof of \cref{makeahole} with  $\sts=\stt\otimes \csigma$ and $m(\csigma,\ctau)=3$. The first term after the second equality is zero by induction, the latter two terms require  an application of the fork-braid relation pictured in  \cref{forkbraid}  before they can be deduced to be zero. } 
  \label{apicforme}
 \end{figure}

\smallskip
\noindent
{\bf Case 2. }
It remains to consider the case that {\color{black!99}  $ \csigma,\ctau \not \in {\sf Rem}(\la)$}. 
We first note that if there exists $\crho \in {\sf Rem}(\la)$ with 
$m(\csigma,\crho)=2=m(\ctau,\crho)$ then we can assume (by \cref{iftheshoecommutes}) that 
$\sts =\stt \otimes \crho$ for  $\stt \in \Std(\la-\crho)$   and  %by induction 
$${\sf 1}_\sts \otimes {\sf braid}^{\csigma\ctau}_{\ctau\csigma}(m)=
{\sf 1}_\stt \otimes	({\sf braid}^{\crho\ctau\csigma\dots }_{ \ctau\csigma\dots \crho}
(	 {\sf braid}^{\csigma\ctau}_{\ctau\csigma}(m)\otimes {\sf 1}_\crho 	)
	 {\sf braid}_{\crho\ctau\csigma\dots }^{ \ctau\csigma\dots \crho})		=0$$where the second equality follows by the
	  commuting braid relation \eqref{R10-com-braid} pictured in \cref{combraid}. 
and the third follows by induction  (as $\ell(\stt)=\ell(\sts)-1$).
{\color{black!99}Thus  for the remainder of Case 2  we can assume without loss of generality that $\sts=\stt \otimes \crho$ 
such that   $\crho\neq\ctau,\csigma$ and that   $m(\csigma,\crho)>2$ (and that  $ \csigma , \ctau \not \in {\sf Rem}(\la)$). }

\smallskip
\noindent 
{\bf Subcase when $\csigma \not \in {\sf Add}(\la)$. }  
We first consider the  ``generic" case in which
$\csigma \not \in {\sf Add}(\la)$.  Our assumptions (and inspection of the  $\mathscr{A}_{(W,P)}$) imply that 
 $\csigma \in {\sf Rem}(\la-\crho)$; 
   furthermore if $m( \crho,\csigma)=4$, then  $\crho \in {\sf Rem}(\la-\crho-\csigma)$. 
By \cref{iftheshoecommutes} we can assume that 
\begin{align}\label{sdfhhkdrsdhsdgjhdfjhaf}
\sts =
\begin{cases}
\stt \otimes  \csigma \otimes \crho 		&\text{for }m(\crho,\csigma)=3 \text { and some }  \sts \in \Std(\la -\crho-\csigma)			\\
\stt \otimes   \crho\otimes  \csigma \otimes \crho &\text{for }m(\crho,\csigma)=4 \text { and    some }\sts \in \Std(\la -\crho-\csigma-\crho)
%\x \crho\csigma \crho &\text{if }m(\crho,\csigma)=4	 
\end{cases}
\end{align}
this covers all instances of types $(A,A\times A)$, $(C,A)$, and $(D,A)$ cases.  
  We consider the first generic case in which  $m(\crho,\csigma)=3	$ (and $m=m(\ctau,\csigma)>2$).  We have  
\begin{align}
\begin{split}\label{dfkjdjfkdfk}
&{\sf 1}_\stt\otimes {\sf 1}_{\csigma\crho}
 \otimes {\sf braid}^{\csigma\ctau }
_{ \ctau\csigma }(m)
\\
 =&
 {\sf 1}_\stt\otimes
 ((
{\sf braid}^{\csigma\crho\csigma}_{\crho\csigma\crho}
{\sf braid}_{\csigma\crho\csigma}^{\crho\csigma\crho}\otimes 
{\sf1}_{\ctau\csigma}^{m-1} 
-{\sf spot}^{\csigma\crho\csigma}_{\csigma\emptyset\csigma}
{\sf dork}^{\csigma\csigma}_{\csigma \csigma }
{\sf spot}_{\csigma\crho\csigma}^{\csigma\emptyset\csigma}
))
 ( {\sf 1}_{\csigma\crho}
 \otimes {\sf braid}^{\csigma\ctau }
_{ \ctau\csigma }(m))
\end{split}
\end{align}
using the double-braid relation.  We now observe that 
\begin{align*}
{\sf 1}_\stt\otimes
{\sf braid}^{\csigma\crho\csigma}_{\crho\csigma\crho}\otimes 
{\sf1}_{\ctau\csigma}^{m-1}  =0\qquad \qquad 
 {\sf 1}_\stt\otimes   {\sf 1}_{\csigma}\otimes {\sf spot}_\crho^\emptyset 
 \otimes {\sf braid}^{\csigma\ctau }
_{ \ctau\csigma }(m) =0
\end{align*}
by induction (since $\ell(\stt)$,  $\ell(\stt\otimes \csigma)<\ell(\sts)$) and so both terms in \eqref{dfkjdjfkdfk} are zero, as  required.   
We now consider the second generic case in which $m(\crho,\csigma)=4	$ (which implies that $m=m(\ctau,\csigma)=3$).  We have   that 
\begin{align*}
 &{\sf 1}_\stt\otimes {\sf 1 }_{\crho\csigma\crho }
 \otimes \braid^{\csigma\ctau \csigma }
_{ \ctau\csigma \ctau } %\otimes {\sf 1}_{\ctau\csigma}
\\
 =
 &{\sf 1}_\stt\otimes
 (  \braid^{\crho\csigma\crho\csigma }_{\csigma\crho\csigma\crho }
   \braid_{\crho\csigma\crho\csigma}^{\csigma\crho\csigma\crho }\otimes {\sf 1}_{\ctau\csigma}
 + 
 \langle \alpha_\csigma^\vee , \alpha_\crho 
  \rangle 
%\spot^{\crho\csigma\crho\csigma\ctau\csigma}
%_{\crho\emptyset\crho\csigma\ctau\csigma}
%\fork^{\crho\crho\csigma\ctau\csigma}_{\crho\csigma\ctau\csigma}
%\fork_{\crho\crho\csigma \ctau\csigma}^{\crho\csigma\ctau\csigma}
%\spot^{ \crho\emptyset\crho \csigma\ctau\csigma}
%_{\crho\csigma\crho\csigma\ctau\csigma}
\spot^{\crho\csigma\crho }
_{\crho\emptyset\crho  }
{\sf dork}^{\crho\crho  }
%_{\crho  }
%\fork
_{\crho\crho   }%^{\crho  }
\spot^{ \crho\emptyset\crho   }
_{\crho\csigma\crho  }\otimes {\sf 1}_{\csigma\ctau\csigma}
\\& + 
 \langle \alpha_\crho^\vee , \alpha_\csigma
  \rangle 
\spot^{\crho\csigma\crho\csigma } %\ctau\csigma}
_{\crho\csigma \emptyset\csigma} %\ctau\csigma}
\fork^{\crho \csigma\csigma}%\ctau\csigma}
_{\crho\csigma} %\ctau\csigma}
\fork_{\crho\csigma\csigma }%\ctau\csigma}
^{\crho\csigma}%\ctau\csigma}
\spot^{ \crho  \csigma \emptyset \csigma}%\ctau\csigma}
_{\crho\csigma\crho\csigma}%\ctau\csigma}
\otimes {\sf 1}_{\ctau\csigma}
\\
&-
\spot^{\crho\csigma\crho\csigma\ctau\csigma}_{\crho\emptyset\crho\csigma\ctau\csigma}
\fork^{\crho\crho\csigma\ctau\csigma}_
{ \crho\csigma\csigma\ctau\csigma}
\spot^{ \crho\csigma\emptyset\csigma\ctau\csigma}_{ \crho\csigma\crho\csigma\ctau\csigma})
-\spot^{\crho\csigma\crho\csigma\ctau\csigma}_{\crho\csigma\emptyset\csigma\ctau\csigma}
\fork^{\crho\csigma \csigma\ctau\csigma}_{\crho\crho  \csigma\ctau\csigma}
\spot^{\crho\emptyset\crho  \csigma\ctau\csigma}
_{\crho\csigma\crho  \csigma\ctau\csigma})
({\sf 1}_{\crho\csigma\crho}\otimes \braid^{\csigma\ctau\csigma}_{ \ctau\csigma\ctau})
%  \end{split}
\end{align*}
and all of these terms are zero by induction on length (similarly to the $m(\crho,\csigma)=3$ case, above).
The righthand-side of this equation is depicted  in \cref{RHSfuck} in type $(C,A)$.

\begin{figure}[h!]

$$ 
{\sf 1}_{\stt}\otimes \left(\; \begin{minipage}{2.4cm}
\end{minipage}
\;
\right)$$

\caption{\color{black!99} Rewriting the diagram ${\sf 1}_\stt\otimes {\sf 1 }_{\crho\csigma\crho } \otimes \braid^{\csigma\ctau \csigma }_{\ctau\csigma\ctau}$ in type $(C,A)$. The type $(B,B)$ case is similar, but with the coefficients on the last two terms switched. }
\label{RHSfuck}

\end{figure}

\begin{figure}[h!]
  $$
   %
  $$ 
\caption{\color{black!99}
 Some examples of the ``exceptional" examples
 for which $\csigma \in {\sf Add}(\la)$ in Case 2 of the proof of \cref{jfgdkhgkjfdhglkjsdfhglsjkdfghlskdjfg}.
  The region $\mu\subseteq \la$  is pictured in yellow (with $\mu$ pictured in grey). 
   In the first case $( \csigma,\ctau )=({\color{magenta}s_2},{\color{cyan}s_1})$ and $\crho={\color{green!80!black}s_3}$.
  In the second case, 
  $( \csigma,\ctau )=({\color{magenta}s_4},{\color{cyan}s_5})$ and $\crho={\color{green!80!black}s_3}$.
  In the third case
   $( \csigma,\ctau )=({\color{magenta}s_3},{\color{cyan}s_6})$ and $\crho  \in\{s_2,s_4\}$. 
    The colouring of nodes is chosen to 
    emphasise the roles of  $\csigma,\ctau,\crho$  in each case (and so is inconsistent with the colouring of \cref{hhhh}).
    }
\label{jfgdkhgkjfdhglkjsdfhglsjkdfghlskdjfg}
\end{figure}

\smallskip
\noindent 
{\bf Exceptional subcases where $\csigma  \in {\sf Add}(\la)$. }  
If $\csigma  \in {\sf Add}(\la)$, this implies that $\ctau \not  \in {\sf Add}(\la)$, 
\color{purple}
because the addable tiles for a tile partition must commute with each other. 
(This is a general fact about fully commutative elements of Coxeter groups; for the analogous statement regarding removable tiles see e.g.~the proof of \cite[Theorem 4.2]{MR1406459}.)
\color{black}
 It remains to consider this case.  
We remark that  our assumptions on $\la$ and the fact that
$\csigma  \in {\sf Add}(\la)$ implies that we must be  
  in one of types $(D_n,D_{n-1})$, $(B_{n},B_{n-1})$ and exceptional type and so we refer to this as the ``exceptional" case.   
%In such types, we do not {\em always} have words of the form 
%in \ref{sdfhhkdrsdhsdgjhdfjhaf}.  
In this case, we can write $\mu\subseteq \la$ where $|\la|-|\mu|=L$ is  maximal  such that 
$\stt_{\mu\setminus\la}=s_{i_1}\dots s_{i_L}$ and $m(s_{i_k},\ctau)=2$ for all $1\leq k \leq L$ and 
%there exists  $\sts' \in \Std(\mu)$ satisfying the conditions of \cref{sdfhhkdrsdhsdgjhdfjhaf}. 
such that  ${\sf 1}_{\sts \otimes \ctau}=0$ for any  $\sts  \in \Std(\mu)$ by (possibly repeated application of) the double-braid 
(pictured in \cref{2braid1}) 
and  cyclotomic relations \eqref{R11}.
% This case can be handled in an identical fashion to the above, except that one must also apply the (commutative!)  double braid relation for $m=2$  in order to move the $\ctau$-strand   past the idempotent ${\sf 1}_{\stt_{\la\setminus\mu}}$. 

 Rather than go into the word combinatorics for each exceptional case in detail (as they are all very similar) 
   we simply check one of these exceptional cases here. 
Further 
    illustrative examples are in \cref{jfgdkhgkjfdhglkjsdfhglsjkdfghlskdjfg} (and we leave these as an exercise for the reader).  
\color{black!99}    Let  $\la=(1^2,2^2) \in \mptn$ for $(W,P)=(E_6,D_5)$
    with 
    $\csigma = {\color{magenta}s_2} \in {\sf Add}(\la)$ and 
        $\crho = {\color{green!80!black}s_3} \in {\sf Rem}(\la)$ with 
        $m(\csigma,\crho)=3$ as pictured in \cref{jfgdkhgkjfdhglkjsdfhglsjkdfghlskdjfg}.  We have that 
        $ (1^2)=\mu \subseteq \la=(1^2,2^2)$ 
        and setting $\ctau=\color{cyan}s_1$ we note that  $m(\ctau, s_i)=2$ for $s_i \in \mu\setminus\la=     {\color{green!80!black}s_3}s_4 s_6 
         {\color{green!80!black}s_3}         	 
$. Using the colouring of the leftmost diagram in \cref{jfgdkhgkjfdhglkjsdfhglsjkdfghlskdjfg}, we have that 
%\begin{equation}\label{anexampleexcep}
%         {\sf 1}_{{\color{cyan}s_1}{\color{magenta}s_2}
%         {\color{green!80!black}s_3}s_4 s_6 
%         {\color{green!80!black}s_3}         	}
%         \otimes {\sf braid}^{{\color{cyan}s_1}{\color{magenta}s_2}}
%         _{{\color{magenta}s_2}{\color{cyan}s_1}}
%     =
%  {\sf braid}^
%       {{\color{cyan}s_1}{\color{magenta}s_2}
%         {\color{green!80!black}s_3}s_4 s_6 
%         {\color{green!80!black}s_3}         {\color{cyan}s_1}{\color{magenta}s_2}	}
%      _
%       {{\color{cyan}s_1}{\color{magenta}s_2 {\color{cyan}s_1}}
%         {\color{green!80!black}s_3}s_4 s_6 
%         {\color{green!80!black}s_3}         {\color{magenta}s_2}	}
% %     
%%     
%%     
%%     
%%     
%       {\sf braid}_
%       {{\color{cyan}s_1}{\color{magenta}s_2}
%         {\color{green!80!black}s_3}s_4 s_6 
%         {\color{green!80!black}s_3}         {\color{cyan}s_1}{\color{magenta}s_2}	}
%         ^
%       {{\color{cyan}s_1}{\color{magenta}s_2 {\color{cyan}s_1}}
%         {\color{green!80!black}s_3}s_4 s_6 
%         {\color{green!80!black}s_3}         {\color{magenta}s_2}	}
%       ( {\sf 1}_{{\color{cyan}s_1}{\color{magenta}s_2}
%         {\color{green!80!black}s_3}s_4 s_6 
%         {\color{green!80!black}s_3}         	}
%         \otimes {\sf braid}^{{\color{cyan}s_1}{\color{magenta}s_2}}
%         _{{\color{magenta}s_2}{\color{cyan}s_1}})
%     =0.   \end{equation}
$$\color{black}\scalefont{0.8}   \begin{minipage}{3.525cm}
\end{minipage}  
=0 $$The first equality follows by the 
double braid  relation    for $\color{black}m( s_i , \ctau)=2$  of \eqref{braidrelation1}  (pictured on the left in \cref{2braid1}) for $i=3,4,6$. 
The second equality follows by 
 the double-braid relation  \eqref{braidrelation1} (pictured in \cref{2braid1})
 and the    cyclotomic relation \eqref{R10}; the third equality follows by the commutation relation \eqref{braidrelation1}; the fourth by the    cyclotomic relation \eqref{R10}.
\end{proof}

  We now state the obvious corollary, for ease of reference.

\begin{defn}
 We   define a  {\sf simple Soergel diagram} to be any Soergel diagram which  does not  contain any barbells or  ${\sf braid}^{\csigma\ctau}_{\ctau\csigma}(m)$ for $m=m(\csigma,\ctau)>2$.  \end{defn}

 \begin{cor}\label{inlightof}
 Let $( W ,P)$ be a Hermitian symmetric pair.  We can define  $\mathcal{H}_{(W,P)}$ to be the locally-unital  associative $\Bbbk$-algebra spanned by all  
simple  Soergel   diagrams   with multiplication given by vertical concatenation of diagrams modulo   relations  
 \ref{R1}, \ref{R2}, \ref{R3}, \ref{R4}, \ref{R5},  
  \ref{R8}, \ref{R9}, \ref{R10}, \ref{R11}, 
    for   $(\orange, \ctau ,\crho)\in S^3$ 
with   $m(\orange, \crho)=
m(  \crho,\ctau) =m(\orange ,\ctau)=2$,  we have the commutation relations 
\begin{equation} \label{commutingrelations}
\begin{minipage}{1.3cm}
\end{minipage}
 \end{equation}
 and for   $m(\csigma,\ctau)=3$, we have the null-braid relation 
 \begin{align*}
    \begin{minipage}{1.6cm}
 %
\end{minipage}
  =0 
   \end{align*}
and for   $m(\csigma,\ctau)=4$, we have the null-braid relation 
     \begin{align*}   &\begin{minipage}{2cm} %
\end{minipage}\!\!\!=0
  \end{align*}
 and their horizontal flips.
        \end{cor}

\section{Light leaves   for the Hecke categories  of Hermitian symmetric pairs}\label{lightleaves}
In this section, we recall Libedinsky--Williamson's construction of the light leaves basis  in the case  of Hermitian symmetric pairs.   
This could have been done in \cref{coulda}, however we delayed until now so that we could  simplify the presentation of this material by virtue of \cref{simpligy}.   
We regard $\mathcal{H}_{(W,P)}$  as a locally unital associative algebra in the sense of 
\cite[Section 2.2]{MR4684337} via the following idempotent decomposition 
$$\mathcal{H}_{(W,P)}=\bigoplus _{
\begin{subarray} c 
\x \in   {\sf exp}(x) 
\\
\y \in   {\sf exp} (y)
\\
 x,y \in { ^PW}
\end{subarray}
}{\sf 1}_\x \mathcal{H}_{(W,P)}{\sf 1}_\y$$
%and hence regard $\mathcal{H}_{(W,P)}$  as a locally unital associative algebra in the sense of 
%\cite[Section 2.2]{BSBS}. 

\begin{rmk}\label{jkbhxlkhbkjhcxbvhgbvjxkhvkxcv}
Given $\sts,\stt \in \Std(\la)$, by \cref{iftheshoecommutes} we have that 
$${\sf braid}^\sts_\stt \circ {\sf 1}_\stt\circ {\sf braid}^\stt_\sts = {\sf 1}_\sts \qquad {\sf braid}_\sts^\stt 
\circ {\sf 1}_\sts\circ {\sf braid}_\stt^\sts = {\sf 1}_\stt$$
Thus from now on, we may fix any preferred choice of  $\stt_\la \in \Std(\la)$, for each $\la \in \mptn$. 

 \end{rmk}

By \cref{jkbhxlkhbkjhcxbvhgbvjxkhvkxcv}, we can truncate the set of weights to be ``as small as possible".

\begin{defn} \label{asinthisthing}
We set
 $$\textstyle
{\sf 1}_{(W,P)}
=
 \sum_{\mu \in \mptn } {\sf 1}_{\stt_\mu}  
 \qquad 
\text{and} 
\qquad   h_{(W,P)}=
  {\sf 1}_{(W,P)}
   \mathcal{H}_{(W,P)}
  {\sf 1}_{(W,P)}.
$$
\end{defn}

{ Let $\la,\mu \in \mptn$ and}  $\SSTT \in \Path_\ell(\la,\stt_\mu)$ be a path of the form
$$
\SSTT: \varnothing ={\la}_0 \to {\la}_1 \to {\la}_2  \to \dots 
 \to {\la}_\ell  = \la 
$$and we let  $\ctau \in {\rm Add}(\mu )$.   
 If   $ \ctau   \in {\rm Add}(\la)   $, we set $\la^+=\la\ctau=\la + \ctau$ and $\la^-= \la$.  
  If   $ \ctau   \in {\rm Rem}(\la)   $, we set $\la^+=\la  $ and $\la^-= \la-\ctau$.  
We set $\SSTT^+$ and $\SSTT^-$ to be the paths 
$$\SSTT^+: \varnothing ={\la}_0 \to {\la}_1 \to {\la}_2  \to \dots 
 \to {\la}_\ell  \to \la^+ 
\qquad
\SSTT^-: \varnothing ={\la}_0 \to {\la}_1 \to {\la}_2  \to \dots 
 \to {\la}_\ell  \to  \la^-.$$For the  empty path $\SSTT ^\varnothing$  we set $c_{\SSTT ^\varnothing}={\sf 1}_\varnothing$ to be empty diagram. 
 We now inductively define the basis via certain ``add" and ``remove" operators 
 (denoted $A^\pm_\ctau$ and $R^\pm_\ctau$ respectively).  
\color{black!99} If   $ \ctau   \in {\rm Add}(\la)   $, then 
  we define  
$$
A^+_\ctau(c_\SSTT):= 
 {\sf braid}^{\stt_{\la + \ctau}}_{\stt_\la \otimes \ctau}
( c_{\SSTT } \otimes {\sf 1}_\ctau )
 \qquad 
A^-_\ctau(c_\SSTT) := 
   c_{\SSTT} \otimes {\sf spot}_\ctau^\emptyset  .
$$and we set  $c_{{\SSTT^+}}  =A^+_\ctau(c_\SSTT)$ and 
$ c_{{\SSTT^-}}  =A^-_\ctau(c_\SSTT)$. 
 If   $ \ctau   \in {\rm Rem}(\la)   $, then $\la=\la'\ctau$ and  
  we define 
\begin{align*}
 R^+_\ctau(c_\SSTT) &:= 
 {\sf braid}^{\stt_{\la }}_{\stt_{\la'} \otimes \ctau}
  ({\sf 1}_{\stt_{\la'}} \otimes  {\sf fork}_{\ctau\ctau}^\ctau)
 ( {\sf braid}_{\stt_{\la }}^{\stt_{\la'} \otimes \ctau} c_{\SSTT } \otimes {\sf1}_\ctau )
 \\
  R^-_\ctau(c_\SSTT) &:= 
   ({\sf 1}_{\stt_{\la'}} \otimes  {\sf cap}_{\ctau\ctau}^\emptyset)
 ( {\sf braid}_{\stt_{\la}}^{\stt_{\la'} \otimes \ctau} c_{\SSTT } \otimes {\sf1}_\ctau )
\end{align*}and we set 
$c_{{\SSTT^+}}=R^+_\ctau(c_\SSTT)$ and $ c_{{\SSTT^-}}= 	  R^-_\ctau(c_\SSTT)$. 
 An example is given in the rightmost diagram in  \cref{basis-construct}. 
 \color{black}

\begin{figure}[ht!]
  $$
   \begin{minipage}{2.2cm}
\end{minipage}$$
\caption{Construction of a light leaves basis element (for $(A_5, A_2\times A_2)$, whose Bruhat graph is the leftmost of \cref{coxeterlabelA22NOW}) 
using the diagrammatic composition defined in  \cref{compose}.  
The rightmost diagram   is equal to  
$
 A_{{\color{green!80!black} s_1 }}^+
 R_{{\color{orange} s_4 }}^+
 R_{{\color{magenta} s_2 }}^+  
A_{{\color{cyan} s_3 }}^-  
 A_{{\color{orange} s_4 }}^+
 A_{{\color{magenta} s_2 }}^+ 
A_{{\color{cyan} s_3 }}^+ 
({\sf 1}_\emptyset)
$ 
}
\label{basis-construct}
\end{figure}

 \color{black!99}
We are hence able to equip the algebras $h_{(W,P)}$ with powerful ``light leaves'' graded cellular bases  which  encode a great deal of representation theoretic information. % (in the sense of \cite{CPS1}) 
  
%  
%  \color{blue}
%  
%  
%  
%\begin{thm}[{\cite[Theorem 5.3]{antiLW}}] \label{thm:univrealff}
%Let $\mathcal{O}$ be a complete local ring in which $2$ is invertible, and let $(a_{\csigma \ctau}^{\mathcal{O}})_{\csigma,\ctau \in S}$ be a balanced Cartan matrix for $W$ over $\mathcal{O}$.
%If the universal realisation for $W$ with respect to $(a_{\csigma \ctau}^{\mathcal{O}})_{\csigma,\ctau \in S}$ is faithful, then \Cref{assump:antispherLL} holds for $\mathcal{H}_{P \backslash W}^{\mathcal{O}}$, the anti-spherical Hecke  category defined over $\mathcal{O}$ with respect to $(a_{\csigma \ctau}^{\mathcal{O}})_{\csigma,\ctau \in S}$.
%Moreover, if there is a ring homomorphism $\mathcal{O} \rightarrow \Bbbk$ such that $(a_{\csigma \ctau})_{\csigma,\ctau \in S}$ is the image in $\Bbbk$ of $(a_{\csigma \ctau}^{\mathcal{O}})_{\csigma,\ctau \in S}$, then \Cref{assump:antispherLL} holds for $\mathcal{H}^\Bbbk_{P \backslash W}$.
%\end{thm}
%
%\begin{proof}
%  \Cref{assump:antispherLL} is preserved by base change, so the second statement follows immediately from the first.
%The special case where $\mathcal{O}=\mathbb{R}$ and $(a_{\csigma \ctau}^{\mathcal{O}})_{\csigma,\ctau \in S}$ is the ``geometric'' Cartan matrix for $W$ over $\mathbb{R}$ is proved in \cite[Proposition 5.5]{antiLW}.
%In fact this proof is valid for any universal realisation over a complete local ring in which the ``parabolic property'' \cite[2.3]{antiLW} holds, including faithful realisations.
%\end{proof}
%  
%  
%  
%  
%  
%  \color{black}
%  
%  

\begin{thm}[{\cite[Section 6.4]{MR3555156} and {\cite[Theorem 5.3]{MR4437613}}}]  \label{LEW1}
%Fix some reduced expression $\x$  for  
% each $x \in \mathcal{P}_{\leq  w}$.   
Let  $\la_1\prec  \la_2 \prec  \dots \prec  \la_t$  be any total refinement of the Bruhat order $\leq $ on $\mptn$. 
  The  algebra   
$ {h}_{(W,P)}  $  
has a   chain of two-sided ideals 
 $$0 \subset {h}_{(W,P)} {\sf1}_{\la_1}		{h}_{(W,P)} \subset
 {h}_{(W,P)} ({\sf1}_{\la_1}+{\sf1}_{\la_2}){h}_{(W,P)} \subset
  \dots \subset
 {h}_{(W,P)} ({\sf1}_{\la_1}+{\sf1}_{\la_2}+\dots+ {\sf1}_{\la_t}){h}_{(W,P)} ={h}_{(W,P)}$$such that 
\begin{equation}\label{basis}
\{c^{\la_k}_{\SSTS\SSTT}:=c^*_{\SSTS}c_{\SSTT}  \mid \SSTS \in \Path _{(W,P)}( \la_k,\stt_\mu) ,  \SSTT \in   \Path _{(W,P)}( \la_k,\stt_\nu),   \mu,\nu \in  \mathscr{P}_{(W,P)} \}
\end{equation}
is a $\Bbbk$-basis of ${h}_{(W,P)} {\sf1}_{\la_k}{h}_{(W,P)} /
 {h}_{(W,P)} {\sf1}_{\la_{k-1}}{h}_{(W,P)} $. Thus ${h}_{(W,P)} $ is 
 is a graded cellular algebra  in the sense of \cite{hm10} %.
 and a  quasi-hereditary algebra in the sense of  \cite{CPS1}.
\end{thm}

\color{purple}

 \begin{rmk} \label{parabolicprop}
 The careful reader will notice that we have \emph{not} assumed the {\em parabolic property} from \cite[\S 2.3]{MR4437613}, as this can fail in positive characteristic. 
 It turns out that this condition is not always necessary to obtain a light leaves basis as in \cref{LEW1} above. 
  In our setting, the trick in \cite[Example 1.11(1)]{MR4510171} is always enough to show that the light leaves construction yields a basis in any characteristic (including $p=2$, by way of \cref{demazure}). 
\color{magenta}
In more detail: let $(\mathbb{K},\mathcal{O},\Bbbk)$ be a $p$-modular system. 
In other words, suppose $\mathcal{O}$ is a complete discrete valuation ring whose residue field is $\Bbbk$ and whose field of fractions $\mathbb{K}$ is of characteristic $0$. 
Since $W$ is a Weyl group there is a standard $\mathcal{O}$-form $\mathcal{H}^{\mathcal{O}}_{(W,P)}$ of the algebra $\mathcal{H}_{(W,P)}$, which arises from the $\mathcal{O}$-form of the geometric realisation of $W$. 
This realisation is faithful (and thus satisfies the parabolic property) because the fraction field $\mathbb{K}$ is of characteristic $0$. 
By \cite[Theorem 5.3]{MR4437613} the light leaves construction yields a basis for $\mathcal{H}^{\mathcal{O}}_{(W,P)}$, which descends to a light leaves basis for $\mathcal{H}_{(W,P)}$. 
\color{purple} 

 \end{rmk}
 
 \color{black!99}

\begin{rmk}
Whilst the quasi-heredity property is not mentioned explicitly in {\cite[Section 6.4]{MR3555156} and {\cite[Theorem 5.3]{MR4437613}}},  one of the first   theorems in the literature on  cellular algebras
was that they are quasi-hereditary if and only if 
 each layer of the cell-filtration has an idempotent, as above   (see \cite[Proposition 4.1]{MR1648638}).  We refer to a  quasi-hereditary algebra as having  a {\sf highest weight structure} as in \cite{CPS1}.
\end{rmk}

\color{black}
%\begin{proof}
%For a proof in the language of this paper, see \cite[Theorem 1.7]{BHN}.  
%\end{proof}
%  
%  
%
When it cannot result in confusion, we write $c_{\SSTS\SSTT}$ for $c_{\SSTS\SSTT}^\la$. For $\mu \in \mptn$, we define  one-sided ideals 
\begin{align*}
h_{(W,P)} ^{ \leq  \mu  }  = {\sf1}_{	\leq \mu	} h_{(W,P)}  
 &&
h_{(W,P)} ^{< \mu}  =
h_{(W,P)} ^{ \leq  \mu }   \cap \Bbbk \{ c^\la_{\SSTS\SSTT}  \mid 
 \SSTS,\SSTT  \in \Path 
 ( \la ), \la <   \mu \}   
 \end{align*}
and we hence define the   {\sf standard} or {\sf cell}  modules of $h_{(W,P)}$ as follows 
  \begin{equation}  \label{identification}
        \Delta ( \mu) = \{ c^\mu_\SSTS:=c^\mu_{\stt^\mu\SSTS} +  h_{(W,P)} ^{< \mu} 
       \mid \SSTS \in \Path ( \mu, \stt_\nu  ), \nu \in \mptn \} . \end{equation}
   We recall that the   cellular structure allows us to define, for each
 $ \mu \in \mptn$,  a bilinear form 
  $\langle\ ,\ \rangle^{ \mu} $ 
   on $\Delta(\mu) $    which
is determined by
\begin{equation}\label{geoide}
   c^\mu_{\SSTS \SSTT}  c^\mu_{\SSTU \SSTV}\equiv
  \langle c _\SSTT,c _\SSTU \rangle^ \mu  c^\mu_{\SSTS\SSTV}
  \pmod{ h_{(W,P)}^{< \mu}} 
  \end{equation}
for any $\SSTS,\SSTT,\SSTU,\SSTV\in \Path(\mu, -)$.  
% When $\Bbbk$ is a field,  w
 We obtain a complete set of non-isomorphic simple modules 
 for $h_{(W,P)} $   as follows 
 $$ 
L ( \mu) =
\Delta  ( \mu) /
  \rad(\langle \;\; ,\;\;  \rangle^{ \mu} )
   $$
for $ \mu\in\mptn$.  
The projective indecomposable $h_{(W,P)} $-modules  are the direct summands 
$$
{\sf 1}_{\stt_\mu} h_{(W,P)} = 
\bigoplus_{ \la \leq  \mu	}
\dim_q( L (\la){\sf 1}_{\stt_{\mu}})  P(\la).
$$For $\Bbbk$ a field of characteristic $p\geq 0$, the  {\sf anti-spherical $p$-Kazhdan--Lusztig polynomials}  are defined  as follows, 
\begin{equation}\label{hjhjhjhas}
{^p}{n}_{\la,\mu}(\grade)
:=
 %\dim_\grade
 \dim_\grade (\Hom_{ h_{(W,P)} }(P (\la),\Delta (\mu))
 =
\sum_{k\in \ZZ}   [\Delta  ( {\mu})	: L  (\la)		\langle k \rangle ] \grade^k   
\end{equation} for any  $\la,\mu\in \mptn$.  
These polynomials were first defined via the {diagrammatic character} of 
  \cite[Definition 6.23]{MR3555156} and \cite[Section 8]{MR4437613} and  rephrased as above in \cite[Theorem 4.8]{MR3591153}. 
\color{black!99}
  Elias--Williamson and Libedinsky--Williamson \cite{MR3245013,MR4437613} proved that 
 over a field of characteristic   $p=0$, the (anti-spherical) 
$p$-Kazhdan--Lusztig polynomials are, in fact, equal to the 
 classical (anti-spherical)  Kazhdan--Lusztig polynomials of \cite{DD1,MR1444322} thus justifying the nomenclature.   
Interesting families of   $p$-Kazhdan--Lusztig polynomials 
calculated in the literature include   Williamson's torsion explosion 
examples \cite{w13} and examples of 
Fiebig, 
Lanini--McNamara and  
Libedinsky--Williamson, 
\cite{MR2999126,MR2667425,MR4287720,MR3693504};  algorithms for calculating the  $p$-Kazhdan--Lusztig polynomials  are considered in 
\cite{MR4633007,MR3611719}.

\color{black}  
 
 \begin{rmk}
 We note that $\mathcal{H}_{(W,P)}$ and $h_{(W,P)}$  are graded Morita equivalent (as the latter is obtained from the former by a truncation which 
 does not kill any simple module). \end{rmk}

\color{black!99}

\begin{rmk}\label{spherical}
The trivial and sign representations of the Hecke algebra give rise to the spherical and anti-spherical Hecke categories, respectively. 
The latter of which are the focus of this paper and are more well-studied  than their spherical counterparts (see for example \cite{MR3805034,MR3766576,MR4633007,MR4381198,MR4510171,MR4611117}).  
The   spherical  Hecke categories  are much more mysterious: constructing presentations of these categories (generalising \cref{the algebra}) is an important open problem. 
The problem of constructing bases of these categories for type $(A_n,A_{k-1}\times A_{n-k}) $   was solved in \cite{MR4323501} using tiling combinatorics akin to this paper; generalising this to all parabolic Coxeter systems has been the subject of much recent work \cite{MR4563864,singlightleaves,elias2024reduced}. 
   The $p$-(in)dependence of  spherical  
$p$-Kazhdan--Lusztig polynomials  is the focus of a recent preprint by Baine 
  \cite{baine2024cominusculeheckecategories}. 

%  Since this paper   first  appeared online in 2022, an independent 
%  proof that the $p$-Kazhdan--Lusztig polynomials are independent of $p\geq 0$ was given in 
%  \cite{baine2024cominusculeheckecategories}.

\end{rmk}

\color{black}

\subsection{A ``singular"  horizontal concatenation}
Singular Soergel bimodules were first considered in Williamson's thesis, where it was proven that they categorify the Hecke algebroid \cite{MR2844932}.  
At present, %there is not a   complete
we do not have a 
 diagrammatic construction of the category of singular Soergel bimodules (although some progress has been made, see \cite[Chapter 24]{MR4220642}).  
In this paper, we will give a complete realisation of   singular Soergel bimodules  within the diagrammatic Hecke category for $(W,P)$ a Hermitian symmetric pair. 
 In order to accomplish this goal, we first need to provide a Soergel-diagrammatic analogue of the tensor product (denoted $\otimes_{R^{\ctau}}$) for singular Soergel bimodules ``on a $\ctau$-hyperplane" for $\ctau \in S_W$ (which we will denote by $\color{cyan}\circledast$).  
 
 \begin{defn}
 We suppose that a
 diagram $D\in {\sf 1}_{\x\ctau} \mathcal{H}_{(W,P)} {\sf 1}_{\y\ctau}$
% has a final exposed propagating 
% $\ctau$-strand if 
%the  
 is such that 
 $(i)$ the
 rightmost $\ctau $ in the northern 
boundary  is connected to the  rightmost  $\ctau$ in the  southern boundary by a strand 
and  $(ii)$ there are no barbells to the right of this strand.  
We say that such a diagram has a 
{\sf a final exposed propagating 
 $\ctau$-strand}.
Similarly, we define a 
 {\sf  first exposed propagating 
 $\ctau$-strand} by reflecting this definition through the vertical axis.  
 \end{defn}

 \begin{defn} \label{compose} 
  We let 
 $D_1  \in {\sf 1}_{\x\ctau} \mathcal{H}_{(W,P)}  {\sf 1}_{\y\ctau}
 $  and 
 $D_2  \in {\sf 1}_{\ctau\underline{u}} \mathcal{H}_{(W,P)}  {\sf 1}_{\ctau\underline{v} }
$.  
We suppose that $D_1$ (respectively $D_2$) has a final (respectively first) exposed propagating $\ctau$-strand.  
 We define 
\begin{align}\label{adfsjkhahakfhshjkafsd}
D_1   
  \compose  D_2 =
 (D_1 \otimes {\sf 1}_{\underline{u}})
  ({\sf 1}_\y \otimes D_2)
    = 
  ({\sf 1}_\x \otimes D_2)
( D_1 \otimes {\sf 1}_{\underline{v}})
\end{align}
    Now suppose that   $  	D_1' = \braid_{\x\ctau}^{\underline{w}}	D_1 
  \braid^{\y\ctau}_{\underline{z}}  
   \neq0 $.
       We extend the above definition as follows
\begin{align}\label{adfsjkhahakfhshjkafsd2}
D_1 '  
  \compose  D_2   = 
( \braid_{\x\ctau}^{\underline{w}}     \otimes {\sf 1}_{\underline{u}})
 ( D_1 
\compose D_2)
 (\braid^{\y\ctau}_ {\underline{z}}\otimes {\sf 1}_{\underline{v}}).
\end{align}
%        We extend the above definition as follows
%\begin{align}\label{adfsjkhahakfhshjkafsd2}
%D_1   
%  \compose  D_2   = 
%( \braid_{\x\ctau}^{\underline{w}}     \otimes {\sf 1}_{\underline{u}})
% ((\braid^{\x\ctau}_{\underline{w}} D_1\braid^{\underline{z}}_{\y\ctau} )
%\compose D_2)
% (\braid^{\y\ctau}_ {\underline{z}}\otimes {\sf 1}_{\underline{v}}).
%\end{align}

\begin{figure}[ht!]

$$
   \begin{minipage}{3.5cm}
\end{minipage}$$
\caption{An example of $\compose$ merging the rightmost and leftmost $\ctau$-strands.}
\label{hhghghghghghg}
\end{figure}

    \end{defn}

\begin{rmk}
 Diagrammatically, we can think of $\compose$ as identifying the rightmost 
 $\ctau$-strand of $D_1^{(')}$ %(or more generally $D_1'$)  
  with the leftmost $\ctau$-strand of $D_2$.  
For examples, see \cref{basis-construct,hhghghghghghg,illustrative}.   
 \end{rmk}
  
  \begin{prop} \label{interchange}
   The operation $\compose $ satisfies the interchange law: $$(D_1\compose D_2)\circ (D_3 \compose D_4)= (D_1\circ D_3)\compose  (D_2 \circ  D_4).$$      
\end{prop}
 
 \begin{proof}
For \cref{adfsjkhahakfhshjkafsd}, the result follows immediately by diagram chasing using the fact that 
\begin{align}\label{askjhghdjkghdsghdsg}
(D_1 \otimes {\sf 1}_{\underline{u}})
  ({\sf 1}_\y \otimes D_2)
    = 
  ({\sf 1}_\x \otimes D_2)
( D_1 \otimes {\sf 1}_{\underline{v}})
\end{align}
The case of \cref{adfsjkhahakfhshjkafsd2} follows by applying commuting braid generators.  
 \end{proof}

We will  abuse notation   and use $
\compose$ as a shorthand as follows.

 \begin{defn} \label{compose2}
   We let 
 $D_1  \in {\sf 1}_{\x\ctau} \mathcal{H}_{(W,P)}  {\sf 1}_{\y\ctau}
 $  and 
 $D_2  \in {\sf 1}_{\ctau\underline{u}} \mathcal{H}_{(W,P)}  {\sf 1}_{\ctau\underline{v} }
$.  
We suppose that $D_1$ (respectively $D_2$) has a final (respectively first) exposed propagating $\ctau$-strand.  
 We define 
\begin{align}\label{adfsjkhahakfhshjkafsd3}
D_1   
  \compose  ({\sf spot}_\ctau^\emptyset \otimes {\sf 1}_{\underline{u}}) D_2)   = 
({\sf 1}_{\x} \otimes {\sf spot}_\ctau^\emptyset \otimes {\sf 1}_{\underline{u}}) 
( D_1   
  \compose D_2).
\end{align}
We extend this via commuting braid generators in an analogous fashion to 
\cref{adfsjkhahakfhshjkafsd2}. 
  \end{defn}

\begin{rmk}
We note that the operation in \cref{adfsjkhahakfhshjkafsd3} considers non-propagating $\ctau$-strands and therefore does not satisfy the interchange law (as \cref{askjhghdjkghdsghdsg} no longer holds).  
\end{rmk}

 \section{Categorical combinatorial invariance }\label{cat-cob}
 
%We now propose a  categorical lift  of 
% the famous combinatorial invariance conjecture and we prove this statement for Hermitian symmetric pairs (thus categorifying Brenti's result for the Kazhdan--Lusztig polynomials \cite[Corollary 5.2]{MR2465813}).  
In this section we provide new  presentations of the anti-spherical Hecke categories $\mathcal{H}_{(W,P)}$ for ${(W,P)}$ a simply laced Hermitian symmetric pair.  Our presentations encode the combinatorics of the Bruhat graph more effectively than those of \cref{the algebra,inlightof}.
% While monoidal presentations have many advantages, they 
% are ill-equipped for tackling the combinatorial invariance conjecture. 
% This is because monoidal presentations are ``too local" and they cannot reflect the structure of the Bruhat graph. 
%

\color{black}
% First, we let $(W,P)$ be an arbitrary parabolic Coxeter system and we recall the classical notions of combinatorial invariance, so that we can categorify the question in full generality.

% \begin{defn}
% Let  $\Pi $ denote a subset of ${^P}W$.  
%We say that $\Pi$ is {\sf saturated} if for any $\alpha\in \Pi$ and $\beta \in {^P}W$ with $ \alpha<\beta$, we have that $\beta \in \Pi$. We say that $\Pi$ is {\sf co-saturated} if its complement in $^PW$ is saturated.  
%If a set  is saturated, co-saturated, or the intersection of a saturated and a co-saturated set, we shall say that it is {\sf closed} under the partial order.  
%\end{defn}

%Given any closed set, we can define a corresponding  subquotient of $\mathcal{H}_{(W,P)}$ as follows.  
%(These subquotients also arise in the work of Achar--Riche--Vay in \cite{MR4045966}.)  

\begin{defn}\label{subquot}\color{black!99}
We let $\la,\mu \in {^PW}$ and we set  
$\Pi= [\la,\mu]:= \{  \nu \mid \la\leq \nu\leq \mu\}  $. 
 We let
\begin{align}\label{idempotennnnnssss}
e = \sum_{   \la \not \leq \nu } {\sf 1}_\nu \quad \text{and} \quad
f = \sum_{\nu \leq \mu} {\sf 1}_\nu 
\end{align}
in $\mathcal{H}_{(W,P)}$.  
We let $\mathcal{H}_{(W,P)}^{\Pi} $ denote the subquotient of $\mathcal{H}_{(W,P)}$ given by
\[
\mathcal{H}_{(W,P)}^\Pi= f(\mathcal{H}_{(W,P)}/(\mathcal{H}_{(W,P)} e \mathcal{H}_{(W,P)}))f.
\]
We set $h_\Pi = {\sf 1}_{(W,P)} \mathcal{H}_{(W,P)}^\Pi {\sf 1}_{(W,P)}$  as in \cref{asinthisthing}.
\end{defn}

 \begin{thm}\label{HSP-comb}\color{black!99}
% Let $\Bbbk$ be a field and let $(W,P)$ and $(W',P')$ be  Hermitian symmetric pairs. Let $\Pi  \subseteq    {^P}W$ 
% and $\Pi' \subseteq   {^{P'}}W' $ 
%be closed subsets and   
% suppose that $\Pi$ and $\Pi'$ are isomorphic as partially ordered sets.  Then   
% $ \mathcal{H}^{\Pi}_{(W,P)}$ 
%and 
% $ \mathcal{H}^{\Pi'}_{(W',P')}$ are highest-weight   graded Morita equivalent.  
  Let  $\Pi=[\la,\mu]$ and $\Pi=[\la',\mu']$ be  subquotients of the Bruhat graphs of   Hermitian symmetric pairs $\la,\mu\in (W,P)$ and $\la',\mu'\in (W',P')$.  If  $\Pi$ and $\Pi'$   are isomorphic as partially ordered sets, then the corresponding subquotients   
 $ \mathcal{H}^{\Pi}_{(W,P)}$ 
and 
 $ \mathcal{H}^{\Pi'}_{(W',P')}$ %are highest-weight   graded Morita equivalent.  
% Let  $(W,P)$, $(W',P')$ be    Hermitian symmetric pairs. 
%  Let  $\Pi \subseteq \mathscr{A}_{(W,P)}$ and $\Pi'\subseteq \mathscr{A}_{(W,P)}$ be 
%% tilings of the same shape.
%  Then the corresponding subquotients   
% $ \mathcal{H}^{\Pi}_{(W,P)}$ 
%and 
% $ \mathcal{H}^{\Pi'}_{(W',P')}$
  are       Morita equivalent 
 (in the sense of \cite[Section 2.2]{MR1644252})
 and this equivalence preserves the 
 grading, cellular, and highest-weight structures of these algebras.   
 \end{thm}
 
%We remark that 
%  $\Pi\cong \Pi'$ is a weakening of the condition in \cref{ques} and so this gives a positive answer to the question.
This section is dedicated to the proof of  \cref{HSP-comb}.  In order to do this, we must first provide ``Tetris-style" presentations of these categories.

%    \subsection{The Tetris-style presentation for Hecke categories of simply laced type } \label{barbells}
% 
% In this section we provide   new, non-monoidal, presentations of the anti-spherical Hecke categories $\mathcal{H}_{(W,P)}$ for ${(W,P)}$ a simply laced Hermitian symmetric pair.  
% While monoidal presentations have many advantages, they 
% are ill-equipped for tackling the combinatorial invariance conjecture (for the statement of which, see the next section).  
% This is because monoidal presentations are ``too local" and there cannot reflect the structure of the Bruhat graph (we will discuss this more in the next section).

\subsection{Tetris combinatorics for ``gaps" in reduced words}  \label{zeroooooooo}
 In what follows we let $\csigma,\ctau \in S_W$ with $m(\csigma,\ctau)=3$ or $4$.  
 If  $m(\csigma,\ctau)= 4$ and    $(W,P)=(C_n,A_{n-1})$, then suppose that 
 $(\csigma,\ctau)= ( s_1 , s_2)$ or if that $(W,P)=(B_n,B_{n-1})$, then 
 $(\csigma,\ctau)= (s_2,s_1)$.   
Then    $\langle\alpha_\csigma^\vee,\alpha_\ctau\rangle=-1$, and the two-colour $\csigma\ctau$-barbell relation is
% Suppose $(W,P)=(C_n,A_{n-1})$, $\csigma=\color{magenta}s_2$, and $\ctau=\color{cyan}s_1$ or that $(W,P)=(B_n,B_{n-1})$, $\csigma=s_1$, and $\ctau=s_2$.   
%Then $m(\csigma,\ctau)=4$, $\langle\alpha_\csigma^\vee,\alpha_\ctau\rangle=-1$, and the two-colour $\csigma\ctau$-barbell relation is
\begin{equation}\label{useful2} 
  \begin{minipage}{1.4cm}
\end{minipage} \end{equation}
Using the one-colour barbell relation we then obtain  
\begin{equation}\label{useful1}
\begin{minipage}{1.4cm}%
\end{minipage}
 \end{equation}
 Finally, for  $m(\csigma,\ctau)=3=m(\csigma,\crho)$, by summing over the one and two colour barbell relations we obtain the following
 \begin{equation}\label{useful3}  
 \begin{minipage}{1.4cm}%
\end{minipage} 
  \end{equation}
We now provide inductive versions of the \cref{useful1,useful2}.  
 First, we define a  {\sf type $A$ string} $T\subseteq S_W$ to be an ordered  set of reflections 
 $s_{i_1}, \dots, s_{i_t}$ such that $m(s_{i_j}, s_{i_{k}})=2+\delta_{j,k+1}+\delta_{j,k-1}$  for $1\leq j,k\leq t$.  By induction on \cref{useful1} we have the following:
 
% 
% First, we define a  {\sf type $A$ string} $T\subseteq S_W$ to be an ordered multiset of reflections 
% $s_{i_1}, \dots, s_{i_t}$ such that $\langle s_{i_1}, \dots, s_{i_t}\rangle$ is a symmetric group and 
%  %$m(s_{i_j}, s_{i_{k}})=2+\delta_{j,k+1}+\delta_{j,k-1}$  for $1\leq j,k\leq t$.  
% $m(s_{i_j}, s_{i_{j+1}})=3$ for all $1\leq j <t$.  
% By induction on \cref{useful1} we have the following:

 \begin{lem}\label{lem1}
 Let $T\subseteq S_W $ be a    type $A$ string.  
 We have  {\color{black!99} the following local relation}
 $${\sf 1}_{s_{i_1} s_{i_2} \dots  s_{i_{t-1}}}\otimes {\sf bar}(s_{i_t})
 =\sum_{k=1} ^{t} {\sf bar}(s_{i_k})	\otimes 
{\sf 1}_{s_{i_1} s_{i_2} \dots  s_{i_{t-1}}}
  -
\sum_{k=1} ^{t-1}{\sf 1}_{s_{i_1} s_{i_2} \dots  s_{i_{k-1}}} \otimes 
{\sf gap}(s_{i_k}) \otimes 
   {\sf 1}_{s_{i_{k+1}}  \dots   s_{i_{t-1}}}
 $$
 \end{lem}
 
  By induction on \cref{useful2} we have the following:
 
 \begin{lem}\label{lem2}
 Let $T\subseteq S_W$ be a    type $A$ string.  
 We have   {\color{black!99} the following local relation}
 $$
  \sum_{k=1} ^t 
  {\sf 1}_{s_{i_2} s_{i_3} \dots  s_{i_{t}}}\otimes   
 {\sf bar}(s_{i_k})	
 =
{\sf bar}(s_{i_1})	\otimes {\sf 1}_{s_{i_2} s_{i_3} \dots  s_{i_{t}}}
% {\sf 1}_{i_1}\otimes {\sf 1}_{i_2}\otimes \dots \otimes {\sf 1}_{i_{t-1}}
 +
\sum_{k=2} ^t
{\sf 1}_{ s_{i_2}s_{i_3} \dots  s_{i_{k-1}}} \otimes 
 {\sf gap}(s_{i_k}) \otimes 
  {\sf 1}_{s_{i_{k+1}}  \dots   s_{i_{t}}}
 $$
 \end{lem}

 \begin{lem} \label{Astuffiszero}
 Let $\betar, \gam \in S_W $  be such that $m( \betar, \gam)=3$.   
We have the  following local relation 
\begin{align}\label{typeAtoprove}
 \on_{	  \gam	}\otimes 
 \left( {\sf bar}(\betar) +  
  {\sf bar}(\gam )   \right)\otimes {\sf 1}_\betar=0.
   \end{align}
    \end{lem}
 \begin{proof} 
 We have that 
\begin{align*}
 \on_{	  \gam	}\otimes 
 \left( {\sf bar}(\betar) +  
  {\sf bar}(\gam )   \right)\otimes {\sf 1}_\betar& = 
  \on_{	  \gam	}\otimes {\sf 1}_\betar \otimes {\sf bar}(\gam)
 +\on_{	  \gam	}\otimes {\sf gap}(\betar) \\
 & =
 {\sf spot}_{\gam\betar\gam}^{\gam\betar\emptyset}
 \left({\sf 1}_{\gam \betar \gam} +  
 {\sf spot}^{\gam\betar \gam}_{\gam\emptyset\gam}
 {\sf dork}^{\gam\gam}_{\gam\gam}
% {\sf fork}_{\gam\gam}^\gam
 {\sf spot}_{\gam\betar \gam}^{\gam\emptyset\gam} \right)
 {\sf spot}^{\gam\betar\gam}_{\gam\betar\emptyset}
\end{align*}
and so the result follows %by induction (on the length of $\mu$)  
% from \cref{remove a barbell,underphi}.  
from the $\betar\gam$-null-braid relation.
 \end{proof} 
 
 \begin{lem} \label{underphi}
 Suppose that 
  $m(\csigma,\ctau)=3=m(\orange,\ctau)$ and $m(\csigma,\orange)=2$. 
    We have that 
% \begin{align}\label{movethehole}
%{\sf 1}_\ctau	\otimes {\sf gap}(\csigma)		\otimes  {\sf 1}_\orange \otimes		 {\sf 1}_\ctau
% ={\sf 1}_\ctau\otimes  {\sf 1}_\csigma \otimes {\sf gap}(\orange)	\otimes {\sf 1}_\ctau.
% \end{align}
%  \end{lem}
%  
%  
%   
%
%   \begin{figure}[ht!]
%  $$   
\begin{equation}\label{movethehole}   \begin{minipage}{1.85cm}
\end{minipage}
% $$ 
% \caption{
% The far left and right depict the two sides of the equation from \cref{movethehole}.
%   The middle diagram depicts how we rewrite these terms   using the null-braid relation.  
%   }
% \label{figureTL-S}
% \end{figure}
\end{equation}
\end{lem}

  \begin{proof}
%   We set $\blue=\color{cyan}s_{ [r+1,c]}$, $\pink = \color{magenta}s_{ [r ,c]}$, and $\orange= \color{orange}s_{ [r+1,c-1]}$.  
  We apply the $\blue\orange$- and $\csigma\ctau$-null-braid relation  to the left and righthand-sides of \eqref{movethehole} respectively.
%  , as depicted in
%    \cref{figureTL-S}.  
 The result follows. 
  \end{proof}

 \begin{lem} \label{underphi2}\color{black!99}
  Suppose that 
  $m(\csigma,\ctau) =m(\orange,\ctau)=m(\crho,\ctau)=3$ and $m(\csigma,\orange)=m(\crho,\orange)=m(\csigma,\crho)=2$. 
    We have that \color{black}
 \begin{equation}\label{movethehole2}%\label{figureTL-S-D}
   \begin{minipage}{1.6cm}

  \end{minipage} 
 \end{equation} 
  \end{lem}

  \begin{proof}
%   We set $\blue=\color{cyan}s_{ [r+1,c]}$, $\pink = \color{magenta}s_{ [r ,c]}$, and $\orange= \color{orange}s_{ [r+1,c-1]}$.  
  We apply the $\blue\crho$- (respectively $\blue\orange$-) and $\csigma\ctau$-null-braid relation  to the lefthand-side (respectively righthand-side) 
  of  
   \cref{movethehole2} respectively. %, as depicted in  \cref{movethehole2}.  
   The result follows. 
  \end{proof}

 \begin{prop}\label{theresultwethoughtweneed}
 \color{black!99}
 Suppose that we are in one of the following cases: 
 $(i)$ 
 $(W,P)=(D_n,A_{n-1})$ and $\{\csigma,\ctau\}=\{s_0,s_1\} $ 
 $(ii)$ 
  $(W,P)=(A_n,A_{n-1})$ and $m(\csigma,\ctau)=2 $ 
  $(iii)$
   $(W,P)=(D_n,D_{n-1})$ and $m(\csigma,\ctau)=2 $  with  $\{\csigma,\ctau\}\neq \{s_0,s_1\} $. 
% \begin{itemize}[le
% 
%
% \end{itemize}
 Then we have the following local relation: ${\sf 1}_{\csigma\ctau} =0$ in $\mathcal{H}_{(W,P)}$.
 \end{prop}

 \begin{proof}
 \color{black!99}As in \cref{locality}, it will suffice to  prove that 
 $
{\sf 1}_{\sts } \otimes {\sf 1}_{\csigma\ctau}=0,
 $ for all   $\sts \in \Std(\mu)$ and all    $\mu \in \mptn$. 	
We will proceed by induction on $\ell(\mu)=\ell\geq 0$ with the base case ${\sf 1}_{\csigma\ctau}=0$ being trivial for any    $m(\csigma,\ctau)=2 $ (by the commuting relations \eqref{commutingrelations} and the cyclotomic relation \eqref{R11}).
  Now assume that $\ell(\mu)\geq 1$. We let  $\crho \in {\sf Rem}(\mu)$ and set $\mu'=\mu-\crho$. 
  If $\crho=\csigma $ (or  similarly $\ctau$), then
 $${\sf 1}_{\mu}= {\sf 1}_{\mu'} \otimes {\sf 1}_{\crho} \otimes {\sf 1}_{\csigma\ctau}
= {\sf 1}_{\mu'} \otimes {\sf 1}_{\csigma} \otimes {\sf 1}_{\csigma\ctau}
= {\sf 1}_{\mu'} \otimes {\sf 1}_{\csigma} \otimes {\sf braid}^{\csigma\ctau}_{\ctau\csigma}{\sf 1}_{ \ctau\csigma}{\sf braid}_{\csigma\ctau}^{\ctau\csigma}=0
 $$and similarly, if $m(\crho, \csigma)=2=m(\crho,\ctau)$ then 
  $${\sf 1}_{\mu}= {\sf 1}_{\mu'} \otimes {\sf 1}_{\crho} \otimes {\sf 1}_{\csigma\ctau}
= {\sf 1}_{\mu'} \otimes  {\sf braid}^{\crho\csigma\ctau}_{\csigma\ctau\crho}{\sf 1}_{ \csigma\ctau\crho }{\sf braid}_{\csigma\ctau\crho }^{\crho \csigma\ctau}=0
 $$by induction (as $\ell(\mu') < \ell(\mu)$) by the leftmost and rightmost equations in \eqref{commutingrelations} respectively. 
 Thus we can assume, without loss of generality, that $m(\csigma,\crho)=3$. We consider each type in turn.

 For $(W,P)=(D_n,A_{n-1})$ it remains to consider the case that $\rho={\color{green!80!black}s_2} \in {\sf Rem}(\mu)$ is the unique removable node. We will assume that $\stt_\mu =\stt_{\mu''}\otimes {\color{magenta}s_0}\otimes {\color{green!80!black}s_2}$ (the case  $\stt_\mu =\stt_{\mu''}\otimes {\color{cyan}s_1}\otimes {\color{green!80!black}s_2}$ is identical).  
 We have that 
 $${\sf1}_{  \stt_{\mu}}\otimes{\sf 1}_{	\color{magenta}s_0\color{cyan}s_1}=
{\sf1}_{  \stt_{\mu''}}\otimes{\sf 1}_{ {\color{magenta}s_0} \color{green!80!black}s_2	\color{magenta}s_0\color{cyan}s_1}
= -
{\sf1}_{  \stt_{\mu''}}\otimes   \begin{minipage}{2.06cm}
  \begin{tikzpicture}[scale=1]
   
   \draw[densely dotted, rounded corners] (0.5*7+0.75,0) rectangle (6.25,1.4);

          \draw[magenta ,line width=0.08cm](0.5*9,0)   --++(90:1.4);

                 \draw[magenta ,line width=0.08cm](0.5*11,0)   to [out=90,in=0] (0.5*9,0.6);

                 \draw[magenta ,line width=0.08cm](0.5*11,1.4)   to [out=-90,in=0] (0.5*9,1.4-0.6);

                 \draw[green!80!black ,line width=0.08cm](0.5*10,0)   --++(90:0.3) coordinate (hi);
   \fill[green!80!black] (hi) circle (3pt);
 
           \draw[green!80!black ,line width=0.08cm](0.5*10,1.4)   --++(-90:0.3) coordinate (hi);
   \fill[green!80!black] (hi) circle (3pt);

           \draw[cyan ,line width=0.08cm](0.5*12,1.4)   --++(-90:1.4) coordinate (hi);
   \end{tikzpicture}\end{minipage}=0
$$by induction on $\ell(\mu'')<\ell(\mu)$.  (We note that ${\sf1}_{  \stt_{\mu}}\otimes{\sf 1}_{	 \color{cyan}s_1\color{magenta}s_0}=0 $ by further application of the commutativity relation \eqref{commutingrelations}).

 For $(W,P)=(A_n,A_{n-1})$ we can assume without loss of generality that $\csigma= {\color{magenta}s_i}$ and $\ctau = {\color{cyan}s_j}$ for $i\leq j-2$.
Let $\mu=s_1 \dots s_k$ and  $j\leq k+1$, we have that 
  ${\sf 1}_{\stt_\mu }\otimes {\sf1}_{\color{magenta}s_i}\otimes {\sf1}_ {\color{cyan}s_j}
=0$ by    the $(s_{i-1}, 	{\color{magenta}s_i})$-null-braid relation, the  leftmost commutativity relation of \eqref{commutingrelations}, and the cyclotomic relation \eqref{R11}.  For  $j\geq k+2$ if follows immediately from the commutativity and  cyclotomic relations.

 For $(W,P)=(D_n,D_{n-1})$ we can assume without loss of generality that $\csigma= {\color{magenta}s_i}$ and $\ctau = {\color{cyan}s_j}$ for $i\leq j-2$.
 If $\mu \subseteq s_{n-1}s_{n-2}\dots s_{2}s_0$ or $s_{n-1}s_{n-2}\dots s_{2}s_1$ the result follows as in type $A$. 
 It remains to consider $   s_{n-1}s_{n-2}\dots s_{2}s_1 s_0			\subseteq \mu$ and we assume without loss of generality that 
% $\csigma= {\color{magenta}s_i}$ and $\ctau = {\color{cyan}s_j}$ for 
 $j\geq 3$ if $i=0$.
 If $\mu =  s_{n-1}s_{n-2}\dots s_{2}s_1 s_0	$ then $j-2\geq i\geq  2$ and we can apply the $({\color{cyan}s_j},s_{j-1})$-null-braid relation,
  the  leftmost commutativity relation of \eqref{commutingrelations}, and the cyclotomic relation \eqref{R11}. 

Finally, if $\mu = s_{n-1}s_{n-2}\dots s_{2}s_1 s_0	s_2 \dots s_k$ for $k\geq 2$.   
If $j\geq k+2$ (respectively $j\leq k+1$), we apply the $({\color{cyan}s_j},s_{j+1})$-null-braid   (respectively 
 $({\color{magenta}s_i},s_{i+1})$-null-braid) relation and the commutativity and  cyclotomic relations.
  \end{proof}

 \begin{defn}\label{zerobraidgenerator}
 For $(W,P)$ and $\csigma, \ctau$ as in \cref{theresultwethoughtweneed}, we will call the corresponding braid generator ${\sf braid}^{\csigma  \ctau}_{\ctau \csigma}$ a {\sf zero braid generator}. 
 \end{defn}

 \begin{prop}\label{gapiszero}
 \color{black!99}
(1) Let $W$ be simply laced and  $[r,c]\in \mu \in \mptn$. 
 If  $[r,c-1] \not \in \mathscr{A}_{(W,P)}$ (respectively  $[r-1,c] \not \in \mathscr{A}_{(W,P)}$) 
then $ {\sf gap}(\stt_\mu-[r,c-1])=0$ %. 
% For $\mu \in \mptn$ with $[r,c]\in \mu$
%  and $[r-1,c] \not \in \mathscr{A}_{(W,P)}$, we have that 
(respectively  $ {\sf gap}(\stt_\mu-[r-1,c])=0$). 

(2) For $(W,P)=(C_n, A_{n-1})$, $\mu\in \mathscr{P}_{(W,P)}$ and $[r,c]\in \mu$ with $[r,c-1]\notin \mathscr{A}_{(W,P)}$ we have $ {\sf gap}(\stt_\mu-[r-1,c])=0$.

(3) For $(W,P)=(B_n, B_{n-1})$ and $\mu\in \mathscr{P}_{(W,P)}$, if $[1,c]\in \mu$ then $ {\sf gap}(\stt_\mu-[1,c-1])=0$ and if $[r,n]\in \mu$ with $r\geq 3$ then $ {\sf gap}(\stt_\mu-[r-1,n])=0$. 
   \end{prop}

 \begin{proof}
 \color{black!99}
(1) First assume that  $r=1$ or $c=1$.  In either case the result follows by 
   the  leftmost commutativity relation of \eqref{commutingrelations}, and the cyclotomic relation \eqref{R11}. 
 We now consider the types in turn. Note that 
   $(W,P)= (A_n , A_{k} \times A_{n-k-1})$ follows from the $r=1$ and $c=1$ cases.  
For $(W,P)=(D_n,A_{n-1})$   we can assume that $s_{[r,c-1]}= s_2$ and $s_{[r,c]}= s_0$ or $s_1$ and the result follows from 
\cref{theresultwethoughtweneed}. Type  $(D_n,D_{n-1})$ also follows from \cref{theresultwethoughtweneed}. 
For the  exceptional Weyl groups, 
there are two types of subcase to consider.  
If $[r,c]$ is a $\spadesuit$ tile in \cref{spadeclub} then  
then $  {\sf gap}(\stt_\mu-[r,c-1])$
 or  $ {\sf gap}(\stt_\mu-[r-1,c])$ is zero immediately by the 
 leftmost commuting relation of \eqref{commutingrelations} and the cyclotomic relation \eqref{R11}.

\begin{figure}[h!]

\color{black}
 $$
 
$$

 \caption{\color{black!99}The tilings for $(W,P)$ of type $(E_6,D_5)$ or $(E_7,E_6)$. The   cases that $[r,c]$ is a $\spadesuit$ tile 
  all follow  from commutativity and cyclotomic relations. The $\clubsuit$ cases further require some applications of the null-braid relations. }
 \label{spadeclub} 
 \end{figure}

 We now consider the more interesting subcases 
 where $[r,c]$ is a $\clubsuit$ tile in \cref{spadeclub}. 
The $\clubsuit$ cases can all be treated uniformly using some applications of the null-braid relations and the  leftmost commuting relation of \eqref{commutingrelations} and the cyclotomic relation \eqref{R11}; rather than checking them all explicitly, we will just provide an illustrative example.  
For $(W,P)=(E_6,D_5)$ with colouring as in \cref{spadeclub,hhhh} and   ${\color{cyan}[r,c] }= {\color{cyan}[4,3]}$, with $\la_{\color{cyan}[4,3]}=(1^2,2,3)$ we have that 
$$
{\sf gap}(\stt_{(1^2,2,3)}-{\color{green!80!black}[4,2]} )=
  \begin{minipage}{3.2cm}
\end{minipage}
=   0$$he first equality is by definition, the second follows from the $({\color{cyan}s_2},{\color{green!80!black}s_3})$-null braid and the leftmost commutativity relation of  \eqref{commutingrelations}, 
the penultimate equality is again by commutativity,  
and the final equality holds by the cyclotomic relation \eqref{R11}.

Statement  (2) and the first part of statement  (3) are immediate from the commutation and cyclotomic relation \eqref{R11}. The second  part of statement  (3) follows by commutativity, repeated application of the nullbraid relation and cyclotomic relation. 
 \end{proof}

 \subsection{Gaps as basis elements}\label{gapsbasis}\color{black!99}
%Let $\mu \in \mptn$  and $[r,c] \in \mu$.  Let $a,b\geq 0$ be maximal such that 
%$$
%[r-a,c+a],
%[r-a+1,c+a ],
%%[r-a+1,c+a-1],
%\dots,
%%[r-c,c-1] ,
%[r-1,c],
%[r,c] 
%,[r+1,c],
%%[r+1,c-1],
%\dots ,
%[r+b,r-b+1],
%[r+b, c-b] \in \mu.
%$$that is, the 

%\begin{defn}
Fix $\mu \in \mptn$ for $(W,P) = (A_n , A_{k-1} \times A_{n-k})$ or $(D_n,A_{n-1})$ and let $[r,c]\in \mu$. 
Let $l, k$ be the maximal non-negative integer such that $[r-i, c+i] \in \mu$ for all $0\leq i\leq l$, $[r-i+1, c+i]\in \mu$ for all $1\leq i\leq l$ and $[r+j, c-j]\in \mu$ for all $0\leq j\leq k$, $[r+j, c-j+1]\in \mu$ for all $1\leq j\leq k$. Then define the {\sf  path  generated by the tile $[r,c]\in \mu$}, denoted by $\langle r,c \rangle_\mu$, to be the   collection  of tiles 
%$$[r{-}l, c{+}l], [r{-}l{+}1, c{+}l] ,\ldots[r{-}c,c{-}1] ,
%[r{-}1,c],
%[r,c] ,[r{+}1,c],[r{+}1,c{-}1],
% \ldots , [r{+}k, c{-}k] .$$
$$[r - k, c + k], [r - k + 1, c + k] ,\ldots[r - 1,c + 1] ,
[r ,c+1],
[r,c] ,[r + 1,c],[r + 1,c - 1],
 \ldots , [r + l, c - l]   $$see for example, the pink regions in  \cref{typesofwiggle}. Now, assume that $[r-k, c+k+1] \not \in \mu$ 
 and $[r+l + 1 ,c-l ] \not \in \mu$.  
\color{black} Suppose first that $[r-k+1, c+k+1] \notin \mu$.  In this case we choose $\la\subseteq \nu \subseteq \mu$ such that $\nu\setminus\la= \langle r,c\rangle_\mu$ 
 and $\stt_\mu$ such that $\stt_\mu = \stt_{\la }\circ \stt_{\nu\setminus\la }\circ \stt_{\mu\setminus\nu}$ 
 and we let ${s_{[r-k,c+k]}}= s_i$ and 
 ${s_{[r+l,c-l]}}= s_j$ with $i\leq j$.  We define 
 $$
 c_\SSTS=
( (R^+_{j-1}A^-_{j} )
\dots 
(R^+_{i+3}A^-_{i+4} )
(R^+_{i+1}A^-_{i+2} )
A^-_i ( {\sf 1}_{\stt_\la}))
\otimes {\sf 1}_{\stt_{\mu\setminus\nu}}.$$
 Now suppose $[r-k+1, c+k+1]\in \mu$. Note that, as we assumed $[r-k, c+k+1]\notin \mu$, this case can only happen when $(W,P)=(D,A)$ and $s_{[r-k,c+k]} = s_0$ or $s_1$. 
 Now take $m\geq 1$ maximal such that 
 $$
 [r-k+1,c+k+1],
 [r-k+2,c+k],
 \dots 
  [r-k+m,c+k-m+2] \in \mu.$$
We let $\langle \langle r,c\rangle \rangle _\mu $ denote the collection of tiles 
$$
[r-k+1,c+k+1],[r-k,c+k-1],
% [r-k+2,c+k+2], [r-k-1,c+k-2],
 \dots 
  [r-k+m,c+k-m+2] ,   [r-k+m-1,c+k-m]
 $$see for example, the blue region in  \cref{typesofwiggle}. \color{black}

 \begin{figure}[h!]
 \color{black}
 $$    $$

\caption{The first two cases depict $\color{magenta}\langle r,c\rangle_\mu$ for $[r,c]= [4,4], [5,4] \in \mu=(7^2,6^2,5^3,2) \in \mptn $ with $(W,P)=(A,A\times A)$. \color{black}In the first case we have $k=3$ and $l=3$. In the second case we have $k=2$ and $l=3$. \color{black} Notice that in the first case the $\times$ denotes the box $[r+l + 1 ,c-l ]   \in \mu$.
On the right we depict 
$\color{magenta}\langle r,c\rangle_\mu$ 
and 
$\color{cyan}\langle\langle r,c\rangle\rangle_\mu$  for $[r,c]=[7,5] \in (1,2,3,4,5,6,7,8^2,3,1^4)$, we note that in this final case $k=1, l=3$ and $m=3$.   
 }
\label{typesofwiggle}
\end{figure}

 \noindent Now we choose $\la\subseteq \nu \subseteq \mu$ such that $\nu\setminus\la=
  \langle r,c\rangle_\mu \sqcup  \langle \langle r,c\rangle\rangle_\mu$ 
 and $\stt_\mu$ such that $\stt_\mu = \stt_{\la }\circ \stt_{\nu\setminus\la }\circ \stt_{\mu\setminus\nu}$. 
By assumption,  ${\color{black}s_{[r-k,c+k]}}= \color{black}s_0$ or $s_1$
 and in the former case 
we define 
 $$
 c_\SSTS=
((R^+_{2k+2l}A^-_{2k+2l+1} )
 \dots
 (R^+_{2m+2}A^-_{2m+3} ))
((R^+_{2m-1}R^-_{2m}A^-_{2m+1}A^+_{2m} ) 
\dots 
(R^+_{1}R^-_{2}A^-_{3}A^+_{2} ))
A^-_0 ( {\sf 1}_{\stt_\la}))
\otimes {\sf 1}_{\stt_{\mu\setminus\nu}}.
 $$and in  the latter case 
we define 
 $$
 c_\SSTS=
((R^+_{2k+2l}A^-_{2k+2l+1} )
 \dots
 (R^+_{2m+2}A^-_{2m+3} ))
((R^+_{2m-1}R^-_{2m}A^-_{2m+1}A^+_{2m} ) 
\dots 
(R^+_{0}R^-_{2}A^-_{3}A^+_{2} ))
A^-_1 ( {\sf 1}_{\stt_\la}))
\otimes {\sf 1}_{\stt_{\mu\setminus\nu}}.
 $$Examples are depicted in \cref{afiguredweneed,afiguredweneed2}.

 Now, if $[r-k, c+k+1]   \in \mu$ 
or  $[r-l + 1 ,c-l ]  \in \mu$ we claim that 
\begin{align}\label{iszero}
{\sf gap}(\stt_\mu - [r,c])={\sf gap}(\stt_\mu - [r+l,c-l])=
{\sf gap}(\stt_\mu - [r-k,c+k])=
0.
\end{align}The first two equalities follow by repeated applications of 
 \cref{underphi}. The final equality is immediate from 
  the leftmost commuting relation of \eqref{commutingrelations} and the cyclotomic relation \eqref{R11}.  
 Otherwise, 
 \begin{align}\label{isnonzero}
{\sf gap}(\stt_\mu - [r,c])=
{\sf gap}(\stt_\mu - [r+l,c-l])=
{\sf gap}(\stt_\mu - [r-k,c+k])=
(-1)^{k+l-m}c_{\SSTS}^*c_{\SSTS}.
\end{align}where $c_\SSTS$ is defined above.  
%By \cref{underphi} we can assume that $k=0$ as in \cref{afiguredweneed2,afiguredweneed}.  
The first two equalities follow directly from  \cref{underphi}.
To verify the  $m=0$ case of the third equality, we simply apply   
\cref{underphi}   $(k+l)$ times  from left-to-right as in 
 \cref{afiguredweneed2}.  
For $m>0$, we apply 
\cref{underphi2} $m$ times, and then apply 
\cref{underphi} $k+l-m$ times.

\begin{figure}[h!]
\color{black}\scalefont{0.9}
$$  
        $$
        \caption{\color{black!99}For $\mu=(7^2,6^2,5^3,2)$ and $[r,c] = [3,6	]$ we we depict on  the left the 
 operators  applied in the definition of $c_\SSTS$. 
 The thick lines break up the operators according to the bracketed terms in the definition of $c_\SSTS$. 
 On the right we depict  $\mu-{\color{magenta} 
   \langle 3,6 \rangle_\mu}  $, obtained by  deleting the pink  tiles and letting the white tiles
  fall under gravity. 
        }
        
        \label{afiguredweneed2}
        
        \end{figure}
 
\begin{figure}[ht!] \color{black}\scalefont{0.8}
   $$  \scalefont{0.9}   %
   $$
        \caption{\color{black!99}
%        On the left we have drawn the algorithm decorations on the tile partition
For  $\mu=(1,2,3,4,5,6,7,8^2,3,1^4)$ and $[r,c]=[6,6]$ we depict on  the left the 
 operators  applied in the definition of $c_\SSTS$. 
 The thick lines break up the operators according to the bracketed terms in the definition of $c_\SSTS$. 
 On the right we depict  $\mu-{\color{magenta}   \langle 6,6 \rangle_\mu}-\color{cyan}
  \langle \langle 6,6\rangle\rangle_\mu $, obtained by  deleting the pink and blue  tiles and letting the white tiles
  fall under gravity. 
  }
        \label{afiguredweneed}
   \end{figure}

\color{black} 
 \subsection{Tetris combinatorics for barbells}
We now provide closed combinatorial formulas for removing barbells from diagrams.

   \begin{figure}[ht!]
  $$
  
$$ 
 \caption{
Examples of trails $ {\tt T}^\mu_{[x,y]\to [1,1]}$, the first 4 of which are maximal length (the 5th is not).  
Here $[x,y]=[5,4], [7,4], [8,3]$  and  $(W,P)=(A_9,A_3\times A_5)$,
$(D_8,A_7)$, $(E_6,D_5) $
with   $\mu=(4^5,3)$, $(1,2,3,4,5^2,4)$, and $(1^2,2,4^2,2,1^2)$ respectively (only the distinct cases are listed).  
The 5th trail is non-maximal length as it has length  6, whereas the 4th
trail has length 8.  
The grey tiles $[a,b]$ are those such that ${\sf gap}(\stt_\mu - [a,b])=0$ by \cref{underphi,gapiszero}}
    \label{hooks-pic}
\end{figure}

\begin{defn}\label{hook:Def}
\color{black!99}  Let $\mu \in \mptn$. 
Given $[x,y]$ a (possibly non-admissible) tile, we  set   $SW[x,y]=[x-1,y]$ and $SE[x,y]=[x,y-1]$.  
For a pair of such tiles    $[x,y]$ and $[x',y']$			we define a {\sf trail}, denoted $ {\tt T}_{[x,y]\to [x',y']}$,
 to be a (possibly empty) set of tiles
\begin{align}\label{asequeneee}
%{\sf trail}[x,y] = (
[x,y]=T_1, T_2, \dots, T_{x+y-x'-y'+1}=[x',y']
%)
\end{align}
such that $T_{i+1}=SW(T_i)$ or $T_{i+1}=SE(T_i)$ for $1<i \leq x+y-x'-y'+1$.  
We  
 write 
 $$ {\tt T}^\mu_{[x,y]\to [x',y']}:= \mu\cap  {\tt T}_{[x,y]\to [x',y']}$$\color{black}and we define 
 the {\sf $\mu$-length} of the  trail to be $|{\tt T}^\mu_{[x,y]\to [x',y']}|$.  
Given  $\ctau=[r,c] \in \Add(\mu)$,        we let  ${\rm Hook}_{\ctau}(\mu) $ denote any multiset of the form 
     $$
     {\rm Hook}_{\ctau}(\mu) = 
          \begin{cases}
          2  \cdot {\tt T}^\mu_{[r,c-1]\to [1,1]}			&\text{if ${[r-1,c]\not \in \mathscr{A}_{(W,P)}}$  
          in type $(C_n,A_{n-1})$}				\\
 {\tt T}^\mu_{[r-1,c]\to [1,1]}\sqcup {\tt T}^\mu_{[1,n]\to [1,1]}			&\text{if ${[r,c-1]\not \in \mathscr{A}_{(W,P)}}$  
          in type $(B_n,B_{n-1})$}				\\
      {\tt T}^\mu_{[r-1,c]\to [1,1]}\sqcup 
  {\tt T}^\mu_{[r,c-1]\to [1,1]}				&\text{otherwise}
  \end{cases}$$for any preferred  choices   of   maximal $\mu$-length trails on the right-hand side.  
 \end{defn}

 This allows us to provide a closed combinatorial formula for rewriting barbells in diagrams, as follows.  
 This formula is essential to our proof of combinatorial invariance.  Finding such formulas for general $(W,P)$ seems to be an impossible task. 
 
\begin{prop}\label{hook:Def}\label{remove a barbell}
Let $(W,P)$ be a Hermitian symmetric pair.  
  Let $\mu\in \mptn $  and  $\ctau={\color{cyan}[r,c]} \in \Add(\mu)$.  
 We have that 
%\begin{align}\label{Chook}{
%\sf 1}_{ \stt_{\mu}}  \otimes {\sf bar}(\ctau  )  
% = 
%- 
%\sum_{y<r } 
%2 {\sf gap}(\stt_\mu- [r,y]) 
%  \qquad \text{if } \ctau ={\color{cyan}s_{[r,r]} }   \text{ and $(W,P)=(C_n,A_{n-1})$ } 
%  \end{align}
%and in all other cases, we have that %define 
\begin{align}\label{notChook} 
{\sf 1}_{ \stt_{\mu}}  \otimes {\sf bar}(\ctau  )  
  = 
-\!\!
\sum_{[x,y] \in {\sf Hook}_\ctau(\mu)}
  {\sf gap}(\stt_\mu- [x,y]). 
  \end{align}
 \end{prop}

Before embarking on the proof, we emphasise that  \cref{notChook} has a lot of redundancy.  
Many of the terms on the righthand-side of this sum are zero, using the results of \cref{zeroooooooo} (some of these are highlighted in grey in  \cref{hooks-pic}).
We can  pick preferred choices of the maximal length trails in the definition of ${\sf Hook}_\ctau(\mu)$ and delete some of these redundant terms.  
\color{black!99} In particular, in classical types we have the following simplification. 
\color{black}

\begin{lem}\label{irrelevant}\color{black!99}
Let   $\mu\in \mptn$   and $\ctau=[r,c] \in \Add(\mu)$.
For $(W,P)= (A_n, A_{k-1}\times A_{n-k})$,  $(D_n,A_{n-1})$, or $(D_n,D_{n-1})$, or finally 
           $(W,P)=(C_n,A_{n-1})$ and     $r\neq c$, we set   		  
 $$
 \underline{{\sf Hook}}_\ctau (\mu)=
 {\sf T}_{[r-1,c]\to [r-1,2]}^\mu
 \sqcup
  {\sf T}_{[r ,c-1]\to [r,1]}^\mu.
$$
If $(W,P)=(C_n,A_{n-1})$ with     $r=c$, we set   		  
 $$
 \underline{{\sf Hook}}_\ctau (\mu)=  2 {\sf T}_{[r ,r-1]\to [r,1]}^\mu.$$
If $(W,P)= (B_n, B_{n-1})$ then we set 
$$
 \underline{{\sf Hook}}_\ctau (\mu)
 =
 \begin{cases}
 [1,c-1]		 			&		\text{if }r=1
 \\
 2[1,n]		 			&		\text{if }[r,c]=[2,n]
 \\
 2[1,n]		\sqcup  [r-1,n]	 &		\text{if }[r,c]=[r,n] \text{ with }r\geq 3.
  \end{cases}
 $$We have that 
 $$
\sum_{[x,y] \in  {\sf Hook}_\ctau(\mu) } 
 {\sf gap}(\stt_\mu- [x,y])
=
\sum_{[x,y] \in    \underline{{\sf Hook}}_\ctau(\mu) } 
 {\sf gap}(\stt_\mu- [x,y]) .
 $$
   \end{lem}
\begin{proof}
We need only to show that  ${\sf gap}(\stt_\mu -[x,y])=0$ for $[x,y]
\in
 {{\sf Hook}}_\ctau (\mu) \setminus   \underline{{\sf Hook}}_\ctau (\mu)
$. 
This follows from  \cref{underphi,gapiszero}. 
 \end{proof}

Examples of the multisets  $\underline{{\sf Hook}}_\ctau (\mu)$ are provided in \cref{hooks-pic2}.

 \begin{figure}[ht!]
  $$
  
$$ 
 \caption{
The multisets $\underline{{\sf Hook}}_\ctau (\mu)$. 
The first  case  is  
  $\ctau=\color{cyan}[6,4]$ and $\mu=(5^2,4^3,3,1)$ in type $(A_{11},A_{4}\times A_{6})$.  
The second  case is    $\ctau=\color{cyan}[6,4]$ and $\mu=(1,2,3,4^2,3,1)$ in  type $(D_8,A_7)$.
  The third and fourth cases are for $\mu=  (1,2,3,4,5^2 )$ and  $\ctau=\color{cyan}s_{[6,6]}$ in types    $(D,A)$  and $(C,A)$  respectively.  
 We   highlight the tiles  in 
 $\underline{{\sf Hook}}_\ctau (\mu)$
  by placing a gap diagram in the tile and the multiplicity of that tile within the multiset.   
 }
    \label{hooks-pic2}
\end{figure}

\begin{proof}[Proof of \cref{remove a barbell}] 
We consider the  cases in which the parabolic   is of type $A$.  The other cases are left as an exercise for the reader.    
 By the commutativity relations, it is enough to prove the result for  
  $\ctau  $ and  $\mu$ such that $\ctau=  \color{cyan}s_{[r,c]}$ and $\mu\ctau=\la_{ \color{cyan}[r,c]}$ for some $r,c\geq 1$
  (in the notation of \cref{flashback}).  
 If $r=c=1$ then the result is immediate from the cyclotomic relation.  
 If $r>1$ and  $c=1$ or $r=1$ and $c>1$, then   the result follows by \Cref{lem1} and the cyclotomic relation.  
Thus by  \cref{irrelevant}  it is enough to show that 
$ 
{\sf 1}_{ \stt_{\mu}}  \otimes {\sf bar}(\ctau  )  = - \sum_{[x,y] \in \underline{{\sf Hook}}[r,c ]}      
 {\sf gap}(\stt_\mu- [x,y]) 
 $ for $r,c>1$.

\smallskip
\noindent {\bf Case 1. }  We now assume that $r,c>1$, and consider the case where $\color{cyan} {[r,c]}		\color{black}	 $ is such that 
 $[r,c-1], \color{magenta}  [r-1,c]\color{black}	\in \mu$ as this is uniform across all types.  
We set $\nu$  
to be the partition obtained by removing the final box from each column of $\mu$, that is $\nu\csigma=\la_{[r-1,c]}$.  
We have that 
 \begin{align*}
 {\sf 1}_{ {\stt_\mu}}  \otimes  {\sf bar}([r,c])=
 {\sf 1}_{\stt_{\nu\csigma}} \otimes 		  {\sf bar}(	 \color{cyan}  [r,c]\color{black}		) \otimes {\sf 1}_{ \stt_{\mu\setminus \nu\csigma} }
 +\sum_{y< c} 
   {\sf 1}_{\stt_{ \nu\csigma}} \otimes 	{\sf bar}([r,y])
 		\otimes {\sf 1}_{	\stt_{\mu  \setminus \nu\csigma}}	 	
  - \!\sum_{ 1\leq y < c  } \!\!
 {\sf gap}(\stt_\mu-[r,y]) 
 \end{align*}  
 by \cref{lem1} and applying  \cref{useful2} we obtain 
\begin{align*}
 {\sf 1}_{ {\stt_\mu}}  \otimes  {\sf bar}([r,c])=
 {\sf 1}_{\stt_{\nu}} \otimes 		  {\sf bar}(	 \color{magenta}  [r-1,c]\color{black}		) \otimes {\sf 1}_{ \stt_{\mu\setminus \nu} }
 +\sum_{y\leq c} 
   {\sf 1}_{\stt_{ \nu}} \otimes 	{\sf bar}([r,y])
 		\otimes {\sf 1}_{	\stt_{\mu  \setminus \nu}}	 	
  - \!\sum_{ [x,y] \in \mu \setminus \nu } \!\!
 {\sf gap}(\stt_\mu-[x,y]).
 \end{align*}     
We set $\pi $  
to be the partition obtained by removing the final two rows of $\mu$, that is $\pi=\la_{[r-2,c]}$.  
By \cref{lem2}, we have that
 $$
\sum_{y\leq c} 
   {\sf 1}_{\stt_{ \nu}} \otimes 	{\sf bar}([r,y])
 		\otimes {\sf 1}_{	\stt_{\mu  \setminus \nu}}			=
		{\sf 1}_{\stt_{\pi	}}\otimes {\sf bar}([r,1]) \otimes 
 {\sf 1}_{\stt_{\mu \setminus  \pi}}
 + \sum_{  y <  c } 
 {\sf gap}(\stt_\mu-[r-1,y])  $$  
 and we note that the first term after the equality is zero by the 
 commutativity and cyclotomic relations.  
Putting these two equations above together, we have that 
 $$  {\sf 1}_{ {\stt_\mu}}  \otimes  {\sf bar}(		\color{cyan}  [r,c]\color{black}		)=
 {\sf 1}_{\stt_{\nu}} \otimes 		  {\sf bar}(\color{magenta}  [r-1,c]\color{black}	) \otimes {\sf 1}_{ \stt_{\mu\setminus \nu} }
 +
	\Big(	\sum_{  y< c } 
 {\sf gap}(\stt_\mu-[r-1,y]) 
    - \!\sum_{ [x,y] \in \mu \setminus \nu } \!\!
 {\sf gap}(\stt_\mu-[x,y]) \Big)$$  
and so the result follows by induction and \cref{underphi,gapiszero}.
This inductive step is  visualised in  \cref{typeA ind}; we have bracketed the latter two terms above in order to facilitate comparison with the  rightmost diagram in \cref{typeA ind}.

\begin{figure}[ht!]
$$  \scalefont{0.9}
  \begin{minipage}{3.65cm}
  \end{minipage}    
      $$
        \caption{
        The    first term on the left-hand side  
      depicts  ${\sf 1}_{ {\stt_{\nu}}}  \otimes  {\sf bar}(		\color{magenta}  [r-1,c]\color{black}	) \otimes {\sf 1}_{\stt_{\mu\setminus \nu}} $ (known by induction)
    the   second term  depicts the coefficients of the gap   terms in  the inductive step in the proof.   
    The first equality records the cancellations; the second equality follows from \cref{underphi,gapiszero}. 
        The rightmost  diagram  depicts 
         ${\sf 1}_{ {\stt_\mu}}  \otimes  {\sf bar}(		\color{cyan}  [r ,c]\color{black}			)$ (for $r\neq c$ in types $C$ and $D$).     
        }
        \label{typeA ind}
        \end{figure}

\noindent {\bf Case 2. } 
Now consider the type   $C$ and $D$ cases for $\ctau=\color{cyan}  s_{[r,r]}$ with $r> 1$ and we let $\csigma=\color{magenta} s_2\color{black}  \in W$.     
 $$
{\sf 1}_{\stt_\mu}\otimes {\sf bar}(	 \color{cyan} [r,r]	\color{black}		) = 
\begin{cases}
  {\sf 1}_{\stt_{\mu-\csigma}  }\otimes ({\sf bar}(\ctau) +2 {\sf bar}(\csigma) )\otimes {\sf 1}_\csigma
- 2{\sf gap}(\stt_\mu - \color{magenta}  [r,r-1]\color{black}		) 	& \text{ in type $C$ }\\
{\sf 1}_{\stt_{\mu-\csigma}  }\otimes ({\sf bar}(\ctau) + {\sf bar}(\csigma))\otimes {\sf 1}_\csigma
- {\sf gap}(\stt_\mu - \color{magenta}  [r,r-1]\color{black})	& \text{ in type $D$ } 
\end{cases}
$$
%$$
%{\sf 1}_{\stt_\mu}\otimes {\sf bar}(	 \color{cyan} [r,r]	\color{black}		) = 
%\begin{cases}
%  {\sf 1}_{\stt_{\mu-\csigma}  }\otimes ({\sf bar}( \color{cyan} [r,r]	\color{black}) +2 {\sf bar}  (\color{magenta} [r,r-1]	\color{black}) )\otimes {\sf 1}_\csigma
%- 2{\sf gap}(\stt_\mu - \color{magenta}  [r,r-1]\color{black}		) 	& \text{ in type $C$ }\\
%{\sf 1}_{\stt_{\mu-\csigma}  }\otimes ({\sf bar}( \color{cyan} [r,r]	\color{black}) + 
%{\sf bar}(\color{magenta} [r,r-1]	\color{black}))\otimes {\sf 1}_\csigma
%- {\sf gap}(\stt_\mu - \color{magenta}  [r,r-1]\color{black})	& \text{ in type $D$ } 
%\end{cases}
%$$
and the $r=2$ case now follows by the cyclotomic relation.  

Now suppose $r>2$.   
We first consider the type $C$ case.  
We set $\nu$  
to be the partition such that $\nu\ctau=\la_{[r-1,r-1]}$. 
By the commutativity and one-colour barbell relations, we have that 
$$
{\sf 1}_{\stt_{\mu-\csigma}} \otimes {\sf bar}(\ctau) \otimes {\sf 1}_\csigma =  
-{\sf 1}_{\stt_\nu } \otimes  {\sf bar}(\ctau) \otimes{\sf 1}_{\stt_{\mu\setminus \nu} } 
+ 
2   {\sf gap}   ({\stt_{\mu}}  - \color{cyan} [r-1,r-1]	\color{black}	)$$
and 
so 
\begin{align*}
  {\sf 1}_{ {\stt_\mu}}  \otimes  {\sf bar}(	
   \color{cyan} [r,r]	\color{black}
  )= &
 	2\times   {\sf 1}_{\stt_{\mu-\csigma}  }\otimes  {\sf bar}(
 \color{magenta} [r,r-1]	\color{black}	
	) \otimes {\sf 1}_\csigma  
   -{\sf 1}_{\stt_{\nu}} \otimes 		
     {\sf bar}(		 \color{cyan} [r-1,r-1]	\color{black}	) 
     \otimes {\sf 1}_{ \stt_{\mu\setminus \nu} }    \\
&  +2 (  
 {\sf gap}(\stt_\mu
 -
   \color{cyan} [r-1,r-1]	\color{black}
 ) -  
 2
   {\sf gap}(\stt_\mu-
\color{magenta} 	[r,r-1] \color{black}
)  ) 
\end{align*}
%and the result follows by induction (see \cref{typeC ind} for a visualisation of this step).   
%The bracketed pair of terms correspond
and so the result follows by induction.  
This inductive step is  visualised in  \cref{typeC ind}; we have bracketed the latter two terms above in order to facilitate comparison with the  rightmost diagram in \cref{typeC ind}.

\begin{figure}[ht!]
$$  \scalefont{0.9}\begin{minipage}{2.55cm}
  \end{minipage}  
   $$
        \caption{
     The left-hand side  depicts 
         ${\sf 1}_{ {\stt_\mu}}  \otimes  {\sf bar}(
\color{cyan}         [r,r]\color{black}
         )$ in type $C$.  
        The    first term on the right-hand side  
      depicts  
      $2\times {\sf 1}_{ {\stt_{\mu-\csigma}}}  \otimes  {\sf bar}(
\color{magenta}       [r,r-1]\color{black} 
      )\otimes {\sf 1}_{\csigma} $; 
    the   second term  depicts 
    $-{\sf 1}_{ {\stt_{\nu}}}  \otimes 
     {\sf bar}
     (\color{cyan} [r-1,r-1]
     \color{black} )\otimes {\sf 1}_{\stt_{\mu\setminus \nu}} $;
     the third term depicts the coefficients of the gap   terms in  the inductive step.      }
        \label{typeC ind}
        \end{figure}

We now consider  type $D$. 
 %For notational purposes, we suppose that $r$ is even so that $\ctau=
%\color{cyan}s_{[r,r]}\color{black}=\color{cyan}s_{1}$
%and $ 
We colour $\color{violet}s_{[r-1,r-1]} $ violet and set  
  $\nu$,   $\pi$, and $\rho$ to be the partitions 
     $\nu+
{\color{violet} {[r-1,r-1]} } =\la_{[r-1,r-1]}$,
  $\pi +{\color{cyan}[r-2,r-2]}
     = \la_{\color{cyan}[r-2,r-2]}$,  
and 
$\rho  =\la_{[r-3,r-3]}$. 
We have that 
 $$
{\sf 1}_{ \stt_\mu } \otimes {\sf bar}(\color{cyan}	[r,r]	\color{black}	)	 =
{\sf 1}_{ \stt_\nu } \otimes       {\sf bar}(\color{violet}[r-1,r-1 ] \color{black})   \otimes 
{\sf 1}_{ \stt_{\mu\setminus \nu} }+
\sum_{ y=1}^r  
{\sf 1}_{ \stt_\nu } \otimes       {\sf bar}([r,y])   \otimes 
{\sf 1}_{ \stt_{\mu\setminus \nu} }
-  \! \sum_{[x,y] \in \mu\setminus \nu}  \!\!\! {\sf gap}   (\stt_{\mu} - [x,y])
  $$by  \cref{lem1}. The second term can be rewritten as follows 
 \begin{align*}
\sum_{ y=1 }^r   {\sf 1}_{\stt_\nu} \otimes     
{\sf bar}([r,y]) \otimes 
{\sf 1}_{\stt_{\mu\setminus \nu} }
&=
  \sum_{y=1 }^r  {\sf 1}_{\stt_{\pi\blue}}  
  \otimes 
  {\sf bar}([r,y]) 
 \otimes {\sf 1}_{\stt_{\mu	-\pi\blue }} 
 \\
 &=
 \sum_{y=1 }^{r-1}  {\sf 1}_{\stt_\pi}   \otimes 
  {\sf bar}([r,y]) 
 \otimes {\sf 1}_{\stt_{\mu	- \pi  } }
 + {\sf gap}(\stt_\mu - \color{cyan}[r-2,r-2] \color{black})  
\\ &=
\sum_{y=1}^{r-2}{\sf gap}(\stt_\mu - [r-2,y])     +   {\sf 1}_{\stt_\rho} \otimes ( {\sf bar}([r,1]) +  {\sf bar}([r,2])) \otimes \on_{\stt_{\mu-\rho}}  
 \\&
=\sum_{y=1}^{r-2}{\sf gap}(\stt_\mu - [r-2,y]) 
\end{align*}
where the first equality follows  by repeated applications of   \cref{useful3}; the second from    \cref{useful2};
 the third from  \cref{lem2} and the commutation relations;
 the fourth from  the commutation and cyclotomic  relations (notice that 
no tile in $\pi$ has colour label corresponding to the reflections 
$s_{[r,1]}$ or $s_{ [r,2]}$).  
 Substituting this into the above, we obtain 
 $$
 {\sf 1}_{ \stt_\mu } \otimes {\sf bar}(
 \color{cyan} [r,r]\color{black}
 ) =
{\sf 1}_{ \stt_\nu } \otimes       {\sf bar}(
\color{violet} [r-1,r-1]\color{black}
)   \otimes 
{\sf 1}_{ \stt_{\mu\setminus \nu} }
+ % 
 \sum_{y=1}^{r-2}{\sf gap}(\stt_\mu - [r-2,y]) 
-  \!\!\!\sum_{[x,y] \in \mu\setminus \nu}  {\sf gap}   (\stt_{\mu} - [x,y])
$$
The result follows by induction (see \cref{typeD ind} for a visualisation of this step).  
 \end{proof}

 \begin{figure}[ht!]
$$\scalefont{0.9}  \begin{minipage}{2.55cm}
  \end{minipage}  
   $$
        \caption{
     The lefthand-side  depicts 
         ${\sf 1}_{ {\stt_\mu}}  \otimes  {\sf bar}(	\color{cyan}	[r,r] 	\color{black}	)$.  
        The    first term on the righthand-side  
      depicts  
      $ {\sf 1}_{ \stt_\nu } \otimes       {\sf bar}( \color{violet}	[r-1,r-1]		\color{black}	)   \otimes 
{\sf 1}_{ \stt_{\mu\setminus \nu} }$ (known by induction); 
    the   second and third terms  depict 
    $+   
 \sum_{y\leq r }{\sf gap}(\stt_\mu - [r-2,y]) 
$
and 
$-   \sum_{[x,y] \in \mu\setminus \nu}  {\sf gap}   (\stt_{\mu} - [x,y])$ respectively, which provide  the gap   terms in  the inductive step in the proof.      }
        \label{typeD ind}
        \end{figure}

  \subsection{The Tetris-style presentation}  
We are now ready to provide a new presentation for the Hecke categories of  %simply laced
 simply laced  Hermitian symmetric pairs.   One should notice that this presentation is mainly given in terms of the tiling combinatorics and {\em not} the usual Dynkin diagram combinatorics (the exception to this being discussion of commuting relations which are ``far apart" in the Dynkin diagram).

 \begin{thm}\label{easydoesit}   \renewcommand{\vvv}{{\underline{w} }} 
\renewcommand{\w}{{\underline{x}}}
\renewcommand{\x}{{\underline{y}}}
\renewcommand{\y}{{\underline{z}}}
\newcommand{\zz}{{\underline{w}}}
 
Let $(W,P)$ denote a  simply laced Hermitian symmetric pair.  
 The   algebra  $ \mathcal{H}_{(W,P)}  $  can be defined as 
 the locally-unital  associative $\Bbbk$-algebra 
spanned by  simple  Soergel   diagrams 
 with    multiplication given by vertical concatenation of diagrams modulo the following local relations and their horizontal and vertical flips.   Firstly,  for any   $ \csigma \in S_W$  we have the  relations 
\begin{align*}
{\sf 1}_{\csigma} {\sf 1}_{\ctau}& =\delta_{\csigma,\ctau}{\sf 1}_{\csigma} & {\sf 1}_{\emptyset} {\sf 1}_{\csigma} & =0 & {\sf 1}_{\emptyset}^2& ={\sf 1}_{\emptyset}\\
{\sf 1}_{\emptyset} {\sf spot}_{\csigma}^\emptyset {\sf 1}_{\csigma}& ={\sf spot}_{\csigma}^{\emptyset} & {\sf 1}_{\csigma} {\sf fork}_{\csigma\csigma}^{\csigma} {\sf 1}_{\csigma\csigma}& ={\sf fork}_{\csigma\csigma}^{\csigma} &
{\sf 1}_{\ctau\csigma}  {\sf  braid}_{\csigma\ctau}^{\ctau\csigma}
 {\sf 1}_{\csigma\ctau}  & ={\sf  braid}_{\csigma\ctau}^{\ctau\csigma} 
\end{align*}
where the final relation holds for all ordered pairs  $( \csigma,\ctau)\in S_W^2$  with $m(\csigma, \ctau)  =2  $.  
For each  $\csigma \in S_W $  we have   fork-spot contraction, the double-fork, and circle-annihilation  
 relations:  
  \begin{equation*} 
({\sf spot}_\csigma^\emptyset \otimes {\sf 1}_\csigma){\sf fork}^{\csigma\csigma}_{\csigma}
=
{\sf 1}_{\csigma},
 \quad 
  ({\sf 1}_\csigma\otimes {\sf fork}_{\csigma\csigma}^{ \csigma} )
({\sf fork}^{\csigma\csigma}_{\csigma}\otimes {\sf 1}_{\csigma})
=
{\sf fork}^{\csigma\csigma}_{\csigma}
{\sf fork}^{\csigma}_{ \csigma\csigma},
\quad
{\sf fork}_{\csigma\csigma}^{\csigma}
{\sf fork}^{\csigma\csigma}_{\csigma}=0,
 \end{equation*} 
   For   $(\csigma, \ctau ,\crho)\in S^3$ 
with   $m_{\csigma \crho}=
m_{  \crho\ctau} =m_{\csigma \ctau}=2$,  we have the commutation relations  
 \begin{equation*} 
\begin{minipage}{1.3cm}
\end{minipage}
 \end{equation*}
  For $\mu$ any partition tiling and $\csigma\in {\rm Add}(\mu)$, we have the monochrome Tetris relation
$$  {\sf 1}_{\stt_{\mu\csigma}}\otimes {\sf 1}_\csigma = 
 {\sf 1}_{\stt_{\mu}} \otimes ({\sf spot}_\csigma ^\emptyset \otimes {\sf fork}_\csigma^{\csigma\csigma}
 +{\sf spot}^\csigma _\emptyset \otimes {\sf fork}^\csigma_{\csigma\csigma})
 +\!\!
 \sum_{[r,c] \in {\rm Hook_\csigma} (\mu)} 
 {\sf gap}(\stt_\mu- [r,c])  \otimes {\sf dork}^{\csigma\csigma}_{\csigma\csigma}.
$$
For any  $\csigma,\ctau \in S_W$ with $m(\csigma,\ctau)=3$,  
%\sout{such that 
% $ \csigma\not\in \Add(\mu\csigma\ctau) $}, 
 we have the null-braid relation
\begin{align*} 
% {\sf 1}_{\stt_\mu} \otimes 
% \left(
 {\sf 1}_{\csigma\ctau\csigma} + 
 ({\sf 1}_\csigma \otimes {\sf spot}^\ctau _\emptyset \otimes  {\sf 1}_\csigma )
 {\sf dork}^{\csigma\csigma}_{\csigma\csigma}
  ({\sf 1}_\csigma \otimes {\sf spot}_\ctau ^\emptyset \otimes  {\sf 1}_\csigma )
%  \right) 
 =0
\end{align*}
 For   $\mu \in \mptn $  and 
 $\ctau\in \Add(\mu)$, we have the bi-chrome Tetris relation 
 $${\sf 1}_{ \stt_{\mu}}  \otimes {\sf bar}(\ctau  )  
 = 
-\!\!
\sum_{[r,c] \in {\rm Hook_\ctau} (\mu)} 
 {\sf gap}(\stt_\mu- [r,c]). $$Further,  we require the interchange law and the monoidal unit relation 
$$
% \big(({\sf D  }_1\circ	{\sf 1}_{ {\w}}   )\otimes  ({\sf D}_2 \circ {\sf 1}_{ {\x}}) \big)
%\big(({\sf 1}_{\w} \circ {\sf D}_3)  \otimes ({\sf 1}_\x \circ{\sf D }_4)\big)
%=
% ({\sf D}_1 \circ {\sf 1}_{{\w}} \circ {\sf D_3}) \otimes ({\sf D}_2 \circ {\sf 1}_{{\x}} \circ {\sf D}_4)
 \big( {\sf D  }_1 \otimes   {\sf D}_2   \big)\circ  
\big(  {\sf D}_3  \otimes {\sf D }_4 \big)
=
 ({\sf D}_1 \circ   {\sf D_3}) \otimes ({\sf D}_2 \circ  {\sf D}_4)
%$$
%and the monoidal unit relation
%$$
\qquad {\sf 1}_{\emptyset} \otimes {\sf D}_1={\sf D}_1={\sf D}_1 \otimes {\sf 1}_{\emptyset}
$$
for all diagrams ${\sf D}_1,{\sf D}_2,{\sf D}_3,{\sf D}_4$.  
Finally, we require the   non-local  cyclotomic relations
\begin{align*}
{\sf bar}(\csigma)\otimes {\sf D}  =0 \qquad &\text{for  all   $\csigma \in S_W$ and ${\sf D}$ any diagram  }
%$$
%%  and the parabolic annihilation relation
%$$
\\
{\sf 1}_\ctau \otimes {\sf  D}=0 \qquad &\text{for  all   $\ctau \in S_P\subseteq S_W$ and ${\sf D}$ any diagram.  }
\end{align*}

 \end{thm}
\begin{proof}
In light of \cref{inlightof}, we need only show that 
   the one and two colour barbell relations can be replaced by the 
   monochrome and bi-chrome   Tetris relations.

  Given two simple Soergel   diagrams ${\sf D}_1$ and  ${\sf D}_2$, 
% the one and two colour barbell relations tell 
%  us how to remove   barbells  from the product ${\sf D}_1{\sf D}_2$ in order to rewrite ${\sf D}_1{\sf D}_2$ as a linear combination of simple Soergel   diagrams.   
%Similarly, 
the one and two colour barbell relations allow us to inductively move leftwards any barbell {\em anywhere} in ${\sf D}_1{\sf D}_2$; once all barbells are at the leftmost edge of the diagram these are zero by the cyclotomic relation.  
Thus the one and two colour barbell relations allow us to rewrite a product of simple Soergel   diagrams as a linear combination of simple Soergel   diagrams.

 We   now show that we can rewrite, using only	the
 relations of \cref{easydoesit},  any 
diagram %${\sf D} \otimes {\sf bar}(\ctau)$
${\sf D}_1{\sf D}_2$
 as a linear combination of simple Soergel   diagrams.   
Any diagram can be rewritten in terms of the cellular basis
and the cellular basis elements are all of the form 
$c_\SSTS^*{\sf 1}_{\stt_\la} c_\SSTT$.  Thus it suffices to have a list of rules which rewrites 
${\sf 1}_{\stt_\la}\otimes {\sf bar}(\ctau)$ as a linear combination of simple Soergel diagrams.  This is precisely what the   monochrome  and bi-chrome Tetris relation do. The result follows.  
 \end{proof}

 \subsection{Combinatorial invariance for simply laced types. } Equipped with our Tetris style presentations, we are now ready to prove \cref{HSP-comb}.  We begin by restricting our explicit attention to simply laced types only, as this case follows easily from our Tetris-style presentation.

 \color{black!99}
\begin{prop}\label{jkfgjkgffgjkfgkj}Let $(W,P)$ and $(W',P')$ be simply laced Hermitian symmetric pairs.  
Let $\Pi=[\la,\mu]$ and $\Pi'=[\la',\mu']$ be closed subsets of $\mptn$ and $\mathscr{P}_{(W',P')}$ respectively.  
 Given a map $\MAP :\Pi\to \Pi'$,   $\MAP  ( \alpha)= \alpha'$ for $\la \leq \alpha \leq \mu$,
  we have that $\MAP $ is a poset isomorphism if and only if  $\MAP $ sends like-coloured tiles in $\alpha$ to like-coloured tiles in $\alpha'$.  
\end{prop}

\begin{proof}
We consider  the classical types, as the exceptional cases can be verified by exhaustion.  
%%  We restrict our attention to the case that $\Pi$ has a unique maximal element $\mu$ and minimal element 
%%  $\la$ (for ease of exposition) as any general $\Pi$ can be obtained as a union of such posets.  
%We proceed by induction on $|\Pi|$. 
%The set  $\Pi$ has  some set of maximal elements $\mu^{(1)},\mu^{(2)},\dots $ and some set of 
%minimal elements  
%$\la^{(1)},\la^{(2)},\dots $. %  (for ease of exposition) as any general $\Pi$ can be obtained as a union of such posets.  
%%  Any such $\Pi$ consists of a rectangle intersected with the admissible region.  
%We first consider the case that  $\Pi$ has a unique maximal and minimal element
We identify   $\la\subseteq \mu $ and  $\la'\subseteq \mu' $ with their tilings  $\la\subseteq \mu $ and  $\la'\subseteq \mu' $.  
 If  the tiling $\la\subseteq \mu$  can appear as a subregion of $\mathscr{A}_{(W,P)}$ for $W$ of type $A$ (that is,  $[r,c] ,[r+1,c+1]  \in (\mu\setminus\la)$  implies 
 $[r,c+1], [r+1,c] \in (\mu\setminus\la)$) 
  then the colouring on each $\alpha \in \Pi$, $\alpha' \in \Pi'$ is simply given by shifting the $x$-coordinates of the diagonals (up to reflecting through the $y-$axis, see the first pair in \cref{jk343434}). 
 Otherwise,  $\la\subseteq \mu $ and  $\la'\subseteq \mu' $ each  have a single diagonal  in which the colouring alternates (either alternating between $s_0$ and $s_1$ in type $(D,A)$ as in the rightmost examples in \cref{jk343434},   or alternating between $s_1$ and $s_3$ in type $(D,D)$).  
If such a $\Pi$ fits within a  $3\times 3$ rectangle in type $(D,A)$, then one can check that there exists an isomorphic $\Pi'$ in type $(D,D)$ and that the colourings match-up in these small cases.  Otherwise, if $\Pi$ and $\Pi'$ are two distinct isomorphic posets, then they are both of type $(D,A)$ and can be obtained by vertical translation (perhaps swapping $s_0$ and $s_1$ in the process, depending on the parity of the translation).  
  \end{proof}
  \color{black}

\begin{figure}[ht!]
$$
$$

\caption{\color{black!99}
Two pairs of isomorphic tilings.  The pair $\Pi=\{\alpha \mid \varnothing  \leq \alpha \leq (1,2,3)		\}$
and $\Pi'=\{\alpha \mid (1,2,3^4)  \leq \alpha \leq (1,2,3,4,5,6)		\}$
	 on the right can only appear in type $D$ (note the jagged edge on the left) and the recolouring of tiles 
 is 	given by  $\iota({\color{magenta}s_0})= {\color{pink}s_1}	$,  
 	 $\iota( {\color{pink}s_1}	)={\color{magenta}s_0} 	$, 
	$\iota(   {\color{brown}s_2}) =   {\color{brown}s_2}$, 
		$\iota(   {\color{orange}s_3}) =   {\color{orange}s_3}$).
	  The pair on the left  are $\Pi=\{\alpha \mid	 (4,3^2,2)   \leq \alpha \leq (5^3,4) 	\}$
and $\Pi'=\{\alpha \mid (3,1)  \leq \alpha \leq (4^2,3)		\}$
and the recolouring of tiles is given by   
 $\iota( {\color{violet}s_1}	)={\color{cyan}s_7} 	$,
  $\iota( {\color{yellow}s_2}	)={\color{brown}s_6} 	$,
    $\iota( {\color{orange}s_3}	)={\color{magenta}s_5} 	$,
  $\iota( {\color{green!80!black}s_4}	)={\color{green!80!black}s_4} 	$,  
    $\iota(  {\color{magenta}s_5} )= {\color{orange}s_3}	$,	
      $\iota( {\color{brown}s_6} )= {\color{yellow}s_2}	$ (note the flip through the vertical axis).}
      \label{jk343434}
\end{figure}

\begin{defn}\label{recolourme} 
%Let $D$ be a Soergel diagram 
%formed from  horizontal and vertical concatenation 
%of Soergel generators 
Let $\Gamma$ (respectively $\Gamma'$)
 be the set of  colours of the tiles in $\Pi$ (respectively $\Pi'$).  
 \color{black!99} We let
 $\colMAP: \Gamma \to \Gamma'$, 
 be a surjective map.  \color{black}
We lift this to a  recolouring map on Soergel diagrams as follows. 
For $\gamma, \delta \in \Gamma$ with   $\colMAP (\gamma)=\gamma'$ and   $\colMAP (\delta)=\delta'$    we set 
$$
\colMAP ({\sf 1}_\gamma  )=
{\sf 1}_{\gamma'}  
\quad 
\colMAP ({\sf spot}_\gamma^\emptyset )=
{\sf spot}_{\gamma'}^\emptyset 
\quad 
\colMAP ({\sf fork}_{\gamma\gamma}^{\gamma} )=
{\sf fork}_{\gamma'\gamma'}^{\gamma'}
\quad  \color{black!99}
\colMAP ({\sf braid}^{\gamma \delta }_{\delta\gamma } )=
{\sf braid}^{\gamma '\delta'}_{\delta' \gamma'}		 \color{black}
$$and we set $\colMAP({\sf D}^\ast)=(\colMAP({\sf D}))^\ast$.  We then inductively define
$$
\colMAP({\sf D}_1 \otimes {\sf D}_2)
=
\colMAP({\sf D} _1) \otimes\colMAP( {\sf D}_2)
\qquad
\colMAP({\sf D}_1 \circ {\sf D}_2)
=
\colMAP({\sf D}_1 ) \circ\colMAP( {\sf D}_2)
$$and extend this map $\Bbbk$-linearly.  
\end{defn}

\begin{prop}\label{recolourme} 
Suppose we have a poset isomorphism $\MAP:\Pi \to \Pi'$ for 
 $\Pi=[\la,\mu]$ and $\Pi'=[\la',\mu']$  as in \Cref{jkfgjkgffgjkfgkj} inducing a recolouring map $\colMAP$.
Then $\MAP$ extends to a unique graded isomorphism of $\Bbbk$-algebras  $\MAP :   {h}_\Pi    \to   {h}_{\Pi '} $ 
satisfying
$$
  \MAP  ({\sf 1}_{\stt_{\la }}
    \otimes
D )
={\sf 1}_{\stt_{\la' }} 
   \otimes
\colMAP( {\sf D})
$$for  ${\sf D}$ any Soergel diagram %.
, and
$$
  \MAP  ({\sf 1}_{\stt_{\la }}
   \circledast_\gamma 
D )
={\sf 1}_{\stt_{\la' }} 
  \circledast_{\gamma'}
\colMAP( {\sf D})
$$for   $\gamma \in \Rem(\la)$, and ${\sf D}$ any   Soergel diagram for which the $\circledast_\gamma$- and $\circledast_{\gamma'}$-products makes sense.
 \end{prop}

\begin{rmk} We can regard the diagram ${\sf D}$ as being ``coloured by"  \color{black!99} the initial  Coxeter system $W$    \color{black}  and the effect of the map is to ``recolour" this diagram according to the  {\color{black!99}  Coxeter system $W'$}.  The effect of changing $\stt_{\la}$ for $\stt_{\la'}$ is merely to identify the different regions $\Pi$ and $\Pi'$ within our Bruhat graphs $^PW$ and ${^{P'}}W'$.   
 An example is given in \cref{recolour}.
 \end{rmk}
 
 We are now almost ready to prove \cref{recolourme}.  We first 
 observe that the Tetris relations are compatible with restriction to 
 a closed subregion.

 \begin{lem}\label{whenisitzeroinquotient}
\color{black!99} Let $  \Pi = [\la,\mu]$. % with $\color{cyan}[r,c]=\ctau \color{black}\in \Add(\la)$.  
%  We let $\nu \not \in \Pi$ be a maximal element such that $\nu <\la$.  
Given  $\color{magenta}[x,y]
 \color{black} \in \la  %\cap \nu
 $,
  we have that ${\sf gap}(\stt_\la-\color{magenta}[x,y]
 \color{black})=0 \text{ in } \color{black!99}h_\Pi.$
 \end{lem}
 \begin{proof}\color{black!99}
   We suppose $W$ is classical as the general case is similar.  
 By    \eqref{iszero} and    \eqref{isnonzero} we can rewrite ${\sf gap}(\stt_\la-\color{magenta}[x,y]
 )=\pm c^\alpha_{\SSTS\SSTS}$ (or zero) for $\alpha \subset \la $ and the result follows as ${\sf 1}_\alpha=0 \in h_\Pi$ by definition.
  \end{proof}

  \begin{figure}[ht!]
   $$   
   \begin{minipage}{4.4cm}
\end{minipage}
 $$ 

\caption{
The left diagram, $D$, is an element from $h_{\Pi}$ for 
$\Pi=\{ \alpha \mid (1^2)\leq \alpha\leq (3^2)\} \subset \mathscr{A}_{(A_8,A_4\times A_3)}$ and the right diagram is 
the corresponding, $\colMAP(D)$,   
 in  $h_{\Pi'}$ for $\Pi'=\{\alpha' \mid (3,1^2)\leq \alpha\leq (3^3)\} \subset \mathscr{A}_{(A_8,A_4\times A_3)}$.  Compare the colouring with that of 
 ${(A_8,A_4\times A_3)}$  in \cref{typeAtiling}.  
  }
\label{recolour}
\end{figure}

 \begin{proof}[Proof of \cref{recolourme}] Let $\alpha \in \Pi = [\la,\mu]$ for $\la,\mu \in \mptn$. 
Each cellular basis element $c^\alpha_{\SSTS\SSTT}$ in $h_{\Pi}$ can be written in the form ${\sf 1}_{\stt_\la} \otimes {\sf D}$ 
 or ${\sf 1}_{\stt_{\la'}}   \circledast_\gamma {\sf D}$
 for some simple Soergel diagram ${\sf D}$, so $\MAP$ is well defined. 
 The monochrome and idempotent relations are trivially preserved by  the map $\MAP$.  
 Two   commuting reflections,  $\ctau, \csigma\in S_W$ correspond to tiles $\color{cyan}[x,y]$, $\color{magenta}[r,c]$ from $\Pi$  
if and only if  $(x-y)-(r-c)\neq \pm 1$; this distance is preserved by the map $\MAP:\Pi \to \Pi'$ (by \cref{jkfgjkgffgjkfgkj}) and so the commuting relations are preserved.  

The Tetris   relations for $\mathscr{H}_\Pi$ are written entirely in terms of the addable and removable nodes of tilings and 
the sets  ${\rm Hook}_\ctau(\alpha) $ for $\la \leq \alpha \leq \mu$ and this is compatible with restriction to $\Pi$  (using \cref{whenisitzeroinquotient}).
The sets 
${\rm Hook}_\ctau(\alpha) $  depend only on information which is preserved under $\MAP:\Pi \to \Pi'$  (using \cref{underphi}   to flip left versus right in the definition of ${\sf Hook}_\ctau(\mu)$, if necessary). 
Therefore the  Tetris  relations  go through $\MAP $. 
Finally, we note that the cyclotomic relations 
  follow from the   Tetris relations
and \cref{whenisitzeroinquotient}.  Thus the map $\MAP $ is an algebra homomorphism.

One can similarly define $\MAP ^{-1}$ as the recolouring map in the opposite direction.  We have that  $ \MAP \circ \MAP ^{-1}$ and 
$\MAP ^{-1}\circ \MAP $ are both identity maps (as they amount to recolouring and recolouring again) and so the map is indeed an algebra isomorphism.  
  \end{proof}

  \subsection{Fixed point subgroups and non-simply laced types}\label{nonsimply}

  We now consider the group automorphisms, $\sharp$,  
  for type $A_{2n-1}$ and $D_{n+1}$ given by flipping the Coxeter 
  diagrams through the horizontal and vertical axes, respectively.  
Explicitly, the map $\sharp$ is determined by $\sharp(s_i)= s_{2n-1-i}$ for the group of type $A_{2n-1}$.  
The map  $\sharp$ is determined by $\sharp(s_0)= s_{1}$, 
$\sharp(s_1)= s_{0}$ and $\sharp(s_i)= s_{i}$
 for the group of type $D_{n+1}$.  
 The fixed point groups of these automorphism are the groups 
 $\langle s_i s_{2n-1-i}, s_{n}\mid 1\leq i <n\rangle$ of type 
 $B_n$ and 
 $\langle s_0s_1, s_i \mid 2\leq i \leq n\rangle$ of type $C_n$.  
By restricting our attention from the group to its fixed point subgroup, we obtain a surjective   map $\colMAP$ on the tile-colourings.  
These are depicted in \cref{sjkhdsdhjkfasjkhfjkfhlsad,sjkhdsdhjkfasjkhfjkfhlsad2}.

\begin{rmk}\label{assumpttt}For the remainder of this section, we will fix 
%$\csigma,\ctau \in S_W$ are such that $m(\csigma,\tau)= 4$;
  $ \csigma=  s_2 $ and $\ctau=s_1$  for   $(W,P)=(C_n,A_{n-1})$ 
and
  $ \csigma=  s_1 $ and $\ctau=s_2$  for   $(W,P)=(B_n,B_{n-1})$.  
  
 \end{rmk}

  We extend our colouring convention from \cref{assumpttt} by 
setting 
$\csigma= \color{magenta}s_n $ in type $A_{2n-1}$ and 
$  \csigma= \color{magenta}s_2 $
in type $D_{n+1}$.  
We similarly set 
$\crho={\color{green!80!black}s_{n-1}}, \cpi={\color{violet}s_{n+1}}$ in type $A_{2n-1}$ and 
$\crho={\color{green!80!black}s_0}, \cpi={\color{violet}s_1}$ in type $D_{2n+1}$ 
so that green and purple map to blue in both cases.  
We can easily extend this colouring map to a map, $\colMAP$, on the level of paths and moreover we have the following:

\begin{lem}\label{onpathsssss}
The colouring maps on paths map bijectively onto the subsets of 
 parity preserving paths from \cref{pm3}.  In other words,
  \begin{align*}
   \colMAP &  : \Path_{  ( D_{n+1}, A_{n} ) } (\la,\stt_\mu) \to 
 \Path^{\pm} _{  ( C_n,A_{n-1})}(\la,\stt_\mu),   
%\\
\quad     \colMAP  : \Path_{  ( A_{2n-1}, A_{2n-2} ) } (\la,\stt_\mu) \to 
 \Path^{\pm} _{  ( B_n,B_{n-1})}(\la,\stt_\mu) 
 \end{align*}
are both grading-preserving  bijections.
\end{lem}

 We now prove that  
$h_{(C_n,A_{n-1})}$ and 
$h_{(B_n,B_{n-1})}$ are graded Morita equivalent to 
 $h_{(D_{n+1},A_{n})}$ and $h_{(A_{2n-1},A_{2n-2})}$ respectively.

 \begin{figure}[ht!]
 $$
 \begin{minipage}{2.4cm}
\end{minipage}
  $$
  \caption{An example  of  the colouring  map $\colMAP  $ from   type  $D_{n+1}$ to type $C_n$.  }
  \label{sjkhdsdhjkfasjkhfjkfhlsad}
  \end{figure}

\vspace{-0.1cm}

 \begin{figure}[ht!]
 $$
 \begin{minipage}{2.4cm}%
\end{minipage}
  $$
  \caption{An example  of  the colouring  map $\colMAP  $ from   type  $A_{2n-1}$ to type $B_{n}$.  }
  \label{sjkhdsdhjkfasjkhfjkfhlsad2}
  \end{figure}

 \begin{lem}
Let  $m(\betar,\gam)=3$ or  $m(\betar,\gam)=4$.
 If  $m(\betar,\gam)= 4$, then suppose that 
 $(\betar,\gam)= (  \ctau,\csigma)$  as in \cref{assumpttt}.  
We have  that  
\begin{align}\label{repeat11}
{\sf fork}^\betar_{\betar \betar }
({\sf 1}_\betar \otimes {\sf bar}(\gam) \otimes {\sf 1}_\betar)
 {\sf fork}_\betar^{\betar \betar }
 &=- {\sf 1}_\betar 
\\
\label{repeat12}
{\sf cap}^{\emptyset}_{\betar \betar }
({\sf 1}_\betar \otimes {\sf bar}(\gam) \otimes {\sf 1}_\betar)
 {\sf cup}_{\emptyset}^{\betar \betar }
 &=-{\sf bar}(\betar) 
 \\
\label{repeat13}
{\sf fork}^{\betar}_{\betar \betar }
({\sf 1}_\betar \otimes {\sf bar}(\gam)\otimes {\sf bar}(\gam) \otimes {\sf 1}_\betar)
 {\sf fork}_{\betar}^{\betar \betar }
 &=
 - {\sf 1}_\betar \otimes (2{\sf bar}(\gam) +  {\sf bar}(\betar) )
\end{align}
\end{lem}
\begin{proof}
\Cref{repeat11} follows by  applying   the $\gam\betar$-barbell relation, followed by the $\betar$-circle annihilation relation and $\betar$-fork-spot contraction relation.   
\Cref{repeat12} follows from \cref{repeat11} by apply the spot generator on top and bottom.  
\Cref{repeat13}  follows by  applying  the $\gam\betar$-barbell relation to the lefthand-side, followed by 
\cref{repeat11}.
 \end{proof}

We will find the following shorthand useful,  
$${\sf trid}^\blue_{\blue\csigma \blue}= {\sf fork}^{\blue}_{\blue\blue}
({\sf 1}_\blue\otimes {\sf spot}_\csigma ^\emptyset \otimes {\sf 1}_\ctau)
\qquad
{\sf trid}^\emptyset _{\blue\csigma \blue}= {\sf spot}_\ctau^\emptyset  %{\sf fork}^{\blue}_{\blue\blue}
%({\sf 1}_\blue\otimes {\sf spot}_\csigma ^\emptyset \otimes {\sf 1}_\ctau)
{\sf trid}^\blue_{\blue\csigma \blue}
$$
the former of which   can be pictured as a ``trident".  We set $${\sf trid}^{\blue\csigma \blue}_\blue=({\sf trid}^\blue_{\blue\csigma \blue})^\ast \qquad 
 {\sf trid}^{\blue\csigma \blue}_{\blue\csigma \blue} = 
\trid^{\blue\csigma \blue}_\blue {\sf trid}_{\blue\csigma \blue}^\blue.$$ 
By \cref{repeat11}, we have that
 $-{\sf trid}^{\blue\csigma \blue}_{\blue\csigma \blue} $  
is an idempotent and that 
\begin{align}\label{bpb}
({\sf 1}_{\blue\csigma \blue}+{\sf trid}^{\blue\csigma \blue}_{\blue\csigma \blue} )
{\sf trid}^{\blue\csigma \blue}_{\blue\csigma \blue} =0.
\end{align}

\begin{defn}\label{defn of idempt}
 Let  $(W,P)=(C_n,A_{n-1})$ or $(B_n,B_{n-1})$ and suppose $\csigma,\ctau\in S_W$ satisfy the assumptions of \cref{assumpttt}.  
For  $\mu\in \mptn $ and  $1< k \leq \ell_{\ctau}(\mu)$, we set 
$$\rho_k=\begin{cases}
(1,2,3,\dots,k)	\in \mptn 	&\text{if }W=C_n\\ 
(n,1)	\in \mptn 	&\text{if }W=B_n
\end{cases}
%  \kappa= 
%\rho_k - %   {\color{cyan} [k,k]} - {\color{magenta}[k,k-1]} - {\color{cyan}[k-1,k-1]} ,
%\ctau-\csigma-\ctau
$$and we set $  \kappa= 
\rho_k - 
\ctau-\csigma-\ctau$ 
% let  $
%    \SSTT_k = \stt_{\kappa}  \circ   
%    \stt_{\mu - \kappa}$.  We   define 
and define 
\begin{align}\label{redundaaaa}
 {\sf e}_{ \mu}^k= 
{\sf braid}_{ \stt_{\kappa}  \circ   
    \stt_{\mu - \kappa} }^{\stt_\mu} \Big(
{\sf 1}_{\stt_\kappa}\otimes 
{\sf trid}^{\blue\csigma \blue}_  {\blue\csigma \blue} 
 \otimes {\sf 1}_{\stt_{\mu-\rho_k}}\Big) 
{\sf braid}^{ \stt_{\kappa}  \circ   
    \stt_{\mu - \kappa}}_{\stt_\mu} .
\end{align}
We define the idempotent 
 $$\textstyle 
{\sf e}_{ \mu}=	 \prod_{ 1 < k \leq \ell_{\ctau}(\mu) } ({\sf 1}_{\stt_\mu}+ {\sf e}_{ \mu}^k )
$$ 
and  set $e_\mu={\sf 1}_{\stt_\mu}$ if $\ell_{\ctau}(\mu)\leq1$.
We    define  ${\sf e}=\sum_{\mu \in \mptn} {\sf e}_{ \mu}$.  
\end{defn}

  \begin{figure}[ht!]
   $$   
%   \begin{minipage}{2.5cm}
%  \end{minipage}$$ 
	 \caption{
The element ${\sf e}_\mu$ for   $\mu =(1,2,3)$
in type  $C$.    The colouring is the same as that of \cref{sjkhdsdhjkfasjkhfjkfhlsad}.  The   summands   are 
${\sf 1}_\mu$, 
${\sf e}_\mu^2$,  ${\sf e}_\mu^3$, and ${\sf e}_\mu^2{\sf e}_\mu^3$  respectively.  
   }
\label{eg of idemp}
\end{figure}
  
%  \vspace{-0.1cm}

  \begin{figure}[ht!]
   $$   
   \begin{minipage}{3.5cm}
 %
  \end{minipage}  \qquad\quad
 $$

	 	 \caption{
%On the left w
We picture the element  ${\sf e}_\mu$ for $\mu=(4,1^3)$ for 
 type  $(B_4,B_3)$ as in \cref{defn of idempt}.   
% On the right we picture $\trid^{[3,3]}_{[3,3]}$ for 
% type  $(B_4,B_3)$ as in \cref{tridex}.   
 The colouring is the same as that of \cref{sjkhdsdhjkfasjkhfjkfhlsad2}.   
  }

  \label{eg of idemp2}
\end{figure}

%In fact we are now ready to state the main theorem of this section.  
We can extend the maps  $\colMAP$  from \cref{onpathsssss} to   injective $\Bbbk$-linear maps 
  $ \colMAP: h_{(D_{n+1},A_{n})} \hookrightarrow  h_{(C_n,A_{n-1})}$
  and 
  $ \colMAP: h_{(A_{2n-1},A_{2n-2})} \hookrightarrow  h_{(B_n,B_{n-1})}$ by setting 
  $\colMAP(c_{\SSTS\SSTT})=c_{\colMAP(\SSTS)\colMAP(\SSTT)}$.  
We note that $\colMAP$ is not a $\Bbbk$-algebra homomorphism, but we will prove the following:

\begin{thm}  \label{C--A}   \label{B--A} 
The maps $\Theta:  h_{(D_{n+1},A_{n})} \to {\sf e} h_{(C_n,A_{n-1})}{\sf e} $
and  		$\Theta:  	h_{(A_{2n-1},A_{2n-2})} \to  {\sf e}h_{(B_n,B_{n-1})}{\sf e}$	 
defined by 
$\Theta(a)=e \circ \colMAP(a) \circ e$ 
%$$ {\sf e} h_{(C_n,A_{n-1})}{\sf e} \cong 
% h_{(D_{n+1},A_{n})}
% \qquad
%  {\sf e} h_{(B_n,B_{n-1})}{\sf e} \cong 
% h_{(A_{2n-1},A_{2n-2})} $$
are graded  $\Bbbk$-algebra isomorphisms. Moreover, as ${\sf e}$ is a full idempotent these maps   give  rise to   graded Morita equivalences between
$h_{(D_{n+1},A_{n})} $ and $  h_{(C_n,A_{n-1})}$ and between
$h_{(A_{2n-1},A_{2n-2})} $ and  $ h_{(B_n,B_{n-1})}$.
\end{thm}
 
We note that the first isomorphism  categorifies an observation of Boe in \cite{MR957071}.   
This section is dedicated to the proof.  We begin with the simpler result for orthogonal groups.  

   \subsubsection{The orthogonal case, \color{black!99} type $(B_n,B_{n-1})$ }   
   We first consider the case of the orthogonal group.  
We can simplify the proof by 
  focussing on the cellular basis.  We prove that if 
  $c_{\SSTS\SSTT}c_{\SSTU\SSTV}
  = \sum a_{\SSTX\SSTY}c_{\SSTX\SSTY} 
$ for coefficients $a_{\SSTX\SSTY} \in \Bbbk$, then we have that 
\begin{align}\label{shjfkldgdklfghl}\Theta(c_{\SSTS\SSTT}) \Theta(c_{\SSTU\SSTV})
   =\sum a_{\SSTX\SSTY}\Theta( c_{\SSTX\SSTY} )    \end{align}
for $\SSTS,\SSTT,\SSTU,\SSTV,\SSTX,\SSTY \in \Path_{(A_{2n-1},A_{2n-2})} $    and hence deduce that \cref{B--A}  holds for type $(B_n,B_{n-1})$.  
For     $\mu=(c-1)$ with  $ {\color{orange}[1,c]} \in \Add(\mu)$ with $\gam=\color{orange}s_{[1,c]}$ and $1\leq c \leq n$ the elements $c^\mu_{\SSTS\SSTT}$ are of the form
      $$  {\sf 1}_{\stt_{\mu }},   \quad  {\sf 1}_{\stt_\mu} \otimes  {\sf spot}_\gam^\emptyset,\quad
          {\sf 1}_{\stt_\mu} \otimes  {\sf spot}^\gam_\emptyset, \quad \text{ and }\quad 
   {\sf 1}_{\stt_\mu} \otimes  {\sf gap} (\gam).$$ 
%      up to duality and so 
%  we can ignore the fork and braid relations (and the idempotent relations are trivial).  
Thus rewriting products in \cref{shjfkldgdklfghl} requires only the idempotent, bi-chrome Tetris,  
commutativity and cyclotomic relations.   
We  consider the bi-chrome Tetris relation as the others are trivial.  
 By \cref{remove a barbell,irrelevant}, we have that  
$$
c _\SSTT c _\SSTT^\ast=  {\sf 1}_{\stt_\mu}  \otimes {\sf bar} ( \gam)
=-{\sf gap} ({\stt_{\mu}-{[1,c-1]}}  )
$$
for $\SSTT \in \Path(\mu,\stt_{\mu+\gam})$. %   and $\gam \in \Rem(\mu)$, $\gam \in \Add(\mu)$.
      For $ c\leq n$, we have that 
      ${\sf e}_{\colMAP(\mu\gam)}={\sf 1}_{\colMAP(\stt_{\mu\gam})}={\sf 1}_{(c)}$  
         and  
           $$
      \Theta(  {\sf 1}_{\stt_{(c-1)} }\otimes{ \sf bar}(\gam) )
       = - {\sf 1}_{ (c-1) }  \otimes {\sf gap} (\colMAP(\gam))
       =\Theta (  - {\sf gap} ({\stt_{\mu}-{[1,c-1]}}  ) ) 
           $$
as required.  For $c=n+1$ we have that $\colMAP(\gam)=\ctau$ and 
\begin{align*}
      \Theta(      {\sf 1}_{\stt_{(n)} }\otimes{ \sf bar}(\gam) )
    &  =
       \colMAP(c_\SSTT)\circ ({\sf 1}_{ (n,1) }+{\sf 1}_{ (n-2) } \otimes  \trid^{\blue\pink\blue}_{\blue\pink\blue}) \circ  
      \colMAP (c _\SSTT^\ast)
      \\
      &= {\sf 1}_{(n)} \otimes {\sf bar}( \ctau) 
      + {\sf 1}_{(n-1)}  \otimes {\sf gap}( \csigma) 
        \\
      &=  -2 \cdot {\sf 1}_{(n-1) }  \otimes  {\sf gap} ( \csigma)+{\sf 1}_{ (n-1) }  \otimes {\sf gap} ( \csigma)
        \\
      &=      
        \Theta ( -{\sf gap} ({\stt_{\mu}-{[1,n ]}}  ) ) 
\end{align*}where 
as required. Here the first equality follows from the definition of $\eee$;
  the second from the $\ctau$-fork-spot contraction relation; 
the third equality   from \cref{remove a barbell,irrelevant}; and the fourth is trivial. For $c=n+2$, we have that 
 \begin{align*}
      \Theta(      {\sf 1}_{\stt_{(n+1)} }\otimes{ \sf bar}(\gam) )
   &   =
  {\sf 1}_{\stt_{(n,1)} }\otimes{ \sf bar}(\gam)  
  +  {\sf e}_{\stt_{(n-2)} }
  \otimes \trid^{\ctau\csigma\ctau}_{\ctau\csigma\ctau} \otimes{ \sf bar}(\gam)   
 \\
    &   =
-  {\sf gap}({\stt_{(n,1)} } -[n,2])- 2  {\sf gap}({\stt_{(n,1)} } -[n,1])
  - {\sf e}_{\stt_{(n-2)} }
  \otimes  \trid^{\ctau\csigma\ctau}_\emptyset 
 \trid_{\ctau\csigma\ctau}^\emptyset 
 \\
    &   =
-  {\sf gap}({\stt_{(n,1)} } -[n,2])-  
  {\sf e}_{\stt_{(n-2)} }
  \otimes  \trid^{\ctau\csigma\ctau}_\emptyset 
 \trid_{\ctau\csigma\ctau}^\emptyset 
 \end{align*}
 where the first equality is trivial; 
the second follows by  applying \cref{remove a barbell,irrelevant} to the first term and applying the $(\blue, \colMAP(\gam))$-null-braid and $\colMAP(\gam)$-fork-spot-contraction to the second term;
the third  follows by applying (\ref{shorter}) to the $\ctau$-strands in the second term, followed by \cref{remove a barbell}  
and the cyclotomic and commutativity relations.
On the other hand,
\begin{align*}
  &\;   \Theta ( {\sf 1}_{\stt_{(n) }} \otimes  {\sf gap}(\gam) )    
 \\   = &\; - ({\sf 1}_{ (n,1) }+{\sf 1}_{ (n-2) } \otimes  \trid^{\blue\pink\blue}_{\blue\pink\blue})
({\sf 1}_{\stt_{(n)  }}
 \otimes  {\sf gap}(\ctau)      ) ({\sf 1}_{ (n,1) }+{\sf 1}_{ (n-2) } \otimes  \trid^{\blue\pink\blue}_{\blue\pink\blue})
\\
 =&\;
-{\sf 1}_{ (n-2) } \otimes  (	{\sf 1}_{\ctau\csigma}\otimes {\sf gap}(\blue)
+ 
\trid^{\blue\csigma\blue}_\blue{\sf spot}_{\blue\csigma\blue}^{\blue\emptyset\emptyset}
+
  {\sf spot}^{\blue\csigma\blue}_{\blue\emptyset\emptyset}		 
 \trid_{\blue\csigma\blue}^\blue
 +\trid^{\blue\csigma\blue} _{\blue\csigma\blue} \otimes {\sf bar}(\csigma)
 )
 \\
 =&\;
-{\sf 1}_{ (n-2) } \otimes  (	{\sf 1}_{\ctau\csigma}\otimes {\sf gap}(\blue)
- 
\trid^{\blue\csigma\blue} _\emptyset \trid^\emptyset_{\blue\csigma\blue}
-{\sf bar}(\ctau)\otimes {\sf dork}^{\blue\blue}_{\blue\blue}+
{\sf 1}_\ctau\otimes {\sf gap}(\csigma)\otimes {\sf1}_\ctau
%-\trid^{\blue\csigma\blue} _{\blue\csigma\blue} 
 )
\\  =&\;
-  {\sf gap}({\stt_{(n,1)} } -[n,2])-  
  {\sf e}_{\stt_{(n-2)} }
  \otimes  \trid^{\ctau\csigma\ctau}_\emptyset 
 \trid_{\ctau\csigma\ctau}^\emptyset 
\end{align*}
as required.  Here the penultimate equality follows by 
applying \ref{shorter} to the middle two terms and applying  \cref{remove a barbell,irrelevant} to the final term; the final equality follows from  \cref{remove a barbell,irrelevant} and the commutativity and cyclotomic relations.  
Finally, we suppose that $c>n+2$.  We have that 
\begin{align*}
      \Theta(      {\sf 1}_{\stt_{(c-1)} }\otimes{ \sf bar}(\gam) )
%    &  =
%     \eee_{(n,1^{c-1-n})} 
%      \otimes{ \sf bar}(\gam) )
%\\
   &  =
(   {\sf 1}_{\stt_{(n,1^{c-1-n})}}				  +   \eee_{(n,1^{c-1-n})} )
      \otimes{ \sf bar}(\gam) 
\\
   &        =- 2 \gap (\stt_{(n,1^{c-1-n})} - [1,n]) -   \gap (\stt_{(n,1^{c-1-n})} - [c-1-n,n]) 
\\
   &        = -   \gap (\stt_{(n,1^{c-1-n})} - [c-1-n,n]) 
   \\
   &        = \Theta(- \gap (\stt_{(c+1)} - [1,c+1]) )
\end{align*}
as required.  For the second equality, we apply 
  \cref{remove a barbell,irrelevant} to the first term
   and 
observe that the  second term is zero by applying the bull-braid relations followed by  \cref{remove a barbell,irrelevant}  and the commutativity and cyclotomic relations. The other equalities are trivial.  
Thus the bi-chrome Tetris relation holds in all cases and we are done.

   \subsubsection{The symplectic case, \color{black!99} type $(C_n,A_{n-1})$}
We now consider the, more difficult, case of the symplectic group.
\begin{lem}\label{jkhsfjkhdsfgjkhdsfgdgsfhjkjkhdgfkjdhgsfjkhdsgf}
For 
 ${\sf e}_{\ctau\csigma\ctau\csigma\ctau} = 
 ( {\sf1}_{\ctau\csigma\ctau\csigma\ctau} +
\trid_ {\ctau\csigma\ctau}^{\ctau\csigma\ctau}\otimes {\sf 1}_{\csigma\ctau} )
 ( {\sf1}_{\ctau\csigma\ctau\csigma\ctau} +
  {\sf 1}_{ \ctau\csigma}\otimes \trid_ {\ctau\csigma\ctau}^{\ctau\csigma\ctau} )$, 
we have that 
\begin{align*}%\label{LHS}
( \on_\blue \otimes \trid_{\csigma\ctau\csigma}^\csigma  \otimes \on_\ctau){\sf e}_{\ctau\csigma\ctau\csigma\ctau} 
&=
- \trid^{\blue\csigma\blue}_{\blue\csigma\blue}
(\on_\blue \otimes \trid_{\csigma\ctau\csigma}^\csigma \otimes \ctau)%{\sf e}_{\ctau\csigma\ctau\csigma\ctau}  
- \trid^{\blue\csigma\blue}_{\blue\csigma\blue} \compose \trid^{\blue }_{\blue\csigma\blue}
\\
&=
- \trid^{\blue\csigma\blue}_{\blue\csigma\blue}
(\on_\blue \otimes \trid_{\csigma\ctau\csigma}^\csigma \otimes \ctau){\sf e}_{\ctau\csigma\ctau\csigma\ctau} 
\end{align*}
\end{lem}
 \begin{proof}
 We prove the first equality, the second follows as $\eee_{\ctau\csigma\ctau\csigma\ctau} $ is an idempotent which kills the second term. 
Consider the $m(\csigma,\ctau)=4$ braid relation and tensor it on the left by ${\sf 1}_\ctau$.  Vertically concatenating  
$\trid_{\blue\csigma\blue}^\ctau \otimes {\sf 1}_{\csigma\blue}$ on top of this combination of diagrams, we obtain 
%Rewriting 
% ${\sf 1}_{\blue\csigma }\otimes {\sf trid}_{\blue\csigma \blue}^\blue$ using the Tetris relation, we obtain 
$$
{\sf 1}_{\blue\csigma }\otimes {\sf trid}_{\blue\csigma \blue}^\blue
+
  {\sf trid}_{\blue\csigma \blue}^\blue
  \otimes {\sf 1}_{ \csigma \blue} 
  +
  {\sf trid}^{\blue\csigma \blue}_{\blue\csigma \blue}
(
 {\sf 1}_\blue\otimes   {\sf trid}_{ \csigma \blue\csigma }^\csigma  \otimes {\sf 1}_\blue)
    +2   {\sf trid}^{\blue\csigma \blue} _{\blue\csigma \blue} 
      \compose   {\sf trid}_{\blue\csigma \blue}^\blue
      +
      {\sf 1}_\blue \otimes {\sf trid}_{ \csigma \blue\csigma }^\csigma  \otimes {\sf 1}_\csigma =0 .
$$Moving the third term and one copy (of the two available) of the fourth term to the right, the result follows.
%and we refer to this as the trident-Jones--Wenzl relation.  
%Now, we have that 
%\begin{align}
%\begin{split}
%\label{split1}
%( \on_\blue \otimes \trid_{\csigma\ctau\csigma}^\ctau \otimes \on_\ctau)\eee 
%=\; &
% \on_\blue \otimes \trid_{\csigma\ctau\csigma}^\ctau \otimes \on_ \ctau 
%+
% \trid_{ \ctau\csigma\blue}^\ctau \otimes  \on_{\csigma\blue}  
% \\
% &+
%   \on_{\blue\csigma }  \otimes   \trid_{ \ctau\csigma\blue}^\ctau
%   +\trid_{\blue\csigma\blue}^{\blue\csigma\blue} ( \on_{\blue\csigma} \otimes \trid_{\blue\csigma\blue}^\blue)
%\end{split}
%\end{align}
%We note that 
%$$\trid_{\blue\csigma\blue}^{\blue\csigma\blue}
%(   \on_{\blue\csigma }  \otimes   \trid_{ \ctau\csigma\blue}^\ctau
%   +\trid_{\blue\csigma\blue}^{\blue\csigma\blue} ( \on_{\blue\csigma} \otimes \trid_{\blue\csigma\blue}^\blue)
%)=0
%$$ by  
%    \cref{repeat11}.
%Therefore,     applying the element $\trid_{\blue\csigma\blue}^{\blue\csigma\blue}$ on top of each term in  \cref{split1}, we obtain 
% \begin{align} \label{split2}
%\trid_{\blue\csigma\blue}^{\blue\csigma\blue}( \on_\blue \otimes \trid_{\csigma\ctau\csigma}^\ctau \otimes \on_\ctau)\eee 
%=\; &
%\trid_{\blue\csigma\blue}^{\blue\csigma\blue}( \on_\blue \otimes \trid_{\csigma\ctau\csigma}^\ctau \otimes \on_ \ctau )
%+
%\trid_{\blue\csigma\blue}^{\blue\csigma\blue}( \trid_{ \ctau\csigma\blue}^\ctau \otimes  \on_{\csigma\blue}  )
%  \end{align}
%The result follows by summing over \cref{split1,split2} and applying the trident-Jones--Wenzl relation above. 
 \end{proof}

%In what follows, to simplify notations we write   $d_{\SSTS\SSTT}$ and $c_{\SSTS\SSTT}:= \colMAP (d_{\SSTS\SSTT})$ for the corresponding cellular basis elements of 
%  $ h_{  ( D_{n+1}, A_{n})}$ and $  h_{ ( C_n,A_{n-1})} $
%  for $\SSTS,\SSTT \in \Path_{ ( D_{n+1},A_{n })}  $. 
We will split the proof of (the symplectic case of) \cref{C--A}  into two propositions.  the first one, \cref{vectorproof}, shows that $\Theta$ is an isomorphism of graded vector spaces.  The second one, \cref{vectorproof2}, shows that $\Theta$ is an algebra homomorphism.  
 
  \begin{prop}\label{vectorproof}
 We have that the map $\Theta: h_{{ ( D_{n+1},A_{n }  )}} \to 
 {\sf e}h_{ ( C_{n},A_{n-1 })}{\sf e}$ given by 
  $ 
\Theta ( c_{\SSTS\SSTT} )= 
 {\sf e} \circ \colMAP( c_{\SSTS\SSTT})\circ {\sf e} $ 
is an isomorphism of graded $\Bbbk$-spaces.
  \end{prop}
  \begin{proof}
  We will show that the set 
  \begin{align}
  \{ {\sf e} \circ c_{\SSTS\SSTT} \circ {\sf e}  \mid \SSTS , \SSTT \in  
   \Path^{\pm} _{  ( C_n,A_{n-1} )}(\la,\stt_\mu)		\}
    \end{align}
    form a basis of $ {\sf e}h_{(  C_{n},A_{n-1 }) }{\sf e}$ and thus deduce the result.  
    We do this by considering 
    $ \Delta (\la)\eee $ for all $\la \in \mathscr{P}_{(  C_{n},A_{n-1 }) }$.
  We will prove the following claim:
  $$ ( c_\SSTS + h^{<\la}_{(  C_{n},A_{n-1 }) } )\eee 
  =
  \begin{cases}
  c_\SSTS + \sum_{\SSTT \not \in 
   \Path^{\pm} _{  (  C_{n},A_{n-1 }) }(\la,\stt_\mu)
  } a _\SSTT c_\SSTT+ h^{<\la}_{(  C_{n},A_{n-1 }) } 
  &\text{if }\SSTS   \in 	 \Path^{\pm} _{  (  C_{n},A_{n-1 }) }(\la,\stt_\mu)				\\
  0 &\text{otherwise} 
  \end{cases}$$
  from which we will immediately deduce the result.  
We first note that we can choose our $\stt_\mu$ for each $\mu\in\mathscr{P}_{(C_n,A_{n-1})}$  in such a way that $\csigma$ always occurs immediately prior to a $\ctau$.  
  We  prove this   for $\stt_{\mu + \gam}$ assuming it holds for $\stt_\mu$  (with the $\ell(\stt_\mu)=0$ case being trivial). 
  Let $\SSTS \in \Path(\la,\stt_\mu)$. 
For $\gam \neq \ctau$, we have that 
$$ ( A^\pm _\gam (c_\SSTS) )\eee = A^\pm _\gam (  c_\SSTS \eee )
\qquad
 (R^\pm _\gam (c_\SSTS))\eee = R^\pm _\gam (  c_\SSTS\eee )$$
for $\gam\in \Add(\la)$   or $\gam\in \Add(\la)$, respectively.  
(Whence $\ell_\ctau(\mu+\gam)=\ell_\ctau(\mu)$ implies 
${\sf e}_{\mu+\gam}= {\sf e}_\mu \otimes {\sf 1}_\gam$.)

Thus we may now assume that $\gam= \ctau \in \Add(\mu)$.  
In which case  %we set $\alphar = \color{magenta}s_2$ and 
  $\csigma= {\color{magenta}s_2} \in \Rem(\mu)$ by our choice of $\stt_\mu$. We let $\mu' = \mu - \csigma$.  
We suppose $\ell_\ctau(\mu)$ is odd (the even case is identical) so that 
 $\cpi \in \Add (\colMAP  ^{-1}(\mu))$
 and 
  $\crho \in \Rem (\colMAP  ^{-1}(\mu'))$.  
Given   $\la'\subseteq  \mu'$, we let $c_{\SSTS'} \in \Path_{(  C_{n},A_{n-1 }) } (\la',\stt_{\mu'})$.
We construct $c_\SSTS$ for $\SSTS \in   \Path_{(  C_{n},A_{n-1 }) } (\la,\stt_{\mu\ctau})$ by applying the inductive process twice: once for $\csigma$ and once for $\ctau$ as follows,
$$c_\SSTS = X_\ctau ^\pm X_\csigma ^\pm (c_{\SSTS'})$$
for $X \in \{A,R\}$.  Note that 
\begin{align}\label{askjdfhgjkdfgjdsgfjdsgjhdsgjhdfgbdjgfjkd}
 c_\SSTS \eee  = 
(X_\ctau ^\pm X_\csigma^\pm 
( 
 c_{\SSTS'}\eee )
 ) ({\sf 1}_{\stt_{\mu\ctau}}+ \eee _{\mu + \ctau}^{\ell_\ctau(\mu)+1}) 
 .\end{align}
We assume, by induction, that the claim holds for $c_{\SSTS'}$. %, so it is enough to consider the  $c_{\SSTS}$ constructed from  $\SSTS' \in  \Path_{(  C_{n},A_{n-1 }) } (\la',\stt_{\mu'})$.  
% Now, by induction  
So we have 
\begin{align}\label{ajsfhgkdfhgkhgjds2222}
\textstyle c_{\SSTS'}\eee = c_{\SSTS'} 
  + \sum_{\SSTT' \not \in 
   \Path^{\pm} _{  (  C_{n},A_{n-1 }) }(\la',\stt_{\mu'})
  } a _{\SSTT'} c_{\SSTT'}+ h^{<\la'}_{(  C_{n},A_{n-1 }) } 
.
\end{align}
 Since
$\SSTT' \not \in \Path^{\pm} _{  (  C_{n},A_{n-1 }) }(\la',\stt_{\mu'})$, this implies 
by definition
 $X_\ctau ^\pm X_\csigma ^\pm (c_{\SSTT'}) \not \in \Path^{\pm} _{  (  C_{n},A_{n-1 }) }(\la,\stt_\mu)$.  
We will now consider 
$$c_\SSTS = X_\ctau ^\pm X_\csigma ^\pm (c_{\SSTS'})$$
for $\SSTS' \in  \Path^{\pm} _{  (  C_{n},A_{n-1 }) }(\la',\stt_{\mu'})$.  
Before considering the above case-wise, we remark  that  
 either 
$\crho \in \Rem(\colMAP  ^{-1}(\la'))$ or $ 
 \Add(\colMAP  ^{-1}(\la'))$ (because it appears at the edge of the region).

\bigskip
\noindent{\bf Case 1. } Suppose $\csigma \in \Add( \la') $.  
This implies that $\crho \in \Rem(\colMAP  ^{-1}(\la'))$.  
The first two subcases which we consider simultaneously are 
$$
A^+_\ctau 
A^+_\csigma (c_{\SSTS'})= c_{\SSTS'} \otimes {\sf 1}_{\csigma}\otimes {\sf 1}_{\ctau}
\qquad
A^-_\ctau 
A^+_\csigma (c_{\SSTS'})= c_{\SSTS'} \otimes
{\sf 1}_{\csigma}  \otimes  {\sf spot}_{\ctau}^\emptyset . $$
Here we have that $c_\SSTS = A^{\pm}_\ctau A^{\pm}_\csigma  (c_{\SSTS'})$ satisfies $\SSTS \in \Path^{\pm}_{(C_n,A_{n-1})} $.  We have that  
\begin{align*}
  (A^+_\ctau 
A^+_\csigma (c_{\SSTS'}) + h^{<\la}_{(  C,A) })
(1+ \eee_{\mu + \ctau}^{\ell_\ctau(\mu)+1})
&=
c_\SSTS + c_{\SSTS'} \compose 
{\sf trid}^{\blue\csigma \blue}_  {\blue\csigma \blue} 
 + h^{<\la}_{(  C,A) }
=c_\SSTS + h^{<\la}_{(  C,A) }
 \\
 (A^-_\ctau 
A^+_\csigma (c_{\SSTS'}) + h^{<\la}_{(  C,A) })
   (1+ \eee_{\mu + \ctau}^{\ell_\ctau(\mu)+1})
&=
c_\SSTS + c_{\SSTS'} \compose 
( 
{\sf trid}_{\blue\csigma \blue}^\blue
\otimes {\sf spot}^\csigma _\emptyset)
 + h^{<\la}_{(  C,A) }
=c_\SSTS+ h^{<\la}_{(  C,A) }
\end{align*}
where in both cases the diagram $c_{\SSTS'}\compose \dots$ 
factors through  the idempotent 
labelled by $ {\stt_{\la'}}$ and so 
 belongs to the ideal 
$h^{<\la}_{(  C,A) }$.
The final two subcases which we will consider simultaneously are 
$$
R^+_\ctau 
A^-_\csigma (c_{\SSTS'})= c_{\SSTS'} 
  \compose 
  {\sf trid}_{\ctau\csigma \ctau}^\ctau 
\qquad
R^-_\ctau 
A^-_\csigma (c_{\SSTS'})=
 c_{\SSTS'}  
\compose 
   {\sf trid}_{\ctau\csigma \ctau}^\emptyset  .
  $$
  Here $c_\SSTS = R^{\pm}_\ctau A^-_\csigma (c_{\SSTS'})$ satisfies $\SSTS \not \in \Path^{\pm}_{(C_n,A_{n-1})} $. 
  We note that 
 $  \eee_{\mu + \ctau}^{\ell_\ctau(\mu)+1} = {\sf 1}_{\stt_{\mu-\csigma} }\compose 
({\sf 1}_{\ctau\csigma\ctau}+
{\sf trid}^{\blue\csigma \blue} _{\blue\csigma \blue}  )$ and hence   applying \cref{repeat11} we obtain 
$$( R^\pm _{\ctau} A_\csigma ^-(c_{\SSTS'}))(1+\eee_{\mu + \ctau}^{\ell_\ctau(\mu)+1})
=0 $$
as required. 

 \noindent{\bf Case 2. } 
 Suppose $\csigma \in \Rem( \la') $.  
This implies that $\crho \in \Add(\colMAP  ^{-1}(\la'))$. 
Since any two $\csigma $-tiles in $ \mu'$  are separated by some $\blue$-tile (and 
$\ctau \not\in 
\Rem(\la')$ but $\csigma  \in \Rem(\la')$) we have that 
 $c_{\SSTS'}=c_{\SSTS''}\compose {\sf spot}_\blue^\emptyset$ 
for some $\SSTS'' \in \Path^{\pm} _{  (  C_{n},A_{n-1 }) }(\la' ,\mu'-\ctau)$.  
Here we have that %  The first two subcases which we consider simultaneously are 
$$
 R^+_\alphar (c_{\SSTS'})= c_{\SSTS'} {\;\color{magenta}\circledast\;}
  {\sf fork}_{\alphar\alphar}^\alphar
  =  c_{\SSTS''}    \compose \spot_\ctau^\emptyset 
 {\;\color{magenta}\circledast\;}
  {\sf fork}_{\alphar\alphar}^\alphar
 \quad
R^-_\alphar (c_{\SSTS'})
= 
c_{\SSTS'} {\;\color{magenta}\circledast\;}
 {\sf  cap}_{\csigma  \csigma}^\emptyset
=
 c_{\SSTS''} \compose 
 \spot_
 \ctau^\emptyset  {\;\color{magenta}\circledast\;}
 {\sf  cap}_{\csigma  \csigma}^\emptyset  
 $$
%$$
% R^+_\alphar (c_{\SSTS'})= c_{\SSTS'} {\;\color{magenta}\circledast\;}
%  {\sf fork}_{\alphar\alphar}^\alphar
%  = c_{\SSTS''} {\;\color{magenta}\circledast\;}
% {\sf trid}_{ \csigma \blue\csigma }^\csigma 
% \qquad
%R^-_\alphar (c_{\SSTS'})
%= c_{\SSTS'} {\;\color{magenta}\circledast\;}
%{\sf spot}_\alphar^\emptyset   {\sf fork}_{\alphar\alphar}^\alphar
%= c_{\SSTS''} {\;\color{magenta}\circledast\;}
% (
% {\sf spot}^\emptyset_\alphar 
% {\sf trid}_{ \csigma \blue\csigma }^\csigma )
% $$
 We start with$$
 (A^{\pm}_\ctau R^+_\csigma (c_{\SSTS '}))({\sf 1}_{\stt_{\mu\ctau}}+\eee_{\mu   \ctau}^{\ell_\ctau(\mu)})
({\sf 1}_{\stt_{\mu\ctau}}+\eee_{\mu  \ctau}^{\ell_\ctau(\mu)+1})
.
$$
and we first  consider $A^+_\ctau R^+_\csigma (c_{\SSTS'})$. We note that $\ctau \in 
\Add(\la')$ but that $\violet \not \in  \Add(\colMAP  ^{-1}(\la'))$ (rather, the ``wrong" colour $\color{green!80!black}\rho $ is).  
 Therefore $\SSTS \not \in  \Path^{\pm} _{  (  C_{n},A_{n-1 }) }(\la,\stt_{\mu})$ and using % \cref{whenisitzeroinquotient}
  \cref{jkhsfjkhdsfgjkhdsfgdgsfhjkjkhdgfkjdhgsfjkhdsgf} we have  
\begin{align*}
( A^+_\ctau R^+_\csigma (c_{\SSTS '}) ) \circ ({\sf 1}_{\stt_{\mu\ctau}} + \eee_{\mu   \ctau}^{\ell_\ctau(\mu)})\circ 
({\sf 1}_{\stt_{\mu\ctau}}+\eee_{\mu  \ctau}^{\ell_\ctau(\mu)+1})
 \in h^{<\la}_{(  C,A) }
\end{align*}
%   belongs to $h^{<\la}_{(  C,A) }$
  since  the terms in the sum factor through the idempotent $\la' - \ctau$.  
 Therefore $  c_\SSTS  \eee =0 $ modulo $h^{<\la}_{(  C,A) }$  as required. 
Arguing in an identical manner, (or by simply ``putting a blue spot on top of the above calculation") we have that 
\begin{align*}
& 
 (A^-_\ctau R^+_\csigma (c_{\SSTS '})   )({\sf 1}_{\stt_{\mu\ctau}}+\eee_{\mu  \ctau}^{\ell_\ctau(\mu)})
( {\sf 1}_{\stt_{\mu\ctau}} +\eee_{\mu   \ctau}^{\ell_\ctau(\mu)+1})
   \in  h^{<\la}_{(  C,A) } 
\end{align*}
as required.  
The final two subcases which we will consider simultaneously are 
\begin{align*}c_{\SSTS}	=
 R^+_ \blue   R^-_\csigma  (c_{\SSTS'})
&=
c_{\SSTS''}  \compose \spot_\ctau^\emptyset 
{\;\color{magenta}\circledast\;} %{\sf spot}_\csigma ^\emptyset 
{\sf cap}_{\csigma  \csigma }^\emptyset 
\compose {\sf fork}_{\blue\blue}^\blue
\\
c_{\SSTS}=  R^-_ \blue   R^-_\csigma  (c_{\SSTS'})
&=
c_{\SSTS''}  
\compose \spot_\ctau^\emptyset 
{\;\color{magenta}\circledast\;} 
{\sf cap}_{\csigma  \csigma }^\emptyset 
\compose {\sf cap}_{\blue\blue}^\emptyset
\end{align*}
 In both cases, $\SSTS \in  \Path^{\pm} _{  (  C_{n},A_{n-1 }) }$.   We have that 
\begin{align*}
( R^+_\ctau R^-_\csigma (c_{\SSTS '}))({\sf 1}_{\stt_{\mu\ctau}}+\eee_{\mu   \ctau}^{\ell_\ctau(\mu)+1})
&=
 c_{\SSTS''}  \compose \spot_\ctau^\emptyset 
{\;\color{magenta}\circledast\;}  
{\sf cap}_{\csigma  \csigma }^\emptyset 
\compose {\sf fork}_{\blue\blue}^\blue
+ 
c_{\SSTS''}  
{\;\color{magenta}\circledast\;} {\sf spot}_\csigma ^\emptyset 
\compose 
%{\sf fork}^\ctau_{\ctau\ctau}({\sf 1}_\ctau \otimes {\sf trid}_{\blue\csigma \blue}^\blue)
 ( 
 \fork^\ctau_{\ctau\ctau} \compose {\sf trid}_{\blue\csigma \blue}^\blue )
 \\ 
(R^-_\ctau R^-_\csigma (c_{\SSTS '}))({\sf 1}_{\stt_{\mu\ctau}}+\eee_{\mu   \ctau}^{\ell_\ctau(\mu)+1})
&=
 c_{\SSTS''}  \compose \spot_\ctau^\emptyset 
{\;\color{magenta}\circledast\;}  
{\sf cap}_{\csigma  \csigma }^\emptyset 
\compose {\sf cap}_{\blue\blue}^\blue
+ 
c_{\SSTS''}  
{\;\color{magenta}\circledast\;} {\sf spot}_\csigma ^\emptyset 
\compose ( 
 \fork^\ctau_{\ctau\ctau} \compose {\sf trid}_{\blue\csigma \blue}^\emptyset )
\end{align*}
In each case  the former term on the righthand-side of the equality  is  equal to $c_\SSTS$ and the latter term 
is equal to $c_\SSTT$ for  $\SSTT \not \in
\Path^{\pm} _{  (  C_{n},A_{n-1 }) }(\la ,\stt_{\mu })$. 
 The result follows.  \end{proof}

  \begin{lem}\label{typeClemmer}
%Let $\ctau =  \color{cyan}   s_1   \color{black}  ,
%\csigma =  \color{magenta}   s_2   \color{black}   \in W_C$.  
Let $\eee_{\ctau\csigma\ctau\csigma}=( {\sf1}_{\ctau\csigma\ctau\csigma}+\trid^{\ctau\csigma\ctau}_{\ctau\csigma\ctau}\otimes
 {\sf 1}_\csigma)$,
 then we have 
 $$\eee _{\ctau\csigma\ctau\csigma} {\sf 1}_{\ctau\csigma\ctau\csigma} \eee_{\ctau\csigma\ctau\csigma}
 =
 -
 \eee_{\ctau\csigma\ctau\csigma}({\sf 1}_\ctau \otimes \trid^{\csigma\ctau\csigma}_{\csigma\ctau\csigma})\eee_{\ctau\csigma\ctau\csigma}$$ 
  \end{lem}  

\begin{proof}Applying $\eee_{\ctau\csigma\ctau\csigma}$ to both sides of the $m(\csigma,\ctau)=4$ null-braid relations and using \cref{bpb} immediately gives the result. 
%One first rewrites the element $\on_{\stt_\mu} \otimes \on_{\blue\csigma\blue\csigma} $ using the $\blue\csigma$-Tetris relation on the final 4 strands.  
%All but one of the result diagrams contains a $\trid_{\blue\csigma \blue}^\blue$ or its dual;   
%we apply the idempotent $\eee$ to each of these diagrams, all of which are zero by \cref{bpb}.  The remaining diagram is the required one, $-\eee  \circ   
%(\on_{\stt_\mu} \otimes \on_{\blue}\otimes \trid_{\csigma\blue\csigma} ^{\csigma\blue\csigma})\circ \eee$. 
\end{proof}  

\begin{cor}\label{ZEROC}\color{black!99}
Let $  \mu \in\mathscr{P}_{(C_{n+1},A_{n})}$ and $[r,c]\in \mu$. Define $k,l,m$ as in  \cref{gapsbasis}.
If $[r-k, c+k+1] \notin \mu$ or $[r+l+1, c-l]\notin \mu$ then we have
  $\eee \circ 
  \gap(\stt_\mu-[r,c])\circ \eee =0$.  
Otherwise
 we have 
  $$\eee \circ 
  \gap(\stt_\mu-[r,c])\circ \eee = (-1)^{k+l-m} 	\eee \circ \iota (c_{\SSTS\SSTS})\circ \eee 	$$
  where $c_{\SSTS \SSTS}$ is defined in \cref{gapsbasis}. 
   \end{cor}
\begin{proof}\color{black!99}
The proof follows exactly the same arguments as for the proof the corresponding statement in type $(D_{n+1}, A_n)$ given in \cref{iszero,isnonzero}. There are only two additional things to check. First we need to prove that for $[r,r]\in \mu$ we have 
 $\eee \circ 
  \gap(\stt_\mu-[r,r-1])\circ \eee =0$. 
This follows directly  from \cref{shorter,bpb}.  
The second thing is that the $(s_0,s_2)$ and $(s_1,s_2)$-nullbraid relation is preserved under the map $\Theta$. This is precisely the statement of \cref{typeClemmer}.
\end{proof}
  
  \begin{lem}\label{ajkfhsdjkahsfakjhsfdhjsfsfdhjafsdyuasfhgjasfdhbjdfas}
 For $\SSTS \in \Path _{  ( D_{n+1},A_{n})}(\la ,\stt_{\mu })$, we have that
 $ \colMAP(c_{ \SSTS }) \circ  e  = e \circ  \colMAP(c_{\SSTS})  \circ  e.$ 
 
  \end{lem}
  \begin{proof}
For $1\leq k \leq \ell_\ctau(\colMAP(\la))$, we will show that 
$\eee^k_\la \colMAP(c_\SSTS) = 
\colMAP(c_\SSTS) \eee^j_\mu$ for some $1\leq j \leq \ell_\ctau(\colMAP(\mu))$ and hence deduce the result.  
Assume $k$ is even (the odd case is identical).
We can assume that 
$\stt_\mu$ is such that each $\violet$-strand is immediately preceded by a $\csigma$-strand, so 
that $\stt_\la= \underline{v_0}
\crho \underline{v_1}\csigma\cpi\underline{v_2}$
where  $\ell( \colMAP(\underline{v_0}
\crho \underline{v_1}\csigma\cpi))=k$
and where $ \underline{v_1}$ does not contain $\crho,\csigma,\cpi$.  
We can write $\stt_\mu$ in the form 
$\underline{w_0}
\crho \underline{w_1}\csigma\underline{w_2}\cpi\underline{w_3}$ 
such that the $\crho,\csigma,\cpi$ in this expression
%$\underline{w_0}
%\crho \underline{w_1}\csigma\underline{w_2}\cpi\underline{w_3}$ 
are connected to the  
$\crho,\csigma,\cpi$ in the expression 
$\underline{v_0}
\crho \underline{v_1}\csigma\cpi\underline{v_2}$ by strands in the Soergel diagram.  
Moreover, we can assume that $\ell(\underline{w_0})$
and $\ell(\underline{w_3})$ are maximal with respect to this property.  
We claim that $ \underline{w_1} $ and $
\underline{w_2}$ do not contain 
any occurrences of $\crho,\csigma,\cpi$.  
Thus  the specified  $\violet$ and $\green$ strands commute with  all strands lying between them, except for  the specified $\csigma$-strand.  
  Under $\colMAP$ these correspond to $\ctau$-strands which commute with  all strands lying between them, except for  the specified $\csigma$-strand. Thus applying the trident on top/bottom of these strands  we get the same result, as required.  
  
  It only remains to verify the claim. 
Suppose one of the three colours does occur  in $ \underline{w_1} $. The first colour to appear must be $\csigma$ (because it follows $\crho$) and this must be a $A_\csigma^-$ step in the basis (because $\underline{v}_1$ has no $\csigma$ and so it cannot be a  $X_\csigma^+$  step 
 and the prior $\crho$ step was an $X^+_\crho$ and so it cannot be an $R_\csigma^-$ step). 
After this $\csigma$, there must be a $\cpi$
but this cannot be an $R_\cpi^{\pm}$ (as the  prior $\crho$ step was an $X^+_\crho$) or be  $A^{\pm}_\cpi$ (because
the prior $\csigma$ was a $A_\csigma^{-}$).  Thus the claims follows for $ \underline{w_1} $.  The case of 
$ \underline{w_3} $ is similar. 
%We choose $\stt_\la=\stt_\kappa \circ \stt_{\la-\kappa}$ as in \cref{defn of idempt}
% so that the $k$th blue strand in the northern reading word of $c_\SSTS$ (which we denote $\color{cyan}X_1$)
%  has a red strand immediately preceding it    (which we denote $\color{magenta}X_2$)
%  and a blue strand immediately preceding  that (which we denote $\color{cyan}X_3$).  
%   Our assumption that $\SSTS$ is {\em parity preserving} implies that $\color{magenta}X_2$ connects two precisely one point on the southern boundary, which we set to be the $x_2$th point in the southern reading word $\stt_\mu$; 
%   letting ${\color{cyan}x_3}=\max\{i > x_2 \mid \stt_\mu^{-1}(i)= \ctau\}$  and ${\color{cyan}x_1}=\min\{i  > x_2 \mid \stt_\mu^{-1}(i)= \ctau\}$ we have that 
%   the strands $\color{cyan}X_3$ and $\color{cyan}X_1$ terminate at the points $\color{cyan}x_3$ and $\color{cyan}x_1$ respectively.  In particular, the last $\ctau$-strand before (and first $\ctau$-strand after) $\color{magenta}X_2$ is propagating.  %This implies that $\ctau$ commutes with every colour 
%   This further implies that there are $\color{magenta}x_2$ is the unique pink  point in the southern reading word 
%This implies that 
%$\color{magenta}x_2$ is the unique pink  point within the region of points ${\color{cyan}x_3}\leq y \leq {\color{cyan}x_1}$    of the southern reading word. 
%Thus  $\color{cyan}X_1$ and $\color{cyan}X_3$ commute with  all strands lying between them, except for $\color{magenta}X_2$.  
%The result follows.
\end{proof}

%  
% \begin{thm}\label{C--A}
% We have that   $ h_{  ( D_{n+1},A_{n})}$ and 
% $eh_{  (  C_{n},A_{n-1 }) } e$
% are isomorphic as graded $\Bbbk$-algebras via the map
% $$ 
%\Theta:  c_{\SSTS\SSTT} \mapsto 
% {\sf e} \circ \colMAP( c_{\SSTS\SSTT}		)\circ {\sf e}.$$   
% Thus  $ h_{  ( D_{n+1},A_{n})}$ and 
% $ h_{  (  C_{n},A_{n-1 }) }  $ are graded Morita equivalent.  
% \end{thm}

% We are now ready to prove the main theorem in  the symplectic case.  

\begin{prop}\label{vectorproof2}
The map  $\Theta: h_{{ ( D_{n+1},A_{n }  )}} \to 
 {\sf e}h_{ ( C_{n},A_{n-1 })}{\sf e}$ given by 
  $ 
\Theta ( c_{\SSTS\SSTT} )= 
 {\sf e} \circ \colMAP( c_{\SSTS\SSTT})\circ {\sf e} $  is a  $\Bbbk$-algebra homomorphism. 

\end{prop}

\begin{proof}%[Proof of   \cref{C--A} in the symplectic case]
%We have already shown that $\Theta$ is an isomorphism of graded $\Bbbk$-spaces.  
%It remains to check that $\Theta(ab) =  \Theta(a) \Theta(b)$ for $a,b \in h_{  ( D_{n+1},A_{n})}$.  
We check this on the cellular basis by showing that 
$$ {\sf e} \circ \colMAP( c_{ \SSTS \SSTT } c_{ \SSTU\SSTV})\circ 
 {\sf e} = \Theta(c_{\SSTS\SSTT}c_{\SSTU\SSTV}) = 
  \Theta(c_{\SSTS\SSTT}) \Theta(c_{\SSTU\SSTV})
  ={\sf e} \circ 
\colMAP(  c_{ \SSTS  \SSTT })\circ 
  {\sf e} \circ 
\colMAP(c_{ \SSTU \SSTV })\circ 
  {\sf e} $$
  for $\SSTS\in \Path_{  ( D_{n+1},A_{n})}(\nu,-)$,
  $\SSTT\in \Path_{  ( D_{n+1},A_{n})}(\nu,\stt_\mu)$, 
  $\SSTU\in \Path_{  ( D_{n+1},A_{n})}(\eta,\stt_\mu)$,
  $\SSTV\in \Path_{  ( D_{n+1},A_{n})}(\eta,-)$.  
%In order to check this, we  define the cellular structure coefficients $a_{\SSTX,\SSTY}\in \Bbbk$ as follows 
%$$
%c_{\SSTT'}c_{\SSTU'}^\ast = \sum_{\SSTX,\SSTY}a_{\SSTX,\SSTY} c_{\SSTX\SSTY}
%$$
  By \cref{ajkfhsdjkahsfakjhsfdhjsfsfdhjafsdyuasfhgjasfdhbjdfas}, we have that 
  $$
   {\sf e}\circ  \colMAP(  c_{ \SSTS  \SSTT })\circ 
    {\sf e}\circ  \colMAP(  c_{ \SSTU  \SSTV })\circ  {\sf e}  
    =
   {\sf e}\circ  \colMAP(c_{ \SSTS  }^\ast)    \circ  {\sf e}  
\circ   \colMAP( c_{ \SSTT  } ) \circ     {\sf e}  
\circ  \colMAP(  c_{ \SSTU  }^\ast) \circ {\sf e}  
\circ  \colMAP(      c_{ \SSTV  }   ) \circ   {\sf e}  .$$
%      Diagrammatically, this is visualised in \cref{}.  
%
   We   proceed by induction on $\ell(\mu)$, the base case $\ell(\mu)=0$ is trivial.  
   We can assume   $\ell(\pi),\ell(\rho) < \ell(\mu)$ as if 
      $\ell(\pi)=\ell(\rho)=\ell(\mu)$  then ${\sf 1}_\pi={\sf 1}_{\rho}={\sf 1}_\mu$ 
      and this product becomes ${\sf e}        \circ \colMAP (c_{ \SSTS  \SSTV })        \circ{\sf e}$ as required.  Similarly, if 
       $\ell(\pi)= \ell(\mu)$ and $\ell(\rho)< \ell(\mu)$ (or vice versa) this product becomes
       $$
       {\sf e} 
       \circ 
       \colMAP(c_{\SSTS}^\ast) 
       \circ              {\sf e}       \circ
               \colMAP(c_{\SSTU}^\ast)        \circ
              {\sf e}       \circ
               \colMAP(c_{\SSTV}) 
               =
  (  {\sf e}       \circ \colMAP(c_{\SSTS}^\ast) 
              \circ       {\sf e}       \circ
               \colMAP(c_{\SSTU}^\ast)       \circ       
              {\sf e})
        (       \eee        \circ \colMAP(c_{\SSTV})         \circ \eee )
        =
        \Theta (c_\SSTS^\ast  c _\SSTU^\ast )
                \Theta (c_\SSTV )
       $$
and so we can again appeal to our inductive assumption.  
We will focus on the middle of the product and prove that 
\begin{align}
{\sf e}       \circ \colMAP( c_{ \SSTT })        \circ 
{\sf e}        \circ \colMAP(c_{\SSTU}^\ast)        \circ {\sf e}
=
{\sf e}        \circ \colMAP( c_{ \SSTT  \SSTU })        \circ {\sf e}.
\end{align} 
 As $\ell(\eta),\ell(\nu) \leq  \ell(\mu)$ we can then apply induction to deal with the products with ${\sf e}        \circ \colMAP(c_{\SSTS}^\ast)$ and $ \colMAP(c_{\SSTV})        \circ{\sf e}$.  
Now, the basis elements $c_\SSTT$ and $c_\SSTU^\ast$ are constructed inductively and we will consider cases depending on the last step in this inductive procedure.  

%We will find it useful (for the purposes of notation) to work with reduced expressions other than $\stt_\mu \in \Std(\mu)$ 
%and $\stt_\nu \in \Std(\nu)$.  Of course, this makes no difference to the result because any two reduced expressions differ only by application of the commutation relations. 
%We let 
%  $\sts_\mu\in \Std(\mu)$, $\sts_\nu\in \Std(\nu)$ denote such reduced paths. 
 
\smallskip
\noindent {\bf Case 1. } We first consider the case that 
$c_\SSTT = A_\grey  ^+(c_{\SSTT'})$ and 
$c_\SSTU = A_\grey  ^+(c_{\SSTU'})$.  By induction, we can assume that 
$$\eee        \circ\colMAP(c_{ \SSTT' })        \circ \eee      \circ\colMAP( c_{ \SSTU' }^\ast)        \circ \eee 
=\textstyle\sum _{\SSTX,\SSTY  } a_{\SSTX,\SSTY}\eee       \circ \colMAP(c_{\SSTX\SSTY})        \circ\eee
\quad \text { where } \quad  
c_{\SSTT'}c_{\SSTU'}^\ast =
\textstyle \sum_{\SSTX,\SSTY}a_{\SSTX,\SSTY} c_{\SSTX\SSTY}
$$
 If $\grey   \neq \cpi, \crho$, then $\ell_\ctau(\colMAP(\mu))=\ell_\ctau(\colMAP(\mu-\grey ))$ and therefore
$$\eee  
       \circ
       \colMAP(c_{\SSTT})
              \circ \eee 
                     \circ
                     \colMAP(c_{\SSTU}^\ast)        \circ \eee 
=\textstyle\sum _{\SSTX,\SSTY  } a_{\SSTX,\SSTY}
(\eee
( \colMAP(c_{ \SSTX \SSTY })\otimes {\sf 1}_\grey )\eee).$$
%\noindent{\em Subcase 1.2. } 
  If  $\grey  = \cpi$ (the $\crho $ case is identical) then 
$c_\SSTT= A^+_\cpi A^+_\csigma (c_{\SSTT''})$, 
$c_\SSTU= A^+_\cpi A^+_\csigma (c_{\SSTU''})$   
and \begin{align*}
\eee       \circ \colMAP(c_{\SSTT})       \circ \eee        \circ\colMAP(c_{\SSTU}^\ast)        \circ \eee
&=
\eee
\left((\colMAP(c_{\SSTT''})        \circ \eee_{\mu-\ctau-\csigma}       \circ \colMAP(c_{\SSTU''}   ^\ast))
\compose ( \on_{\ctau \csigma \ctau } + {\sf trid}_{\ctau \csigma\ctau }^{\ctau \csigma\ctau } ) \right) 
 \eee 
 \\
  &=
 \eee
\left(	(\colMAP(c_{\SSTT''})        \circ \eee_{\mu-\ctau-\csigma}        \circ\colMAP(c_{\SSTU''}^\ast))
 \compose  \on_{\ctau \csigma \ctau }  \right) \eee
 \\
 &=
 \eee
\left((\colMAP(c_{\SSTT'} )
\circ  \eee_{\mu-\ctau} \circ  \colMAP(c_{\SSTU '}^\ast))
 \otimes   \on_{\ctau }  \right) \eee
 \\
 &=\textstyle\sum _{\SSTX,\SSTY  } a_{\SSTX,\SSTY}\eee
 (\colMAP( c_{\SSTX \SSTY })\otimes {\sf 1}_\ctau )
 \eee 
\end{align*}
the first   equality follows  from the definition of the idempotents; for the second equality, we note that 
the trident term in the sum is zero by \cref{bpb};
  the third equality follows by definition of the cellular basis elements and the idempotents; the final equality holds by induction.  
Thus in all cases, we have that 
\begin{align}\label{huzzah}
\eee       \circ \colMAP(c_{\SSTT})       \circ \eee        \circ\colMAP(c_{\SSTU}^*)        \circ \eee 
=\sum _{\SSTX,\SSTY  } a_{\SSTX,\SSTY}\eee
( \colMAP (c_{ \SSTX \SSTY })\otimes {\sf 1}_{\colMAP(\grey)} )\eee.
\end{align}
 It remains to show that every $\eee (\colMAP(c_{\SSTX\SSTY})\otimes {\sf 1}_{\colMAP(\grey)})\eee
=\eee (\colMAP(c_{\SSTX\SSTY} \otimes {\sf 1}_{ \grey }))\eee$ for every $\SSTX,\SSTY$ appearing in the above sum.  
We set $\la:= \Shape(\SSTX)= \Shape(\SSTY)$.

\begin{figure}[ht!]

  $$ 
 \begin{minipage}{4.14cm}
 \begin{tikzpicture}[scale=0.9]

\fill[gray!30] (0.5,2)--(3.25,2) to [out=-90,in=90] (4.25,0.5)--(0.5,0.5)--(0.5,2) ;

\fill[gray!30] (0.5,-1)--(3.25,-1) to [out=90,in=-90] (4.25,0.5)--(0.5,0.5)--(0.5,-1);

 \draw[very thick, gray] (3.5,2)  to [out=-90,in=90] (4.5,0.5)  to [out=-90,in=90] (3.5,-1);

 \path(0.5,2)  rectangle (3.25,0.5) node [pos=.5]{$\colMAP({\sf c_{\SSTT'}})$};;
 \path(0.5,0.5)  rectangle (3.25,-1) node [pos=.5]{$\colMAP({\sf c_{\SSTU'}^*})$};;

   \draw[fill=white , rounded corners] (0.25,2) rectangle (3.75,2.4) node [pos=.5]{$\sf e$};;

   \draw[fill=white ,rounded corners] (0.25,2) rectangle (3.75,2.4) node [pos=.5]{$\sf e$};;

     \draw[fill=white, rounded corners] (0.25,-1) rectangle (3.75,-1.4) node [pos=.5]{$\sf e$};;
     \draw[fill=white, rounded corners] (0.25,0.5-0.2) rectangle (4.75,0.5+0.2) node [pos=.5]{$\sf e$};;

\end{tikzpicture}\end{minipage} 
=
\sum a_{\sf X, Y}
\begin{minipage}{3.24cm}
 \begin{tikzpicture}[scale=0.9]

\fill[gray!30] (0.5,2)--(3.25,2) to [out=-90,in=90] (2.25,0.5)--(0.5,0.5)--(0.5,2) ;

\fill[gray!30] (0.5,-1)--(3.25,-1) to [out=90,in=-90] (2.25,0.5)--(0.5,0.5)--(0.5,-1);

 \draw[very thick, gray] (3.5,2)  to [out=-90,in=90] (2.5,0.5)  to [out=-90,in=90] (3.5,-1);
 
  \draw[ rounded corners] (0.25,2) rectangle (3.75,2.4) node [pos=.5]{$\sf e$};;
 
  \draw[ rounded corners] (0.25,-1) rectangle (3.75,-1.4) node [pos=.5]{$\sf e$};;
 
 \path(0.5,2)  rectangle (2.25,0.5) node [pos=.5]{$\colMAP({\sf c_{\sf X}^\ast}) $};;
 \path(0.5,0.5)  rectangle (2.25,-1) node [pos=.5]{$\colMAP({\sf c_{\sf Y}})$};;

\end{tikzpicture}\end{minipage} 
 $$
\caption{Case 1:  a diagrammatic version of \cref{huzzah}.}
\label{diagram-huzzah}
\end{figure}

\smallskip
\noindent{\em Subcase 1.1. }
 If $\grey \in \Add(\la)$ and $\colMAP  (\grey) \in \Add (\colMAP  (\la))$ then
$\colMAP( c_{\SSTX\SSTY})\otimes {\sf 1}_{\colMAP(\grey)}$  and 
$  c_{\SSTX\SSTY} \otimes {\sf 1}_{ \grey }$ are both cellular basis elements 
  and we are done.  
 
 \smallskip
\noindent{\em Subcase 1.2. }
If $\grey \not \in \Add(\la)$ and $\colMAP  (\grey) \in \Add (\colMAP  (\la))$ then 
we can assume that $\grey=\cpi$ (the $\grey=\crho$ case  is identical).  
This implies that $\crho  \in \Add(\la)$  (but by assumption $\csigma \in \Rem(\la)$ and $\cpi \in \Rem(\la-\csigma)$) 
 this implies that we can write $c_{\SSTX}$ and $c_{\SSTY}$ as 
 \begin{align}\label{maudstar}
c_{\SSTX}=c_{\SSTX'} {\color{green!80!black}\;  \circledast \; }\spot_\crho^\emptyset \mcompose \fork^\csigma_{\csigma\csigma}
\qquad
c_{\SSTY}=c_{\SSTY'} {\color{green!80!black}\;  \circledast \; }\spot_\crho^\emptyset \mcompose \fork^\csigma_{\csigma\csigma}
\end{align}
for some $\SSTX',\SSTY' \in \Path_{(D_{n+1},A_{n})}(\la+\crho, \stt_{\nu-\crho-\csigma})$.  Now, as $\cpi \in \Rem(\la-\csigma)$, using the $\cpi\csigma$-bull-braid relations we get
 \begin{align}\label{maudlabel1}
c_{\SSTX\SSTY} \otimes \on_\cpi
&= 
 - (c_\SSTX \mcompose \spot_\csigma^\emptyset \pcompose \fork^{\cpi}_{\cpi\cpi})^\ast 
  (c_\SSTY \mcompose \spot_\csigma^\emptyset \pcompose \fork^{\cpi}_{\cpi\cpi}) 
 \end{align}
 
On the other hand, using \cref{maudstar,jkhsfjkhdsfgjkhdsfgdgsfhjkjkhdgfkjdhgsfjkhdsgf} we have 
$$(\colMAP(c_\SSTX \otimes  {\sf 1}_\ctau)\eee = -({\sf 1}_{\stt_\la} \mcompose \spot_\csigma^\emptyset \compose \fork^{\blue}_{\blue\blue})^\ast 
( \colMAP(c_\SSTX) \mcompose \spot_\csigma^\emptyset \compose \fork^{\blue}_{\blue\blue})  \eee$$
and similarly for $(\colMAP(c_\SSTY)\otimes {\sf 1}_\ctau)\eee $.
Thus we get 
 \begin{align} \nonumber
&\; \eee (\colMAP(c_{\SSTX\SSTY})\otimes {\sf 1}_\ctau)\eee
\\ 
\nonumber
  =&\;
 \eee
( \colMAP(c_\SSTX) \mcompose \spot_\csigma^\emptyset \compose \fork^{\blue}_{\blue\blue})  ^\ast 
 ({\sf 1}_{\stt_\la} \mcompose \spot_\csigma^\emptyset \compose \fork^{\blue}_{\blue\blue}) 
 ({\sf 1}_{\stt_\la} \mcompose \spot_\csigma^\emptyset \compose \fork^{\blue}_{\blue\blue}) ^\ast 
( \colMAP(c_\SSTY) \mcompose \spot_\csigma^\emptyset \compose \fork^{\blue}_{\blue\blue})  \eee
\\
\label{maudlabel2}
=&\;
-\eee
( \colMAP(c_\SSTX) \mcompose \spot_\csigma^\emptyset \compose \fork^{\blue}_{\blue\blue})  ^\ast 
 ( \colMAP(c_\SSTY) \mcompose \spot_\csigma^\emptyset \compose \fork^{\blue}_{\blue\blue})  \eee
 \end{align}
 comparing \cref{maudlabel1,maudlabel2} we are done.

\begin{figure}[ht!]

$$
\begin{minipage}{3.2cm}
\end{minipage} $$
\caption{Subcase 1.2}
\label{seond-thing-in-M-proof}
\end{figure}

 \smallskip
 \noindent{\em Subcase 1.3. } 
If $\grey \in \Rem(\la)$ and 
 $\colMAP  (\grey) \in \Rem(\colMAP  (\la))$  then  
the monochrome Tetris relation implies that 
\begin{align*}
%\begin{split}\label{jkdgfjhkdgfjhkgjhkdfgwww2} 
c_{\SSTX\SSTY} \otimes \on_\grey 
 = &\;
 (
c_{\SSTX } {\color{gray} \;\circledast\; }\fork^\grey_{\grey\grey})^\ast 
  (
c_{\SSTY } {\color{gray} \;\circledast\; }\spot^\emptyset_ \grey \otimes {\sf 1}_\grey) 
+
 (
c_{\SSTX }  {\color{gray} \;\circledast\; }\spot^\emptyset_ \grey \otimes {\sf 1}_\grey )^\ast 
 (
c_{\SSTY } {\color{gray} \;\circledast\; }\fork^\grey_{\grey\grey}) 
\\
   &\quad +
   (
c_{\SSTX } {\color{gray} \;\circledast\; }\fork^\grey_{\grey\grey})^\ast 
(\textstyle\sum_{[x,y]\in {\sf \underline{Hook}}_\grey (\la-\grey)}\gap(\stt_\la-[x,y]))
   (
c_{\SSTY } {\color{gray} \;\circledast\; }\fork^\grey_{\grey\grey}) 
%\end{split}
\end{align*}
Similarly, we obtain 
\begin{align*}
%\begin{split}\label{jkdgfjhkdgfjhkgjhkdfgwww} 
&\;\eee(\colMAP(c_{\SSTX\SSTY}) \otimes \on_{\colMAP(\grey )})\eee
\\= &\;
\eee (\colMAP(c_\SSTX)  {\color{gray} \;\circledast\; }\fork^{\colMAP(\grey)}_{\colMAP(\grey)\colMAP(\grey)})^\ast 
 (
\colMAP(c_{\SSTY }) {\color{gray} \;\circledast\; }\spot^\emptyset_ {\colMAP(\grey)} \otimes {\sf 1}_{\colMAP(\grey)}) \eee
\\
&\; +\eee
 (
\colMAP(c_{\SSTX })  {\color{gray} \;\circledast\; }\spot^\emptyset_ {\colMAP(\grey)} \otimes {\sf 1}_{\colMAP(\grey)} )^\ast 
 (
\colMAP(c_{\SSTY }) {\color{gray} \;\circledast\; }\fork^{\colMAP(\grey)}_{{\colMAP(\grey)}{\colMAP(\grey)}}) \eee
\\
   &\;  +
 \eee  (
\colMAP(c_{\SSTX }) {\color{gray} \;\circledast\; }\fork^{\colMAP(\grey)}_{{\colMAP(\grey)}{\colMAP(\grey)}})^\ast 
\eee (\textstyle\sum_{[x,y]\in {\sf \underline{Hook}}_{\colMAP(\grey)} (\colMAP(\la-{ \grey)})}
 \gap(\stt_{\colMAP(\la)}-[x,y]))\eee 
   (
\colMAP(c_{\SSTY }) {\color{gray} \;\circledast\; }\fork^{\colMAP(\grey)}_{{\colMAP(\grey)}{\colMAP(\grey)}}) \eee
%\end{split}
\end{align*}
where we have inserted extra idempotents $\eee$ in the final summand using \cref{ajkfhsdjkahsfakjhsfdhjsfsfdhjafsdyuasfhgjasfdhbjdfas}. Recall that we need to check that 
$$
\eee( \colMAP(c_{\SSTX\SSTY} \otimes {\sf 1}_\grey) )\eee
=
\eee( \colMAP(c_{\SSTX\SSTY} )\otimes {\sf 1}_{\colMAP(\grey)} )\eee
$$
% for \cref{jkdgfjhkdgfjhkgjhkdfgwww};
where  the first two terms in each of the above equations obviously agree. For the final term, note that if $\grey \neq \crho,\cpi$ we have 
 $$
{ \sf \underline{Hook}}_{\grey}(\la-\grey)={\sf \underline{Hook}}_{\colMAP(\grey)} (\colMAP(\la-  \grey) ) .
 $$
If $\grey=\crho$ then $\colMAP(\grey)=\ctau$ 
 (the $\grey-\cpi $  case is identical).  Say $ {\color{green!80!black}s_{[r,r]}}=\crho \in \Rem(\la)$.  
We note that 
\begin{align}\label{almostthere}
{\sf \underline{Hook}}_\crho (\la - \crho)\sqcup \{{\color{magenta}[r,r-1]}\} = 
{\sf \underline{Hook}}_{\ctau } (\colMAP(\la)-  \ctau ) 
\end{align}
where ${\color{magenta}s_{[r,r-1]}}=\csigma$.   
Note that $\eee \circ {\sf gap}(\stt_\mu - [r,r-1]) \circ \eee = 0$ 
 using \cref{ZEROC} so the result holds in subcase 1.3. 

\begin{figure}[ht!]
$$
\begin{minipage}{2.7cm}
\end{minipage} 
 $$
 \caption{Subcase 1.3.    }
 \end{figure}

\smallskip
\noindent{\em Subcase 1.4. }
If $\grey \not \in \Rem(\la)$ and $\colMAP  (\grey) \in \Rem (\colMAP  (\la))$ then 
we can assume that $\grey=\cpi$ (the $\grey=\crho$   is identical).  
Then we must have that $ {\color{green!80!black}s_{[r,r]}}=\crho \in \Rem(\la)$ 
and $ {\color{magenta}s_{[r+1,r]}}=\csigma$ and $ {\color{violet}s_{[r+1,r+1]}}=\cpi$. We have that 
$$c_\SSTX = c_{\SSTX'} \mcompose \spot_\csigma^\emptyset
\quad 
c_\SSTY = c_{\SSTY'} \mcompose \spot_\csigma^\emptyset$$
  for   $\SSTX',\SSTY' \in \Path_{(D_{n+1},A_n)}(\la+ \csigma,-)$.  Therefore 
\begin{align*}
 c_{\SSTX\SSTY} \otimes \on_\cpi  
&=
( c^\ast _{\SSTX'} \otimes {\sf 1}_\cpi )\gap(\stt_{\la+\csigma+\cpi}-{\color{magenta} {[r+1,r]}})
( c  _{\SSTY'} \otimes {\sf 1}_\cpi )=0
\end{align*}
using \cref{underphi2}.  On the other hand
\begin{align*}
\eee\circ (\colMAP(c_{\SSTX\SSTY})\otimes {\sf 1}_\ctau)\circ \eee &=
\eee\circ  (\colMAP(c_{\SSTX'}^\ast)\otimes {\sf 1}_\ctau)\circ \eee 
\circ \gap(\stt_{\colMAP(\la)+\csigma+\ctau}- {\color{magenta} {[r+1,r]}}) \circ \eee  \circ (\colMAP(c_{\SSTY'}) \otimes \on_\ctau)\circ \eee =0 
\end{align*}
where the first equality follows from \cref{ajkfhsdjkahsfakjhsfdhjsfsfdhjafsdyuasfhgjasfdhbjdfas} (inserting extra idempotents, $\eee$) and the second follows by \cref{ZEROC}.  

\smallskip
\noindent{\em Subcase 1.5. }
If $\colMAP(\grey)\not \in  \Add(\colMAP(\la))$ nor $\Rem(\colMAP(\la))$, then 
$\grey \not \in \Add (\la)$ or $\Rem(\la)$.  Take ${\color{gray} {[x,y]}}$ with $x+y$ minimal such that 
${\color{gray}s_{[x,y]}}=\grey$ and ${\color{gray} {[x,y]}}\not \in \la$.
Then precisely one of 
$[x,y-1]$ or $[x-1,y] \in \la$.
We assume ${\color{orange}[x,y-1]}
\in \la$ and we set   
$\gam={\color{orange}s_{[x,y-1]}}$. 
We have 
%{\color{gray} {[x,y]}}
  $$c_{\SSTX\SSTY}\otimes \on_\grey 
  = -( c_{\SSTX } \ocompose {\sf spot}_\orange^\emptyset 
   \greycompose {\sf fork}_{\grey\grey}^{ \grey})^*
   ( c_{\SSTY} \ocompose {\sf spot}_\orange^\emptyset 
   \greycompose {\sf fork}_{\grey\grey}^{ \grey}).
  $$
by the $\gam\grey$-null-braid relation. 
 This might not be a cellular basis diagram, but can be rewritten as such using \cref{iszero,isnonzero}.  
 Similarly $\eee (\colMAP(c_{\SSTX\SSTY})\otimes {\sf 1}_{\colMAP(\grey)})\eee$ can be rewritten in the same form using \cref{ZEROC}.  
 Subcase 1.5 follows.

\smallskip
\noindent {\bf Case 2. } 
 We now consider the case that 
$c_\SSTT = A_\grey  ^+(c_{\SSTT'})$ and 
$c_\SSTU = A_\grey  ^-(c_{\SSTU'})$ (the dual case with $\SSTT$ and $\SSTU$ swapped is similar).  
If $\grey \neq \cpi, \crho$, then 
\begin{align*}
{\sf e}       \circ \colMAP( c_{ \SSTT' }) \otimes \on_{\colMAP(\grey)}        \circ 
{\sf e}        \circ \colMAP(c_{\SSTU}^\ast) \otimes {\sf spot}^ {\colMAP (\grey)}_\emptyset      \circ {\sf e}
&=
{\sf e}     ( ( \colMAP( c_{ \SSTT' })  \circ 
{\sf e}        \circ 
\colMAP(c_{\SSTU}^\ast ) ) \otimes {\sf spot}^\grey_\emptyset)       ) {\sf e}
\\
&=
{\sf e}  \big(\colMAP( {\sf 1}
_{\stt_{\nu- \grey }}\otimes   {\sf spot}^\grey_\emptyset ) \big) \eee \circ \eee	   \big(    \colMAP( c_{ \SSTT' })  \circ 
{\sf e}        \circ 
\colMAP(c_{\SSTU}^\ast )   \big) \eee 
\end{align*}
  and so the result follows by induction on $\ell(\mu)$.

 \begin{figure}[ht!]
 $$ \begin{minipage}{3.24cm}
\end{minipage} $$
\caption{Case 2 for $\gsigma\neq \cpi$}

 \end{figure}

   If $\grey = \cpi$, then 
  $c_\SSTT= A^+_\cpi A^+_\csigma (c_{\SSTT''})$ and 
  $c_\SSTU= A^-_\cpi A^+_\csigma (c_{\SSTU''})$ and we also set $c_{\SSTT'}=A^+_\csigma(c_{\SSTT''})$ and 
  $c_{\SSTU'}=A^+_\csigma(c_{\SSTU''})$. 
  Expanding out the final term of the middle idempotent $\eee$ and applying \cref{bpb}, we obtain 
\begin{align*}
&{\sf e}       \circ( \colMAP( c_{ \SSTT'' } )\otimes \on_{\csigma\ctau})        \circ 
{\sf e}        \circ
 (\colMAP(c_{\SSTU''}^\ast)	\otimes  \on_{\csigma} \otimes {\sf spot}^\ctau_\emptyset)         \circ {\sf e}
\\
 =&
{\sf e}       \circ 
\big(
\big(\colMAP( c_{ \SSTT'   })         
  \circ 
{\sf e}        \circ
 \colMAP(c_{\SSTU'} )^\ast    \big)  
 \otimes \spot^\ctau_\emptyset      \big) \circ {\sf e}
 +\eee \circ  (\colMAP(c_{\SSTT''}) \compose \trid^{\ctau\csigma\ctau}_{\ctau})
 \circ \eee \circ (\colMAP(c_{\SSTU''})^\ast \otimes \spot^\emptyset_\csigma)\circ \eee
\\
 =&
{\sf e}       \circ 
\big(
\big(\colMAP( c_{ \SSTT' }  )       
  \circ 
{\sf e}        \circ
 \colMAP(c_{\SSTU'}  )^\ast    \big)  
 \otimes 
\spot^\ctau_\emptyset 
 \big)  \circ {\sf e}
 \end{align*}
and so the result again follows by induction on length, as above.

 \begin{figure}[ht!]
 $$ \begin{minipage}{3.24cm}
\end{minipage}  $$
\caption{Case 2 for $  \cpi$.}

 \end{figure}

\noindent {\bf Case 3. }  We now consider the case that 
$c_\SSTT = R_\grey  ^+(c_{\SSTT'})$ and 
$c_\SSTU = A_\grey  ^+(c_{\SSTU'})$ (the dual case with $\SSTT$ and $\SSTU$ swapped is identical).  
By  the same inductive argument as we used in Case 1 (in order to deduce \cref{huzzah}), we have that 
\begin{align} \label{huzzah2222}
\eee       \circ \colMAP(c_{\SSTT})       \circ \eee        \circ\colMAP(c_{\SSTU}^*)        \circ \eee 
=\sum _{\SSTX,\SSTY  } a_{\SSTX,\SSTY}\eee
( \colMAP (c_{ \SSTX }^\ast)  \greycompose \fork_{{\colMAP(\grey)}{\colMAP(\grey)}}^{\colMAP(\grey)})
( \colMAP (c_{ \SSTY })\otimes {\sf 1}_{\colMAP(\grey)} )\eee.
\end{align}
We set $\la=\Shape(\SSTX)=\Shape(\SSTY)$.  
%\color{black}
 Observe that $\grey \in \Add(\la)$ or $\Rem(\la)$.  
\color{black}
%If $\grey\in \Add(\la)$, then the final northern $\grey$-strand of $c_\SSTX^\ast$ does not connect to the southern edge of the diagram (because the final grey step in the construction of $c_\SSTX$ was either a spot or a spotted-fork); therefore 
%$$c_\SSTX^\ast\compose \fork_{\grey\grey}^\grey$$
If $\grey\in \Add(\la)$, then 
%the final northern $\grey$-strand of $$(c_\SSTX^\ast) = (X^-_\grey(c_{\SSTX''})^ \ast $$
%for $x\in \{R,A\}$  
%% does not connect to the southern edge of the diagram (because the final grey step in the construction of $c_\SSTX$ was either a spot or a spotted-fork); 
%and therefore 
$$ c_\SSTX^\ast \greycompose \fork_{\grey\grey}^\grey
= (X^-_\grey(c_{\SSTX'})) ^\ast
\greycompose \fork_{\grey\grey}^\grey
=
 (X^+_\grey (c_{\SSTX'})) ^\ast  $$
for $X\in \{R,A\}$ and some $\SSTX'$. We have that 
$  c_\SSTY \otimes \on_\grey=  A^+_\grey (c_{\SSTY}) $.  
In particular, both diagrams are light leaves basis elements.  
If $\grey \in \Rem(\la)$ then we move the fork through the centre of the product $c_{\SSTX\SSTY}$ and notice that 
$$c_\SSTY \greycompose \fork_{\grey\grey}^\grey=R^+_\grey(c_\SSTY)$$
and $c_\SSTX$ are both light leaves basis elements.  
Exactly the same is true replacing $c_{\SSTX}$, $c_{\SSTY}$, and $\alpha$ with their images under $\colMAP$.   
The result follows.

\begin{figure}[ht!]

  $$ 
 \begin{minipage}{4.6cm}
\end{minipage}  
 $$
\caption{Case 3:  a diagrammatic version of \cref{huzzah2222}.}
\label{diagram-huzzah2222}
\end{figure}

\smallskip
\noindent {\bf Case 4. }
 Now suppose  
$c_\SSTT = R_\grey  ^-(c_{\SSTT'})$ and 
$c_\SSTU = A_\grey  ^+(c_{\SSTU'})$.  In this case, we are simply placing a dot on the diagrams from case 3 and so the result follows from case 3 and induction on $\ell(\mu)$.

 \begin{figure}[ht!]
 $$ \begin{minipage}{3.24cm}
 %
\end{minipage} 
 $$
\caption{Case 5 for $  \grey =\csigma$. }
\label{case5}
 \end{figure}

\noindent {\bf Case 5. }
Let 
$c_\SSTT = A_\grey  ^-(c_{\SSTT'})$ and 
$c_\SSTU = A_\grey  ^-(c_{\SSTU'})$. If $\colMAP(\grey)\neq \ctau$ then $\ell_\ctau(\colMAP(\mu))=\ell_\ctau(\colMAP(\mu-\grey))$ and so
\begin{align*}
\eee \circ \colMAP(c_\SSTT) \circ \eee \circ \colMAP(c_{\SSTU})^\ast \circ \eee
&=
(\eee \circ \colMAP(c_{\SSTT'}) \circ \eee \circ \colMAP(c_{\SSTU'})^\ast \circ \eee)
\otimes 
{\sf bar}(\colMAP(\grey))
\\
&= \eee \circ(\on_{\stt_{\colMAP(\nu)}}\otimes {\sf bar}(\colMAP(\grey)) \circ\eee \circ \colMAP(c_{\SSTT'}) \circ \eee \circ \colMAP(c_{\SSTU'})^*\circ \eee
\end{align*}
and if $\colMAP(\grey)=\ctau$ (say $\grey=\crho$ as the $\cpi$ case is identical) we have 
\begin{align}\begin{split}\label{fjkgjkdgjksdfhgkjdshgjkfdhjghfjghfjghfjghfjgh}
\eee \circ \colMAP(c_\SSTT)\circ \eee \circ \colMAP(c_{\SSTU})^\ast \circ \eee
&=
\eee \circ 
({\sf 1}_{\stt_{\colMAP(\nu)}}
\otimes {\sf bar}(\colMAP(\grey))
\circ \eee \circ 
\colMAP(c_{\SSTT'})
\circ \eee \circ 
\colMAP(c_{\SSTU'})^\ast 
\eee
\\
&\quad + \eee \circ 
({\sf 1}_{\stt_{\colMAP(\nu)-\csigma}}
\otimes {\sf gap}(\csigma))
\circ \eee \circ 
\colMAP(c_{\SSTT'})
\circ \eee \circ 
\colMAP(c_{\SSTU'})^\ast 
\eee 
\end{split}
\end{align}
this is picture in \cref{case5}.  
As observed in subcase 1.3, the rules for resolving ${\sf 1}_{\stt_\nu} \otimes {\sf bar}(\grey)$   in type  $(D_{n+1},A_n)$ and 
${\sf 1}_{\stt_{\colMAP(\nu)}}\otimes {\sf bar}(\colMAP(\grey))$ in type 
$(C_n,A_{n-1})$ 
 are identical 
{\em except} when $\colMAP(\grey)=\ctau$ in which case we get an extra term; this term cancels with the second summand on the right of  \cref{fjkgjkdgjksdfhgkjdshgjkfdhjghfjghfjghfjghfjgh}. 
Using  \cref{iszero,isnonzero} versus \cref{ZEROC}
we see that 
$$\eee \circ  \colMAP(c_\SSTT) \circ \eee\circ 
 \colMAP(c_\SSTU) ^\ast \circ \eee
 =
 \eee \circ  \colMAP(c_{\SSTT \SSTU})  \circ \eee$$
by induction on $\ell(\mu)$.

\smallskip
\noindent {\bf Case 6. }
 Let 
$c_\SSTT = R_\grey  ^+ (c_{\SSTT'})$ and 
$c_\SSTU = A_\grey  ^- (c_{\SSTU'})$ (the dual case is similar). 
If $\colMAP(\grey)\neq \ctau$ then $\ell_\ctau(\colMAP(\mu))=\ell_\ctau(\colMAP(\mu-\grey))$ and using the fork-spot relation we have 
$$\eee \circ \colMAP(c_{\SSTT})
\circ \eee \circ 
\colMAP(c_{\SSTU})^\ast 
\circ \eee =
\eee \circ 
\colMAP(c_{\SSTT'})
\circ \eee \circ
\colMAP(c_{\SSTU'})^\ast \eee$$
and so the result follows by induction.  
If $\colMAP(\grey)=\ctau$ then   $ \ell_\ctau(\colMAP(\mu))=\ell_\ctau(\colMAP(\mu-\grey))+1$
and we set $\grey=\crho$ (the $\grey=\cpi$ case is identical)
and we must have $c_\SSTT= R^+_\crho R^- _\csigma(c_{\SSTT''})$.  We have that
\begin{align*}
\colMAP (c_\SSTT) \eee= \colMAP(c_\SSTT)e_{\colMAP(\mu)}
= \colMAP(c_\SSTT) ({\sf 1}_{\stt_{\colMAP(\mu)}}+e_{\colMAP(\mu) 		}^{\ell_\ctau(\colMAP(\mu))})
( e_{\colMAP(\mu)-\ctau} \otimes {\sf 1}_\ctau)
 = \colMAP(c_\SSTT)  
( e_{\colMAP(\mu)-\ctau} \otimes {\sf 1}_\ctau)
\end{align*}
as illustrated in \cref{afigurebelow221}.
So we have 
$$\eee \circ \colMAP(c_\SSTT) \circ e_{\colMAP(\mu)} \circ \colMAP(c_{\SSTU})^\ast \circ \eee
=
\eee
\circ  \colMAP(c_\SSTT)
\circ ( e_{\colMAP(\mu)-\ctau} \otimes {\sf 1}_\ctau) \circ 
 \colMAP(c_{\SSTU})^\ast \circ \eee
 =
\eee
\circ  \colMAP(c_{\SSTT'})
\circ  \eee  \circ 
 \colMAP(c_{\SSTU'})^\ast \circ \eee
$$ using the fork-spot relation as above.  Again we are done by induction.

 \begin{figure}[ht!]
 $$ \begin{minipage}{4.3cm}
\end{minipage} 
 $$
\caption{Case 6 for $  \grey =\csigma$. }
\label{afigurebelow221}
 \end{figure}

\smallskip
\noindent {\bf Case 7. } 
The case  
$c_\SSTT = A_\grey  ^-(c_{\SSTT'})$ and 
$c_\SSTU = R_\grey  ^-(c_{\SSTU'})$ follows from case 6 (in the manner that case 4 followed from case 3).

\smallskip
\noindent {\bf Case 8. }
 We now consider the case that 
$c_\SSTT = R_\grey  ^+(c_{\SSTT'})$ and 
$c_\SSTU = R_\grey  ^+(c_{\SSTU'})$.   
By  the same inductive argument as we used in Case 1 (in order to deduce \cref{huzzah}), we have that 
\begin{align*} 
\eee       \circ \colMAP(c_{\SSTT})       \circ \eee        \circ\colMAP(c_{\SSTU}^*)        \circ \eee 
=\sum _{\SSTX,\SSTY  } a_{\SSTX,\SSTY}\eee
( \colMAP (c_{ \SSTX }^\ast)  \greycompose \fork_{\colMAP(\grey)\colMAP(\grey)}^{\colMAP(\grey)})
( \colMAP (c_{ \SSTY })
 \greycompose
  \fork_{\colMAP(\grey)}^{\colMAP(\grey)\colMAP(\grey)} )\eee.
\end{align*}
We set $\la=\Shape(\SSTX)=\Shape(\SSTY)$.  
Note that either $\grey \in \Add(\la)$ or in $\Rem(\la)$.
 If $\grey \in \Add(\la)$, then 
 arguing as in case 3 we get that 
$$ c_{ \SSTX }^\ast   \otimes \fork_{\grey\grey}^\grey 
\quad
\colMAP(c_{ \SSTX })^\ast   \otimes \fork_{\colMAP(\grey)\colMAP(\grey)}^{\colMAP(\grey)}
\quad
c_{ \SSTY}    \otimes \fork_{\grey\grey}^\grey 
\quad
\colMAP(c_{ \SSTY })    \otimes \fork_{\colMAP(\grey)\colMAP(\grey)}^{\colMAP(\grey)}$$
are   cellular basis elements and so we are done.   
If $\grey \in \Rem(\la)$, then we have $$
 (c_{ \SSTX }^\ast   \greycompose \fork_{\grey\grey}^\grey)
(c_{ \SSTY }  \greycompose  \fork_\grey^{\grey\grey} )
=
 c_{ \SSTX\SSTY }     \greycompose \fork_{\grey\grey} ^\grey  \fork_\grey^{\grey\grey}  =0$$
% and similarly
 $$
 (\colMAP(c_{ \SSTX })^\ast   
 \greycompose 
 \fork_{\colMAP(\grey)\colMAP(\grey)}^{\colMAP(\grey)})
(\colMAP(c_{ \SSTY })
 \greycompose \fork_{\colMAP(\grey)}^{\colMAP(\grey)\colMAP(\grey)} )
=
\colMAP( c_{ \SSTX\SSTY } )    \greycompose \fork_{\colMAP(\grey)\colMAP(\grey)} ^
{\colMAP(\grey)}  \fork_{\colMAP(\grey)}^{\colMAP(\grey)\colMAP(\grey)}  =0$$ 
and  so we are done.  

\smallskip
\noindent {\bf Cases 9 and 10. } 
The case  in which   
$c_\SSTT = R_\grey  ^+(c_{\SSTT'})$ and 
$c_\SSTU = R_\grey  ^-(c_{\SSTU'})$
and the case in which 
$c_\SSTT = R_\grey  ^-(c_{\SSTT'})$ and 
$c_\SSTU = R_\grey  ^-(c_{\SSTU'})$
both 
 follow  from case 8 (in the manner that case 4 followed from case 3).  \end{proof}

  \subsection{Proof of \cref{HSP-comb}.}
Using \cref{C--A}, it is enough to consider the simply laced cases. Now the result follows from \cref{recolourme}.

\section{ Coxeter   truncation   }\label{hom1}

In this section we prove one of the main results of this paper: that $\ctau$-singular Hecke 
categories for Hermitian symmetric pairs 
\color{black!99} (defined \textit{\`{a} la} \cite{MR3502025}) 
\color{black}
are graded Morita equivalent to 
(regular) Hecke categories for smaller rank Hermitian symmetric pairs. 
For the underlying  Kazhdan--Lusztig polynomials, 
this was first observed by Enright--Shelton \cite{MR888703}.  Our result lifts theirs 
  to the 2-categorical level and to positive characteristic.  
  By \Cref{C--A}  we can focus on  the      simply laced case without loss of generality.  
Let $( W,P)$   a simply laced Hermitian symmetric pair 
of rank $n$ and fix $\ctau\in S_W$. We define 
$$\mtau=
\{
\mu \in \mptn \mid \ctau \in \Rem(\mu)
\}.
$$
We will show that the subalgebra of $h_{(W,P)}$ spanned by 
$$ 
\{c_{\SSTS\SSTT}\mid \SSTS \in \Path(\la,\stt_\mu) , \SSTT \in \Path(\la,\stt_\nu), \la,\mu,\nu\in \mtau\}
$$ 
is isomorphic to $h_{(W,P)^\ctau}$ for some  Hermitian symmetric pair 
${(W,P)^\ctau}$ of strictly smaller rank. % $n-2$.  

 \color{black}

\subsection{The $\ctau$-contraction tilings}\label{maponAdm}
In what follows, we let $(W ,P)$ be a  simply laced Hermitian symmetric pair and  $\ctau\in S_ W$.  
We now introduce a contraction map  which will allow us to work by induction on the rank.

 \begin{figure}[ht!]
 
$$
$$
\caption{
The contraction tilings for $  (A_8, A_4\times A_3)^\tau$ 
with $\tau= \color{cyan}s_3$ 
 and  $\tau=\color{brown}s_1$ 
  respectively.  
 The tiling is discussed in  \cref{maponAdm} and the  Dynkin diagrams are discussed in  \cref{maponCox}.} 
 \label{fig1thissec}
\end{figure}% 

%\vspace{-0.2cm}
\begin{figure} [ht!]
$$ % \vspace{-0.1cm}
 \begin{minipage}{5cm} %
\end{minipage}\qquad\quad $$

\vspace{-0.4cm}
\caption{
The contraction tilings   $ (D_6, A_5)^\tau$ for  
${  \tau}= \color{cyan}s_0
 \color{black} , \color{green!80!black}s_1 
 \color{black},   \color{magenta}s_2
 \color{black}\in W$ respectively. 
   In the first  two cases, the grey node is labelled by 
 $ \color{magenta}2\color{green!80!black}1\color{orange}3\color{magenta}2
 \color{cyan}0$ and 
 $ \color{magenta}2\color{cyan}0\color{orange}3\color{magenta}2\color{green!80!black}1$. 
%  In the third case the grey nodes are labelled by
% $ \color{green!80!black}1\color{orange}3\color{magenta}2$ and  $ \color{cyan}0\color{orange}3\color{magenta}2$. 
  }
 \label{fig3thissec}
\end{figure}

\begin{figure}[ht!]
$$
\qquad
$$
\caption{The contraction tilings of $(D_7,D_6)^\tau$ 
for $\tau = 
%\color{pink}s_1
\color{magenta}s_0
\color{black},
\color{green!80!black}s_2
\color{black},
\color{violet}s_3
\color{black},
\color{gray}s_5
\color{black}
$ in order. }
\label{typeD------er}

\end{figure}

   \begin{figure} 
  $$
%
 $$

\caption{The % final 3 of 6 
contraction tilings of $(E_6,D_5)$ in order.}

\label{lots of exceptionsal1}
\end{figure}

 %XXX!

\begin{defn} \label{hghghghg-22}
Given $(W,P)$ a Hermitian symmetric pair and $\ctau\in S_W$, we let  $ {[r_1,c_1]}<
 {[r_2,c_2]}<\dots <  {[r_k,c_k]}$ denote  the completely ordered  set (according to the natural ordering on $r_i+c_i\in \NN$) of all $\ctau$-tiles in $ \mathscr{A}_{(W,P)}$. 
 Given two    $\ctau$-tiles, $ {[r_i,c_i]}$ and ${[r_{i+1},c_{i+1}]}$  that are adjacent in this ordering, we define  
$$T^\ctau_{i\to i+1}= \{[x,y]\mid r_i\leq x \leq r_{i+1} \text{ and } c_i\leq y \leq c_{i+1}\text{ and }[x,y]\neq [r_i,c_i] \} \cap   \mathscr{A}_{(W,P)}.
$$
 Associated to  the minimal tile $ {[r_1,c_1]}$ we define a corresponding    null-region 
$$T^\ctau_{0\to1}=   \{[x,y]\mid   x \leq r_1   \text{ and }   y \leq c_1
%\text{ such that }
%[x,y]\neq [r_1,c_1] 
\}\cap  \mathscr{A}_{(W,P)}$$
and for the  maximal tile $ {[r_k,c_k]}$,   we define the maximal  null-region  
$$T^\ctau_ {   \infty  }= 
  \{[x,y]\mid   r_k  \leq  x  \text{ and }   c_k \leq  y \text{ and } [x,y]\neq [r_k,c_k]\}\cap  \mathscr{A}_{(W,P)} $$ 
and    we set 
  $ \mathcal{N}^\ctau = T^\ctau_{0\to1} \cup  T^\ctau_{\infty}.$ 
   We define the {\sf $\ctau$-contraction  tiling}  
   to be  the disjoint union of the $ {T}^\ctau _{i\to i+1}$-tiles and all remaining tiles in  $ \mathscr{A}_{(W,P)}\setminus  \mathcal{N}^\ctau $.  
We refer to any 
  tile in this overall tiling as a {\sf  $\ctau$-contraction tile}.  
Given   $\mathbb{T}  $    a $\ctau$-contraction tile. We 
define a  {\sf reading word} of $\mathbb{T}$ by 
recording the constituent tiles $ {[r,c]}$ within   $\mathbb{T} $  from bottom to top (that is, by the natural order  on 
$r+c\in \ZZ_{\geq0}$).    \end{defn}

\begin{figure}[ht!]
  
   $$\qquad $$

\caption{The first 4 of 7 contraction tilings  for  $(E_7,E_6)$, in order.  }
\label{hhhh5672}
\end{figure}

\begin{figure}[ht!]
  
   $$\qquad    %
 $$

\caption{The %final 3 of 
final 3  contraction tilings  for  $(E_7,E_6)$, in order.  }
\label{hhhh5671}
\end{figure}

\begin{rmk}
We note  that any 
tile-partition $\la \in \mtau$ can be obtained by stacking 
$\ctau$-contraction tiles on top of $T^\ctau_{0\to1}$.  
\end{rmk}

\begin{eg}
In the leftmost diagram in \cref{fig3thissec} the  large contraction tiles are all identical.  
\color{black!99} These identical  tiles both  have two  distinct choices for their reading word (as we can order the tiles of the same height freely);  explicitly, these  reading words are 
\color{black}
 $ \color{magenta}2\color{pink}1\color{orange}3\color{magenta}2
 \color{cyan}0$ 
 and 
  $ \color{magenta}2 \color{orange}3\color{pink}1\color{magenta}2
 \color{cyan}0$.  Notice that these words differ only by the commuting relations in the Coxeter groups.  
\end{eg}

%We say that a pair of contraction tiles   
% are {\sf neighbouring} if they meet at an edge.  
% Given a pair of neighbouring contraction tiles $X$ and $Y$, we say that  $Y$ {\sf  
%  supports} $X$ if some constituent tile of 
% $X$ appears above some constituent tile of  $Y$.  
%That this ordering is well-defined follows from the  definition of our tiles, or can be seen by inspection.  
% We can now generalise  the ideas from \cref{tileeeeee} concerning addable and removable nodes of a contraction tiling, the ideas of (reduced) standard     tableaux  verbatim.  

%\subsection{The shapes of contraction tiles}
   There is, in essence, only  one type   of large contraction tile: this is  the 
 contraction tiles  of 
type  $(A,A\times A)$ 
 and   their {\em augmentations} pictured in \cref{augment}.
We will see in \cref{hom2} that these augmentations merely ``bulk out" the corresponding Soergel diagram (using degree zero strands) without changing its substance.  
 \color{black!99} In more detail,  we define the {\sf  tricorne} to be  formed from three tiles in a formation $\mathbb{T}=\{ [r,c],  [r-1,c],   [r,c-1]\}$.   \color{black}
  Here the tile $ [r,c]$ is the only $\ctau$-tile in $\mathbb{T}$. 
 We augment this picture by adding $k$ tiles symmetrically above and below the brim (thus displacing the $\ctau$-tile) to obtain an augmented tile
${\rm aug}_k(\mathbb{T} ) $ as depicted in the rightmost diagram of 
 \cref{augment}.

 \begin{figure}[ht!]
$$
% \vspace{-0.2cm} 
%  \begin{minipage}{4cm} \end{minipage}
 $$
\caption{
  The first diagram  depicts the construction of the type $A$ ``tricorne-like" $\ctau$-tile. 
  The second diagram depicts the augmented tiles ${\rm aug}_k(\mathbb{T} ) $ for $k\geq0$; the original tiles   are coloured.  
(The right  diagram is sometimes flipped through the vertical axis.)  }
\label{augment}
\end{figure}

\subsection{The   Dynkin types of $\ctau$-contraction tilings }\label{maponCox}
 We now identify the $\ctau$-tilings of $\mathscr{A}_{(W,P)}$ with the tilings 
of the admissible region of a Hermitian symmetric pair of smaller rank.  
The nodes of $(W,P)^\ctau$ will be  labelled by the reading words of the 
tiles in the $\ctau$-contraction tiling.

 \begin{prop} 
 Let $(W,P)$ be a simply laced Hermitian symmetric pair   and let $\ctau \in  S_W$.
 There is an order preserving bijection 
 $\TRNC_\ctau: 		\mathscr{P}_{(W,P)^\ctau} \to 	\mtau$ 
 where   $(W,P)^\ctau=(W^\ctau,P^\ctau)$ is defined by 
 \begin{itemize}[leftmargin=*]
   \item $(A_n,A_{k}\times A_{n-k})^\ctau = (A_{n-2},A_{k-1}\times A_{n-k-1})$; 
\item $(D_n, A_{n-1})^\ctau = (D_{n-2},A_{n-3})$; 
\item $(D_n, D_{n-1})^\ctau = (A_1,A_0)$; 
  \item $(E_6,D_5)^\ctau = (A_{5}, A_{4})$;
     \item $(E_7,E_6)^\ctau = (D_{6}, D_{5})$.

 \end{itemize} 
 Moreover, fixing a reading word for each $\ctau$-contraction tile, this defines a reduced path 
 $\TRNC_\ctau(\stt)\in \Std(\TRNC_\ctau(\la))$ for each $\stt \in \Std(\la)$ and $\la \in \mathscr{P}_{(W,P)^\ctau}$.  
 \end{prop}

\begin{proof}
This follows by inspection, comparing \cref{fig1thissec,fig3thissec,lots of exceptionsal1,hhhh5672,hhhh5671,typeD------er} with \cref{typeAtiling,pm1,pm2,hhhh}.  
\end{proof}

We label the nodes of the smaller rank Coxeter system with the reading words of the $\ctau$-contraction tiles of the larger Coxeter system.
The positions of  these labels can  easily be deduced   from the $\ctau$-tilings, see for example \cref{coxeterlabelA22,coxeterlabelB222,typeDexceptional}.

  \color{black}

    \begin{figure}[ht!]\scalefont{0.9}


 \caption{An example of a graph and its contraction at the   vertices  $3$ and $5$ respectively.   }
\label{coxeterlabelA22}
\end{figure}

    \begin{figure}[ht!]
%

 \caption{An example of a graph and its contractions at    vertices $3$ and 5, respectively.    }
\label{coxeterlabelB222}
\end{figure}

\begin{figure}[ht!]

\begin{minipage}{4.8cm}
 %
 \end{minipage}

 \caption{The graphs obtained from the leftmost graph in \cref{coxeterlabelB222} by contraction at the type $D$ vertices $\nodelabel    $, $1$  and $2$, respectively.    }
\label{typeDexceptional}
\end{figure}

%
%\subsection{The graded bijection for walks in Bruhat gaphs }
% We recall that 
%a reduced path of length $n$ in a Bruhat graph can be identified with a tableau $\stt: A_{(W,P)}\mapsto \{1,\dots,n\}$ by traversing the edges according to  the  order on the tiles.  
%We can identify a tableau  $\stt:A_{(W,P)-\setminus \ctau}   \mapsto \{1,\dots,n\}$ 
%with a ``semistandard" tableau of shape 
%$\stt:A_{(W,P) } \to \mapsto \{1,\dots,n\}$ 
%by simply putting the same entry in each constituent tile of the larger $\ctau$-tile.  
%Given a reduced word $\w\in (W,P)^\ctau$, we let  $\TRNC_\ctau(\w)$ denote any  reduced word in $ (W,P) $ obtained from this semistandard tableau. 
In fact, it can be shown that the map $\TRNC_\ctau$ can be extended from reduced paths (where the observation is trivial) to {\em non-reduced} paths.  The proof of this  involves much more substantial combinatorics and so is postponed to the companion paper, \cite{FARR}.

\begin{prop}[{\cite[Proposition 8.1]{FARR}}]\label{BIJECTION}
We have a graded bijection
$$
\TRNC_\ctau:{\rm Path}_{(W,P)^\ctau }(\la, \stt_\mu )
\xrightarrow{ \ \ \ }
{\rm Path}_{(W,P)  }(\TRNC_\ctau (\la), \stt_{\TRNC_\ctau (\mu)}).  
$$
\end{prop}

%\begin{proof}
%For up or down steps in the Bruhat graph, this follows from immediately from the construction of the tilings.  One can then check the up-down and down-up steps in each case.  
%\end{proof}

 \begin{eg}
 In \cref{coxeterlabelA22} we depict 
the pair $(A_5 ,  A_2 \times A_2  )$ and the truncation at the node $\ctau = \color{cyan}  s_3
 \color{black} \in S_W$.  
We have that $\TRNC_\ctau(\color{magenta} 2\color{orange} 4\color{cyan} 3\color{black} )= \color{magenta} 2\color{black} \otimes \color{orange} 4\color{black} \otimes \color{cyan} 3\color{black} $.  
On the right of  \cref{coxeterlabelA22NOW} we depict  the Bruhat graph  for   $(A_5\setminus A_2 \times A_2)^\ctau $.  The bottommost (and topmost) edge  of this graph is tricoloured  by  
$ \color{magenta} 2\color{black} \otimes \color{orange} 4\color{black} \otimes \color{cyan} 3\color{black} $ and   maps to a concatenate of three distinct edges in the leftmost graph.  
 
\end{eg}

\subsection{The dilation homomorphism    }\label{hom2}
We now lift the map $\TRNC_\ctau$ of \cref{BIJECTION} to the level of a graded $\Bbbk$-algebra isomorphism.     We let $i$ denote a square root of $-1$.  
 We first define the dilation maps on the monoidal generators.  
Let $ 
\gsigma=s_1 s_2 \dots s_\ell$ be a (composite) label of a node of the Coxeter graph $(W,P)^\ctau$.  
(Note that $s_1, s_2, \dots ,s_\ell \in S_W$ belong to the dilated Coxeter group, whereas $\gsigma \in S_{W^\ctau}$.)  
We define the dilation map on the idempotent generators  as follows
$$
\dil    ({\sf 1}_{\gsigma})
=
{\sf 1}_\ctau \otimes {\sf 1}_{s_1}\otimes {\sf 1}_{ s_2}
\otimes  \dots \otimes {\sf 1}_{s_\ell}.
$$For a non-zero braid generator (see \cref{zerobraidgenerator} for the list of zero braid generators) we define the dilation map as follows
$$
\dil    ({\sf braid}_{\gsigma \alphar}^{  \alphar \gsigma} )
=
{\sf 1}_\ctau \otimes {\sf braid}_{s_1 s_2   \dots s_\ell \alphar}^{\alphar s_1 s_2   \dots s_\ell}.
$$Examples are depicted in \cref{gfsjklsdfgjkfgjkdfgsjnmasdljkdsf}.

\begin{figure}[ht!]   
$$
\begin{minipage}{1.3cm}
\end{minipage}
   $$
 
 \caption{Examples of the  $\dil    $ map on   idempotent and braid generators.  The colouring corresponds to that  of \cref{fig1thissec}.   
 The leftmost $\ctau$-strand is drawn as a
  dotted strand  in order to remind the reader that   we horizontally concatenate these diagrams using $\compose$ (which identifies this blue strand with an earlier blue strand in the diagram).  
}
 \label{gfsjklsdfgjkfgjkdfgsjnmasdljkdsf}
 \end{figure}

We now define the dilation map on the fork and spot generators.
 For $\gsigma=s_i  $ a  tile labelled by a singleton $s_i\in S_{W^\ctau}$ (which necessarily commutes with $\ctau\in S_W)$ we define
 $$
 \dil    ({\sf fork}_{\gsigma\gsigma}^\gsigma)=
{\sf 1}_\ctau \otimes  {\sf fork}_{s_i   s_i  }^{s_i  }
\qquad
 \dil    ({\sf spot}_{\gsigma }^\emptyset)=
{\sf 1}_\ctau \otimes  {\sf spot}_{s_i    }^{\emptyset}
. $$We set  $\gsigma  =\alphar   \gam {{\color{cyan}\tau}} $, the reading word of  a tricorne tile.  
  We  define  
$$\dil      ({{\sf fork}}^\gsigma_{{\gsigma\gsigma}})=
i \times {\sf 1}_\ctau \otimes   (  
 {{\sf fork}}_{\alphar\alphar}^{\alphar}
\otimes   
 {{\sf fork}}_{\gam \gam   }^{\gam   }
\otimes  {\sf 1} _{ {{\color{cyan}\tau}} }
)
({\sf1}_{\alphar}\otimes {\sf  braid}^{\alphar\gam}_{\gam\alphar}\otimes {\sf1}_{  \gam {{\color{cyan}\tau}}} )
({\sf1}_{\alphar\gam } \otimes {\sf spot}_{{\color{cyan}\tau}}^\emptyset  \otimes {\sf1}_{ \alphar \gam {{\color{cyan}\tau}}} )
  $$
%  and we define   
  $$ \dil      (   {\sf spot}^\emptyset_{ \gsigma})=-i \times 
  {\sf fork}_{\blue\blue}^\blue  
  (
  {\sf 1}_\blue \otimes {\sf spot}^{ \emptyset }_{  \alphar    }
  \otimes {\sf spot}^{ \emptyset }_{  \gam   } \otimes {\sf 1}_\ctau  ).
$$ 
 Examples are depicted in \cref{popfogfpogfpgofpgofpgf}.

%\tikzstyle over=[draw=orange,double=magenta,line width=2pt, double distance=.4pt]
 \begin{figure}[ht!]
$$  \begin{minipage}{2.5cm}
\end{minipage}$$

\caption{Let $(W,P)=(A_5,A_2 \times A_2)$ and $\ctau=\color{cyan}s_3\color{black}\in W$ (see also \cref{coxeterlabelA22,coxeterlabelA22NOW}).   
We depict $\TRNC_\ctau ({\sf fork}_{{\color{gray}\sigma \sigma}}^{{\color{gray}\sigma}})$ and $\TRNC_\ctau ({\sf spot}_{{\color{gray}\sigma}}^\emptyset )$ for $\gsigma =
\color{magenta} s_2\color{orange} s_4\color{cyan} s_3\color{black} $.   
The leftmost $\ctau$-strand is drawn as a
  dotted strand  in order to remind the reader that   we horizontally concatenate these diagrams using $\compose$ (which identifies this blue strand with an earlier blue strand in the diagram).  
  }
  \label{popfogfpogfpgofpgofpgf}
\end{figure}

We now describe how one can ``augment" the diagrams of tricornes  to obtain arbitrary diagrams.    
Let ${{\color{gray}\sigma\color{black}}}  $ be a   label of  a vertex  in the graph  $(W ,P )^{\ctau} $ 
of the form
 $\color{gray}\sigma\color{black}=
  \x 
\color{magenta} \alphar 
\color{orange}	\gam 
\color{black}	 \x^{-1}
\color{cyan}\ctau $.  (That is, ${\gsigma}  $ is the reading word of an augmented tricorne.)  
  We  define  $ \dil       ({{\sf fork}}^{\gsigma}_{{{\gsigma}{\gsigma}}})$ to be the element 
$$
i^{\ell(x)+1} (
%\varphi_\ctau  ({{\sf fork}}^{\gsigma}_{{{\gsigma}{\gsigma}}})=  
 {\sf 1}_{ \ctau\x }  \otimes  (
 {{\sf fork}}_{{\alphar }{\alphar }\gam \gam }^{{\alphar\gam }}
  {\sf  braid}^{  \alphar \alphar \gam  \gam  }_{  \alphar  \gam  \alphar \gam }
({\sf1}_{     \alphar  \gam  }
 \otimes
  {\sf cap} _{{\x \x ^{-1}}}^{ \emptyset } 
  \otimes  {\sf 1} _{   \alphar  \gam } 
  )
({\sf1}_{  \alphar \gam  \x ^{-1} }
 \otimes
  {\sf spot}_	{ \ctau   }	^\emptyset \otimes  {\sf 1} _{  \x   \alphar \gam}
 )  )  \otimes  {\sf 1} _{  \x  ^{-1}  \ctau          })
  $$
  and we define $\dil      (   {\sf spot}^\emptyset_{{\gsigma} })$ to be the element 
     $$(-i)^{\ell(x)+1}  
{\sf fork}_{\blue\blue}^\blue( {\sf 1}_\ctau\otimes 
%{\sf spot}_{\x} ^\emptyset 
% {\sf fork}_{\x \x} ^{\x} 
 {\sf cap} _{{\x \x ^{-1}}}^{ \emptyset } 
    \otimes {\sf 1}
   _{
  {\ctau     }}  )
 ( {\sf 1}_{
   {\ctau     }{\x  }} \otimes {\sf spot}^{\emptyset }_{ {\alphar }    } \otimes {\sf spot}^{\emptyset }_{ {\gam}    }   \otimes {\sf 1}_{
{\x^{-1} }  {\ctau     }} ).
$$ Examples  are depicted in \cref{gfkjhdsfjkghsdjkgflhsdljghsdjlkfghsdjfklhgljksdfhgjklsdfhgljsdkf}.

% \vspace{-0.1cm}
 \begin{figure}[ht!]
 
 $$ %i \times 
% \vspace{-0.1cm}
 \qquad  \begin{minipage}{7.1cm}

\end{minipage}$$

 \caption{
 Let $(W,P)=(D_6,A_5)$ and $\ctau=\color{cyan}s_0\color{black}\in W$ (see also \cref{coxeterlabelB222,typeDexceptional}).   
We depict $\dil     ({\sf fork}_{{\color{gray}\sigma \sigma}}^{{\color{gray}\sigma}})$ and $ \dil     ({\sf spot}_{{\color{gray}\sigma}}^\emptyset )$ for ${\gsigma}
 =
 \color{green!80!black} s_2
\color{magenta} s_{1}
\color{orange}	s_3
\color{green!80!black} s_2
\color{cyan}s_{\nodelabel    }
 $.   
These diagrams are obtained from \cref{popfogfpogfpgofpgofpgf} by adding strands of degree zero. }
\label{gfkjhdsfjkghsdjkgflhsdljghsdjlkfghsdjfklhgljksdfhgjklsdfhgljsdkf}
 \end{figure}

Having defined $\dil$ on all Soergel generators, we set 
$  \dil     (D^*)=(  \dil     (D))^*$.  We now extend this definition to arbitrary Soergel diagrams and hence define our contraction homomorphisms.
 
 \begin{defn}\label{theactualdefn}
  Given  diagrams 
  $D_1, D_2
  \in \mathcal{H}_{(W,P)}$, 
  we inductively define 
  $$\dil     (D_1\otimes D_2)=
  \dil     (D_1)\compose \dil     (D_2)
  \qquad
  \dil     (D_1\circ D_2)=
  \dil     (D_1)\circ \dil     (D_2)  $$
and we extend this map $\Bbbk$-linearly.  
We hence define  $\varphi_\ctau:\mathcal{H}_{(W,P)^\ctau}\hookrightarrow
 \mathcal{H}_{(W,P)}$  as follows,
 $$
 \varphi_\ctau(D)
 =
 {\sf 1}_{T_{0\to1}} \compose \dil    (D )  
 $$   
where we recall that ${T_{0\to1}}$ is the null region at the bottom of the $\ctau$-contraction tiling of $(W,P)$.  
   
\end{defn}
 
% \vspace{-0.25cm}
 \begin{figure}[ht!]
$$
 \varphi_\ctau  \left(\; \begin{minipage}{1.9cm}

\end{minipage}  
$$
\caption{The map $\varphi_\ctau$ on a diagram, for the leftmost contraction tiling   in \cref{fig1thissec}.  }
\label{illustrative}
\end{figure}

\begin{rmk}
Each null-region tile $T_{0\to1}$ has a unique reading word and so there is no ambiguity here.  
That the map $\dil    $ is  well-defined  on diagrams follows from the  
interchange law.  
\end{rmk}

%   \subsection{An isomorphism of graded vector spaces}
 
   The map $ \varphi_\ctau$ preserves  the  light leaves basis (because our map is defined on monoidal generators)  and thus lifts  the map of   \cref{BIJECTION} to an isomorphism of  graded $\Bbbk$-modules
between 
$$ 
h_{(W,P)^\ctau}=\Bbbk\{c_{\SSTS\SSTT}\mid \SSTS \in \Path(\la,\stt_\mu), \SSTT \in \Path(\la,\stt_\nu) , \la,\mu,\nu\in \mathscr{P}_{(W,P)^\ctau}\}
$$ 
and $h_{(W,P)}^\ctau \subset h_{(W,P)}$ which we define to be the subspace with basis 
$$  
\Bbbk\{c_{\SSTS\SSTT}\mid \SSTS \in \Path(\TRNC_\ctau(\la),
\stt_{ \TRNC_\ctau(\mu) }) ,
\SSTT \in \Path(\TRNC_\ctau(\la),
\stt_{\TRNC_\ctau( \nu )}) ,
 \TRNC_\ctau(\la),\TRNC_\ctau(\mu),\TRNC_\ctau(\nu)\in \mtau\} .
$$    
In fact we will now lift this  to the level of graded $\Bbbk$-algebras.

 \begin{thm}   \label{veccer}
Let $(W,P)$ be a Hermitian symmetric pair and $\ctau \in S_W$.  
 We have  a graded $\Bbbk$-algebra isomorphism $\varphi_\ctau(h_{(W,P)^\ctau})\cong 
 h_{(W,P) }^\ctau$. 
 \end{thm}
 
%  The next section is dedicated to the proof of this theorem.

 \section{Proof of the Coxeter dilation homomorphism}
\label{hom3}
 
This section is dedicated to the proof that the map of $\varphi_\ctau$  is a homomorphism.   
This amounts to checking the relations for these algebras.  
By \cref{theactualdefn}, we have that 
\begin{align}\label{ppr1}
\varphi_\ctau( D)
\varphi_\ctau( D') 
&=
( {\sf 1}_{T_{0\to 1}}\compose \dil     ( D ))
( {\sf 1}_{T_{0\to 1}}\compose \dil     ( D '))
 = {\sf 1}_{T_{0\to 1}}\compose\dil     ( D\circ D') 
\end{align}
using the interchange law and 
moreover 
\begin{align}\label{ppr2}
\varphi_\ctau( D\circ D') = {\sf 1}_{T_{0\to 1}}\compose \dil     ( D\circ D') .
\end{align}
Therefore for the local relations,  it suffices to show that 
\begin{align}\label{thisisit}
 \dil     ( D\circ D') 
 =
 \dil     ( D ) \circ
 \dil     (   D') . \end{align}
Most of this section is  dedicated to  the proof that  
 the relations of \cref{inlightof} 
 are preserved under \cref{thisisit} (but replacing $\otimes $ with $\compose$).    
For the non-local relations, we check that 
 \cref{ppr1,ppr2} coincide at the end of the section.  
We now  turn to the local relations.  Relations \ref{R1} and \ref{R2},    and  \ref{R9}  are all trivial.    Relation \ref{R8} is satisfied by \cref{interchange}.

In what follows, we  let $\gsigma $ be a reading word of some $\ctau$-contraction tile. 
In the case that  $\gsigma$ is 
an augmented tricorne $\gsigma= \vx\alphar \gam \vx^{-1} \ctau $, we set 
 and  $\vx =\color{violet}\vsk\dots \vsone$.  In diagrams we 
put a gradient which reflects the ordering of the purple strands (with 
 $\color{violet!65!white}k$ the lightest and 
$\color{violet}1$ the darkest). 
We label strands in a diagram simply by $1\leq j \leq k$ (rather than by $\vsj$) for brevity.  
 
\subsection{\bf The dilated fork-spot relation} 
We first consider the leftmost relation in \ref{R3}, namely the fork-spot relation 
\begin{align}\label{FS1}
(
\dil   ({\sf 1}_{{\color{gray}\sigma}} ) 
\compose
\dil   ( {\sf spot}^\emptyset_{{\color{gray}\sigma}} )
) \circ 
 \dil   ({\sf fork} ^{{{\color{gray}\sigma}}{{\color{gray}\sigma}}}_{{\color{gray}\sigma}} ) 
= \dil     (({\sf 1}_{{\color{gray}\sigma}} )
%\qquad %\label{FS2}
%=( \dil     (
%   {\sf spot}^\emptyset_{{\color{gray}\sigma}} )\compose 
%\dil   ({\sf 1}_{{\color{gray}\sigma}} )
%)
%\dil    ({\sf fork} ^{{{\color{gray}\sigma}}{{\color{gray}\sigma}}}_{{\color{gray}\sigma}}  ).
 \end{align}
 For $\gsigma=s_i\in S_W$, it follows  trivially.   
For $\gsigma= \vx\alphar \gam \vx^{-1} \ctau $, 
the equality follows by  
one application of each of the $\alphar$- $\gam$- and $\ctau$-fork-spot contractions
 and   the $\alphar\gam$-commutativity relation (and monoidal unit relation);
and by ``straightening out" the $\vx$-strands via application of the fork-spot and double-fork relations (this is sometimes referred to simply as ``isotopy" in the literature).   
For a tricorne, this relation is depicted in \cref{forkspotproof}.  
For an augmented tricorne, one side of this relation is depicted in \cref{forkspotproof2} and it is easy to see that the argument goes through unchanged. 
The argument for the horizontal and vertical flips of \cref{FS1} is similar.

\begin{figure}[ht!]
$$\scalefont{0.8}
\begin{minipage}{3.75cm}

 \end{minipage} $$
   \caption{
The   dilation of the 
 fork-spot relation and its   flip through the vertical axis, for a tricorne $\gsigma=\alphar\gam\ctau$. 
(We note that the scalar coefficient for both these products is $i \times -i=1$.)
 }
 \label{forkspotproof}
\end{figure}

 \begin{figure}[ht!]
\scalefont{0.8}   %
\caption{
The lefthand-side of the dilated fork-spot relation for an augmented tricorne
$\gsigma= \vx\alphar \gam \vx^{-1} \ctau $. 
(We note that the scalar coefficient for both these products is $i \times -i=1$ or $1\times 1=1$ depending on the parity of $k\geq 1$.)
 }
 \label{forkspotproof2}
\end{figure}

\subsection{\bf The dilated double-fork relation}   
  We now consider the rightmost relation in \ref{R3}, namely, the double-fork relation
$$
(\dil     ({\sf 1}_{{\gsigma}}) \compose 
\dil ( {\sf fork}^{{{\gsigma}}}_{{{\gsigma}}{{\gsigma}}} )
)\circ 
(\dil ({\sf fork} ^{{{\gsigma}}{{\gsigma}}}_{{\gsigma}} )\compose
\dil ( {\sf 1}_{{\gsigma}}))
= \dil    ( 	{\sf fork} ^{{{\gsigma}}{{\gsigma}}}_{{\gsigma}}
)
\circ \dil (  {\sf fork}^{{{\gsigma}}}_{{{\gsigma}}{{\gsigma}}}  	).
$$
We apply the double-fork relation   to   
every constituent doubly-forked strand   in the   diagram in turn,
 and the result follows.  
See \cref{2forkA} for the corresponding picture for tricornes,
 the augmented tricorne picture can be obtained in a similar fashion 
 to \cref{forkspotproof2}. % the above by adding purple strands.   

\begin{figure}[ht!]

$$-\;\begin{minipage}{4.625cm}
 \end{minipage}$$
 \caption{The dilated double-fork relation for   $\gsigma=\alphar  \gam \ctau $ the reading word of a tricorne. The equality follows     by applying the double-fork relation to the $\alphar$- and $\gam$-strands. 
 }
 \label{2forkA}
\end{figure}

\subsection{\bf The dilated circle annihilation relation} 
 
We now verify the leftmost relation in \ref{R4}, namely, the circle-annihilation relation
\begin{equation}\label{circleannih}
\dil 
({\sf fork} _{{{\gsigma}}{{\gsigma}}}^{{\gsigma}})
\dil(  
{\sf fork} ^{{{\gsigma}}{{\gsigma}}}_{{\gsigma}}  ) 
=0 .
\end{equation} 
  For a tricorne  $\gsigma\color{black}=\alphar  \gam \ctau $ we
  have that $\dil 
({\sf fork} _{{{\gsigma}}{{\gsigma}}}^{{\gsigma}}  )
\dil (
{\sf fork} ^{{{\gsigma}}{{\gsigma}}}_{{\gsigma}}  ) 
$ is equal to 
  \begin{align}\label{dkjgfhsdflkjghsldiuhglsdiuhdurygtjksdhg}
   &  - 
   {\sf 1}_{\ctau} \otimes 
{\sf fork}^ { \alphar\gam } _{\alphar\alphar \gam\gam }
{\sf braid}^{ \alphar \alphar\gam\gam }_{ \alphar\gam\alphar \gam }
 ({\sf 1}_{ \alphar\gam}\otimes {\sf bar}(\ctau) \otimes {\sf 1}_{\alphar\gam })
 {\sf braid}_{ \alphar \alphar\gam\gam }^{ \alphar\gam\alphar \gam }
{\sf fork}^{ \alphar\alphar \gam\gam }_{ \alphar\gam })\otimes   {\sf 1}_{\ctau} 
\end{align}
by definition (note $(-i)^2=-1$).  Applying \cref{useful1} to the $   \ctau    \gam $-strands in \cref{dkjgfhsdflkjghsldiuhglsdiuhdurygtjksdhg} 
we  obtain 
\begin{align*}
  &  - {\sf 1}_\ctau\otimes 
{\sf fork}^ {\alphar\gam} _{\alphar\alphar \gam\gam}
{\sf braid}^{ \alphar \alphar\gam\gam }_{\alphar\gam\alphar \gam}
 ({\sf 1}_{\alphar }\otimes {\sf bar}(\ctau) \otimes {\sf 1}_{\gam\alphar\gam })
 {\sf braid}_{ \alphar \alphar\gam\gam }^{\alphar\gam\alphar \gam}
{\sf fork}^{\alphar\alphar \gam\gam}_{\alphar\gam}  \otimes 	{\sf 1}_\ctau
\\ &- {\sf 1}_\ctau\otimes 
{\sf fork}^ {\alphar\gam} _{\alphar\alphar \gam\gam}
{\sf braid}^{ \alphar \alphar\gam\gam }_{\alphar\gam\alphar \gam}
 ({\sf 1}_{\alphar  }\otimes {\sf bar}(\gam) \otimes {\sf 1}_{\gam \alphar\gam })
 {\sf braid}_{ \alphar \alphar\gam\gam }^{\alphar\gam\alphar \gam}
{\sf fork}^{\alphar\alphar \gam\gam}_{\alphar\gam}  \otimes 	{\sf 1}_\ctau
\\ &+{\sf 1}_\ctau\otimes 
{\sf fork}^ {\alphar\gam} _{\alphar\alphar \gam\gam}
{\sf braid}^{ \alphar \alphar\gam\gam }_{\alphar\gam\alphar \gam}
 ({\sf 1}_{\alphar  }\otimes {\sf gap}(\gam) \otimes {\sf 1}_{ \alphar\gam })
 {\sf braid}_{ \alphar \alphar\gam\gam }^{\alphar\gam\alphar \gam}
{\sf fork}^{\alphar\alphar \gam\gam}_{\alphar\gam}    \otimes 	{\sf 1}_\ctau
\end{align*}  
 (these three   terms are  depicted in \cref{forkdie1}).  
Now,  the first term   is   zero by the  $\alphar\gam$-commutativity relation and 
 the $\gam$-circle annihilation relation.  
The second and third  terms  are   zero by the  $\alphar\gam$-commutativity relation and the 
$\alphar$-circle annihilation relation.

\begin{figure}[ht!]
  $$\begin{minipage}{3.15cm}

\end{minipage}$$
\caption{
The circle annihilation relation for $\gsigma\color{black}=\alphar  \gam \ctau $. 
 We apply the $\gam\ctau$-barbell relation to the lefthand-side.  
The first term (respectively latter two terms) on the righthand-side is zero by 
the circle annihilation relation for 
$\color{orange} \gam$ (respectively $\alphar$) and the $\pink\orange$-commutativity relations.   }
 \label{forkdie1}
\end{figure}

We now consider the case of an augmented tricorne $\gsigma= \vx\alphar \gam \vx^{-1} \ctau $, with  $\vx =\color{violet} \vsk \dots  \vsone  $. 
 The diagram $\dil 
({\sf fork} _{{{\gsigma}}{{\gsigma}}}^{{\gsigma}})
\dil(  
{\sf fork} ^{{{\gsigma}}{{\gsigma}}}_{{\gsigma}}  ) 
$  has a $\ctau$-barbell in the centre of $k$ concentric circles 
with the innermost circle labelled by $ \vsk $ and the outermost labelled by  $ \vsone  $ (as pictured in the diagram on the lefthand-side of \cref{forkdie}).  
We pull this barbell through these $k$ circles using $k$ applications of \cref{repeat12} and hence obtain 
$$ 
  {\sf cap}^\emptyset
_{\color{violet}\x\x^{-1}} 
({\sf 1}_{\color{violet}  \x}
\color{black}\otimes {\sf bar}(\ctau)
\color{black}\otimes {\sf 1}_{\color{violet}\x^{\color{black}-1} })
 {\sf cup}_\emptyset
^{\color{violet}\x\x^{-1}} 
 =(-1)^{k  }{\sf bar} ({\color{violet}\bm  \vsone  }).
$$
 We therefore have that $ \dil 
({\sf fork} _{{{\gsigma}}{{\gsigma}}}^{{\gsigma}})
\dil(  
{\sf fork} ^{{{\gsigma}}{{\gsigma}}}_{{\gsigma}}  ) 
$ is equal to 
 $$  
 {\sf 1}_{\ctau}\otimes {\sf 1}_{\x }
\otimes (
%{\sf fork}^\csigma_{\pink\pink}\fork^\orange_{\orange\orange}
{\sf fork}_{\pink\pink \orange\orange}^{\pink\orange}
{\sf braid}^{\pink\pink\orange\orange}_
{\pink\orange\pink\orange}
({\sf 1}_{\pink\orange}\otimes  {\sf bar}({\color{violet}\bm  \vsone  })
 \otimes {\sf1}_{ \pink\orange})
{\sf braid}_{\pink\pink\orange\orange}^
{\pink\orange\pink\orange}
%{\sf fork}_\csigma^{\pink\pink}\fork_\orange^{\orange\orange}
{\sf fork}^{\pink\pink \orange\orange}_{\pink\orange}
)\otimes 
{\sf 1}_{ \x^{-1}}\otimes {\sf 1}_{\ctau }.
$$
We can now apply the $\orange   \vsone  $-barbell relation and show  that the three resulting terms are zero exactly as in the case of the tricorne, above.  (See also \cref{forkdie}.)

\definecolor{lightpurp}{HTML}{cdb4db}

 \begin{figure}[ht!]  
$$\scalefont{0.8}
 \begin{minipage}{7.6cm}
\end{minipage}
$$
\caption{
 Simplifying the lefthand-side of the  circle annihilation relation  (\cref{circleannih})
 using  \cref{repeat12}.    
Compare  the righthand-side of the equation  above  
 with the 
 lefthand-side of the equation pictured in \cref{forkdie1}.     }
 \label{forkdie}
\end{figure}

\subsection{\bf The dilated null-braid relations}  \label{asnited}
 Let $\betar \in S_{W^\ctau}$ and $m(\gsigma,\betar)=3$.  By inspection of 
 \cref{fig1thissec,fig3thissec,lots of exceptionsal1,hhhh5672,hhhh5671,typeD------er}, we see that $\betar$ must be a singleton label and either 
$(i)$  $m(\betar,\alphar)=3$, $m(\betar,\gam)=m(\beta,\ctau)=m(\betar,\vsi)=2$ for all $1\leq i \leq k$
$(ii)$ $m(\betar,\gam)=3$, $m(\betar,\alphar)=m(\beta,\ctau)=m(\betar,\vsi)=2$ for all $1\leq i \leq k$.
 We assume without loss of generality that $m(\betar,\alphar)=3$.   
 We must prove that 
  \begin{align}\label{thrittt}
\dil   (  {\sf 1}_{\green \gsigma \green})
 &=
-  \dil   ( 
   ( {\sf spot}^{\betar\gsigma\betar}_{\betar\emptyset  \betar} 
   )
   \dil   ( 
{\sf dork}_{\betar\betar}^{\betar\betar})
\dil   ( 
  {\sf spot}_{\betar\gsigma\betar}^{\betar\emptyset  \betar} ) ,
\\ 
\label{thrittt2}
\dil   (  {\sf 1}_{\gsigma \green\gsigma})
 &=
  -\dil   ( 
     {\sf spot}^{\gsigma\green\gsigma}_{\gsigma\emptyset  \gsigma} )
     \dil   ( 
{\sf dork}_{\gsigma\gsigma}^{\gsigma\gsigma})\dil   ( 
  {\sf spot}_{\gsigma\green\gsigma}^{\gsigma\emptyset  \gsigma} )  .  
\end{align}

\definecolor{lightpurp}{HTML}{cdb4db}

 \begin{figure}[ht!]
 $$\scalefont{0.8}
 \begin{minipage}{6.5cm}
\end{minipage}$$
\caption{The six steps in proving \cref{thrittt}.  Read from left to right, one row at a time.}
\label{thritttpic}
\end{figure}

We first prove \cref{thrittt}.  We first apply the commutativity relations to 
the two $\green$-strands in 
 $\dil   (  {\sf 1}_{\green \gsigma \green})$ in order to 
 bring them as close to the $\pink$-strand as possible (to obtain  the top-right diagram of \cref{thritttpic}) and we then apply the $\pink\green$-null-braid 
  (to obtain $-1$ times the middle-left diagram of \cref{thritttpic}).  
  We then apply the $\orange \vsone$-null-braid 
    (to obtain  the middle-right diagram of \cref{thritttpic})
followed by the $\vsi \vsione$-null-braids for $1\leq i <k$ 
in turn 
    (to obtain $(-1)^{k+1}$ times the bottom-left diagram of \cref{thritttpic}).
Finally, we apply the $\ctau\vsk$-null-braid 
    (to obtain $ (-1)^{k+2}$ times the bottom-right diagram of \cref{thritttpic})
    and hence obtain $-  \dil   ( 
   ( {\sf spot}^{\betar\gsigma\betar}_{\betar\emptyset  \betar} 
   )
   \dil   ( 
{\sf dork}_{\betar\betar}^{\betar\betar})
\dil   ( 
  {\sf spot}_{\betar\gsigma\betar}^{\betar\emptyset  \betar} ) 
$ as required.

\begin{figure}[ht!]
 $$\scalefont{0.8}
 
  $$
  \caption{The three    steps in proving \cref{thrittt2}.   }
\label{thrittt2pic}
\end{figure}

We now prove \cref{thrittt2} in a similar fashion.  
We first apply 
the $ \ctau\vsk $-null-braid relation to 
 $\dil   (  {\sf 1}_{\green \gsigma \green})$
 followed by the $ \vsi\vsione$-null-braid relations for $k> i \geq 1$ 
  (to obtain $(-1)^{k}$ times  the second diagram of \cref{thrittt2pic}).
  We then apply the $\orange\green$-null-braid 
and    $ \alphar\vsone$-null-braid relations 
  (to obtain $(-1)^{k}$ times  the third diagram of \cref{thrittt2pic}) and hence obtain $  -\dil   ( 
     {\sf spot}^{\gsigma\green\gsigma}_{\gsigma\emptyset  \gsigma} )
     \dil   ( 
{\sf dork}_{\gsigma\gsigma}^{\gsigma\gsigma})\dil   ( 
  {\sf spot}_{\gsigma\green\gsigma}^{\gsigma\emptyset  \gsigma} )$  as required.

\subsection{The dilated  barbell relations}\label{onebarbell}
  We now consider the one and two colour barbell relations.    

\begin{lem}\label{substitute} 
For  $ \alphar , \gam, \ctau \in S_W^3 $ with 
$m(\alphar,\ctau)=3= m(\gam,\ctau)$ and $m( \alphar, \gam)=2$ we have that 
\begin{align*} 
{\sf fork}^{\ctau}_{\ctau\ctau}
({\sf 1}_{\ctau} \otimes 
{\sf bar}(\alphar)\otimes 
{\sf bar}(\gam)		\otimes {\sf 1}_{\ctau})
{\sf fork}_{\ctau}^{\ctau\ctau}
&=
-  \left(
{\sf bar}(\alphar)+ 
{\sf bar}(\ctau)+ 
{\sf bar}(\gam) \right)  		\otimes {\sf 1}_{\ctau} 
\\
&=-   {\sf 1}_{\ctau}	\otimes   \left(
{\sf bar}(\alphar)+ 
{\sf bar}(\ctau)+ 
{\sf bar}(\gam) \right)  .	 
\end{align*}
\end{lem}
\begin{proof}
We prove the first equality, the second  is given by \cref{useful3} and recorded here only for reference.  
We   first move the $\alphar$ barbell to the left through the $\ctau $ strand using   \ref{useful1} (and hence obtain  3 terms);  for the first two of these terms (in which the $\ctau$-strand remains in tact) we then again use \ref{useful1} to move the $\gam$ barbell to the left through the $\ctau $-strand.  
We hence obtain a sum involving 7 terms, 4 of which are zero by the $\blue$-circle-annihilation relation;  this leaves us with the required 3 terms. 
\end{proof}

We now ``augment" the previous lemma so that it applies to augmented tricornes.

\begin{lem}\label{lemyytytyt2}
Let $\gsigma  $ be an augmented tricorne, $\gsigma = \vx\alphar \gam \vx^{-1} \ctau $
and   $\vx =\color{violet} \vsk \dots  \vsone  $.  
We have that 
\begin{align}%\begin{split}
%\label{lhsofhhh} 
\begin{multlined}
{\sf fork}^{\ctau}_{\ctau\ctau}
({\sf 1}_{\ctau} \otimes {\sf cap}^\emptyset_{\vx\vx^{-1}}
 \otimes{\sf 1}_{\ctau} )
({\sf 1}_{\ctau\vx} \otimes 
{\sf bar}(\alphar)\otimes 
{\sf bar}(\gam)		\otimes {\sf 1}_{\vx^{-1}\ctau})
({\sf 1}_{\ctau} \otimes {\sf cup}_\emptyset^{\vx\vx^{-1}}
 \otimes{\sf 1}_{\ctau} )
{\sf fork}_{\ctau}^{\ctau\ctau}
\\
\label{lhsofhhh2}
\begin{aligned} 
=&\;(-1)^{k+1}{\sf 1}_\ctau\otimes (  \textstyle\sum_{i=1}^k 2{\sf bar}(\vsi) + {\sf bar}(\alphar)+ {\sf bar}(\gam)+{\sf bar}(\ctau))
\end{aligned}
\end{multlined}
\\
\begin{aligned}
\label{lhsofhhh3} 
=&\;(-1)^{k+1} (  \textstyle\sum_{i=1}^k 2{\sf bar}(\vsi) + {\sf bar}(\alphar)+ {\sf bar}(\gam)+{\sf bar}(\ctau))\otimes {\sf 1}_\ctau 
\end{aligned}
\end{align}
\end{lem}
\begin{proof}
We proceed by induction on $k\geq 0$, with the $k=0$ base case taken care of in \cref{substitute}.  
By induction, we can rewrite the left-hand side of \ref{lhsofhhh2} as follows
$$
  \fork^\ctau_{\ctau\ctau}
   \spot_{\ctau\vsk\ctau}^{\ctau\emptyset\ctau}
({\sf 1}_\ctau \otimes 
(-1)^k  {\sf 1}_{\vsk}\otimes ( \textstyle\sum_{i=1}^{k-1}
 2{\sf bar}(\vsi) + {\sf bar}(\alphar)+ {\sf bar}(\gam)
+  {\sf bar}(\vsk))
\otimes {\sf 1}_\ctau)
   \spot^{\ctau\vsk\ctau}_{\ctau\emptyset\ctau}
     \fork_\ctau^{\ctau\ctau}.
$$
which is equal to 
$$
(-1)^k  \fork^\ctau_{\ctau\ctau}
({\sf 1}_\ctau \otimes 
   {\sf bar}({\vsk})
\otimes ( \textstyle\sum_{i=1}^k 2{\sf bar}(\vsi) + {\sf bar}(\alphar)+ {\sf bar}(\gam)
+  {\sf bar}(\vsk))
\otimes {\sf 1}_\ctau)
     \fork_\ctau^{\ctau\ctau}.
$$
The   term  involving  a tensor product ${\sf bar}(\vsk)\otimes {\sf bar}(\vsk)$ can be rewritten using \cref{repeat13}.  
The remaining  terms  involve a tensor product of two distinctly coloured
 barbells, one of which commutes with the $\ctau$-strand; thus
  we can apply  \cref{repeat11} to   these terms.
  Rewriting all the  terms in the above manner and summing over the resulting 
 elements, we obtain   \ref{lhsofhhh2}.  
\Cref{lhsofhhh3}  follows by \cref{useful3}.  
 \end{proof}

We are now ready to construct the dilated barbell diagrams.  

\begin{lem}
Let $\gsigma  $ be an augmented tricorne, $\gsigma = \vx\alphar \gam \vx^{-1} \ctau $
and   $\vx =\color{violet} \vsk \dots  \vsone  $.     
We have that 
  \begin{align}\label{ppg1}
\dil ({\sf bar}(\gsigma))	&=
 {\sf 1}_\ctau\otimes (  \textstyle\sum_{i=1}^k 2{\sf bar}(\vsi) + {\sf bar}(\alphar)+ {\sf bar}(\gam)+ {\sf bar}(\ctau))
 \\
 &=
   (  \textstyle\sum_{i=1}^k 2{\sf bar}(\vsi) + {\sf bar}(\alphar)+ {\sf bar}(\gam)+ {\sf bar}(\ctau))\otimes  {\sf 1}_\ctau
 \\	\label{ppg2}
\dil ({\sf 1}_ \gsigma)\compose \dil ({\sf bar}(\gsigma)	)&=
{\sf 1}_{\ctau\vx \alphar} \otimes 
{\sf bar}(\alphar)\otimes 
{\sf 1}_{\gam\vx^{-1} \ctau} 
\\ \label{ppg3}
\dil ({\sf bar}(\gsigma))\compose
\dil 
({\sf 1}_ \gsigma) &=
{\sf 1}_{\ctau\vx}\otimes {\sf bar}(\alphar) \otimes {\sf 1}_{\alphar\gam\vx^{-1}\ctau}
\\ \label{ppg4}
\dil ({\sf gap}(\gsigma))
&={\sf 1}_{\ctau \vx} \otimes {\sf gap}(\alphar) \otimes {\sf 1}_{\gam\vx^{-1}\ctau}
\end{align}

 \end{lem}
\begin{proof}
\Cref{ppg1} follows directly from \cref{lemyytytyt2}.
We now consider \cref{ppg2} and  \cref{ppg3}.  We have that 
\begin{align*}
\dil ({\sf 1}_ \gsigma)\compose \dil ({\sf bar}(\gsigma)	)&=
 {\sf 1}_{\ctau\vx\alphar\gam\vx^{-1}\ctau}\otimes (  \textstyle\sum_{i=1}^k 2{\sf bar}(\vsi) + {\sf bar}(\alphar)+ {\sf bar}(\gam)+ {\sf bar}(\ctau))
\\
&=
 {\sf 1}_{\ctau\vx\alphar\gam\vx^{-1} }\otimes (  \textstyle\sum_{i=1}^k 2{\sf bar}(\vsi) + {\sf bar}(\alphar)+ {\sf bar}(\gam)+ {\sf bar}(\ctau))
 \otimes{\sf 1}_\ctau 
\\ &=
 {\sf 1}_{\ctau\vx\alphar\gam\vx^{-1} }\otimes (  \textstyle\sum_{i=1}^{k-1} 2{\sf bar}(\vsi) + {\sf bar}(\alphar)+ {\sf bar}(\gam)+ {\sf bar}(\vsk))
 \otimes{\sf 1}_\ctau 
\end{align*}
where the first equality follows
\cref{ppg1};
the second  from summing over relations \ref{R4} and \ref{R5}; the third from \cref{Astuffiszero}.  We repeat the final two steps above a further $k-2$ times and hence obtain 
\begin{align*}
\dil ({\sf 1}_ \gsigma)\compose \dil ({\sf bar}(\gsigma)	)&=
 {\sf 1}_{\ctau\vx\alphar\gam \vsone }\otimes ( 2{\sf bar}(\vsone) + {\sf bar}(\alphar)+ {\sf bar}(\gam)+ {\sf bar}(\vstwo))
 \otimes{\sf 1}_{\vstwo \dots \vsk \ctau }
\\
&=
 {\sf 1}_{\ctau\vx\alphar\gam \vsone }\otimes (   {\sf bar}(\alphar)+ {\sf bar}(\gam)+ {\sf bar}(\vsone))
 \otimes{\sf 1}_{\vstwo \dots \vsk \ctau }
\\
&=
 {\sf 1}_{\ctau\vx\alphar\gam   }\otimes (   {\sf bar}(\alphar)+ {\sf bar}(\gam)+ {\sf bar}(\vsone))
 \otimes{\sf 1}_{\vx^{-1} \ctau }
 \\
&=
 {\sf 1}_{\ctau\vx\alphar\gam   }\otimes     {\sf bar}(\alphar) 
 \otimes{\sf 1}_{\vx^{-1} \ctau }
 \\
&=
 {\sf 1}_{\ctau\vx\alphar   }\otimes     {\sf bar}(\alphar) 
 \otimes{\sf 1}_{\gam \vx^{-1} \ctau }
\end{align*}
where the second and fourth equalities  follow  from 
\cref{Astuffiszero};
the third from \cref{useful3}; the fifth from the $\orange\pink$-commutativity relations.   
We now consider \cref{ppg3}. We have that 
\begin{align*}
\dil ({\sf bar}(\gsigma))\compose
\dil 
({\sf 1}_ \gsigma) &=
 {\sf 1}_\ctau\otimes (  \textstyle\sum_{i=1}^k 2{\sf bar}(\vsi) + {\sf bar}(\alphar)+ {\sf bar}(\gam)+ {\sf bar}(\ctau)) \otimes {\sf 1}_{\vx\alphar\gam\vx^{-1}\ctau}
 \\
 &=
 {\sf 1}_{\ctau\vx  }\otimes     {\sf bar}(\alphar) 
 \otimes{\sf 1}_{\alphar  \gam \vx^{-1} \ctau }
\end{align*}
where the first equality follows from \cref{ppg1} and the second follows by the exact same argument as for the case of  \cref{ppg2}.
Finally, we   consider \cref{ppg4}. We have that 
\begin{align*}
\dil (\gap(\gsigma))	&= \dil (\spot_\emptyset^\gsigma)
 \dil (\spot^\emptyset_\gsigma)\\
 &= (-1)^{k+1}
 {\sf spot}^{\ctau \vx \pink\orange\vx^{-1} \ctau}
 _{\ctau \vx \emptyset\emptyset \vx^{-1} \ctau}
 ({\sf 1}_\ctau \otimes {\sf cup}^{\vx\vx^{-1}}_\emptyset \otimes {\sf 1}_\ctau)
 {\sf dork}^{\ctau\ctau}_{\ctau\ctau}
  ({\sf 1}_\ctau \otimes {\sf cap}_{\vx\vx^{-1}}^\emptyset \otimes {\sf 1}_\ctau)
   {\sf spot}_{\ctau \vx \pink\orange\vx^{-1} \ctau}
^{\ctau \vx \emptyset\emptyset \vx^{-1} \ctau}
\\
&= {\sf 1}_{\ctau\vx}\otimes {\sf gap}(\alphar)\otimes {\sf 1}_{\orange\vx^{-1}\ctau}
\end{align*}
where the first and second equalities are by definition; the third follows by applying the 
$\ctau\vsk$-null-braid relation followed by the 
$ \vsi\vsione$-null-braid relations for $k>i\geq 1$ followed by the 
$\orange\vsone$-null-braid relation.  
\end{proof}

\subsubsection{The dilated one colour barbell relation}
Let $\gsigma $ be a reading word of some $\ctau$-contraction tile.    We now verify the rightmost relation in \eqref{R4}. 
%, namely, the one-colour barbell relation 
%\begin{equation}\label{onecolourbarbell-words}
%\dil   
%(   {\sf bar}(\gsigma)  ) \compose 
%\dil  ( {\sf 1}_{{\gsigma}}  ) 
%+
%\dil (  {\sf 1}_{{\gsigma}}  ) \compose \dil  (   {\sf bar}(\gsigma)    ) 
%=
%2 \cdot
%\dil 
%( 
% {\sf gap}(\gsigma)   
% )  . 
%\end{equation}
We have that 
\begin{multline*}
\dil({\sf bar}(\gsigma))
\compose 
\dil({\sf 1}_\gsigma)
+
\dil({\sf 1}_\gsigma)\compose\dil({\sf bar}(\gsigma))
\\
\begin{aligned}
 =&\;
{\sf 1}_{\ctau\vx} \otimes {\sf bar}(\alphar) \otimes {\sf 1}_{\alphar\gam\vx^{-1}\ctau}
+
{\sf 1}_{\ctau\vx\alphar} \otimes {\sf bar}(\alphar) \otimes {\sf 1}_{ \gam\vx^{-1}\ctau}
\\
=&\;
2\cdot {\sf 1}_{\ctau\vx} \otimes {\sf gap}(\alphar) \otimes {\sf 1}_{ \gam\vx^{-1}\ctau}
\\
=&\;2\cdot \dil({\sf gap}(\gsigma))
\end{aligned}
\end{multline*}
as required. Here the first equality follows from \cref{ppg2,ppg3}; the second follows from the one-colour-barbell relation; and the third from \cref{ppg4}.

\subsubsection{The dilated two colour barbell relations}

Let $\betar \in S_{W^\ctau}$, as noted in \cref{asnited}, we can assume that 
$\betar$ is a singleton which commutes  every label in $\gsigma$ except 
$\alphar$.  
%We  now prove 
%\begin{align*}%\label{2colourbarbell-words}
%\dil   
% (   {\sf bar}(\gsigma)  ) \compose 
%\dil  ( {\sf 1}_{{\betar}}  ) 
%-
%\dil (  {\sf 1}_{{\betar}}  ) \compose \dil  (   {\sf bar}(\gsigma)    ) 
%=
%\dil (  {\sf 1}_{{\betar}}  ) \compose \dil  (   {\sf bar}(\betar)    ) 
%- \dil 
%( 
% {\sf gap}(\betar)   
% ) \\
%\dil   
% (   {\sf bar}(\betar)  ) \compose 
%\dil  ( {\sf 1}_{{\gsigma}}  ) 
%-
%\dil (  {\sf 1}_{{\gsigma}}  ) \compose \dil  (   {\sf bar}(\betar)    ) 
%=
%\dil (  {\sf 1}_{{\gsigma}}  ) \compose \dil  (   {\sf bar}(\gsigma)    ) 
%- \dil 
%( 
% {\sf gap}(\gsigma)   
% )  
%\end{align*}
%We start with the former relation. 
  We have that 
\begin{multline*}   
\dil (   {\sf bar}(\gsigma)  ) \compose 
\dil  ( {\sf 1}_{{\betar}}  ) 
-
\dil (  {\sf 1}_{{\betar}}  ) \compose \dil  (   {\sf bar}(\gsigma)    ) 
\\
\begin{aligned}
=&\;
{\sf 1}_\ctau 
\otimes 
(  \textstyle\sum_{i=1}^k 2{\sf bar}(\vsi) + {\sf bar}(\alphar)+ {\sf bar}(\gam)+ {\sf bar}(\ctau))
\otimes {\sf 1}_\betar
\\
&\quad -
{\sf 1}_{\ctau\betar}
  \otimes (  \textstyle\sum_{i=1}^k 2{\sf bar}(\vsi) + {\sf bar}(\alphar)+ {\sf bar}(\gam)+ {\sf bar}(\ctau))
\\
=&\;
{\sf 1}_\ctau 
\otimes {\sf bar}(\alphar) \otimes {\sf 1}_{\betar}
-
{\sf 1}_{\ctau \betar}
\otimes {\sf bar}(\alphar)  
\\
=&\;
 {\sf 1}_{{\ctau\betar}}    \otimes    {\sf bar}(\betar)     
-  
 {\sf 1}_{{\ctau }}    \otimes   {\sf gap}(\betar)   
\\
=&\;
\dil (  {\sf 1}_{{\betar}}  ) \compose \dil  (   {\sf bar}(\betar)    ) 
- \dil 
( 
 {\sf gap}(\betar)   
 )
\end{aligned}
\end{multline*}
as required.  Here, the first equality follows from \cref{ppg1}; the second from the commutativity relations; the third from the $\alphar\betar$-barbell relation; the fourth follows by definition.  

We now turn to the other two-colour barbell relation (in which the roles of $\betar$ and $\gsigma$ are swapped).  We have that 
\begin{multline*}
\dil   
 (   {\sf bar}(\betar)  ) \compose 
\dil  ( {\sf 1}_{{\gsigma}}  ) 
-
\dil (  {\sf 1}_{{\gsigma}}  ) \compose \dil  (   {\sf bar}(\betar)    ) 
\\
\begin{aligned}
=&\;
{\sf 1}_{ \ctau }
\otimes {\sf bar}(\betar)\otimes {\sf 1}_{\vx\alphar\gam \vx^{-1}\ctau}
-
{\sf 1}_{ \ctau  \vx\alphar\gam \vx^{-1}\ctau}\otimes {\sf bar}(\betar)
\\
=&\;
{\sf 1}_{ \ctau \vx }
\otimes {\sf bar}(\betar)\otimes {\sf 1}_{ \alphar\gam \vx^{-1}\ctau}
-
{\sf 1}_{ \ctau \vx \alphar}
\otimes {\sf bar}(\betar)\otimes {\sf 1}_{  \gam \vx^{-1}\ctau}
\\
=&\;
{\sf 1}_{ \ctau \vx }
\otimes {\sf bar}(\alphar)\otimes {\sf 1}_{ \alphar\gam \vx^{-1}\ctau}
-
{\sf 1}_{ \ctau \vx \alphar}
\otimes {\sf bar}(\alphar)\otimes {\sf 1}_{  \gam \vx^{-1}\ctau}
\\
=&\; 
\dil (  {\sf 1}_{{\gsigma}}  ) \compose \dil  (   {\sf bar}(\gsigma)    ) 
- \dil 
( 
 {\sf gap}(\gsigma)   
 )  
\end{aligned}
\end{multline*}
as required.  Here the first equality follows by definition; 
the second by the commutativity relations; the third by the $\pink\green$-barbell;
the fourth by \cref{ppg2,ppg4}.

%   
%   
%   
%   
%   
%   \begin{align*}
%&		
%{\sf 1}_{\varphi_\ctau(\mu)}
%\otimes  					{\sf 1}_{\violet\pink\orange\violet} \otimes (2\times {\sf bar}(\violet)
%+ {\sf bar}(\pink)+ {\sf bar}(\orange)+ {\sf bar}(\blue))
%\otimes {\sf 1}_\blue
%\\
%=\;&{\sf 1}_{\varphi_\ctau(\mu)}
%\otimes  					{\sf 1}_{\violet\pink\orange\violet} \otimes  ( {\sf bar}(\violet)
%+ {\sf bar}(\pink)+ {\sf bar}(\orange) )
%\otimes {\sf 1}_\blue
%\\
%=\;&{\sf 1}_{\varphi_\ctau(\mu)}
%\otimes  					({\sf 1}_{\violet\pink} 
% \otimes  {\sf bar}(\pink)  \otimes 
%{\sf 1}_{\orange\violet} 
%\otimes {\sf 1}_\blue
%+ {\sf spot}^{\violet\pink\orange\violet\blue} 
%		_{\violet\pink\orange\emptyset\blue} 
%{\sf spot}_{\violet\pink\orange\violet\blue} 
%		^{\violet\pink\orange\emptyset\blue} 
%+
%		{\sf spot}^{\violet\pink\orange\violet\blue} 
%		_{\violet\pink\emptyset\violet\blue} 
%{\sf spot}_{\violet\pink\orange\violet\blue} 
%		^{\violet\pink\emptyset\violet\blue} 	)	
%\\
%=\;&{\sf 1}_{\varphi_\ctau(\mu)}
%\otimes  					({\sf 1}_{\violet\pink} 
% \otimes  {\sf bar}(\pink)  \otimes 
%{\sf 1}_{\orange\violet} 
%\otimes {\sf 1}_\blue
% +
%		{\sf spot}^{\violet\pink\orange\violet\blue} 
%		_{\violet\pink\emptyset\violet\blue} 
%{\sf spot}_{\violet\pink\orange\violet\blue} 
%		^{\violet\pink\emptyset\violet\blue} 	)	
%\end{align*}   

   \subsection{The dilated $m=2$    relations}
For $\gsigma,\betar,\brown \in S_{W^\ctau}$ with $m(\gsigma,\betar)=
m(\brown,\betar)=m(\gsigma,\brown)=2$ we need to check the dilated versions of the relations
\begin{align*}
 {\sf braid}^{\green\gsigma}_{\gsigma\green}  
  {\sf braid}_{\green\gsigma}^{\gsigma\green}  
&= 
   {\sf 1}_{\green\gsigma} 
   \qquad 
   \phantom{ {\sf fork}^{  \green\gsigma}_{ \green\gsigma \gsigma }{\sf braid}^{\green \gsigma \gsigma }_{\gsigma \gsigma \green} }
   {\sf braid}^{\green\brown\gsigma}_{ \brown \green\gsigma}    {\sf braid}^{   \brown  \green\gsigma}_{\gsigma  \brown \green }   
 = 
    {\sf braid}^{  \green \brown \gsigma}_{\gsigma  \green  \brown }      {\sf braid}^ {\gsigma  \green  \brown }  _{\gsigma   \brown \green  }      
  \\
 {\sf braid}^{\green\gsigma}_{\gsigma\green} 
  {\sf fork}^{\gsigma \green}_{\gsigma \gsigma \green} 
&= 
     {\sf fork}^{  \green\gsigma}_{ \green\gsigma \gsigma }{\sf braid}^{\green \gsigma \gsigma }_{\gsigma \gsigma \green} 
\qquad 
\phantom{ {\sf 1}_{\green\gsigma} }
 ({\sf 1}_{\gsigma\betar} \otimes {\sf cap}^{\emptyset}_{\gsigma\gsigma})
 {\sf braid}^{ \gsigma\green\gsigma\gsigma}_
 {\gsigma\gsigma\green\gsigma}  
({ \sf cup}^{\gsigma\gsigma}_{\emptyset} \otimes   {\sf 1}_{\betar\gsigma}) 
=
 {\sf braid}_{\green\gsigma}^{\gsigma\green} 
 \end{align*}
%\begin{align*}
% {\sf braid}^{\green\gsigma}_{\gsigma\green}  
%  {\sf braid}_{\green\gsigma}^{\gsigma\green}  
%&= 
%   {\sf 1}_{\green\gsigma} 
%  \\
% {\sf braid}^{\green\gsigma}_{\gsigma\green} 
%  {\sf fork}^{\gsigma \green}_{\gsigma \gsigma \green} 
%&= 
%     {\sf fork}^{  \green\gsigma}_{ \green\gsigma \gsigma }{\sf braid}^{\green \gsigma \gsigma }_{\gsigma \gsigma \green} 
%\\
% ({\sf 1}_{\gsigma\betar} \otimes {\sf cap}^{\emptyset}_{\gsigma\gsigma})
% {\sf braid}^{ \gsigma\green\gsigma\gsigma}_
% {\gsigma\gsigma\green\gsigma}  
%({ \sf cup}^{\gsigma\gsigma}_{\emptyset} \otimes   {\sf 1}_{\betar\gsigma})&=
% {\sf braid}_{\green\gsigma}^{\gsigma\green} 
% \\
% \varphi_\ctau(  {\sf braid}^{\green\brown\gsigma}_{ \brown \green\gsigma}    {\sf braid}^{   \brown  \green\gsigma}_{\gsigma  \brown \green }  ) 
%&= 
%\varphi_\ctau(    {\sf braid}^{  \green \brown \gsigma}_{\gsigma  \green  \brown }      {\sf braid}^ {\gsigma  \green  \brown }  _{\gsigma   \brown \green  })     
%\end{align*}
and their horizontal and vertical flips, along with the diagrams obtained by swapping the roles of $\green$ and $\gsigma$.  
%We do not explicitly consider the final commutativity relation here (pushing a non-commuting braid through a commuting braid) but this is merely because we have run out of colours (it can be argued in a similar fashion to those above).  
Note that by \cref{theresultwethoughtweneed}, both sides of all of these equations vanish 
when $(W,P)^\ctau=(A_n,A_{n-1})$, 
or 
$(W,P)^\ctau=(D_n,D_{n-1})$,
or $(W,P)^\ctau=(D_n,A_{n-1})$ with $\{\betar,\gsigma\}=\{ \color{cyan} s_0 \color{black}  , \color{magenta}s_1  \color{black}  \}$.  
In all other cases, we have that $\gsigma$ is a (possibly) composite label and $\betar$ (and $\brown$) are singleton  labels which commute with every constituent   label of $\gsigma$.  Thus all these relations are trivially satisfied.

%If $\green$ commutes with every reflection $  \pink,\orange, \blue,  \vsone  ,\dots, \vsk   $ then the relations follow  immediately from the commutativity and cyclotomic relations.  
%(For example, this takes care of the type $(A,A\times A)$ case.).  
%Otherwise $\varphi_\ctau$ maps both sides of the diagram to zero, and so again the result holds.  
% 

\subsection{The cyclotomic relations}
We finish by showing that the dilations of the non-local relation \ref{R10} and \ref{R11} are also preserved by $\varphi_\ctau$. 
It is easy to see that 
$$\varphi_\ctau({\sf 1}_{\gsigma}\otimes {\sf 1}_\w)
={\sf 1}_{T_{0\to1}} \compose \dil ({\sf 1}_\gsigma)
\compose \dil ({\sf 1}_\w)=0$$
whenever $\gsigma \in S_{P^\ctau}$ using (possibly) the null-braid relations, the commutativity relations, and the cyclotomic relation in $\mathcal{H}_{(W,P)}$.  
It remains to show that
$$
\varphi_\ctau({\sf bar}(\gsigma) \otimes 
 {\sf 1}_\w)=0$$
 for $\gsigma$ the unique element of $S_{W^\ctau}\setminus S_{P^\ctau}$.  
 We will show that 
 $${\sf 1}_{T_{0\to1}}\compose  \dil ({\sf bar}(\gsigma) )=0$$
 for such $\gsigma$ and hence deduce the result.  
For the remainder of this section, we set $T_{0\to 1}=\rho_1 \rho_2 \dots \rho_r$ and we note that $\rho_r=\ctau$.  

\smallskip
\noindent 
{\bf Case 1. } Suppose that $\gsigma$ is a singleton.  
Then there exists $1\leq j \leq r$ such that $m(\gsigma,\rho_i)=2$ for all $i\neq j$ 
and $m(\gsigma,\rho_j)=3$; this can be seen by
  inspection of 
 \cref{fig1thissec,fig3thissec,lots of exceptionsal1,hhhh5672,hhhh5671,typeD------er}.  
We have that 
 \begin{align*}
 {\sf1}_{T_{0\to1}}\compose \dil ({\sf bar} (\gsigma))
 &={\sf 1}_{\rho_1\rho_2\dots \rho_r}\otimes  {\sf bar} (\gsigma)
 \\
 &={\sf 1}_{\rho_1\rho_2\dots \rho_j}  \otimes  {\sf bar} (\gsigma) \otimes {\sf 1}_{\rho_{j+1}\dots \rho_r} 
 \\
 &={\sf 1}_{\rho_1\rho_2\dots \rho_{j-1}}  \otimes  ( {\sf bar} (\gsigma)+ {\sf bar} (\rho_j)) \otimes {\sf 1}_{\rho_j\rho_{j+1}\dots \rho_r} 
\\
&\quad -
 {\sf 1}_{\rho_1\rho_2\dots \rho_{j-1}}  \otimes  {\sf gap} (\rho_j) \otimes {\sf 1}_{\rho_{j+1}\dots \rho_r} 
 \\
 &={\sf 1}_{\rho_1\rho_2\dots \rho_{j-1}}  \otimes   {\sf bar} (\rho_j)  \otimes {\sf 1}_{\rho_j\rho_{j+1}\dots \rho_r} 
\\
&=0  \end{align*}
as required.
 Here the first equality is the definition; the second follows  by the commuting relations;
 the third  by the two-colour barbell relation; the fourth by the commuting and cyclotomic relations; 
the fifth follows by repeating the  arguments above.

\smallskip
\noindent 
{\bf Case 2. } We now suppose that $\gsigma=\alphar\gam\ctau$, a tricorne.
By inspecting   \cref{fig1thissec,fig3thissec,lots of exceptionsal1,hhhh5672,hhhh5671,typeD------er}, we deduce that 
  $m(\alphar,\rho_i)=2=m(\gam,\rho_i)$ for all   $1\leq i  \leq r$.  
  We have that 
 \begin{align*}
 {\sf1}_{T_{0\to1}}\compose \dil ({\sf bar} (\gsigma))
 &={\sf 1}_{\rho_1\rho_2\dots \rho_{r-1}}\otimes  {\sf 1}_\ctau \otimes ({\sf bar} (\alphar)+{\sf bar} (\gam)+{\sf bar} (\ctau))
 \\
&={\sf 1}_{\rho_1\rho_2\dots \rho_{r-1}} \otimes ({\sf bar} (\alphar)+{\sf bar} (\gam)+{\sf bar} (\ctau)  )\otimes  {\sf 1}_\ctau  
\\
&={\sf 1}_{\rho_1\rho_2\dots \rho_{r-1}} \otimes  {\sf bar} (\ctau)   \otimes  {\sf 1}_\ctau  
\\
&=0 \end{align*}as required.  
 Here the first equation follows from \cref{ppg1};
 the second by \cref{useful3}; the third by the commuting and cyclotomic relations; the 
 fourth follows as in Case 1.

\smallskip
\noindent 
{\bf Case 3. } We now suppose that $\gsigma=\vx\alphar\gam\vx^{-1}\ctau$, an augmented  tricorne.
By  inspecting  \cref{fig3thissec,lots of exceptionsal1,hhhh5672,hhhh5671,typeD------er}, we deduce that 
$m(\alphar,\rho_i)=m(\gam,\rho_i)=m(\vsj,\rho_i)=2$ for $j\neq k$, $1\leq i \leq r$;
and $m(\vsk,\rho_r)=3$ (recall $\ctau=\rho_r$).  We have that 
\begin{align*}
 {\sf1}_{T_{0\to1}}\compose \dil ({\sf bar} (\gsigma))
 &={\sf 1}_{\rho_1\rho_2\dots \rho_{r-1}}\otimes  {\sf 1}_\ctau 
 \otimes (  \textstyle\sum_{i=1}^k 2{\sf bar}(\vsi) + {\sf bar}(\alphar)+ {\sf bar}(\gam)+ {\sf bar}(\ctau))
 \\
  &={\sf 1}_{\rho_1\rho_2\dots \rho_{r-1}}\otimes  {\sf 1}_\ctau 
 \otimes (   2{\sf bar}(\vsk) + {\sf bar}(\ctau))
 \\
  &={\sf 1}_{\rho_1\rho_2\dots \rho_{r-1}} 
 \otimes (   2{\sf bar}(\vsk) + {\sf bar}(\ctau))
 \otimes  {\sf 1}_\ctau 
  \\
  &={\sf 1}_{\rho_1\rho_2\dots \rho_{r-1}} 
 \otimes {\sf bar}(\ctau) 
 \otimes  {\sf 1}_\ctau 
 \\
 &=0
\end{align*}
as required.  The first equality follows from \cref{ppg1};
the second from the commutativity relations; 
the third from \cref{useful3}; 
the fourth by commutativity relations; the fifth equality follows as in Case 1.

\section{Graded decomposition numbers 
and Koszul resolutions }\label{Koszul:res}

%We continue to assume that $(W,P)$ is a simply laced Hermitian symmetric pair.  
%We now take a short detour in order to consider some basic facts about the algebras   $h_{(W,P)}$ which do not require any of the technical machinery of Coxeter truncation; these facts will be essential in what follows.  
% The non-simply laced types will be dealt with in \Cref{nonsimply}.  
 
\color{black!99}
We are now ready to determine the main structural
 results concerning the Hecke categories of Hermitian symmetric pairs.  Specifically, we will calculate the 
 graded composition multiplicities and radical filtrations of standard modules
 in \cref{yabasic,coincide}.
In order to prove that the  grading and radical layers coincide, we will 
 prove that the algebra  $h_{(W,P)}$ satisfies the   strong cohomological property   of standard Koszulity (see \cite{bgs96,MR1960515} for the definition of standard Koszul); this amounts to constructing linear projective resolutions of standard modules as in \cref{resolutionsss}.  Our treatment of this material is inspired by similar ideas in \cite{MR2600694}.
\color{black}
 
 \begin{prop}[{\cite[Corollary 6.2]{FARR}}]\label{FARR}
 Let $(W,P)$ be a simply laced Hermitian symmetric pair. 
 For any $\la\neq \mu$, we have that 
 $$
% \color{black}\dim_q (\Delta(\la) {\sf 1}_\mu)  \color{black}
 \sum_{\SSTS\in \Path(\la,\stt_\mu)} \grade^{\deg(\SSTS)}
 \in  \grade\ZZ_{\geq 0}[\grade ]  .
 $$\color{black!99}In particular, all the non-zero terms occur in strictly-positive degree.
 \end{prop}

\begin{thm}\label{yabasic}
 Let  $(W,P)$ be an arbitrary  Hermitian symmetric pair and $\Bbbk$ be a field of characteristic $p\geq 0$.  
  The
  \color{black!99}
   $p$-Kazhdan--Lusztig polynoials 
 $$
 {^pn}_{\la,\mu}(q)=   \color{black}
% [\Delta(\la):L(\mu)  ]_\grade
%  = 
  \sum_{k\in \ZZ}[\Delta(\la):L(\mu)\langle k \rangle  ]\grade^k $$\color{black}of $h_{(W,P)}$ are  independent of the prime $p\geq0$.  
 For $(W,P)$ of simply laced type, 
 the algebra  $h_{(W,P)}$ is   basic  and  the modules   ${\sf 1}_{\stt_\la} h_{(W,P)} $ for $\la \in \mptn$ provide a complete set of non-isomorphic projective indecomposable right $h_{(W,P)}$-modules.  \end{thm}

\begin{proof}

By \cref{C--A}, it is enough to restrict our attention to simply laced type.  
 \color{black}  By \cref{FARR}, we have  that $h_{(W,P)}$  is a positively $\ZZ$-graded $\Bbbk$-algebra with 
\begin{align}\label{needsalabel}
\dim_q(h_{(W,P)})|_{q=0} = 
\textstyle (\sum_{\la \in \mptn} (\sum_{\mu \in \mptn } \color{black}\dim_q (\Delta(\la) {\sf 1}_\mu) )^2)|_{q=0} = 
 |\{\la \mid \la \in  \mptn\} |		 \end{align}where the latter equality follows again from \cref{FARR}. %as this will imply that every simple module is    1-dimensional (regardless of the characteristic of the field). 
%\color{black}
%Thus  the result  follows from   and the following pair of facts:  for our algebras 
%$\dim_{\grade}(L(\la){\sf 1}_{\stt_\la})=1$ (by definition) and for general graded  cellular algebras 
Now, we have that $\dim_{\grade}(L(\la) )\in \ZZ_{\geq0}[\grade+\grade^{-1}]$ (by \cite
[Proposition 2.18]{hm10}) and so  by \eqref{needsalabel} we deduce that  $\dim_{\grade}(L(\la) )=1$ 
and that the degree zero subalgebra  of 
 $  h_{(W,P)}$ is isomorphic to $\oplus _{\la \in \mptn }L(\la)$ 
(regardless of the characteristic of $\Bbbk$).  
 Thus the algebra is basic (as all the simple modules are 1-dimensional)  and we have that 
 $$
{^pn}_{\la,\mu}(q)=  \sum_{k\in \ZZ}[\Delta(\la): L(\mu)\langle k \rangle] q^k = 
 \dim_q (\Delta(\la) {\sf 1}_\mu)  = 
 \sum_{\SSTS\in \Path(\la,\stt_\mu)} \grade^{\deg(\SSTS)}
 \in  \grade\ZZ_{\geq 0}[\grade ]  
 $$again by \cref{yabasic}  (again, regardless of the characteristic $p$ of the
  field $\Bbbk$).
\end{proof}

In the remainder of this section, we will prove the Koszulity of the Hecke categories for Hermitian symmetric pairs. Using Section 4, we can reduce to the simply-laced cases where we will use the Coxeter truncation to work by induction on the rank.

\subsection{Induction}

Assume that $(W,P)$ is simply laced and let $\ctau\in S_W$. 
Define
  $$
e_\ctau=
  \sum_{
\begin{subarray}c 
\mu \in \mptn 
\\
\mu <\mu\ctau  
\end{subarray} 
  }{\sf 1}_{\stt_\mu}\otimes {\sf 1}_{ \ctau} 	 	 
% \qquad\text{and}\qquad
%h_{(W,P)}^\ctau=  e_\ctau \mathscr{H}_{(W,P)}  e_\ctau  .
 $$
 We have that $e_\ctau h_{(W,P)}$ carries the structure of a $(h_{(W,P) }^\ctau  , h_{(W,P)})$-bimodule.  The action on the right is by concatenation of diagrams.
   The action on the left is given by 
 first conjugating 
 $h_{(W,P) }^\ctau$ by a (commuting) braid 
 so that the colour sequences match-up, and then concatenating diagrams.  
 (Recall from \cref{jkbhxlkhbkjhcxbvhgbvjxkhvkxcv} that this simply amounts to changing our choice of tableaux.)  
% 
%We have a graded algebra isomorphism  $\iota: h_{(W,P)^\ctau}  \cong h_{(W,P)}^\ctau \subseteq h_{(W,P)}$ given by 
% $$
%\iota(D)= 
%\sum_{\mu,\nu \in \mathscr{P}_{(W,P)^\ctau}}
% \braid^ {\stt_\nu \otimes  \ctau}_{\varphi_\ctau(\stt_\nu)} \circ 
% \varphi_\ctau(D)   \circ  \braid_ {\stt_\mu \otimes  \ctau}^{\varphi_\ctau(\stt_\mu)} .
% $$
%The homomorphism property follows 
%from the commutativity of braid generators together with the homomorphism property for $\varphi_\ctau$ (proven in \cref{hom3}).  
%That the map is a graded vector space isomorphism follows from
%  \cref{BIJECTION} (see also \cref{veccer} and note  that braid elements have degree zero).  
With this isomorphism in place (and the isomorphism of \cref{veccer}) we are now able to define an induction functor 
\begin{align*}
G^\ctau: &
 h_{(W,P)^\ctau }{\rm -mod}
\xrightarrow { \ \  \ } h_{(W,P) }{\rm -mod}
\\
&M \mapsto
  M
 \otimes _{h_{(W,P)^\ctau }}
e_\ctau     h_{(W,P)}\langle -1\rangle
\end{align*}
using the identification $h_{(W,P)^\ctau}  \cong h_{(W,P)}^\ctau \subseteq h_{(W,P)}$.  The degree shift in this definition ensures that the functor $G^\ctau$ commutes with duality (see \cref{duality} below).
We have that $$
\mathscr{P}_{(W,P)}^\ctau:=\{\la\in \mptn  \mid \ctau \in \Rem(\la)\}
\leftrightarrow
\mptntau
$$
and for $\la \in\mathscr{P}_{(W,P)}^\ctau$, 
  we write $\la{\downarrow}_\ctau$ for the image on the righthand-side (so that $\TRNC_\ctau(\la{\downarrow}_\ctau)=\la$).
We say that $\la{\downarrow}_\ctau$ is the {\sf contraction} of $\la$ at $\ctau$.  
In what follows, we will write ${\sf 1}_{\mu}$ instead of ${\sf 1}_{\stt_\mu}$ to simplify notations.

\begin{thm}\label{exactnessisimportnat}
The functor $G^\ctau$ is exact. 
\end{thm}

\begin{proof}
We need to show that $e_\ctau h_{(W,P)}$ is projective as both a right 
 $h_{(W,P)}$-module and as a left
  $h_{(W,P)^\ctau}$-module.  
  As a right  $h_{(W,P)}$-module, $e_\ctau    h_{(W,P)}$ is a direct summand of 
  $h_{(W,P)}$ (as $e_\ctau$ is an idempotent) and so it is clearly projective.  
It remains to show that $e_\ctau h_{(W,P)}$ is projective as a left
 $h_{(W,P)^\ctau}$-module.  We can decompose this module as follows
 $$
e_\ctau  h_{(W,P)} = \oplus_\mu e_\ctau h_{(W,P)}{\sf 1}_\mu.
 $$
 We will show that each of these summands is projective as a left
  $h_{(W,P)^\ctau}$-module.  
For the remainder of the proof, all statements   concerning modules or homomorphisms will be
taken  implicitly  to be  of left
  $h_{(W,P)^\ctau}$-modules.  
In all of the following cases, we will use the fact that $c^\la_{\SSTS\SSTT}\in e_\ctau  h_{(W,P)} $ implies $\SSTS \in \SStd(\la,\stt_\nu)$ such that $\ctau \in \Rem(\nu)$.  This, in turn, implies that  
 $\ctau \in \Rem(\la )$ or in $\Add(\la)$.

\smallskip\noindent{\bf Case 1}.  
We first assume that $\ctau \in \Rem(\mu)$.  We claim that in this case
$$
e_\ctau h_{(W,P)}{\sf 1}_\mu \cong 
h_{(W,P)^\ctau}{\sf 1}_{\mu{\downarrow}_\ctau}
\oplus 
h_{(W,P)^\ctau}{\sf 1}_{\mu{\downarrow}_\ctau}\langle 2 \rangle .
$$
The module $e_\ctau h_{(W,P)}{\sf 1}_\mu$ has a basis
$$B=\{c^\la_{\SSTS\SSTT} \mid \SSTS \in \Path(\la,\stt_\nu), \SSTT \in \Path(\la,\stt_\mu), 
\text{ with }\la \in \mptn \text{ and }\nu\in \mathscr{P}^\ctau_{(W,P)}\}$$
which decomposes as a disjoint union 
 $\{c^\la_{\SSTS\SSTT}\in B \mid \ctau \in \Rem(\la)\}
\sqcup 
\{c^\la_{\SSTS\SSTT}\in B \mid \ctau \in \Add(\la)\}$. 
Now we have 
$$\langle c^\la_{\SSTS\SSTT}\in B \mid \ctau \in \Rem(\la)\rangle = h_{(W,P)}^\ctau {\sf 1}_\mu \cong h_{(W,P)^\ctau}{\sf 1}_{\mu{\downarrow}_\ctau}.$$
Now consider the quotient  $e_\ctau h_{(W,P)}{\sf 1}_\mu / h_{(W,P)}^\ctau {\sf 1}_\mu$. It has a basis given by the elements $c^\la_{\SSTS\SSTT}+ h_{(W,P)}^\ctau {\sf 1}_\mu$ with  $\ctau \in \Add(\la)$.
These satisfy $\SSTS=X_\ctau^-(\SSTS')$ and $\SSTT = X^-_\ctau(\SSTT')$ for a (possibly different) choice of $X = A$ or $R$ for each one. If we take $\SSTU = X^+_\ctau(\SSTS')$ and $\SSTV = X^+_\ctau(\SSTT')$ then we can write
$$c^\la_{ \SSTS \SSTT } = c_\SSTS^* c_\SSTT = c_\SSTU^* ({\sf 1}_\lambda \otimes \gap(\ctau))c_\SSTV$$
If $\SSTT = A_\ctau^-(\SSTT')$ and so $\SSTV = A_\ctau^+(\SSTT')$ then it becomes
$$c^\la_{ \SSTS \SSTT } = c^{\lambda\ctau}_{\SSTU\SSTV}({\sf 1}_{\mu - \ctau} \otimes \gap(\ctau)).$$
If $\SSTT = R_\ctau^-(\SSTT')$ and so $\SSTV = R_\ctau^+(\SSTT')$ then we can factorise $c^\la_{ \SSTS \SSTT }$ as
\begin{equation}\label{removable}c^\la_{ \SSTS \SSTT } = c_\SSTU^* ({\sf 1}_\lambda \otimes (\spot^\ctau_{\emptyset} {\sf cap}^\emptyset_{\ctau \ctau}))(c_{\SSTT'}\otimes {\sf 1}_{\ctau}).\end{equation}
Now applying ${\sf 1}_\ctau\otimes \spot_\ctau^\emptyset$ to equation \cref{shorter} we get
$${\sf 1}_\ctau \otimes \spot_\ctau^\emptyset = \spot_\ctau^\emptyset \otimes {\sf 1}_\ctau + \spot^\ctau_{\emptyset} {\sf cap}^\emptyset_{\ctau \ctau} - {\sf bar}(\ctau) \otimes \fork_{\ctau \ctau}^\ctau.$$
Thus we can rewrite \cref{removable} as 
$$
c^\la_{ \SSTS \SSTT } = c_\SSTU^* ({\sf 1}_\lambda \otimes ({\sf 1}_\ctau \otimes \spot_\ctau^\emptyset ))(c_{\SSTT'}\otimes {\sf 1}_{\ctau})\\ - c_\SSTU^* ({\sf 1}_\lambda \otimes (\spot_\ctau^\emptyset \otimes {\sf 1}_\ctau))(c_{\SSTT'}\otimes {\sf 1}_{\ctau})\\ + c_\SSTU^* ({\sf 1}_\lambda \otimes ({\sf bar}(\ctau) \otimes \fork_{\ctau \ctau}^\ctau))(c_{\SSTT'}\otimes {\sf 1}_{\ctau})
$$
Now note that the last two terms belong to $h_{(W,P)}^\ctau {\sf 1}_\mu$ and the first one can be rewritten as 
$$c^{\lambda\ctau}_{\SSTU\SSTV}({\sf 1}_{\mu - \ctau} \otimes {\sf gap}(\ctau)).$$
where $c^{\lambda\ctau}_{\SSTU\SSTV} \in h_{(W,P)}^\ctau {\sf 1}_\mu$. This shows that the quotient is isomorphic to $h_{(W,P)}^\ctau {\sf 1}_\mu ({\sf 1}_{\mu - \ctau} \otimes {\sf gap}(\ctau))$. As it is projective, it splits and we have
$$e_\ctau h_{(W,P)}{\sf 1}_\mu \cong 
h_{(W,P)}^\ctau {\sf 1}_{\mu}
\oplus 
h_{(W,P)}^\ctau {\sf 1}_{\mu} ({\sf 1}_{\mu - \ctau} \otimes {\sf gap}(\ctau)),
$$
thus proving the claim.

\smallskip\noindent{\bf Case 2}.  
We now assume that $\ctau \in \Add(\mu)$.  
  We claim that in this case
$$e_\ctau h_{(W,P)}{\sf 1}_\mu \cong h_{(W,P)^\ctau}{\sf 1}_{\mu\ctau{\downarrow}_\ctau}\langle 1 \rangle.$$
 To see this, we will show that 
$$e_\ctau h_{(W,P)} {\sf 1}_\mu \cong h_{(W,P)}^\ctau {\sf 1}_{\mu\ctau}
({\sf 1}_\mu \otimes {\sf spot}^\ctau_\emptyset).$$
Indeed for any $c^\la_{\SSTS\SSTT}\in e_\ctau h_{(W,P)} {\sf 1}_\mu$ we have that  
$\SSTS = X_\ctau^{\pm}(\SSTS')$.
If $\SSTS = X_\ctau^{-}(\SSTS')$ then we must have $\ctau \in \Add(\lambda)$ and we define $\SSTU = X_\ctau^+(\SSTS')$ and $\SSTV = A_\ctau^+(\SSTT)$. If $\SSTS = X_\ctau^{+}(\SSTS')$ then we must have $\ctau \in \Rem(\lambda)$ and we define $\SSTU = \SSTS$ and $\SSTV = R_\ctau^+(\SSTT)$. Then in both cases we can write
$$c^\la_{\SSTS\SSTT} = c_{\SSTU\SSTV}({\sf 1}_\mu \otimes \spot_\emptyset^\ctau).$$
Note that $c_{\SSTU\SSTV} \in h_{(W,P)}^\ctau {\sf 1}_{\mu\ctau}$ so we're done.

\smallskip\noindent{\bf Case 3}.  It remains to consider the case that $\ctau\not\in \Rem(\mu)$ or $\Add(\mu)$.  
We now consider the case that $\ctau \not \in \Rem(\mu )$ or $\Add(\mu)$,
 but there exists $\csigma\in \Rem(\mu)$ with $m(\csigma,\ctau)=3$. Note that we can assume that $\tau\in \Rem(\mu-\csigma)$ as otherwise we would be in Case 2.
   This will serve as the base case for the inductive step in Case 4.  
We claim that in this case
$$
e_\ctau h_{(W,P)}{\sf 1}_\mu \cong 
h_{(W,P)^\ctau}{\sf 1}_{(\mu-\csigma){\downarrow}_\ctau}\langle 1 \rangle.
 $$
To see this, we will show that 
$$
e_\ctau h_{(W,P)}{\sf 1}_\mu =  h_{(W,P)}^\ctau{\sf 1}_{\mu-\csigma}({\sf 1}_{\mu-\csigma}\otimes {\sf spot}_\csigma^\emptyset)$$
Our assumptions that $\csigma\in\Rem(\mu)$ and $\ctau \in \Rem(\nu)$ imply that 
there are two cases to consider: $\csigma\in \Rem(\la)$ and $\ctau\in \Add(\la)$ versus 
$\csigma\in \Add(\la)$ and $\ctau\in \Rem(\la)$.  
In the latter case, we have that $\SSTT=A_\csigma^-(\SSTT')$ and so 
$$c_{\SSTS\SSTT}= 
c_{\SSTS}^\ast (c_{\SSTT'}\otimes {\sf spot}_\csigma^\emptyset) = c_{\SSTS \SSTT'}\otimes \spot_\csigma^\emptyset$$
with $c_{\SSTS \SSTT'}\in h_{(W,P)}^{\ctau}{\sf 1}_{\mu-\csigma}$ as required.
In the first case, we have 
$\SSTT=
R_\csigma^+
X_\ctau^-(\SSTT')$ and $\SSTS= X^-_\ctau(\SSTS')$. Setting $\SSTU = X^+_\ctau (\SSTS')$ and $\SSTV = X^+_\ctau(\SSTT')$, we can write
$$c_{\SSTS\SSTT} = c_\SSTU^\ast ({\sf 1}_{\lambda - \csigma} \otimes \trid_{\csigma \ctau \csigma}^\csigma \otimes \spot_{\emptyset}^\ctau) (c_\SSTV \otimes {\sf 1}_\csigma) = - c_{\SSTU\SSTV} \otimes \spot_\csigma^\emptyset$$
where the last equality follows by applying ${\sf 1}_{\csigma\ctau} \otimes \spot_\csigma^\emptyset$ to the $\csigma \ctau$-nullbraid relations.
Again we have that $c_{\SSTU\SSTV}\in h_{(W,P)}^{\ctau}{\sf 1}_{\mu-\csigma}$ so we are done.

\smallskip\noindent{\bf Case 4}. 
If $\mu$ is not as in cases 1 to 3, then we must have $\csigma \in \Rem(\mu)$ with $\csigma$ and $\ctau$ commuting.  
We will show that $e_\ctau h_{W,P)}{\sf 1}_\mu$ is either 0 or projective-indecomposable as a   left 
$h_{(W,P)^\ctau}$-module.
We proceed by induction on the rank of $W$.  Note that as $\csigma$ and $\ctau$ commute, $\csigma$ labels a node in the Dynkin diagram for $(W,P)^\ctau$ and so it makes sense to consider $e_\csigma^{(W,P)^\ctau}\in h_{(W,P)^\ctau}$ and $(W,P)^{\ctau \csigma}$. We claim that 
\begin{align}\label{asdfjhkkhjsdfjkhsdfjkhdsfghjkdsfhjkdsfgkjhvdjhvxchjk}
e_\ctau 
h_{(W,P)}
{\sf 1}_\mu
\cong 
h_{(W,P)^\ctau}
e ^{(W,P)^\ctau}
_\csigma
\otimes _{h_{(W,P)^{\csigma\ctau}}}
e ^{(W,P)^\csigma}
_\ctau
h_{(W,P)^\csigma}
{\sf 1}_{\mu{\downarrow}_\csigma}
\end{align}
as a left  $h_{(W,P)^\ctau}$-module.  Note that any basis element in 
$e_\ctau    h_{(W,P)} {\sf 1}_\mu$ has the form 
$c^\la _{\SSTS\SSTT}$ for $\SSTS \in \SStd(\la,\stt_\nu)$,
$\SSTT \in \SStd(\la,\stt_\mu)$   with $\ctau \in \Rem(\nu)$ and 
$\csigma \in \Rem(\mu)$.  So either 
$\csigma \in \Rem(\la)$ or $\csigma \in \Add(\la)$ 
and similarly either 
$\ctau \in \Rem(\la)$ or $\ctau \in \Add(\la)$. 
To prove the claim, it is enough to show that any such $c^\la_{\SSTS\SSTT}$ can be written as a product 
$$c^\la_{\SSTS\SSTT}=c^\alpha_{\SSTP\SSTQ} c^\beta_{\SSTU\SSTV}   $$
 where 
 $c^\alpha_{\SSTP\SSTQ} \in h_{(W,P)}^{\ctau} $
 and 
  $c^\beta_{\SSTU\SSTV}\in h_{(W,P)}^{\csigma}$. 
  There are four distinct cases to consider. 
   If $\csigma,\ctau \in \Rem(\la)$ then $\SSTS = X^+_\ctau(\SSTS'), \SSTT = X^+_\csigma(\SSTT')$ and we pick $\SSTP = \SSTS$, $\SSTQ = \SSTU = \stt_\lambda$ and $\SSTV = \SSTT$. 
   If $\csigma \in \Rem(\lambda)$ and $\ctau\in \Add(\lambda)$ then $\SSTS = X^-_\ctau(S'), \SSTT = X^+_\csigma(\SSTT')$ and we pick $\SSTP = X^+_\ctau(\SSTS')$, $\SSTQ = A^+_\ctau(\stt_\lambda)$, $\SSTU = A^-_\ctau(\stt_\lambda)$ and $\SSTV = \SSTT$. 
  If $\csigma \in \Add(\lambda)$ and $\ctau\in \Rem(\lambda)$ then $\SSTS = X^+_\ctau(S'), \SSTT = X^-_\csigma(\SSTT')$ and we pick $\SSTP = \SSTS$, $\SSTQ = A^-_\csigma(\stt_\lambda)$, $\SSTU = A^+_\csigma(\stt_\lambda)$ and $\SSTV = X^+_\csigma(\SSTT')$. 
If $\csigma,\ctau \in \Add(\lambda)$  then $\SSTS = X^-_\ctau(S'), \SSTT = X^-_\csigma(\SSTT')$ and we pick $\SSTP = X^+_\ctau(\SSTS')$, $\SSTQ = A^+_\ctau A^-_\csigma(\stt_\lambda)$, $\SSTU = A^-_\ctau A^+_\csigma(\stt_\lambda)$ and $\SSTV = X^+_\csigma(\SSTT')$. 
Hence we have proven \cref{asdfjhkkhjsdfjkhsdfjkhdsfghjkdsfhjkdsfgkjhvdjhvxchjk}.

By induction, $e_\ctau^{(W,P)^\csigma}h_{(W,P)^\csigma}{\sf 1}_{\mu\downarrow_\csigma}$ is either 0, or it  is a projective
indecomposable $h_{(W,P)^{\csigma\ctau}}$-module, say $h_{(W,P)^{\csigma\ctau}}{\sf 1}_\eta$.
Substituting into \cref{asdfjhkkhjsdfjkhsdfjkhdsfghjkdsfhjkdsfgkjhvdjhvxchjk}, we obtain that $e_\ctau h_{(W,P)}{\sf1}_\mu$ is either  0, or 
\begin{align*}
e_\ctau h_{(W,P)}{\sf 1}_\mu
 &
 \cong h_{(W,P)^\ctau}e_\csigma^{(W,P)^\ctau}\otimes _{h_{(W,P)^{\csigma\ctau}}}
 h_{(W,P)^{\csigma\ctau}}{\sf 1}_{\eta}
 \\
 &
 \cong h_{(W,P)^\ctau}{\sf 1}_{\TRNC_\csigma(\eta)}
\end{align*}
which is projective indecomposable.
 \end{proof}

\begin{lem}
There is a graded $(h_{(W,P)^\ctau},h_{(W,P)^\ctau})$-bimodule homomorphism
$$\psi : e_\ctau h_{(W,P)}e_\ctau 
\to
h_{(W,P)^\ctau}\langle 2\rangle .
$$
\end{lem}
\begin{proof}
The module $e_\ctau  h_{(W,P)}e_\ctau$ has basis 
given by 
$$B=\{c^\la_{\SSTS\SSTT} \mid \SSTS \in \Path(\la,\stt_\mu), \SSTT \in \Path(\la,\stt_\nu), 
\text{ with }\la \in \mptn \text{ and }\mu,\nu\in \mathscr{P}^\ctau_{(W,P)}\}$$
which decomposes as a disjoint union 
 $\{c^\la_{\SSTS\SSTT}\in B \mid \ctau \in \Rem(\la)\}
\sqcup 
\{c^\la_{\SSTS\SSTT}\in B \mid \ctau \in \Add(\la)\}$. 
By \cref{veccer},
we have a  $(h_{(W,P)^\ctau},h_{(W,P)^\ctau})$-bimodule isomorphism  
$$h_{(W,P)^\ctau}  \cong h_{(W,P)}^\ctau
=
\langle c^\la_{\SSTS\SSTT}\in B \mid \ctau \in \Rem(\la)\rangle \subseteq e_\ctau  h_{(W,P)}e_\ctau.$$
  Following the proof of case 1 of \cref{exactnessisimportnat}, we see that 
  $$
  e_\ctau  h_{(W,P)}e_\ctau
  / \Bbbk \{c^\la_{\SSTS\SSTT}\in B \mid \ctau \in \Rem(\la)\}
      \cong 
  h_{(W,P)}^\ctau  (\textstyle
      \sum_{\mu \in \mathscr{P}^\ctau_{(W,P)}}
      {\sf 1}_{\stt_{\mu-\ctau}}\otimes {\sf gap}(\ctau))
  $$
  as left $h_{(W,P)^\ctau}$-modules and similarly, flipping diagrams across the horizontal axis we get that 
    $$
  e_\ctau  h_{(W,P)}e_\ctau
  / \Bbbk \{c^\la_{\SSTS\SSTT}\in B \mid \ctau \in \Rem(\la)\}
      \cong 
       (\textstyle
      \sum_{\mu \in \mathscr{P}^\ctau_{(W,P)}}
      {\sf 1}_{\stt_{\mu-\ctau}}\otimes {\sf gap}(\ctau))
 h_{(W,P)}^\ctau 
  $$
as right $h_{(W,P)^\ctau}$-modules.
  This shows that 
    $$
  e_\ctau  h_{(W,P)}e_\ctau
  / \Bbbk \{c^\la_{\SSTS\SSTT}\in B \mid \ctau \in \Rem(\la)\}
      \cong  h_{(W,P)}^\ctau \langle 2 \rangle \cong 
h_{(W,P)^\ctau}\langle 2\rangle$$
as $(h_{(W,P)^\ctau},h_{(W,P)^\ctau})$-bimodules as required.  
\end{proof}

Let $M$ be a right $h_{(W,P)}$-module.  Define the 
right $h_{(W,P)}$-module $M^*$ by 
 $M^*= \Hom_\Bbbk(M,\Bbbk)$ as a vector space and for $f\in M^*$, $a \in h_{(W,P)}$ we define 
 $fa \in M^*$ by 
 $(fa)(m)=f(ma^*)$ where $a^*$ is the dual element in $
 h_{(W,P)}$ (given by flipping a diagram across the horizontal axis).

\begin{thm}\label{duality}
For  $M$  an $h_{(W,P)}$-module we have that 
$G^\ctau(M^\ast )\cong (G^\ctau(M))^\ast $.
\end{thm}

\begin{proof}
We have that 
$$G^\ctau(M^*)=M^* \otimes _{h_{(W,P) ^\ctau}} e_\ctau h_{(W,P)}\langle -1\rangle 
\quad \quad
 (G^\ctau(M))^*=
\Hom _\Bbbk( M
 \otimes _{h_{(W,P) ^\ctau}} e_\ctau h_{(W,P)}\langle -1\rangle, \Bbbk)$$
We define $\vartheta  : G^\ctau (M^*)\to (G^\ctau(M))^*$ by setting 
$f\otimes a \mapsto \vartheta _{f\otimes a}$
for $f\in M^*$ and $a \in e_\ctau h_{(W,P)}\langle -1 \rangle$ 
where 
$$ \vartheta _{f\otimes a}(m\otimes b)=f(m \psi(ba^*))
$$  
for $m\in M$ and $b\in e_\ctau h_{(W,P)}\langle -1\rangle$. 
Note that this makes sense because 
$ba^\ast \in e_\ctau h_{(W,P)}e_\ctau \langle -2 \rangle$ and so 
$\psi(ba^*) \in h_{(W,P)^\ctau}$.  
Also 
$\vartheta$ is well-defined as $\psi$ is a bimodule homomorphism.  

We now show that $\vartheta$ is a $h_{(W,P)}$-homomorphism.  On one hand, we have 
$$\vartheta_{(f\otimes a)x}(m\otimes b)
=
\vartheta_{ f\otimes ax}(m\otimes b)
= f(m \psi(bx^*a^*)).
$$
On the other hand, we have  
$$
(\vartheta_{(f\otimes a)}x)(m\otimes b)
=
 \vartheta_{ f\otimes a } ( (m\otimes b)x^*)
 =
  \vartheta_{ f\otimes a } (  m\otimes b x^*)
  = f(m \psi(bx^*a^*))
$$
as required.  
We now show that $\vartheta$ is a  vector space isomorphism. 
It is enough to check that 
 $\vartheta: 
G^\ctau(M^\ast ) {\sf 1}_{\stt_\mu} \to  (G^\ctau(M))^\ast  {\sf 1}_{\stt_\mu}$
is a vector space isomorphism for each  $\mu \in \mptn$.  We have 
\begin{align*}
G^\ctau(M^\ast ) {\sf 1}_{\stt_\mu} 
  &=M^* \otimes _{h_{(W,P) ^\ctau}} e_\ctau h_{(W,P)}\langle -1\rangle {\sf 1}_{\stt_\mu}
  \\ 
  (G^\ctau(M))^* {\sf 1}_{\stt_\mu }&=
\Hom _\Bbbk( M
 \otimes _{h_{(W,P) ^\ctau}} e_\ctau h_{(W,P)}\langle -1\rangle, \Bbbk) {\sf 1}_{\stt_\mu}
 \\
&=\Hom _\Bbbk( M
 \otimes _{h_{(W,P) ^\ctau}} e_\ctau h_{(W,P)}\langle -1\rangle  {\sf 1}_{\stt_\mu}, \Bbbk) 
 \\
 &=
 (M
 \otimes _{h_{(W,P) ^\ctau}} e_\ctau h_{(W,P)}\langle -1\rangle  {\sf 1}_{\stt_\mu})^\ast 
 \end{align*}
We have seen in the proof of \cref{exactnessisimportnat} that 
$e_\ctau h_{(W,P)}   {\sf 1}_{\stt_\mu}$ is either zero or  isomorphic to (possibly two shifted copies of)
$h_{(W,P)^\ctau}{\sf 1}_{\stt_\nu}$ for some 
$\nu \in \mathscr{P}_{(W,P)^\ctau}$.  
So it is enough  to note that 
$$M^\ast {\sf 1}_{\stt_\nu}
=
M^\ast \otimes _{h_{(W,P)^\ctau}}
h_{(W,P)^\ctau}{\sf 1}_{\stt_\nu}\cong
(M\otimes _{h_{(W,P)^\ctau}}
h_{(W,P)^\ctau}{\sf 1}_{\stt_\nu})^*
=(M  {\sf 1}_{\stt_\nu})^\ast
$$
as required. 
\end{proof}

Using our induction functor, we will  
 relate (sequences of) $h_{{(W,P)^\ctau }} $-modules labelled by $\la \in \mptntau$ with  
 (sequences of) $h_{{(W,P) }} $-modules labelled by 
 $$\la^+ := \TRNC_\ctau(\la) \quad \text{  and }\quad 
   \la^- := \TRNC_\ctau(\la)-\ctau.$$   
  We note that this is the typical Kazhdan--Lusztig ``doubling-up" that we expect.

\begin{prop}\label{projers}
For each $\la \in \mptntau$,  we have 
 $G^\ctau(P (\la))= P(\la^+)\langle -1\rangle $. 
\end{prop}
\begin{proof}
Recall that    $(W,P)$ is a simply laced Hermitian symmetric pair.  
By \cref{yabasic}, the  projective indecomposable modules are   $P(\la)={\sf 1}_{\stt_\la}h_{{(W,P) }}$ for $\la \in \mptn$.  
Therefore
\begin{align*}
G^\ctau(P (\la))		&=
{\sf 1}_{\stt_\la} h_{{(W,P)^\ctau}} 
\otimes _{h_{{(W,P)^\ctau }} } e_\ctau h_{{(W,P) }} \langle -1 \rangle
 =  \on_{\TRNC_\ctau(\stt_\la)}h_{{(W,P) }} \langle -1 \rangle
 =P(\la^+)\langle -1 \rangle
\end{align*}as required. 
\end{proof}

\begin{prop}\label{ind-delta}
For each $\mu \in \mptntau$,  we have 
 $$0 \to \Delta(\mu^-) \to G^\ctau(\Delta (\mu))\to \Delta(\mu^+)\langle -1 \rangle \to 0 $$ 
\end{prop}
\begin{proof}
We have an exact sequence
$$0\to   h^{<\mu}_{(W,P)} \to P(\mu) \to \Delta(\mu)\to 0$$
where 
$  h^{<\mu}_{(W,P)} =  \sum_{\nu<\mu} {\sf 1}_\mu h_{(W,P)}{\sf 1}_\nu h_{(W,P)}$.  
The modules 
 $P(\mu)$  and $  h^{<\mu}_{(W,P)} $  
 have bases $$\{{\sf 1}_\mu c^\nu_{\SSTS\SSTT}\mid \SSTS,\SSTT \in \Path(\nu,-), \nu\leq	\mu\}
\qquad\{{\sf 1}_\mu c^\nu_{\SSTS\SSTT}\mid \SSTS,\SSTT \in \Path(\nu,-), \nu<\mu\}$$
respectively. Since $G^\ctau$ is exact, we obtain an exact sequence
$$
0\to G^\ctau(  h^{<\mu}_{(W,P)} ) \to G^\ctau(P(\mu)) \to G^\ctau(\Delta(\mu))	\to 0
$$
  where $G^\ctau(P(\mu))\cong P(\mu^+)\langle -1\rangle$.  
Therefore $G^\ctau(\Delta(\mu))= P(\mu^+)\langle -1 \rangle /G^\ctau(	h^{<\mu}_{(W,P)}	)$ has basis given by
$$\{ c_\SSTU\langle -1\rangle, c_\SSTV \otimes {\sf spot}^\ctau_\emptyset \langle -1 \rangle \mid 
\SSTU \in \SStd(\mu^+,-),
\SSTV \in \SStd(\mu^-,-)\}.$$
It is clear that, as a right $h_{(W,P)}$-module
$$\{  c_\SSTV \otimes {\sf spot}^\ctau_\emptyset \langle -1 \rangle \mid 
\SSTV \in \SStd(\mu^-, -)\} $$
is a submodule of $G^\ctau(\Delta(\mu))$ isomorphic to $\Delta(\mu^-)$ and the quotient is isomorphic to $\Delta(\mu^+)\langle -1 \rangle$. 
\end{proof}

\subsection{Koszulity}  
We are now able to use the ideas of the previous section in order to prove that $h_{(W,P)}$ is standard Koszul.  
First, we continue to assume that $(W,P)$ is simply laced.

\begin{defn}
For  $\la,\mu \in \mptn$, we define polynomials 
 $p_{\la,\mu}(q)$ inductively on the rank and Bruhat order 
 as follows.  
We set $p_{\la,\la}(q)=1$ and for $\la \not \subseteq \mu$ we set 
$p_{\la,\mu}(q)=0$.  
 If $\la \subset \mu$, pick $\ctau$ such that $\ctau \in \Rem(\la)$.  
 We set 
 $$
p_{\la,\mu}(q)=
 \begin{cases}
 p_{\la{\downarrow}_\ctau,\mu{\downarrow}_\ctau}(q)+ 
 q\times p_{\la-\ctau,\mu}(q) &\text{if }\ctau \in \Rem(\mu);
 \\
 q\times p_{\la-\ctau,\mu}(q) &\text{if }\ctau \not\in \Rem(\mu) .
 \end{cases}
 $$
 We write $p_{\la,\mu}(q)= \sum_{n\geq 0} p_{\la,\mu}^{(n)} \grade^n$.  
 \end{defn}
 
\begin{thm}\label{resolutionsss}
For $\la \in \mptn$, we have an exact sequence
$$
\dots \to P_2(\la) \to P_1(\la) \to P_0(\la) \to \Delta(\la)	\to 0
$$
where $P_0(\lambda)=P(\lambda)$ and for $n\geq 1$ we have $P_n(\la)=\oplus_{\mu\in \mathcal{P}_{(W,P)}}  p_{\la,\mu}^{(n)} P(\mu) \langle n \rangle$.  
\end{thm} 
 \begin{proof}
  We proceed by induction on the rank of $W$ and the Bruhat order on $\mathscr{P}_{(W,P)}$.  
 If $\la =\varnothing$ is the  minimal element in the Bruhat order, then $\Delta(\varnothing)= P(\varnothing)$ and we are done.  
 Assume $  \varnothing\neq \la\in \mptn $, then there exists some $\ctau \in \Rem(\la)$ and we have that 
 $\la-\ctau \in \mptn$ and $\la {\downarrow}_\ctau \in \mptntau$.  By induction we have  exact sequences,
 \begin{align*}
 \dots \to P_2(\la-\ctau ) \to P_1(\la-\ctau ) \to P_0(\la-\ctau ) \to \Delta(\la-\ctau )	\to 0&
\\
 \dots \to P_2(\la {\downarrow}_\ctau) \to P_1(\la {\downarrow}_\ctau) \to P_0(\la {\downarrow}_\ctau) \to \Delta(\la {\downarrow}_\ctau)	\to 0&
 \end{align*}
 in $ h_{(W,P)  }{\rm -mod}$ and  $h_{(W,P)^\ctau }{\rm -mod}$ respectively. 
 Applying the induction functor $G^\ctau$ to the latter sequence, and lifting the injective homomorphism from \cref{ind-delta} we obtain a commutative diagram with exact rows. 
$$
$$
Taking the total complex of this double complex (that is, summing over the dotted lines)
and then taking the quotient by the complex
$$\ldots \to 0 \to 0 \to \Delta(\la-\ctau)  \to \Delta(\la-\ctau)   \to 0$$ we obtain 
$$
\cdots
  \xrightarrow{ \ \   }G^\ctau(P_2(\la{\downarrow}_\ctau))\oplus P_1(\la-\ctau) 
  \xrightarrow{ \ \   } G^\ctau(P_1(\la{\downarrow}_\ctau))\oplus P_0(\la-\ctau) 
 \xrightarrow{ \ \   } G^\ctau(P_0(\la{\downarrow}_\ctau)) %\oplus \Delta(\la-\ctau)\langle 1\rangle 
  \xrightarrow{ \ \   } \Delta(\la)\langle -1\rangle  \to 0.
$$ 
%and by \cref{projers}, this is equal to
%$$
%\cdots
%  \xrightarrow{ \ \   } P_2(\la ) 		\oplus P_1(\la-\ctau)\langle 1\rangle
%  \xrightarrow{ \ \   }  P_1(\la )  \oplus P_0(\la-\ctau)\langle 1\rangle
% \xrightarrow{ \ \   } P_0(\la )  %\oplus \Delta(\la-\ctau)\langle 1\rangle 
%  \xrightarrow{ \ \   } \Delta(\la) \to 0
%$$
We have $G^\ctau(P_0(\la{\downarrow}_\ctau)) = P(\lambda)\langle -1\rangle$.  By induction, for $n\geq 1$ we have that 
\begin{align*}
&  G^\ctau (P_n(\la{\downarrow}_\ctau))\oplus P_{n-1}(\la-\ctau)
\\
&=    \bigoplus _{\mu{\downarrow}_\ctau \in \mptntau}	\!\!\!\!	p^{(n)}_{\la{\downarrow}_\ctau ,\mu{\downarrow}_\ctau }G^\ctau(P(\mu{\downarrow}_\ctau)\langle n \rangle  )
%\oplus
 \bigoplus _{\mu\in \mptn}	p^{(n-1)}_{\la - \ctau, \mu } P (\mu)\langle n-1 \rangle
\\
&=   \bigoplus _{\mu{\downarrow}_\ctau \in \mptntau}	\!\!\!\!	p^{(n)}_{\la{\downarrow}_\ctau ,\mu{\downarrow}_\ctau }P(\mu )\langle n -1 \rangle  
%\oplus
 \bigoplus _{\mu\in \mptn}	p^{(n-1)}_{\la - \ctau, \mu } P (\mu)\langle n -1 \rangle  
\\
&=  \bigoplus _{\mu  \in \mathscr{P}_{(W,P)}^\ctau}	\!\!\!\!	
(p^{(n)}_{\la{\downarrow}_\ctau ,\mu{\downarrow}_\ctau }
 +p^{(n-1)}_{\la - \ctau, \mu } )P (\mu)\langle n-1 \rangle   \bigoplus _{\mu\in \mptn \setminus \mathscr{P}_{(W,P)}^\ctau}	p^{(n-1)}_{\la - \ctau, \mu } P (\mu)\langle n -1 \rangle  \\
 &= \bigoplus _{\mu  \in \mptn  }	\!\!\!\!	
 p^{(n )}_{\la  ,\mu  }
  P (\mu)\langle n-1 \rangle  
\end{align*}
where the last equality follows by the definition of $p_{\la,\mu}(q)$.
Thus we obtain an exact sequence
$$ \dots \to P_2(\la)\langle -1\rangle  \to P_1(\la ) \langle -1 \rangle \to P_0(\la )\langle -1 \rangle  \to \Delta(\la) \langle -1\rangle	\to 0.$$
Applying a degree shift $\langle 1 \rangle$ gives the required linear projective resolution for $\Delta(\la)$.

\end{proof}

 \begin{cor}\label{kosvul} Let $(W,P)$ be any Hermitian symmetric pair.
The algebra  $h_{(W,P)}$ is standard  Koszul.   
 \end{cor}
 
\begin{proof}
Using \cref{C--A}, it is enough to consider the simply laced types. The algebra  $h_{(W,P)}$ is  graded quasi-hereditary algebra with (right) standard modules $\Delta(\la)$; the   linear  projective resolutions of these modules  are given in  \cref{resolutionsss}. 
Twisting with the anti-automorphism $\ast$ we also get that its left standard modules have linear  projective resolutions. 
Therefore $h_{(W,P)}$ is Koszul by \cite[Theorem 1]{MR1960515}. \end{proof}

   \begin{cor}\label{coincide} Let $(W,P)$ be any Hermitian symmetric pair.
  For $\mu \in \mptn$, we 
  have that the radical filtration of $\Delta(\mu)$ coincides with the grading filtration 
 $$\Delta(\mu)= \Delta_{\geq 0}(\mu)\supset 
\Delta_{\geq 1}(\mu)\supset \Delta_{\geq 2}(\mu)\supset \dots$$ 
where we define 
 $\Delta_{\geq k}(\mu)=\{ c_\SSTS \mid \SSTS\in  \Path(\la,\stt_\mu), \deg(\SSTS)\geq k\}.$ 
    \end{cor}

  \begin{proof}
We have  that $h_{(W,P)}$ is standard  Koszul  by \cref{kosvul}. 
 That the radical and grading series coincide  follows from \cite[Proposition 2.4.1]{bgs96}.  
  \end{proof}

\begin{Acknowledgements*}
 The first   and third   authors are grateful for funding from EPSRC  grant EP/V00090X/1 and the Royal Commission for the Exhibition of 1851, respectively.  The authors are very thankful for the referee's careful reading of an earlier version of this manuscript. 
%The first author would like to thank 
%Volodymyr Mazorchuk for helpful conversations regarding the cohomology of parabolic Verma  modules.  

\end{Acknowledgements*}

         \bibliographystyle{amsalpha}   
\bibliography{master}

 \end{document}